\documentclass[a4paper,10pt]{article}

 \usepackage{amssymb, amsmath, amscd, amsthm}
\usepackage{graphicx}
\usepackage{mathptmx}
\usepackage{psfrag}
\usepackage[all]{xy}
\usepackage{rotating} 
 \newtheorem{teo}{\bf Theorem}[section]
 \newtheorem{thm}[teo]{\bf Theorem}
 \newtheorem{lem}[teo]{\bf Lemma}
 \newtheorem{cor}[teo]{\bf Corollary}
 \newtheorem{prop}[teo]{\bf Proposition}

 \newtheorem{defi}[teo]{\bf Definition}
 
 \newtheorem{rem}[teo]{\bf Remark}

\newcommand{\ba}{\begin{array}}
\newcommand{\ea}{\end{array}}

\newcommand{\N}{\mathbb{N}}
\newcommand{\R}{\mathbb{R}}

\newcommand{\U}{\mathcal{U}}
\newcommand{\V}{\mathcal{V}}

\newcommand{\I}{\mbox{\rm{im\,}}\!}

\newcommand{\p}{\varphi}

\newcommand{\szf}{\ell_{(X,\p)}}
\newcommand{\sza}{\ell_{(A,\p|_A)}}
\newcommand{\szb}{\ell_{(B,\p|_B)}}
\newcommand{\szab}{\ell_{(A\cap B,\p|_{A\cap B})}}

\newcommand{\eps}{\epsilon}
\newcommand{\Cech}{{\v{C}ech }}
\newcommand{\rank}{\mbox{\rm{rank}}}

\begin{document}

\title{\Cech homology for shape recognition in the presence of occlusions}
\author{Barbara Di Fabio$^{1}$ \quad Claudia
Landi$^{2}${\footnote{Corresponding author. E-mail Address:
\texttt{clandi@unimore.it}}}\\
\small{$^{1}$Dipartimento di Matematica, Universit\`a di
Bologna,}\\
\small{P.zza di Porta S. Donato 5, I--$40126$ Bologna,
Italia}\\
\small{$^2$Dipartimento di Scienze e Metodi dell'Ingegneria, Universit\`a di Modena e Reggio Emilia,}\\
\small{Via Amendola $2$, Pad. Morselli, I--$42100$ Reggio Emilia,
Italia}}
\maketitle

\begin{abstract}
In Computer Vision the ability to recognize objects in the
presence of occlusions is a necessary requirement for any shape
representation method. In this paper we investigate how the size
function of a shape changes when a portion of the shape is
occluded by another shape. More precisely,  considering a set
$X=A\cup B$ and a measuring function $\p$ on $X$, we establish a
condition so that $\szf=\sza+\szb-\szab$. The main tool we use is
the Mayer-Vietoris sequence of \Cech homology groups. This result
allows us to prove that  size functions are able to detect partial
matching between  shapes by showing a common subset of
cornerpoints.
\end{abstract}

\textbf{Keywords:} Size function, Mayer-Vietoris sequence,
persistent homology, shape occlusion

\textbf{MSC (2000):} 55N05, 68U05

\section{Introduction}
Shape matching and retrieval are key aspects in the design of
search engines based on visual, rather than keyword, information.
Generally speaking, shape matching methods rely on the computation
of a shape description, also called a signature, that effectively
captures some essential features of the object.   The ability to
perform not only global matching, but also partial matching, is
regarded  as one of the most meaningful properties in order to
evaluate the performance of a shape matching method (cf., e.g.,
\cite{VeHa01}).  Basically, the interest in robustness against
partial occlusions is motivated  by the problem of recognizing an
object partially hidden by some other foreground object in the
same image. However, there are also other situations in which
partial matching is useful, such as when dealing with the problem
of identifying  similarities between  different configurations of
articulated objects, or when dealing with unreliable object
segmentation from images. For these reasons, the ability  to
recognize shapes, even when they are partially occluded by another
pattern, has been investigated in the Computer Vision literature
by various authors, with reference to a variety of shape
recognition methods (see, e.g.,
\cite{ChCh05,GhPe05,HoOh03,MoBeMa01,SuSu05,TaVe05}).

Size functions belong to a class of methods for shape description,
characterized by the study of the topological changes in the lower
level sets  of a real valued function defined on the shape to
derive its signature (cf., e.g., \cite{BiFlXX, KaMiMr04}). In this
paper we study the robustness of size functions against partial
occlusions.   Previous works have already assessed the  robustness
of size functions with respect to continuous deformations of the
shape \cite{DAFrLa06}, the conciseness of the descriptor
\cite{FrLa01}, the invariance of the descriptor to transformation
groups \cite{DiFrPa04,VeUr94},  that are further properties
recognized as important for  shape matching methods. Size
functions, like most methods of their class, work on a shape as a
whole. In general, it is argued that global object methods are not
robust against occlusions, whereas methods based on computing
local features may be more suited to this task. Our aim is to show
that size functions are able to  preserve local information, so
that they  can manage uncertainty  due to the presence of occluded
shapes.

We model the presence of occlusions in a shape as follows. The
visible object is a locally connected compact Hausdorff space $X$.
The shape of interest $A$ is occluded by a shape $B$, so that
$X=A\cup B$.  In particular, $A$ and $B$ have the topology induced
from $X$ and are assumed to be locally connected. The shapes of
$X$, $A$, and $B$ are analyzed through the size functions
$\ell_{(X, \p)}$, $\ell_{(A, \varphi_{|A})}$, and $\ell_{(B,
\varphi_{|B})}$, respectively, where $\p:X\rightarrow \R$ is the
continuous function chosen to extract the shape features.

The starting point of this research is the fact that the size
function  $\ell_{(X, \p)}$, evaluated at a point $(u,v)$ of
$\R^2$, with $u<v$, is equal to the rank of the image of the
homomorphism induced by inclusion between the \Cech homology
groups $\check H_0(X_u)$ and  $\check H_0(X_v)$, where $X_u=\{p\in
X: \p(p)\le u\}$ and $X_v=\{p\in X:\p(p)\le v\}$.

Our main result establishes a necessary and sufficient condition
so that the equality \begin{eqnarray}\label{grassmann1} \ell_{(X,
\p)}(u, v) = \ell_{(A, \varphi_{|A})}(u, v) + \ell_{(B,
\varphi_{|B})}(u, v) - \ell_{(A\cap B, \varphi_{|A\cap B})}(u, v)
\end{eqnarray}
holds.  This is proved using the Mayer-Vietoris sequence of \Cech homology groups.

From this result we can deduce that the size function of $X$
contains features of the size functions of $A$ and $B$.  In
particular, when size functions are represented as formal series
of points in the plane through their {\em cornerpoints}
\cite{FrLa01}, relation (\ref{grassmann1}) allows us to prove that
the set of  cornerpoints of   $\ell_{(X, \p)}$ contains a subset
of  cornerpoints of $\ell_{(A, \varphi_{|A})}$. These are a kind
of ``fingerprint'' of the presence of $A$ in $X$. In other words,
size functions are able to detect  a partial matching between two
shapes by showing a common subset of cornerpoints.

The paper is organized as follows. In Section 2 we introduce
background notions about size functions. In Section 3 some general
results concerning the link between size functions and \Cech
homology are proved, with a particular emphasis on the relation
existing between discontinuity points of size functions \cite{FrLa01}
and homological critical values \cite{CoEdHa05}. The reader not
familiar with \Cech homology can find a brief survey of the
subject in Appendices A and B. However, we use \Cech homology only
for technical reasons, so that, after establishing that, in our setting, \Cech
homology groups satisfy all the
ordinary homological axioms, we can use them as ordinary homology groups.
Therefore, the reader acquainted with ordinary homology can easily
go through the next sections. In Section 4 we prove our main
result concerning the relationship between the size function of
$A$, $B$ and $A\cup B$. The relation we obtain holds subject to a
homological condition derived from the Mayer-Vietoris sequence of
\Cech homology. In the same section we also investigate this
homological condition in terms of size functions. Moreover, we
introduce the Mayer-Vietoris sequence of persistent \Cech homology
groups. Section 5 is devoted to the consequent relationship
between cornerpoints for $\sza$, $\szb$ and $\szf$ in terms of
their coordinates and multiplicities. Before concluding the paper
with a brief discussion of our results, we show some experimental
applications in Section 6, demonstrating the potential of our
approach.

\section{Background on size functions}\label{background}

Size functions are a method for shape analysis that is suitable
for any multi-dimensional data set that can be modeled as a
topological space $X$, and whose shape properties can be described
by a continuous function $\p$ defined on it (e.g. a domain of
$\R^2$ and the height function may model  terrain elevations).
Size functions were introduced by P. Frosini at the beginning of
the 1990s (cf., e.g., \cite{Fr91}), and are defined in terms of
the number of connected components of lower level sets associated
with the given space  and  function defined on it.  They belong to
a class of methods that are grounded in Morse theory, as described
in \cite{BiFlXX}.  From the theoretical point of view, the main
properties of size functions that have been studied since their
introduction are the computational issues \cite{d'A00, Fr92}, the
robustness of size functions with respect to continuous
deformations of the shape \cite{DAFrLa06}, the conciseness of the
descriptor \cite{FrLa01}, the invariance of the descriptor to
transformation groups \cite{DiFrPa04,VeUr94},  the  connections of
size functions to the natural pseudo-distance in order to compare
shapes \cite{DoFr04}, their algebraic topological counterparts
\cite{ CaFePo01,FrMu99}, and their generalization to a setting
where many functions are used at the same time to describe the
same space \cite{BiCeXX}. As far as application is concerned, the
most recent papers describe the retrieval of 3D objects
\cite{Bi06} and trademark retrieval \cite{Ce06}.

In this section we provide the reader with the necessary
mathematical background concerning size functions that will be
used in the next sections.

In this paper a pair $(X,\p)$, where $X$ denotes a non-empty
compact and locally connected Hausdorff topological space, and
$\p: X \rightarrow \R$ denotes a continuous function, is called a
{\em size pair}.  Moreover, the function $\p$ is  called a {\em
measuring function}.

Given a size pair $(X,\p)$,  for every $u\in \R$, we denote by
$X_u$ the lower level set $\{ p\in X: \p(p)\le u\}$.

\begin{defi}
Let $(X,\p)$ be a size pair. For every
$u\in \R$, we shall say that two points $p, q \in X$ are
$\langle\p \leq u\rangle$-{\em connected} if and only if a connected
subset  of $X_u$ exists, containing both $p$ and
$q$.
\end{defi}
The relation of being $\langle\p \leq u\rangle$-connected is an
equivalence relation. If two points $p, q \in X$ are $\langle\p
\leq u\rangle$-connected we shall write $p \sim_u q$. For the sake
of simplicity, we are going to use the same symbol $\sim_u$ to
denote the same equivalence relation on subsets of $X$ as well.

In the following, we shall denote by $\Delta^+$ the open half
plane $\{(u,v)\in \R^2: u<v\}$.

\begin{defi}\label{connectedness}
The \emph{size
function} associated with the size pair $(X,\p)$ is the function
$\szf: \Delta^+ \rightarrow \N$ such that, for every $(u, v) \in
\Delta^+$, $\szf(u, v)$ is equal to the number of equivalence
classes into which the set $X_u$ is divided by the relation of
$\langle\p \leq v\rangle$-connectedness.
\end{defi}

In other words,    $\ell _{({X} ,\varphi )}(u,v)$ is equal to the
number of connected components in $X_v$ that contain at least one
point of $X_u$. The finiteness of this number is a consequence of
the  compactness and local connectedness of $X$, and the
continuity of $\p$.

\par
\begin{figure}
\psfrag{P}{$P$} \psfrag{a}{$a$} \psfrag{b}{$b$} \psfrag{c}{$c$}
\psfrag{e}{$e$} \psfrag{x}{$u$} \psfrag{y}{$v$} \psfrag{0}{$0$}
\psfrag{1}{$1$} \psfrag{2}{$2$} \psfrag{3}{$3$}
\centerline{\includegraphics[width=4in]{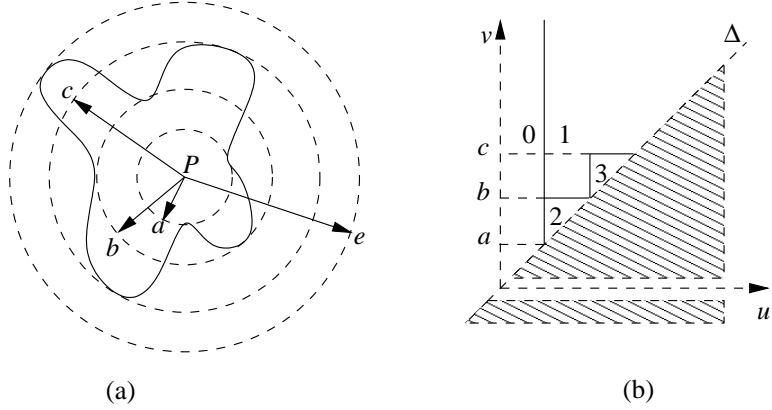}}
\caption{\footnotesize{(b) The size function of the  size pair
$(X,\p)$, where  $X$ is the curve represented by a continuous line
in (a), and $\p$ is the function ``Euclidean distance from the
point $P$''.}}\label{Fig:sf}
\end{figure}
\par

An example of size  function is illustrated in
Figure~\ref{Fig:sf}. In this  example we consider the size pair
$(X,\p)$, where $X$ is the curve of $\R^2$, represented by a
continuous line in Figure~\ref{Fig:sf}~(a), and $\p$ is the
function ``Euclidean distance from the point $P$''.  The  size
function associated with $(X,\p)$ is shown in
Figure~\ref{Fig:sf}~(b). Here, the domain of the size function,
$\Delta^+$,  is divided by solid lines, representing the
discontinuity points of the  size function. These discontinuity
points divide $\Delta^+$ into regions where the size function is
constant. The value displayed in each region is the value taken by
the size function in that region.

For instance, for $a\le u<b$, the set $X_u$ has two connected
components contained in different connected components of $X_v$,
when $u< v<b$. Therefore, $\ell_{\left(X,\varphi\right)}(u,v)=2$
for $a\le u<b$ and $u< v<b$. When $a\le u<b$ and  $v\ge b$, all
the connected components of $X_u$ are contained in the same
connected component of $X_v$. Therefore,
$\ell_{\left(X,\varphi\right)}(u,v)=1$ for $a\le u<b$ and $v\ge
b$. When $b\le u<c$ and $v\ge c$, all of the three connected
components of $X_u$ belong to the same connected component of
$X_v$, implying that in this case
$\ell_{\left(X,\varphi\right)}(u,v)=1$.

As for the values taken on the discontinuity lines, they are
easily obtained by observing that  size functions are
right-continuous, both in the variable $u$ and in the variable
$v$.

We point out that in less recent papers about size functions one
encounters a slightly different definition of size function. In
fact, the original definition of size function was based on the
relation of arcwise-connectedness. The definition used in this
paper, based on connectedness, was introduced in \cite{DAFrLa06}.
This change of definition is theoretically motivated, since it
implies
 the right-continuity of size functions, not only in the
variable $u$ but also in the variable $v$. As a consequence, many
results can be stated more neatly.

An important property of size  functions is that they can be
represented as formal series of points, called cornerpoints.
The main reference here is \cite{FrLa01}.

\begin{defi}
\label{cornerpt} For every point $p=(u,v)\in\Delta^+$, let us
define the number $\mu_X (p)$ as
$$\lim_{\eps\to 0^+}\left(\ell _{(X,\varphi )}(u+\epsilon ,v-\epsilon )-\ell_{(X,\varphi )}(u-\epsilon ,v-\epsilon )-\ell_{(X,\varphi)} (u+\epsilon ,v+\epsilon )+\ell _{(X,\varphi )}(u-\epsilon ,v+\epsilon )\right).$$
The finite number  $\mu_X (p)$ will be called {\em multiplicity
of} $p$ for $\ell_{(X,\varphi )}$. Moreover,  we shall call {\em
proper cornerpoint} for $\ell_{(X,\varphi )}$ any point
$p\in\Delta^+$ such that  the number $\mu_X (p)$ is strictly
positive.
\end{defi}

\begin{defi}
\label{cornerptinfty} For every vertical line $r$, with equation
$u=k$ in the plane $u,v$, let us identify $r$ with the pair
$(k,\infty)$, and define the number $\mu_X(r)$ as
$$\lim_{\eps\to 0^+}\left(\ell_{(X,\varphi )}(k+\epsilon,1/\epsilon)-
\ell _{(X,\varphi )}(k-\epsilon,1/\epsilon)\right).$$
When this finite number, called {\em multiplicity of} $r$ for $\ell
_{(X,\varphi )}$, is strictly positive, we call $(k,\infty)$ a {\em
cornerpoint at infinity} for the  size function.
\end{defi}

As an example of cornerpoints in  size functions, in
Figure~\ref{Fig:value} we see that the proper cornerpoints of the
depicted size function are the points $p$, $q$ and $m$ (with
multiplicity $2$, $1$ and $1$, respectively). The line $r$ is the
only cornerpoint at infinity.

\par
\begin{figure}
\psfrag{x}{$u$} \psfrag{y}{$v$} \psfrag{m}{$r$} \psfrag{p}{$p$}
\psfrag{q}{$q$} \psfrag{s}{$s$} \psfrag{r}{$m$} \psfrag{0}{$0$}
\psfrag{1}{$1$} \psfrag{3}{$3$} \psfrag{4}{$4$} \psfrag{6}{$6$}
\psfrag{7}{$7$} \psfrag{5}{$5$} \centerline{
\includegraphics[width=2in]{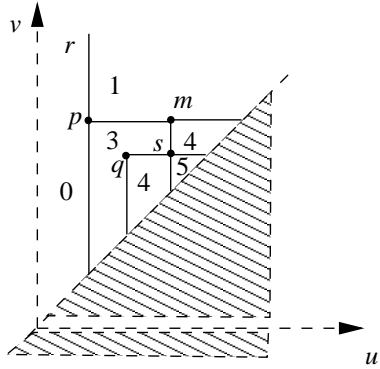}
} \caption{\footnotesize{Cornerpoints of a  size function:  in
this example, $p$, $q$ and $m$ are the only proper cornerpoints,
and have multiplicity equal to $2$ ($p$) and $1$ ($m,q$). The
point $s$ is not a cornerpoint, since its multiplicity vanishes.
The line $r$ is the only cornerpoint at infinity.}
}\label{Fig:value}
\end{figure}
\par

The importance of cornerpoints is revealed by the next  result,
showing that cornerpoints,  with their multiplicities, uniquely
determine  size functions.

The open half-plane $\Delta^+$, extended by the points at infinity
of the kind $(k,\infty)$, will be denoted by $\Delta^*$, i.e.
$$\Delta^*:=\Delta^+\cup\{(k,\infty):k\in\R\}.$$

\begin{thm}
\label{value} For every  $({\bar u},{\bar v})\in\Delta^+$ we have
\begin{eqnarray}\label{reprthm}
\ell_{(X,\varphi )}({\bar u},{\bar v})=\sum _{ (u,v)\in\Delta^*\atop u\le {\bar u}, v>\bar v }\mu_X\big((u,v)\big).
\end{eqnarray}
\end{thm}

The equality~(\ref{reprthm})  can be checked in the example of
Figure~\ref{Fig:value}. The points where the  size function takes
value $0$ are exactly those for which there is no cornerpoint
(either proper or at infinity) lying to the left and above them.
Let us take a point  in the region of the domain where the  size
function takes the value  $3$. According to the above theorem, the
value of the  size function at that point must be equal to
$\mu(r)+\mu(p)=3$.

\section{The link between size functions and \Cech homology}

In this section we prove that the value of the size function can
be computed in terms of rank of \Cech homology groups. We then
analyze the links between homological critical values and size
functions.

The idea of relating size functions to homology groups is not a
new one. Already in \cite{CaFePo01}, introducing the concept of
{\em size functor}, this link was recognized, when the space $X$
is a smooth manifold and $\p$ is a Morse function. Roughly
speaking,  the size functor associated with the pair $(X, \p)$
takes a pair of real numbers $(u,v)\in \Delta^+$ to the image of
the homomorphism from $H_k(X_u)$ to $H_k(X_v)$, induced by
inclusion of $X_u$ into $X_v$. Here homology means singular
homology. This also shows a link between size functions and $0$th
persistent homology groups \cite{{EdLeZo02}}. Later, the relation
between size functions and singular homology groups of closed
manifolds endowed with Morse functions emerged again in
\cite{AlCo07}, studying the {\em Morse shape descriptor}.

The reason for further exploring  the homological interpretation
of size function in the present paper is technical. As explained
in Section 2, our definition of size function is based on the
relation of connectedness (cf. Definition \ref{connectedness}).
This implies that singular homology, whose $0$th group detects the
number of arcwise-connected components, is no longer suited to
dealing with size functions. Adding further assumptions on $X$, so
that connectedness and arcwise-connectedness coincide on $X$, such
as asking $X$ to be locally arcwise-connected, is not sufficient
to solve the problem. Indeed, we emphasize the fact that  in the
definition of $\szf$ we count the components not of the space
$X$ itself, but those of the lower level sets of $X$ with respect
to the continuous function $\p$, and  it is not guaranteed that
locally arcwise-connectedness is inherited by lower level sets.

The  tool we need for counting connected components instead of
arcwise-connected components is \Cech homology (a brief review of
this subject can be found in Appendix \ref{cech}).
 Indeed, in \cite{Wi49}  the
following result is proved, under the assumption that $X$ is a
compact Hausdorff space.

\begin{thm}[\cite{Wi49}, Thm. V 11.3a]\label{ncc}
The number of components of a space $X$ is exactly the rank of the $0$th \Cech Homology group.
\end{thm}

One of the main problems in the use of \Cech homology is that, in
general, the long sequence of the pair may fail to be exact.
However, the exactness of this sequence holds, provided that some
assumptions are satisfied: the space must be compact and the group
$G$ must be either a compact Abelian topological group or a vector
space over a field (see Appendix \ref{exactness}). In view of
establishing a connection between size functions and \Cech
homology, it is important to recall that when $(X,\p)$ is a size
pair, $X$ is assumed to be compact and Hausdorff and $\p$ is
continuous. Therefore, the lower level sets $X_u$ are themselves
Hausdorff and compact spaces. In order that the \Cech homology
sequence of the pair  be available, we will take $G$ to be a
vector space over a field. Therefore, from now on, we will take
the \Cech homology sequence of the pair for granted and we will
denote the \Cech homology groups of $X$ over $G$ simply by $\check
H_p(X)$, maintaining the notation $H_p(X)$ for ordinary homology.
From \cite{EiSt52} we know that $\check H_p(X)$ is a vector space
over the same field.

We shall first furnish a link between  size functions and relative
\Cech homology groups. We need the following preliminary results.

\begin{defi}[\cite{Wi49}, Def. I 12.2]
If $X$ is a space,  and $x,y\in X$, then a finite collection of
sets $X^1$, $X^2$, $\ldots$, $X^n$ will be said to form a {\em
simple chain} of sets from $x$ to $y$  if (1) $X^i$ contains $x$
if and only if $i=1$; (2)  $X^i$ contains $y$ if and only if
$i=n$; (3) $X^i\cap X^j\ne \emptyset$, $i<j$, if and only if
$j=i+1$.
\end{defi}

\begin{prop}[\cite{Wi49}, Cor. I 12.5]\label{simplechain}
A space $X$ is connected if and only if, for arbitrary $x,y\in X$
and covering $\U$ of $X$ by open sets, $\U$ contains a simple
chain from $x$ to $y$.
\end{prop}

Following the proof  used in \cite{Wi49} to prove Theorem
\ref{ncc}, we can also interpret relative homology groups in terms
of the number of connected components.

\begin{lem}\label{nccrel}
For every  pair of spaces $(X, A)$, with $X$ a compact Hausdorff
space and $A$ a closed subset of $X$, the number of connected
components of $X$ that do not meet $A$ is equal to the rank of
$\check{H}_0(X,A)$.
\end{lem}

\begin{proof}
When $A$ is empty, the claim reduces to Theorem \ref{ncc}. When
$A$ is non-empty, if $X$ is connected  then $\check{H}_0(X,A)=0$.
Indeed, under these assumptions, let $z_0=\{z_0(\U)\}$ be a  \Cech
cycle in $X$ relative to $A$, with $z_0(\U)=\sum_{j=1}^ka_j\cdot
U_j$, $a_j \neq 0$. Since $A\subseteq X$ is non-empty, there is an
open set $\bar U\in \U$ such that $\bar U \in \U_A$. Now we can
use Proposition \ref{simplechain} to show that, for every $1\le
j\le k$, there exists a sequence ${\mathcal S}_j$ of elements of
$\U$, beginning with $U_j$ and ending with $\bar U$. So,
associated with ${\mathcal S}_j$, there is a $1$-chain $c_1^j$
such that $\partial c_1^j=U_j-\bar U$. Hence, $\partial  \sum
_{j=1}^ka_j\cdot c_1^j= \sum _{j=1}^ka_j\cdot U_j-\sum
_{j=1}^ka_j\cdot \bar U=z_0(\U)-\sum _{j=1}^ka_j\cdot\bar U$,
proving that $z_0(\U)$ is homologous to $0$ in ${Z}_0(X,A)$. By
the arbitrariness of $\U$, each coordinate of $z_0$ is homologous
to $0$, implying that $\check{H}_0(X,A)=0$.

In general, if $X$ is not connected, the preceding argument shows
that only those connected components of $X$ that do not meet $A$
contain a non-trivial  \Cech cycle relative to $A$. Then the claim
follows from Theorem \ref{ncc}.
\end{proof}

As an immediate consequence of Lemma \ref{nccrel}, we have the
following link between size functions and relative \Cech homology
groups. It is analogous to the link given in \cite{AlCo07} using
singular homology for size functions, defined in terms of the
arcwise-connectedness relation.

\begin{cor}\label{sfrelhom}
For every  size pair $(X,\p)$, and every $(u,v)\in \Delta^+$, it
holds that the value $\szf(u,v)$ equals the rank  of
$\check{H}_0(X_v)$ minus the rank of $\check{H}_0(X_v, X_u)$.
\end{cor}

\begin{proof}
The claim follows from Lemma \ref{nccrel}, observing that
$\szf(u,v)$ is equal to the number of connected components of
$X_v$ that meet $X_u$.
\end{proof}

We now show that the size function can also be expressed as the
rank of the image of the homomorphism between \Cech homology
groups, induced by inclusion of $X_u$ into $X_v$. This link is
analogous to the existing one between the size functor and size
functions, defined using the arcwise-connectedness relation
\cite{CaFePo01}.

Given a size pair $(X,\p)$, and  $(u,v)\in \Delta^+$, we denote by
$\iota^{u,v}$ the inclusion of $X_u$ into $X_v$. This mapping
induces  a homomorphism of \Cech homology groups
$\iota_p^{u,v}:\check H_p(X_u)\rightarrow \check H_p(X_v)$ for
each integer $p\ge 0$.

Following \cite{EdLeZo02}, we can define the
persistent \Cech homology groups.

\begin{defi}\label{persistentcechhom}
Given a size pair $(X,\p)$ and a point $(u,v)\in \Delta^+$, the
{\em $p$th persistent
 \Cech homology group $\check{H}_p^{u, v}$} is the image of the homomorphism $\iota_p^{u,v}$
induced between the $p$th \Cech homology groups by the inclusion
mapping of $X_u$ into $X_v$: $\mathit{\check{H}_p^{u, v}(X) = \I
\iota_p^{u, v}}$.
\end{defi}

\begin{cor}\label{sfhomology}
For every  size pair $(X,\p)$, and every $(u,v)\in \Delta^+$, it
holds that the value $\szf(u,v)$ equals the rank  of the $0$th
persistent
 \Cech homology group $\check{H}_0^{u, v}(X)$.
\end{cor}

\begin{proof}
Let us consider the final terms of the long exact sequence of the pair $(X_v,X_u)$:
$$\ldots\rightarrow\check{H}_0(X_u)\stackrel{\iota_0^{u,v}}{\rightarrow} \check{H}_0(X_v)\rightarrow \check{H}_0(X_v,X_u)\rightarrow 0.$$
From the exactness of this sequence we deduce that
$$\rank\check{H}_0^{u, v}(X)= \rank\I
{\iota}_0^{u, v}=\rank \check{H}_0(X_v)-\rank
\check{H}_0(X_v,X_u).$$ Applying Theorem \ref{ncc} and Lemma
\ref{nccrel}, the rank of $\check{H}_0^{u, v}(X)$ turns out to be
equal to the number of connected components of $X_v$ that meet
$X_u$, that is  $\szf(u,v)$.
\end{proof}

\subsection{Some useful results}
In this section we show the link between homological critical
values and discontinuity points of size functions. Homological critical
values have been  introduced in \cite{CoEdHa05}, and intuitively
correspond to levels where the lower level sets undergo a
topological change. Discontinuity points of size functions have been
thoroughly studied in \cite{FrLa01}.

In particular, we prove that if a point $(u,v) \in \Delta^+$ is a
discontinuity point for a size function, then either $u$ or $v$ is
a level where the $0$-homology of the lower level set changes
(Proposition \ref{0-critValueSF}). Then we show that also the
converse is true when the number of homological critical values is
finite (Proposition \ref{finitecrit}). However, in general,  there
may exist homological critical values not generating
discontinuities for the size function (Remark \ref{patologic}). We
conclude the section with a result concerning the surjectivity of
the homomorphism induced by inclusion (Proposition \ref{surj}).

\begin{defi}\label{p-critValue}
Let $(X,\p)$ be a size pair. A \emph{homological $p$-critical
value} for $(X,\p)$ is a real number $w$ such that, for every
sufficiently small $\eps > 0$, the map ${\iota}_p^{w - \eps, w +
\eps}: \check{H}_p(X_{w - \eps})\rightarrow \check{H}_p(X_{w +
\eps})$ induced by inclusion is not an isomorphism.
\end{defi}

The following results show the behavior of a size function
according to whether it is calculated in correspondence with
homological 0-critical values or not.

\begin{prop}\label{0-critValueSF}
If $w \in \R$ is not a homological $0$-critical value for the size
pair $(X, \p)$, then the following statements are true:
\begin{enumerate}
\item For every $v > w$, $\underset{\eps \rightarrow
0^+}\lim\left(\szf(w + \eps, v)- \szf(w - \eps, v)\right)=0$;
\item For every $u < w$, $\underset{\eps \rightarrow
0^+}\lim\left(\szf(u, w - \eps)- \szf(u, w + \eps)\right)=0.$
\end{enumerate}
\end{prop}
\begin{proof}
We begin by proving \emph{(i)}. Let $v > w$. For every $\eps > 0$
such that $v > w + \eps$, we can consider the commutative diagram:
\begin{eqnarray}
\begin{array}{c}\label{PairSequencev}
\begin{centering}
\hfill \xymatrix {\cdots \ar[r]&\check{H}_0(X_{w -
\eps})\ar[r]^(.6){{\iota}_0^{w -\eps, v}} \ar[d]^{\iota_0^{w -
\eps, w + \eps}}&\check{H}_0(X_v)\ar[r]^(.4){h}
\ar[d]^{\iota_0^{v,v}} &\check{H}_0(X_v, X_{w - \eps})\ar[r]
\ar[d]^{j}&0\ar[d]_{0}
\\\cdots \ar[r]&\check{H}_0(X_{w+ \eps})\ar[r]^(.6){{\iota}_0^{w +\eps, v}}&\check{H}_0(X_v)\ar[r]^(.4){k}&\check{H}_0(X_v, X_{w + \eps})\ar[r]&0}\hfill
\end{centering}
\end{array}
\end{eqnarray}
where the two horizontal lines are exact homology sequences of the
pairs $(X_v, X_{w - \eps})$ and $(X_v, X_{w + \eps})$,
respectively, and the vertical maps are homomorphisms induced by
inclusions. By the assumption that $w$ is not a homological
0-critical value, there exists an arbitrarily small
$\overline{\eps}>0$ such that $\iota_0^{w - \overline{\eps},w +
\overline{\eps}}$ is an isomorphism. Therefore, by applying the
Five Lemma in diagram (\ref{PairSequencev}) with $\eps =
\overline{\eps}$, we deduce that $j$ is an isomorphism. Thus,
$\rank \check{H}_0(X_v, X_{w - \overline{\eps}})= \rank
\check{H}_0(X_v, X_{w + \overline{\eps}})$, and consequently, by
Corollary \ref{sfrelhom}, $\szf(w + \overline{\eps}, v)= \szf(w -
\overline{\eps}, v).$ Hence, since size functions are
non-decreasing in the first variable, it may be concluded that
$\underset{\eps \rightarrow 0^+}\lim\left(\szf(w + \eps, v)-
\szf(w - \eps, v)\right)=0$.

Now, let us proceed by proving \emph{(ii)}. Let $u < w$. For every
$\eps> 0$ such that $u < w - \eps$, let us consider the following
commutative diagram:
\begin{eqnarray}
\begin{array}{c}\label{PairSequenceu}
\begin{centering}
\hfill \xymatrix {\cdots
\ar[r]&\check{H}_0(X_u)\ar[r]^(.55){\iota_0^{u,w- \eps}}
\ar[d]^{\iota_0^{u,u}}&\check{H}_0(X_{w - \eps})\ar[r]^(.45){h}
\ar[d]^{\iota_0^{w - \eps,w + \eps}} &\check{H}_0(X_{w -
\eps},X_u)\ar[r] \ar[d]^{j}&0\ar[d]_{0}
\\\cdots \ar[r]&\check{H}_0(X_u)\ar[r]^(.55){\iota_0^{u,w+ \eps}}&\check{H}_0(X_{w + \eps})\ar[r]^(.45){k}&\check{H}_0(X_{w + \eps},X_u)\ar[r]&0}\hfill
\end{centering}
\end{array}
\end{eqnarray}
where the vertical maps are homomorphisms induced by inclusions
and the two horizontal lines are exact homology sequences of the
pairs $(X_{w - \eps}, X_u)$ and $(X_{w + \eps}, X_u)$,
respectively.  By the assumption that $w$ is not a homological
0-critical value, there exists an arbitrarily small
$\overline{\eps}>0$, for which ${\iota}_0^{w - \overline{\eps}, w
+ \overline{\eps}}: \check{H}_0(X_{w -
\overline{\eps}})\rightarrow \check{H}_0(X_{w + \overline{\eps}})$
is an isomorphism. Therefore, by applying the Five Lemma in
diagram (\ref{PairSequenceu}) with $\eps = \overline{\eps}$, we
deduce that $j$ is an isomorphism. Thus, $\rank \check{H}_0(X_{w -
\overline{\eps}}, X_u)= \rank \check{H}_0(X_{w + \overline{\eps}},
X_u)$, implying $\szf(u, w- \overline{\eps})= \szf(u, w+
\overline{\eps})$. Hence, since size functions are non-increasing
in the second variable, the desired claim follows.
\end{proof}

Assuming the existence of at most a finite number of homological
critical values, we can say that homological critical values give
rise to discontinuities in size functions.

\begin{prop}\label{finitecrit}
Let $(X, \p)$ be a size pair with at most a finite number of
homological 0-critical values. Let $w \in \R$ be a homological
0-critical value. The following statements hold:
\begin{enumerate}
\item If ${\iota}_0^{w - \eps, w + \eps}$ is not surjective for
any sufficiently small positive real number $\eps$, then there
exists $v>w$ such that $w$ is a discontinuity point for
$\szf(\cdot, v)$; \item If ${\iota}_0^{w - \eps, w + \eps}$ is
surjective for every sufficiently small positive real number
$\eps$, then there exists $u<w$ such that $w$ is a discontinuity
point for $\szf(u, \cdot)$.
\end{enumerate}
\end{prop}
\begin{proof}
Let us prove \emph{(i)}, always referring to diagram
(\ref{PairSequencev}) in the proof of Proposition
\ref{0-critValueSF}. Let $v>w$. For every $\eps>0$ such that $v>
w+ \eps$, the map $j$ of diagram (\ref{PairSequencev}) is
surjective. Indeed, $h$, $k$ and $\iota_0^{v, v}$ are surjective.

If we prove that there exists $v>w$ for which, for every $\eps>0$
such that $v>w + \eps$, $j$ is not injective, then, since $j$ is
surjective, it necessarily holds that $\rank \check{H}_0(X_v, X_{w
- \eps})> \rank \check{H}_0(X_v, X_{w + \eps})$, for every
$\eps>0$ such that $v>w + \eps$. From this we obtain $\szf(w -
\eps, v) = \rank \check{H}_0(X_v)-\rank \check{H}_0(X_v, X_{w -
\eps})< \rank \check{H}_0(X_v)-\rank \check{H}_0(X_v, X_{w +
\eps})= \szf(w + \eps, v)$, for every $\eps>0$ such that $v>w +
\eps$. Therefore, $\underset{\eps \rightarrow 0^+}\lim\left(\szf(w
+ \eps, v)- \szf(w - \eps, v)\right)>0$, that is, $w$ is a
discontinuity point for $\szf(\cdot, v)$.

We now show that there exists $v>w$ for which, for every $\eps>0$
such that $v>w + \eps$, $j$ is not injective. \\Since we have
hypothesized the presence of at most a finite number of
homological 0-critical values for $(X, \p)$, there surely exists
$v>w$ such that, for every sufficiently small $\eps>0$, $v>w+
\eps$ and ${\iota}_0^{w +\eps, v}: \check{H}_0(X_{w
+\eps})\rightarrow \check{H}_0(X_{v})$ is an isomorphism. Hence,
from the exactness of the second row in diagram
(\ref{PairSequencev}), taking such a $v$, $\check{H}_0(X_{v},X_{w
+ \eps})$ is trivial. Now, if $j$ were injective, from the
triviality of $\check{H}_0(X_{v},X_{w + \eps})$, it would follow
that $\check{H}_0(X_v, X_{w - \eps})$ is also trivial, and
consequently ${\iota}_0^{w -\eps, v}$ surjective. This is a
contradiction, since we are assuming $\iota_0^{w - \eps, w +
\eps}$ not surjective, and it implies that ${\iota}_0^{w -\eps,
v}$ is not surjective because $\iota_0^{w + \eps, v}$ and
$\iota_0^{v, v}$ are isomorphisms.

As for \emph{(ii)}, we will always refer to diagram
(\ref{PairSequenceu}) in the proof of Proposition
\ref{0-critValueSF}. In this case, by combining the hypothesis
that, for any sufficiently small $\eps>0$, ${\iota}_0^{w - \eps, w
+ \eps}$ is not an isomorphism and ${\iota}_0^{w - \eps, w +
\eps}$ is surjective, it necessarily follows that ${\iota}_0^{w -
\eps, w + \eps}$ is not injective. Hence, $\rank \check{H}_0(X_{w
- \eps})> \rank \check{H}_0(X_{w + \eps})$, for every sufficiently
small $\eps>0$. Let $u<w$. For every $\eps>0$ such that $u+
\eps<w$, the map $j$ of diagram (\ref{PairSequenceu}) is
surjective. Indeed, $h$, $k$ and $\iota_0^{w - \eps, w + \eps}$
are surjective.

Now, if we prove the existence of $u<w$, for which, for every
$\eps>0$ such that $u+ \eps<w$, $j$ is an isomorphism, it
necessarily holds that $\rank \check{H}_0(X_{w - \eps},X_u) =
\rank \check{H}_0(X_{w + \eps},X_u)$, for every $\eps>0$ such that
$u+ \eps<w$. Thus, it follows that $\szf(u, w - \eps) = \rank
\check{H}_0(X_{w - \eps})-\rank \check{H}_0(X_{w - \eps},X_u)>
\rank \check{H}_0(X_{w + \eps})-\rank \check{H}_0(X_{w +
\eps},X_u)= \szf(u, w + \eps)$, for every $\eps>0$ such that $u+
\eps<w$, implying $\underset{\eps \rightarrow
0^+}\lim\left(\szf(u, w - \eps)- \szf(u, w + \eps)\right)>0$, that
is, $w$ is a discontinuity point for $\szf(u, \cdot)$.

Recalling that $j$ is surjective, let us prove that there exists
$u<w$ for which $j$ is injective for every $\eps>0$ with $u+
\eps<w$.
\\Since we have assumed the presence of at most a finite number of
homological 0-critical values for $(X, \p)$, there surely exists
$u<w$ such that, for every sufficiently small $\eps>0$, $u<w-
\eps$ and ${\iota}_0^{u, w -\eps}: \check{H}_0(X_{u})\rightarrow
\check{H}_0(X_{w - \eps})$ is an isomorphism. Hence, for such a
$u$, $\check{H}_0(X_{w - \eps},X_u)$ is trivial, implying $j$
injective.
\end{proof}

Dropping the assumption that the number of homological 0-critical
values for $(X, \p)$ is finite, the converse of Proposition
\ref{0-critValueSF} is false, as the following remark states.

\begin{rem}\label{patologic}
From the condition that $w$ is a homological 0-critical value, it
does not follow that $w$ is a discontinuity point for the function
$\szf(\cdot, v)$, $v>w$, or for the function $\szf(u, \cdot)$,
$u<w$.

In particular, the hypothesis $\rank \check{H}_0(X_{w - \eps})\neq
\rank \check{H}_0(X_{w + \eps})$, for every sufficiently small
$\eps > 0$, does not imply that there exists either $v > w$ such
that $\underset{\eps \rightarrow 0^+}\lim\left(\szf(w + \eps, v)-
\szf(w - \eps, v)\right) \neq 0$ or $u < w$ such that
\\$\underset{\eps \rightarrow 0^+}\lim\left(\szf(u, w - \eps)-
\szf(u, w + \eps)\right)\neq 0.$
\end{rem}

\begin{figure}[htbp]
\begin{center}
\begin{tabular}{cc}
\psfrag{a}{$\scriptstyle{1}$} \psfrag{b}{$\scriptstyle{9/8}$}
\psfrag{c}{$\scriptstyle{5/4}$} \psfrag{d}{$\scriptstyle{3/2}$}
\psfrag{u}{$\scriptstyle{u}$} \psfrag{v}{$\scriptstyle{v}$}
\psfrag{0}{$\scriptstyle{0}$} \psfrag{1}{$\scriptstyle{1}$}
\psfrag{2}{$\scriptstyle{2}$} \psfrag{e}{$\scriptstyle{2}$}
\includegraphics[height=4cm]{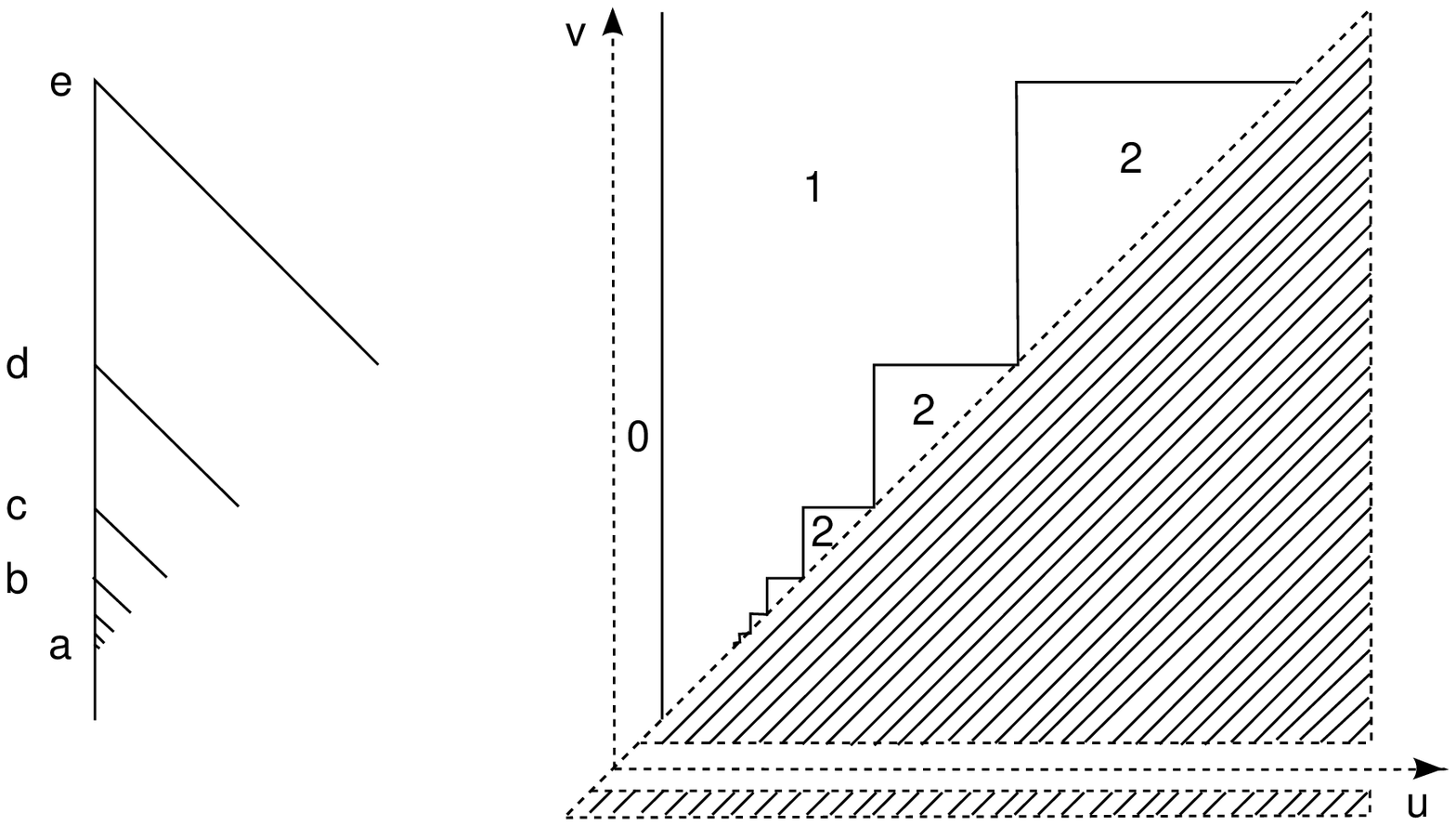}&
\psfrag{a}{$\scriptstyle{1}$} \psfrag{b}{$\scriptstyle{3/2}$}
\psfrag{c}{$\scriptstyle{7/4}$} \psfrag{d}{$\scriptstyle{15/8}$}
\psfrag{u}{$\scriptstyle{u}$} \psfrag{v}{$\scriptstyle{v}$}
\psfrag{0}{$\scriptstyle{0}$} \psfrag{1}{$\scriptstyle{1}$}
\psfrag{2}{$\scriptstyle{2}$} \psfrag{e}{$\scriptstyle{2}$}
\includegraphics[height=4cm]{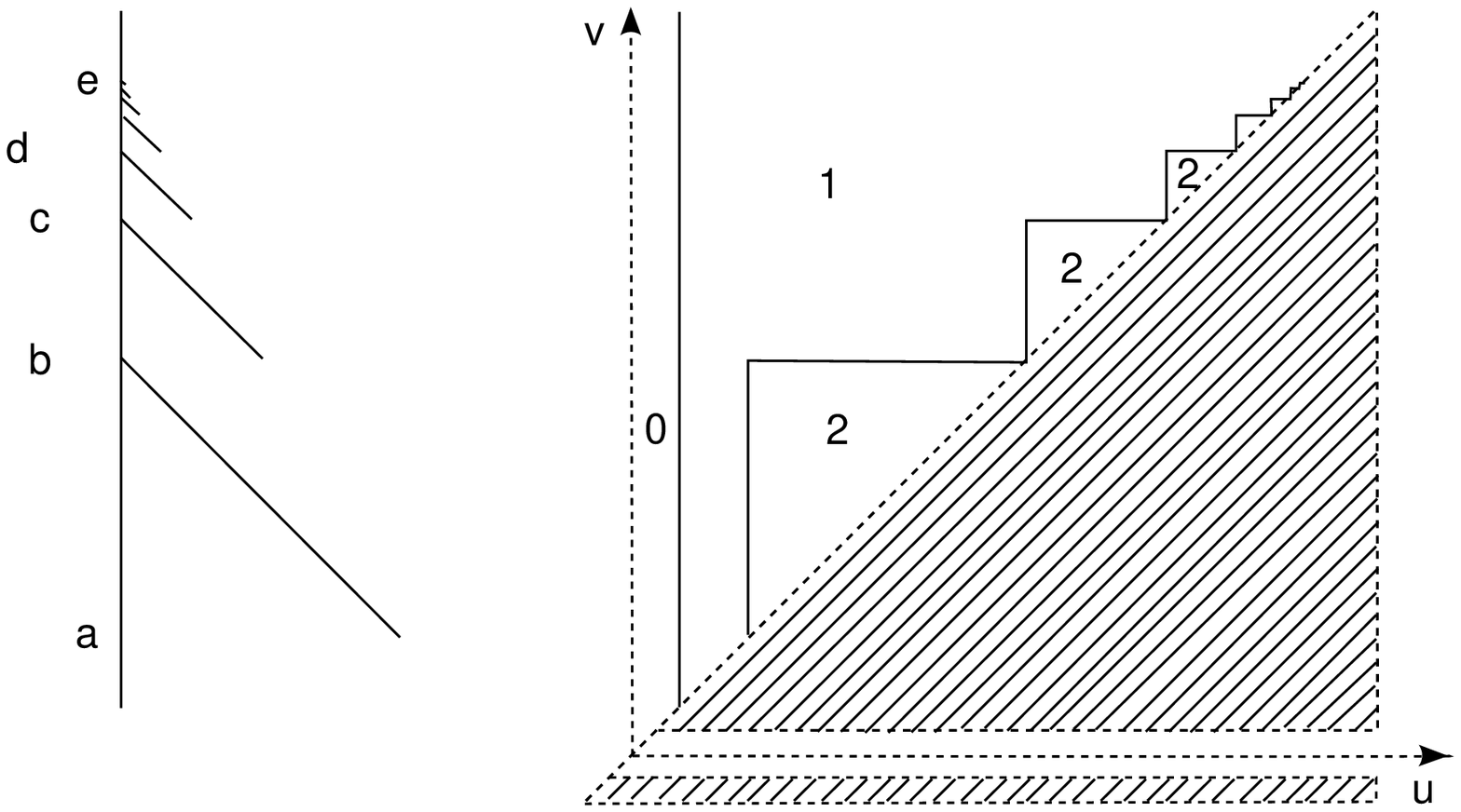}\\
$(a)$&$(b)$
\end{tabular}
\caption{\footnotesize{Two examples showing the existence of a
real number $w$ that is a homological 0-critical value for $(X,
\p)$ but not a discontinuity point for $\szf(\cdot, v)$ or
$\szf(u, \cdot)$.}} \label{patologicEx}
\end{center}
\end{figure}
Two different examples, shown in Figure \ref{patologicEx}, support
our claim.

Let us describe the first example (see Figure \ref{patologicEx},
(a)). Let $(X, \p)$ be the size pair where $X$ is the topological
space obtained by adding an infinite number of branches to a
vertical segment, each one sprouting at the height where the
previous expires. These heights are chosen according to the
sequence $(1 + \frac{1}{2^n})_{n \in \N}$, converging to $1$. The
measuring function $\p$ is the height function. The size function
associated with $(X, \p)$ is displayed on the right side of $X$.
In this case, $w = 1$ is a homological 0-critical value. Indeed,
for $w=1$, it holds that $\rank \check{H}_0(X_{w - \eps}) = 1$
while $\rank \check{H}_0(X_{w + \eps}) = 2$, for every
sufficiently small $\eps
> 0$. On the other hand, for every $v > w$,
and for every small enough $\eps > 0$, it holds that $\szf(w +
\eps, v) = \szf(w - \eps, v) = 1$. Therefore, $\underset{\eps
\rightarrow 0^+}\lim\left(\szf(w + \eps, v)- \szf(w - \eps,
v)\right)= 0$, for every $v>w$. Moreover, it is immediately
verifiable that, for every $u<w$, $\underset{\eps \rightarrow
0^+}\lim\left(\szf(u, w - \eps)- \szf(u, w + \eps)\right)= 0.$

The second example, shown in Figure \ref{patologicEx} (b), is
built in a similar way. In the chosen size pair $(X, \p)$, $\p$ is
again the height function, and $X$ is again obtained by adding an
infinite number of branches to a vertical segment, but this time,
the sequence of heights of their endpoints is $(2 -
\frac{1}{2^n})_{n \in \N}$, converging to $2$. In this case, $w =
2$ is a homological 0-critical value for $(X, \p)$. Indeed, for
every sufficiently small $\eps>0$, $\rank \check{H}_0(X_{w -
\eps}) = 2$ while $\rank \check{H}_0(X_{w + \eps}) = 1$. On the
other hand, for every $u<w$, and for every small enough $\eps >
0$, it holds that $\szf(u, w + \eps) = \szf(u, w - \eps) = 1$ or
$\szf(u, w + \eps) = \szf(u, w - \eps) = 0$. Therefore,
$\underset{\eps \rightarrow 0^+}\lim\left(\szf(u, w - \eps)-
\szf(u, w + \eps)\right)= 0$, for every $u<w$, in both cases.
Moreover, we can immediately verify that, for every $v>w$,
$\underset{\eps \rightarrow 0^+}\lim\left(\szf(w + \eps, v)-
\szf(w - \eps, v)\right)= 0.$

Before concluding this section, we investigate a condition for the
surjetivity of the homomorphism between the $0$th \Cech homology
groups induced by the inclusion map of  $X_u$ into $X_v$,
$\iota_0^{u,v}: \check{H}_0(X_u) \rightarrow \check{H}_0(X_v)$,
because it will be needed in Subsection \ref{suffCond}.
\begin{prop}\label{surj}
Let $(X, \p)$ be a size pair. For every $(u, v) \in\Delta^+ $,
$\iota_0^{u,v}$ is surjective if and only if $\ell_{(X,\p)}(u, v')
= \ell_{(X,\p)}(v, v')$, for every $v'>v$.
\end{prop}

\begin{proof}
For every $v' > v$, let $\frac{X_u}{\sim_{v'}}$ (respectively,
$\frac{X_v}{\sim_{v'}}$) be the space obtained quotienting $X_u$
(respectively, $X_v$) by the relation of $\langle\p\le
v'\rangle$-connectedness. Let us define the map $F_{v'}:
\frac{X_u}{\sim_{v'}}\rightarrow\frac{X_v}{\sim_{v'}}$, such that
$F_{v'}$ takes the class of $p$ in $\frac{X_u}{\sim_{v'}}$ into
the class of $p$ in $\frac{X_v}{\sim_{v'}}$. $F_{v'}$ is well
defined and injective, since $u<v<v'$. The condition that
$\ell_{(X,\p)}(u, v') = \ell_{(X,\p)}(v, v')$  is equivalent to
the bijectivity of $F_{v'}$.

Let $\iota_0^{u, v}: \check{H}_0(X_u) \rightarrow
\check{H}_0(X_v)$ be surjective.  By Corollary \ref{sfrelhom} and
Corollary \ref{sfhomology}, this is equivalent to saying that, for
every $p\in X_v$, there is $q\in X_u$ with $p\sim_v q$. Since
$v<v'$, it also holds that $p\sim_{v'} q$ and this implies
$F_{v'}([q]) = [p]$, for all $v' > v$. So, $F_{v'}$ is bijective
and $\ell_{(X,\p)}(u, v') = \ell_{(X,\p)}(v, v')$, for every
$v'>v$.

Conversely, let $F_{v'}:
\frac{X_u}{\sim_{v'}}\rightarrow\frac{X_v}{\sim_{v'}}$ be a
surjective map, for all $v' > v$. Let $p \in X_v$. Let $(v_n)$ be
a strictly decreasing sequence of real numbers converging to $v$.
The surjectivety of $F_{v_n}$ implies that $q_n\in {X_u}$ exists,
such that $F_{v_n}([q_n]) = [p]$, for all $n \in \N$. Thus $p
\sim_{v_n} q_n$, for all $n \in \N$. Since $X$ is compact and
$X_u$ is closed in $X$, there is a subsequence of $(q_n)$, still
denoted by $(q_n)$, converging in $X_u$. Let $q = {\underset{n
\rightarrow \infty}\lim} q_n\in X_u$. Then, necessarily, $p
\sim_{v_n} q$, for all $n$. In fact, let us call $C_n$ the
connected component of $X_{v_n}$ containing $p$. Since $(v_n)$ is
decreasing, we have $C_{n} \supseteq C_{n+1}$ for every $n\in\N$.
Let us assume that there exists $N \in \N$ such that $p
\nsim_{v_N} q$. Since $C_N$ is closed, if $q\notin C_N$, there
exists  an open neighborhood $U(q)$ of $q$, such that $U(q) \cap
C_N = \emptyset$. Thus, surely, there exists at least one point
$q_n \in U(q)$, with $n > N$ and $q_n \not\in C_N$. This is a
contradiction, because $q_n \in C_n \subseteq C_N$, for all $n >
N$.

Therefore, $p \sim_{v_n} q$ for all $n$, and this implies that $p
\sim_{v} q$, because of Rem. 3 in \cite{DAFrLa06}. Hence,
$\iota_0^{u, v}: \check{H}_0(X_u) \rightarrow \check{H}_0(X_v)$ is
surjective.
\end{proof}

\begin{rem}
The condition  that $\ell_{(X,\p)}(u, v') = \ell_{(X,\p)}(v, v')$,
for every $v'>v$, can be restated saying that $\ell_{(X,\p)}$ has
no points of horizontal discontinuity in the region $\{(x,y)\in
\Delta^+: u< x\le v,\  y>v\}$. In other words, the set $\{(x,y)\in
\Delta^+: u< x\le v,\  y>v\}$ does not contain any cornerpoint
(either proper or at infinity) for $\ell_{(X,\p)}$.
\end{rem}

\section{The Mayer-Vietoris sequence of persistent \Cech homology groups}

In this section, we look for a relation expressing the size
function associated with the size pair $(X,\p)$ in terms of size
functions associated with size pairs $(A,\p_{|A})$ and $(B,
\p_{|B})$, where $A$ and $B$  are closed locally connected subsets
of $X$, such that $X= int(A) \cup int(B)$, and $A\cap B$ is
locally connected. The notations $int(A)$ and $int(B)$ stand for
the interior of the sets $A$ and $B$ in $X$, respectively. The
previous assumptions on $A$, $B$ and $A\cap B$, together with the
fact that the functions $ \varphi_{|A\cap B}$, $\varphi_{|A}$, and
$\varphi_{|B}$ are continuous, as restrictions of the continuous
function $\p:X\rightarrow \R$ to spaces endowed with the topology
induced from $X$, ensure that $(A,\p_{|A})$, $(B, \p_{|B})$, and
$(A\cap B,\p_{|A\cap B})$ are themselves size pairs. These
hypotheses on $X$, $A$, $B$ and $A \cap B$ will be maintained
throughout the paper.

We  find a homological condition guaranteeing a Mayer-Vietoris
formula between size functions evaluated at a point $(u,v)\in
\Delta^+$, that is, $\ell_{(X, \p)}(u, v) = \ell_{(A,
\varphi_{|A})}(u, v) + \ell_{(B, \varphi_{|B})}(u, v) -
\ell_{(A\cap B, \varphi_{|A\cap B})}(u, v)$ (see Corollary
\ref{szrelation}). We shall apply this relation in the next
section in order to show that it is possible to match a subset of
the cornerpoints for $\ell_{(X, \p)}$ to cornerpoints for either
$\ell_{(A, \varphi_{|A})}$ or $\ell_{(B, \varphi_{|B})}$.

Our main tools are the Mayer-Vietoris sequence and the homology
sequence of the pair, applied to the lower level sets of $X$, $A$,
$B$, and $A\cap B$.

Using the same tools, we show that there exists a Mayer-Vietoris
sequence for persistent \Cech homology groups that is of order
$2$. This implies that, under proper assumptions, there is a short
exact sequence involving the $0$th persistent \Cech homology
groups of $X$, $A$, $B$, and $A\cap B$ (see Proposition
\ref{persist}).

We begin by underlining some simple properties of the lower level
sets of $X$, $A$, $B$, and $A\cap B$.

\begin{lem}\label{cupcap}
Let $u \in \R$.  Let us endow $X_u$ with the topology induced by
$X$. Then $A_u$ and $B_u$ are closed sets in $X_u$. Moreover, $X_u =
int(A_u) \cup int(B_u)$ and $A_u\cap B_u=(A\cap B)_u$.
\end{lem}

\begin{proof}
$A_u$ is closed in $X_u$ if there exists a set $C \subseteq
X$, closed in the topology of $X$, such that $C \cap X_u =
A_u$. It is sufficient to take  $C = A$. Analogously for
$B_u$.

About the second statement, the proof that $X_u \supseteq int(A_u)
\cup int(B_u)$ is trivial. Let us prove that $X_u \subseteq
int(A_u) \cup int(B_u)$. If $x \in X_u$ then $x \in int(A)$ or
$x\in  int(B)$.  Let us suppose that $x \in int(A)$. Then there
exists an open neighborhood of $x$ in $X$ contained in $A$, say
$U(x)$. Clearly, $U(x) \cap X_u$ is an open neighborhood of $x$ in
$X_u$ and is contained in $A_u$. Hence $x \in int(A_u)$. The proof
is analogous if $x \in int(B)$. The proof that $A_u\cap B_u=(A\cap
B)_u$ is trivial.
\end{proof}
Lemma \ref{cupcap} ensures that, for $(u,v)\in \Delta^+$, we can
consider the following diagram:
\begin{eqnarray}\label{MVdiag}
\begin{array}{cccccccccc}
&\vdots&&\vdots&&\vdots&&\vdots&&\\
&\downarrow&&\downarrow&&\downarrow&&\downarrow&&\\
&&&&&&&&&\\
\cdots\rightarrow\!\!\!&\!\!\!\check{H}_{p+1}(X_u)\!\!\!&\!\!\!\stackrel{\Delta_u}{\rightarrow}\!\!\!&\!\!\!\check{H}_p((A\cap B)_u)\!\!\!&\!\!\!\stackrel{\alpha_u}{\rightarrow}\!\!\!&\!\!\!\check{H}_p(A_u)\oplus \check{H}_p(B_u)\!\!\!&\!\!\!\stackrel{\beta_u}{\rightarrow}\!\!\!&\!\!\!\check{H}_{p}(X_u)\!\!\!&\!\!\!\rightarrow\!\!\!&\!\!\! \cdots\\
&&&&&&&&&\\
&\downarrow h_{p+1}&&\downarrow f_p&&\downarrow g_p&&\downarrow h_p&&\\
&&&&&&&&&\\
 \cdots\rightarrow\!\!\!&\!\!\!\check{H}_{p+1}(X_v)\!\!\!&\!\!\!
\stackrel{\Delta_v}{\rightarrow}\!\!\!&\!\!\!\check{H}_p((A\cap
B)_v)\!\!\!&\!\!\!
\stackrel{\alpha_v}{\rightarrow}\!\!\!&\!\!\!\check{H}_p(A_v)\oplus
\check{H}_p(B_v)\!\!\!&\!\!\!\stackrel{\beta_v}{\rightarrow}\!\!\!&\!\!\!\check{H}_p(X_v)\!\!\!&\!\!\!\rightarrow\!\!\!&\!\!\!\cdots\\
&&&&&&&&&\\
&\downarrow h'_{p+1}&&\downarrow f'_p&&\downarrow g'_p&&\downarrow h'_p&&\\
&&&&&&&&&\\
 \cdots\rightarrow\!\!\!&\!\!\!\check{H}_{p+1}(X_v,X_u)\!\!\!&\!\!\!
\stackrel{\Delta_{v,u}}{\rightarrow}\!\!\!&\!\!\!\check{H}_p((A\cap
B)_v,(A\cap B)_u)\!\!\!&\!\!\!
\stackrel{\alpha_{v,u}}{\rightarrow}\!\!\!&\!\!\!\check{H}_p(A_v,A_u)\oplus
\check{H}_p(B_v,B_u)\!\!\!&\!\!\!\stackrel{\beta_{v,u}}{\rightarrow}\!\!\!&\!\!\!\check{H}_p(X_v,X_u)\!\!\!&\!\!\!\rightarrow\!\!\!&\!\!\!\cdots\\
&&&&&&&&&\\
&\downarrow&&\downarrow&&\downarrow&&\downarrow&&\\
&\vdots&&\vdots&&\vdots&&\vdots&&
\end{array}
\end{eqnarray}
where the top line belongs to the Mayer-Vietoris  sequence of the
triad $(X_u, A_u, B_u)$, the second line belongs to the
Mayer-Vietoris sequence of the triad $(X_v, A_v, B_v)$, and the
bottom line belongs to the relative Mayer-Vietoris sequence of the
triad $((X_v, X_u), (A_v, A_u), (B_v, B_u))$. For every $p\ge 0$,
the vertical maps $f_p, g_p$, and $ h_p$ are induced by
inclusions of $(A\cap B)_u$ into $(A\cap B)_v$, $(A_u, B_u)$ into
$(A_v, B_v)$, and $X_u$ into $X_v$, respectively. Moreover, $f'_p,
g'_p$ and $ h'_p$ are induced by inclusions of $((A\cap
B)_v,\emptyset)$ into $((A\cap B)_v,(A\cap B)_u)$,
$((A_v,\emptyset),(B_v,\emptyset))$ into $((A_v,A_u),(B_v,B_u))$,
and $(X_v,\emptyset)$ into $(X_v,X_u)$, respectively.

\begin{lem}\label{commut}
Each vertical and horizontal line in diagram $\mathrm{(\ref{MVdiag})}$ is exact. Moreover,
each square in the same diagram  is commutative.
\end{lem}

\begin{proof}
We recall that we are assuming that $X$ is compact and $\p$
continuous, therefore $X_u$ and $X_v$ are compact, as are $A_u$,
$A_v$, $B_u$ and $B_v$ by Lemma \ref{cupcap}. Therefore, since we
are also assuming that the coefficient group $G$ is  a vector
space over a field,  it holds that the homology sequences of the
pairs $(X_v,X_u)$, $((A\cap B)_v,(A\cap B)_u)$, $(A_v,A_u)$,
$(B_v,B_u)$ (vertical lines) are exact (cf. Theorem \ref{pair} in
Appendix \ref{exactness}).

Analogously, the Mayer-Vietoris sequences of $(X_u, A_u, B_u)$ and
$(X_v, A_v, B_v)$, and the relative Mayer-Vietoris sequence of
$((X_v, X_u), (A_v, A_u), (B_v, B_u))$ (horizontal lines) are
exact (cf. Theorems \ref{MV} and \ref{MVrel} in Appendix
\ref{exactness}).

About the commutativity of the top squares, it is sufficient to
apply Theorem \ref{MVhom} in Appendix \ref{exactness}. The same
conclusion can be drawn for the commutativity of the bottom
squares, with $X_v$ replaced by $(X_v,\emptyset)$, $A_v$ by
$(A_v,\emptyset)$ and $B_v$ by $(B_v,\emptyset)$, respectively,
applying Theorem \ref{MVrelhom}.
\end{proof}

The image of the maps $f_p$, $g_p$, and $h_p$ of diagram
(\ref{MVdiag}) are related to the $p$th persistent \Cech homology
groups. In particular, when $p=0$, they are related to size
functions, as the following lemma formally states.

\begin{lem}\label{persCechhom}
For $(u,v)\in\Delta^+$, let $f_p, g_p, h_p$ be the maps induced by inclusions of $(A\cap B)_u$ into $(A\cap B)_v$, $(A_u, B_u)$
into $(A_v, B_v)$, and $X_u$ into $X_v$, respectively. Then $\I
f_p=\check{H}_p^{u, v}(A \cap B)$, $ \I g_p=\check{H}_p^{u, v}(A)
\oplus \check{H}_p^{u, v}(B)$, and $\I h_p=\check{H}_p^{u, v}(X)$.
In
particular, $\rank \I f_0 = \ell_{(A\cap B, \varphi_{|A\cap
B})}(u, v)$, $\rank \I g_0 = \ell_{(A, \varphi_{|A})}(u, v) +
\ell_{(B, \varphi_{|B})}(u, v)$ and  $\rank \I h_0 = \ell_{(X,
\varphi)}(u, v)$.
\end{lem}

\begin{proof}
The proof trivially follows from the definition of $p$th
persistent \Cech homology group (Definition
\ref{persistentcechhom}) and from Corollary \ref{sfhomology}.
\end{proof}

The following proposition proves that the commutativity of squares in
diagram (\ref{MVdiag}) induces a   sequence of Mayer-Vietoris of order $2$
involving the $p$th persistent \Cech homology groups of $X$, $A$,
$B$, and $A\cap B$, for every integer $p \geq 0$.

\begin{prop}\label{order2}
Let us consider the sequence of homomorphisms of persistent \Cech
homology groups
$$\begin{array}{cccccccccccc}
\cdots \rightarrow& \check{H}_{p+1}^{u, v}(X)  &
\stackrel{\Delta}{\rightarrow}& \check{H}_p^{u, v}(A \cap B)&
\stackrel{\alpha}{\rightarrow} & \check{H}_p^{u, v}(A) \oplus
\check{H}_p^{u, v}(B)& \stackrel{\beta}{\rightarrow} &
\check{H}_p^{u, v}(X)&\rightarrow \cdots
\rightarrow&\check{H}_{0}^{u, v}(X)&\rightarrow&0
\end{array}$$
where $\Delta = {\Delta_v}_{|\I h_{p+1}}$, $ \alpha =
{\alpha_v}_{|\I f_p}$, and $ \beta = {\beta_v}_{|\I g_p}$. For every integer $p\ge 0$, the
following statements hold:
\begin{enumerate}
\item $\I \Delta \subseteq \ker \alpha$; \item $\I \alpha
\subseteq \ker \beta$; \item $\I \beta \subseteq \ker \Delta$,
\end{enumerate}
that is, the sequence is of order $2$.
\end{prop}

\begin{proof}
First of all, we observe that, by Lemma \ref{commut}, $ \I \Delta
\subseteq \I f_p$, $\I\alpha\subseteq\I g_p$ and
$\I\beta\subseteq\I h_p$. Now we prove only claim {\em (i)},
considering that {\em (ii)} and {\em (iii)} can be deduced
analogously.

{\em (i)}  Let $c \in \I \Delta$. Then $c \in \I f_p$ and $c \in
\I \Delta_v = \ker \alpha_v$ in diagram (\ref{MVdiag}). Therefore
$c \in \ker \alpha$.
\end{proof}

\subsection{The size function of the union of two spaces}

In the  rest of the section we focus on the ending part of diagram
(\ref{MVdiag}):
\begin{eqnarray}\label{MVend}
\begin{array}{cccccccccc}
&\vdots&&\vdots&&\vdots&&\vdots&&\\
&\downarrow&&\downarrow&&\downarrow&&\downarrow&&\\
&&&&&&&&&\\
\cdots\rightarrow\!\!\!&\!\!\!\check{H}_1(X_u)\!\!\!&\!\!\!\stackrel{\Delta_u}{\rightarrow}\!\!\!&\!\!\!\check{H}_0((A\cap B)_u)\!\!\!&\!\!\!\stackrel{\alpha_u}{\rightarrow}\!\!\!&\!\!\!\check{H}_0(A_u)\oplus \check{H}_0(B_u)\!\!\!&\!\!\!\stackrel{\beta_u}{\rightarrow}\!\!\!&\!\!\!\check{H}_0(X_u)\!\!\!&\!\!\!\rightarrow\!\!\!&\!\!\!0\\
&&&&&&&&&\\
&\downarrow h_1&&\downarrow f_0&&\downarrow g_0&&\downarrow h_0&&\\
&&&&&&&&&\\
 \cdots\rightarrow\!\!\!&\!\!\!\check{H}_1(X_v)\!\!\!&\!\!\!
\stackrel{\Delta_v}{\rightarrow}\!\!\!&\!\!\!\check{H}_0((A\cap
B)_v)\!\!\!&\!\!\!
\stackrel{\alpha_v}{\rightarrow}\!\!\!&\!\!\!\check{H}_0(A_v)\oplus
\check{H}_0(B_v)\!\!\!&\!\!\!\stackrel{\beta_v}{\rightarrow}\!\!\!&\!\!\!\check{H}_0(X_v)\!\!\!&\!\!\!\rightarrow\!\!\!&\!\!\!0\\
&&&&&&&&&\\
&\downarrow h'_{1}&&\downarrow f'_0&&\downarrow g'_0&&\downarrow h'_0&&\\
&&&&&&&&&\\
 \cdots\rightarrow\!\!\!&\!\!\!\check{H}_{1}(X_v,X_u)\!\!\!&\!\!\!
\stackrel{\Delta{v,u}}{\rightarrow}\!\!\!&\!\!\!\check{H}_0((A\cap
B)_v,(A\cap B)_u)\!\!\!&\!\!\!
\stackrel{\alpha_{v,u}}{\rightarrow}\!\!\!&\!\!\!\check{H}_0(A_v,A_u)\oplus
\check{H}_0(B_v,B_u)\!\!\!&\!\!\!\stackrel{\beta_{v,u}}{\rightarrow}\!\!\!&\!\!\!\check{H}_0(X_v,X_u)\!\!\!&\!\!\!\rightarrow\!\!\!&\!\!\!0\\
&&&&&&&&&\\
&\downarrow&&\downarrow&&\downarrow&&\downarrow&&\\
&\vdots&&0&&0&&0&&
\end{array}
\end{eqnarray}
and, in the rest of the paper, the notations we use always refer
to diagram (\ref{MVend}).

We are now ready to deduce the relation between $\ell_{(X, \p)}$ and $ \ell_{(A, \varphi_{|A})}$,
$\ell_{(B, \varphi_{|B})}$.

\begin{thm}\label{szrelationcompl}
 For every $(u, v) \in
\Delta^+$, it holds that
\begin{eqnarray*}
\ell_{(X, \p)}(u, v) = \ell_{(A, \varphi_{|A})}(u, v) + \ell_{(B,
\varphi_{|B})}(u, v) - \ell_{(A\cap B, \varphi_{|A\cap B})}(u, v)
+\rank \ker \alpha_v - \rank \ker \alpha_{v,u}.
\end{eqnarray*}
\end{thm}

\begin{proof}
By the exactness of the second horizontal line of diagram
(\ref{MVend}) and by the surjectivity of the homomorphism
$\beta_v$, repeatedly using the dimensional relation between the
domain of a homomorphism, its kernel and its image, we obtain
\begin{eqnarray}\label{secondrow}
\rank \check{H}_0(X_v) = \rank \I \beta_v &=& \rank
\check{H}_0(A_v)\oplus\check{H}_0(B_v)- \rank \ker \beta_v\nonumber\\
&=&\rank \check{H}_0(A_v)\oplus \check{H}_0(B_v)- \rank
\I\alpha_v\nonumber\\
&=&\rank \check{H}_0(A_v)+ \rank \check{H}_0(B_v)- \rank
\check{H}_0((A\cap B)_v) + \rank \ker \alpha_v.
\end{eqnarray}
Similarly, by the exactness of the third horizontal line of the
same diagram and by the surjectivity of $\beta_{v,u}$, it holds
that
\begin{eqnarray}\label{thirdrow}
\rank \check{H}_0(X_v, X_u)=\rank \check{H}_0(A_v,A_u)+ \rank
\check{H}_0(B_v,B_u)- \rank \check{H}_0((A\cap B)_v,(A\cap B)_u) +
\rank \ker \alpha_{v,u}.
\end{eqnarray}
Now, subtracting equality (\ref{thirdrow}) from equality
(\ref{secondrow}), we have
\begin{eqnarray*}
\rank \check{H}_0(X_v) - \rank \check{H}_0(X_v, X_u)&=&\rank
\check{H}_0(A_v)- \rank \check{H}_0(A_v,A_u)+ \rank
\check{H}_0(B_v)- \rank \check{H}_0(B_v,B_u)\\&&- \rank
\check{H}_0((A\cap B)_v)+ \rank \check{H}_0((A\cap B)_v,(A\cap
B)_u) + \rank \ker \alpha_v - \rank \ker \alpha_{v,u},
\end{eqnarray*}
which is equivalent, in terms of size functions, to the relation
claimed, because of  Corollary \ref{sfhomology}.
\end{proof}

\begin{cor}\label{szrelation}
 For every $(u, v) \in
\Delta^+$, it holds that
\begin{eqnarray*}
\ell_{(X, \p)}(u, v) = \ell_{(A, \varphi_{|A})}(u, v) +
\ell_{(B, \varphi_{|B})}(u, v) - \ell_{(A\cap B, \varphi_{|A\cap
B})}(u, v)
\end{eqnarray*}
if and only if $\rank \ker \alpha_v = \rank \ker \alpha_{v,u}$.
\end{cor}

\begin{proof}
Immediate from Theorem \ref{szrelationcompl}.
\end{proof}

We now show that combining  the assumption that $\alpha_v$ and
$\alpha_{v,u}$ are both injective with Proposition \ref{order2},
there is a short exact sequence involving the $0$th persistent
\Cech homology groups of $X$, $A$, $B$, and $A\cap B$.

\begin{prop}\label{persist}
 For every $(u, v) \in
\Delta^+$, such that the maps $\alpha_v$ and $\alpha_{v,u}$ are
injective, the sequence of maps
\begin{eqnarray}\label{exactsequence}
\begin{array}{ccccccc}
0& {\rightarrow}& \check H_0^{u,v}(A\cap B)&
\stackrel{\alpha}{\rightarrow} & \check H_0^{u,v}(A)\oplus\check
H_0^{u,v}(B)& \stackrel{\beta}{\rightarrow} & \check
H_0^{u,v}(X)\rightarrow 0,
\end{array}
\end{eqnarray}
where $ \alpha = {\alpha_v}_{|\I f_0}$ and $ \beta =
{\beta_v}_{|\I g_0}$, is exact.
\end{prop}

\begin{proof}
By Proposition \ref{order2}, $\I \alpha \subseteq \ker \beta$, so
we only have to show that $\beta$ is surjective, $\alpha$ is
injective, and $\rank\I \alpha=\rank\ker \beta$.

We recall that $\check H_0^{u,v}(A\cap B)=\I f_0$, $\check
H_0^{u,v}(A)\oplus\check H_0^{u,v}(B)=\I g_0$, and $\check
H_0^{u,v}(X)=\I h_0$ (Lemma \ref{persCechhom}).

We begin by showing that $\beta$ is surjective. Let $c \in \I
h_0$. There exists $d \in \check{H}_0(X_u)$ such that $h_0(d) =
c$. Since $\beta_u$ is surjective, there exists $d' \in
\check{H}_0(A_u)\oplus \check{H}_0(B_u)$ such that $h_0\circ
\beta_u(d') = c$. By Lemma \ref{commut}, $\beta_v\circ g_0(d') =
c$. Thus, taking $c' = g_0(d')$, we immediately have $\beta(c') =
c$.

As for the injectivity of $\alpha$, the claim is immediate because
$\ker\alpha\subseteq \ker\alpha_{v}$ and we are assuming
$\alpha_v$ injective.

Now we have to show that  $\rank\I \alpha=\rank \ker \beta$. In
order to do so, we observe that for every $(u,v)\in \Delta^+$ it
holds that
\begin{eqnarray}\label{generalRelation}
\ell_{(X, \p)}(u, v)=\rank \check H_0^{u,v}(X)=\rank\I \beta=\rank
H_0^{u,v}(A)\oplus\check H_0^{u,v}(B)-\rank\ker\beta \nonumber\\
=\ell_{(A, \varphi_{|A})}(u, v) + \ell_{(B, \varphi_{|B})}(u,
v)-\rank\ker\beta.
\end{eqnarray}
On the other hand,  by Corollary \ref{szrelation}, when
$\rank\ker\alpha_v=\rank\ker\alpha_{v,u}$ it holds that
$$\ell_{(X, \p)}(u, v) = \ell_{(A, \varphi_{|A})}(u, v) + \ell_{(B,
\varphi_{|B})}(u, v) - \ell_{(A\cap B, \varphi_{|A\cap B})}(u,
v).$$ Hence, if $\rank\ker\alpha_v=\rank\ker\alpha_{v,u}$, then
$\rank\ker\beta=\ell_{(A\cap B, \varphi_{|A\cap B})}(u, v)$.
Moreover, since $\ell_{(A\cap B, \varphi_{|A\cap B})}(u, v)=\rank
\check H_0^{u,v}(A\cap B)=\rank\ker\alpha+\rank\I\alpha$, when
$\rank\ker\alpha_v=\rank\ker\alpha_{v,u}$, we have
$\rank\ker\beta=\rank\ker\alpha+\rank\I\alpha$. Therefore, when
$\rank\ker\alpha_v=\rank\ker\alpha_{v,u}$, $\alpha$ is injective
if and only if  $\rank\I \alpha=\rank\ker \beta$.
\end{proof}

The condition $\rank\ker\alpha_v=\rank\ker\alpha_{v,u}=0$ in the
previous Proposition \ref{persist} cannot be weakened, in fact:

\begin{rem}\label{controes}
The equality $\rank\ker\alpha_v=\rank\ker\alpha_{v,u}$ does not
imply the injectivity of $\alpha$.
\end{rem}
\begin{figure}[htbp]
\begin{center}
\includegraphics[height=4cm]{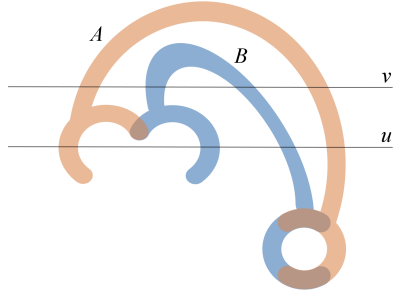}
\end{center}
\caption{\footnotesize{The sets $A$ and $B$ used in Remark
\ref{controes}.}} \label{alphaNoInjec}
\end{figure}
Indeed, Figure \ref{alphaNoInjec} shows an example of a
topological space $X = A \cup B$ on which, taking $\p$ equal to
the height function and $u, v \in \R$ as displayed, it holds that
$\rank\ker\alpha_v=\rank\ker\alpha_{v,u}\ne 0$, but $\rank \ker
\alpha >0$, making the sequence (\ref{exactsequence}) not exact.
To see this, we note that the equalities (\ref{secondrow}) and
(\ref{thirdrow}) imply $\rank \ker \alpha_{v}=\rank \ker
\alpha_{v,u}=1$.

As far as the homomorphism $\alpha$ is concerned, let us consider
the homology sequence of the pair $(X_v, X_u)$
\begin{eqnarray*}
\begin{array}{ccccccccccc}
\cdots& {\rightarrow}& \check H_2(X_v, X_u)&
\stackrel{h}{\rightarrow} &\check H_1(X_u)&
\stackrel{h_1}{\rightarrow}& \check H_1(X_v)&
\stackrel{h'_1}{\rightarrow} & \check H_1(X_v, X_u)& {\rightarrow}
& \cdots
\end{array}
\end{eqnarray*}
that is, the leftmost vertical line in diagram (\ref{MVend}). In
this instance, $\check H_2(X_v,X_u)=0$, so it follows that $h_1$
is injective. Now, recalling that, by Proposition \ref{order2},
$\I\Delta\subseteq\ker\alpha$, where $\Delta=\Delta_{v|\I h_1}$,
the triviality of $\check H_1(A_v)\oplus \check H_1(B_v)$, implies
that $\Delta_v$ is surjective. So, since $\rank \check
H_1(X_u)=1$, it follows that $\rank \I\Delta=1$, and hence
$\ker\alpha = 1$ because $\ker\alpha\subseteq\ker\alpha_v$.

As shown in the proof of Proposition \ref{persist}, it holds that
$\ell_{(X, \p)}(u, v)=\ell_{(A, \varphi_{|A})}(u, v) + \ell_{(B,
\varphi_{|B})}(u, v)-\rank\ker\beta$ for every $(u, v) \in
\Delta^+$ (see equality (\ref{generalRelation})). So, as an
immediate consequence, we observe that
\begin{rem}
For every $(u, v) \in \Delta^+$, it holds that $\ell_{(X, \p)}(u,
v) \le  \ell_{(A, \varphi_{|A})}(u, v) + \ell_{(B,
\varphi_{|B})}(u, v)$.
\end{rem}

\subsection{Examples}\label{examples}
In this section, we give two examples illustrating the previous results.

In both these examples, we consider a ``double open-end wrench''
shape $A$, partially occluded   by another shape $B$, resulting in
different shapes $X = A \cup B \subset \R^2$. The size functions
$\ell_{(A, \varphi_{|A})}$, $\ell_{(B, \varphi_{|B})}$,
$\ell_{(A\cap B, \varphi_{|A\cap B})}$, $\ell_{(X, \p)}$ are
computed taking $\p:X\rightarrow \R$, $\p(P)=-\|P-H\|$, with $H$ a
fixed point in $\R^2$.

\begin{figure}[htbp]
\begin{center}
\psfrag{A}{$A$} \psfrag{B}{$B$} \psfrag{AuB}{$A\cup B$}
\psfrag{AnB}{$A\cap B$} \psfrag{AnB}{$A\cap B$} \psfrag{u}{$u$}
\psfrag{v}{$v$} \psfrag{0}{$0$} \psfrag{1}{$1$} \psfrag{2}{$2$}
\psfrag{3}{$3$} \psfrag{4}{$4$} \psfrag{-a}{$-a$}
\psfrag{-b}{$-b$} \psfrag{-c}{$-c$} \psfrag{-d}{$-d$}
\psfrag{a}{$a$} \psfrag{b}{$b$} \psfrag{c}{$c$} \psfrag{d}{$d$}
\psfrag{H}{$H$} \psfrag{D+}{$\Delta^+$}
\psfrag{la}{$\ell_{(A,\p_{|A})}$}
\psfrag{lb}{$\ell_{(B,\p_{|B})}$} \psfrag{laub}{$\ell_{(A\cup
B,\p)}$} \psfrag{lanb}{$\ell_{(A\cap B,\p_{|A\cap B})}$}
\begin{tabular}{cc}
\includegraphics[height=4.2cm]{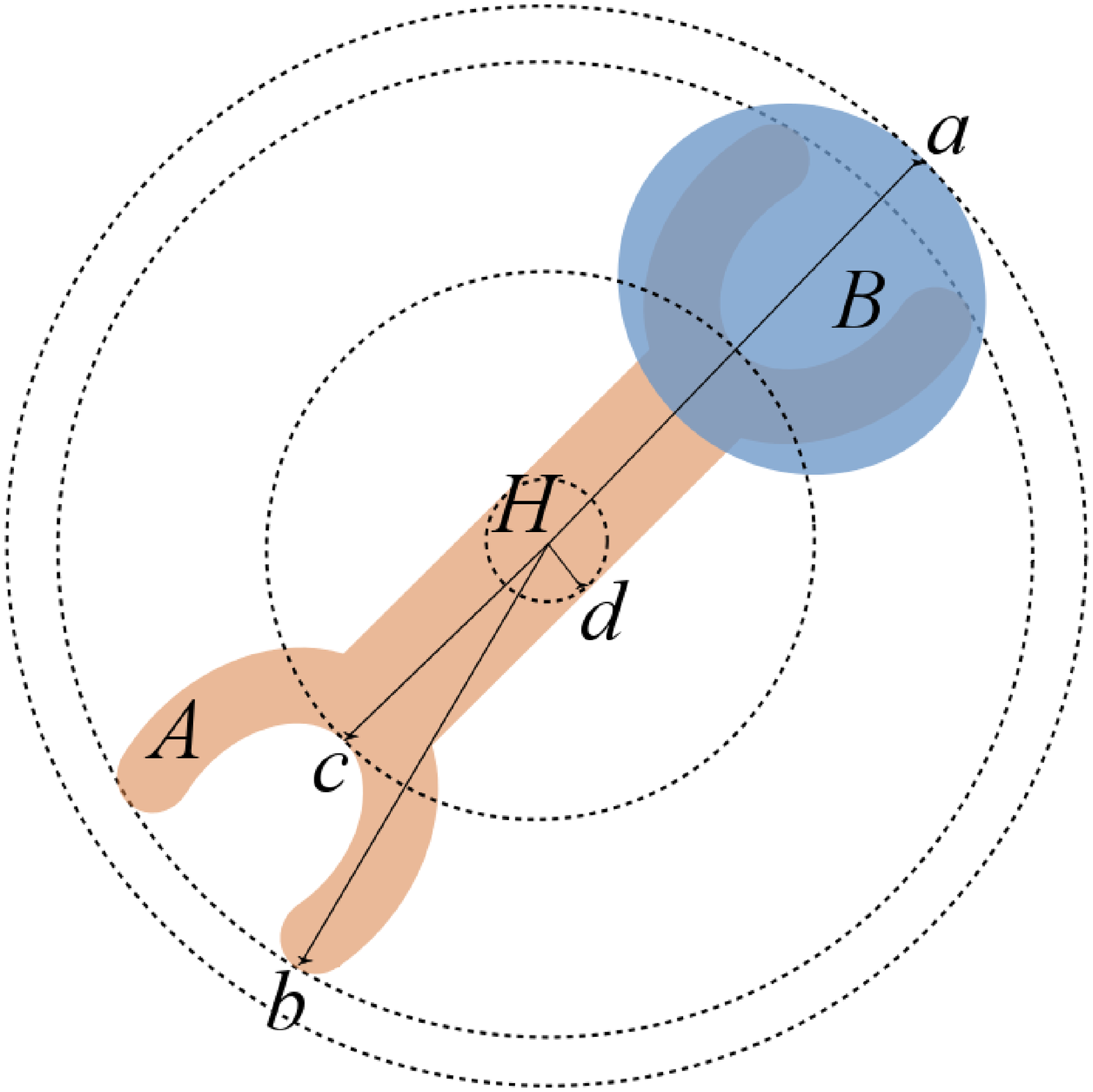}&
\includegraphics[height=4cm]{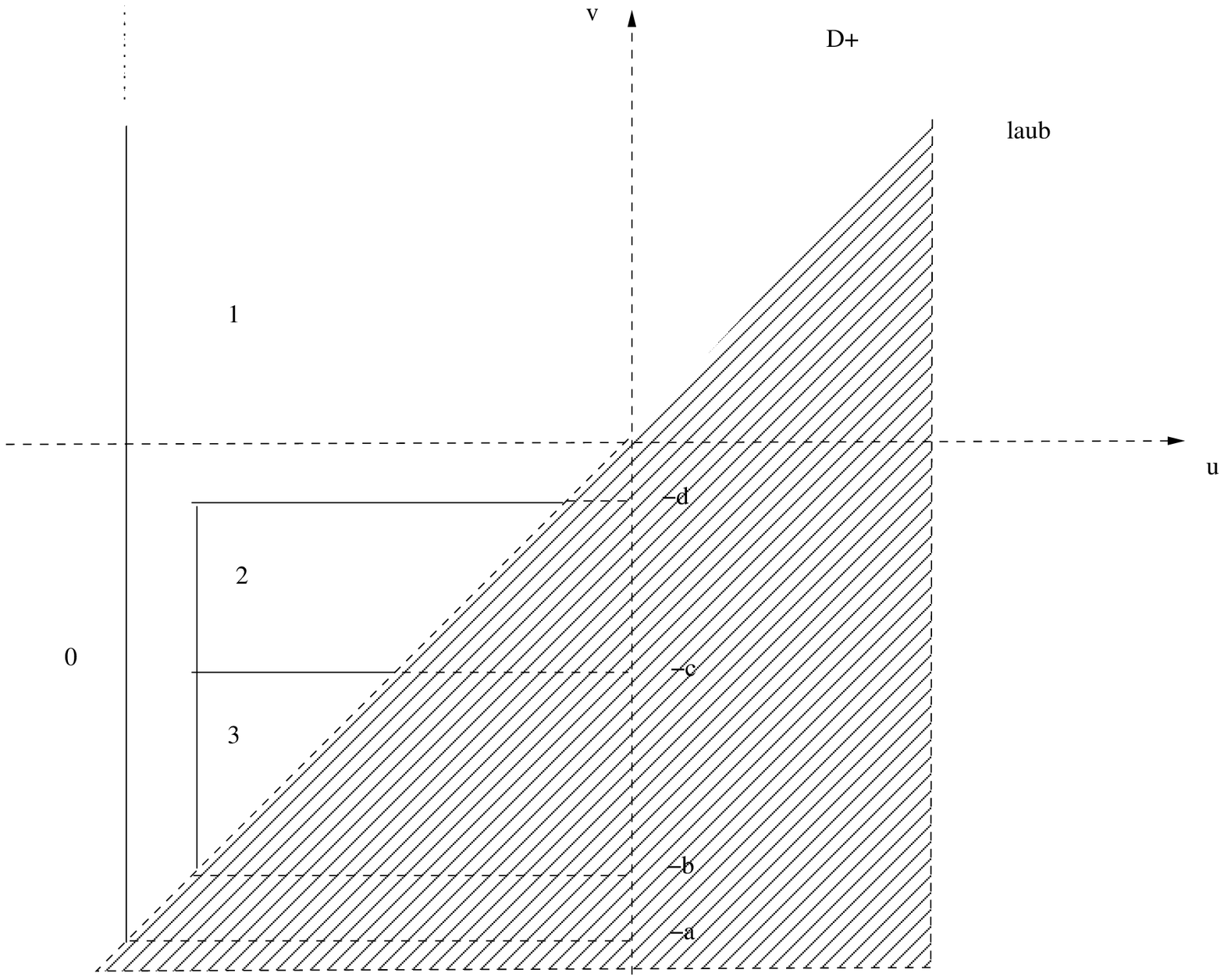}\\
$(a)$&$(b)$\\
\end{tabular}
\begin{tabular}{ccc}
\includegraphics[height=4cm]{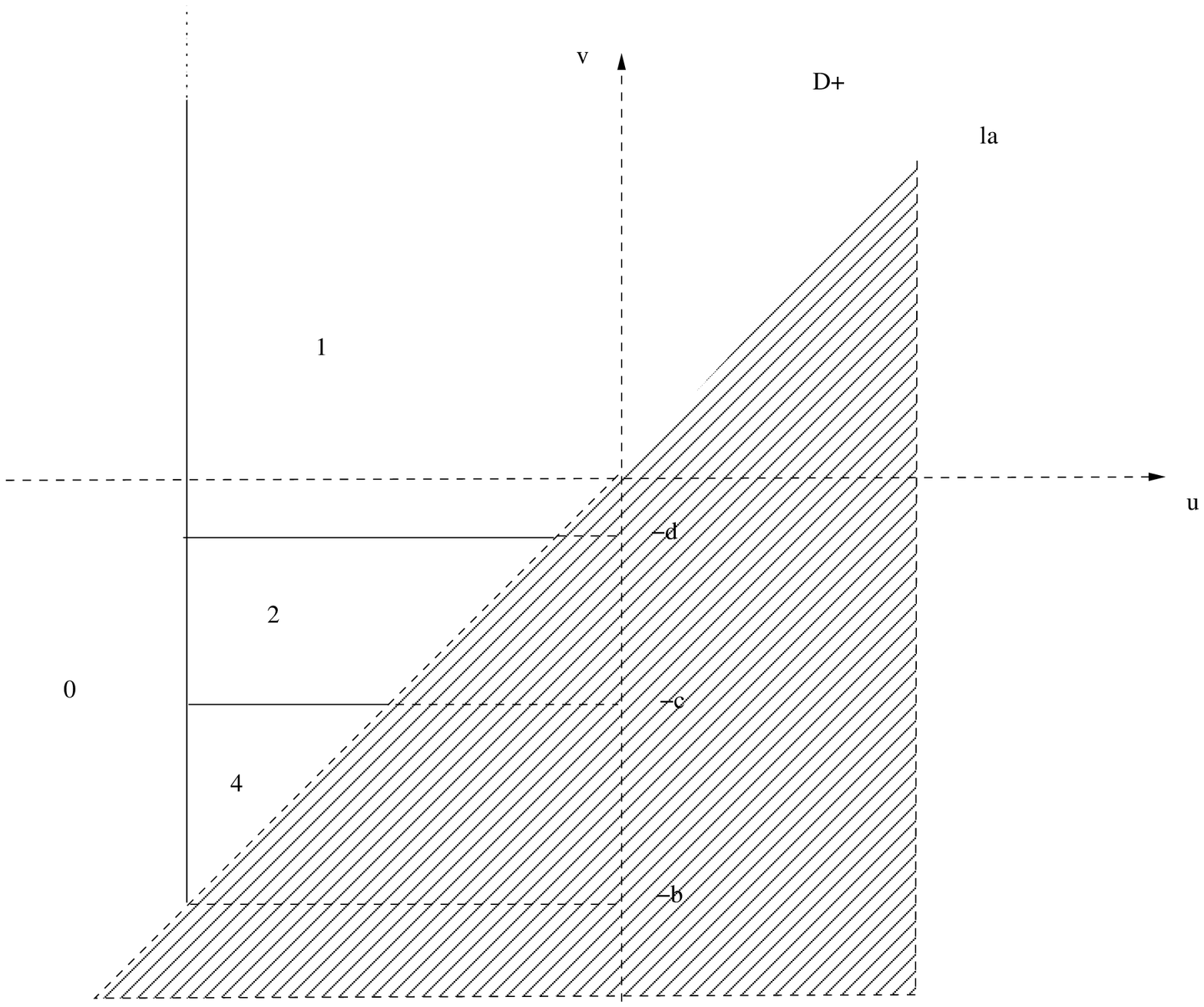}&
\includegraphics[height=4cm]{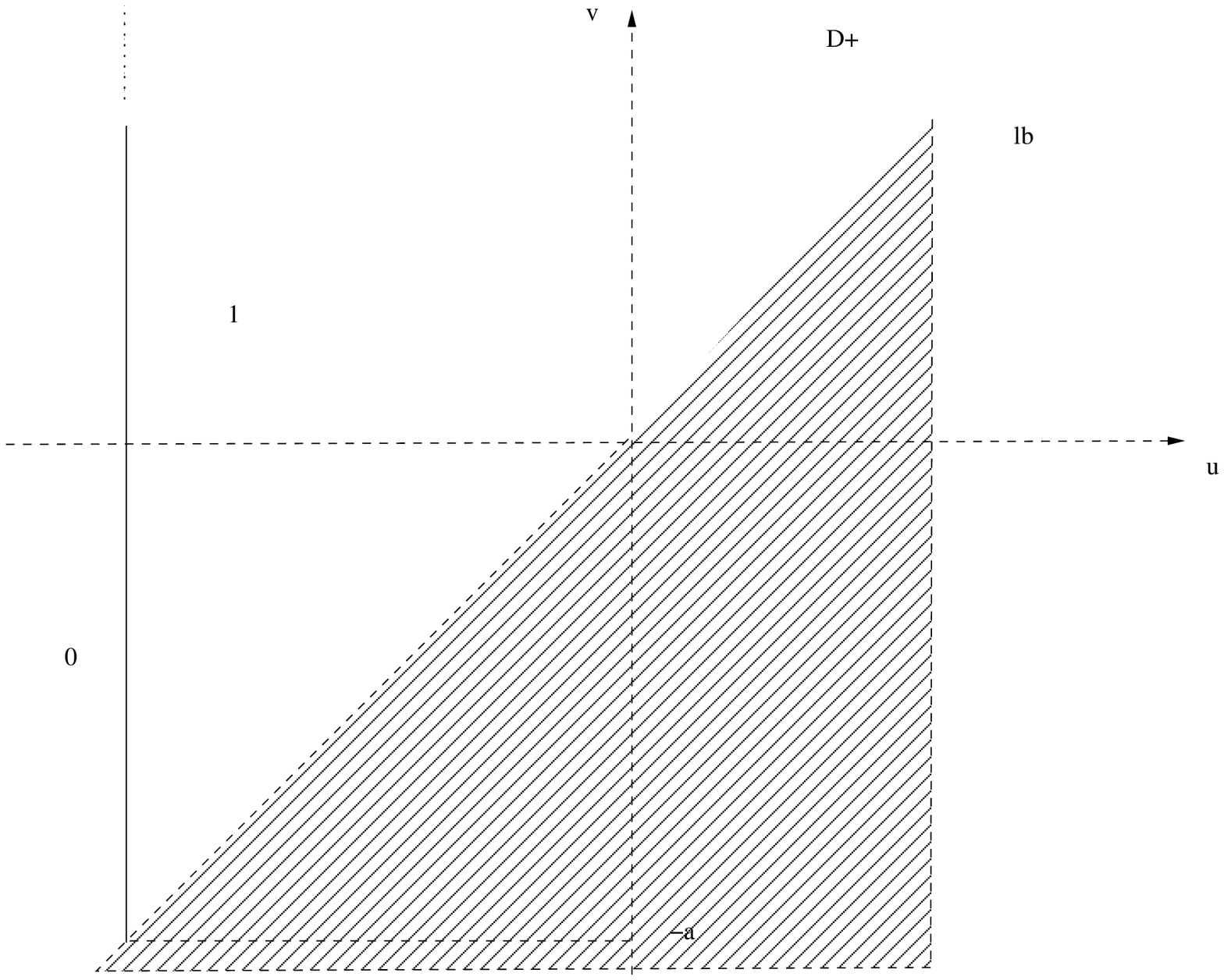}&
\includegraphics[height=4cm]{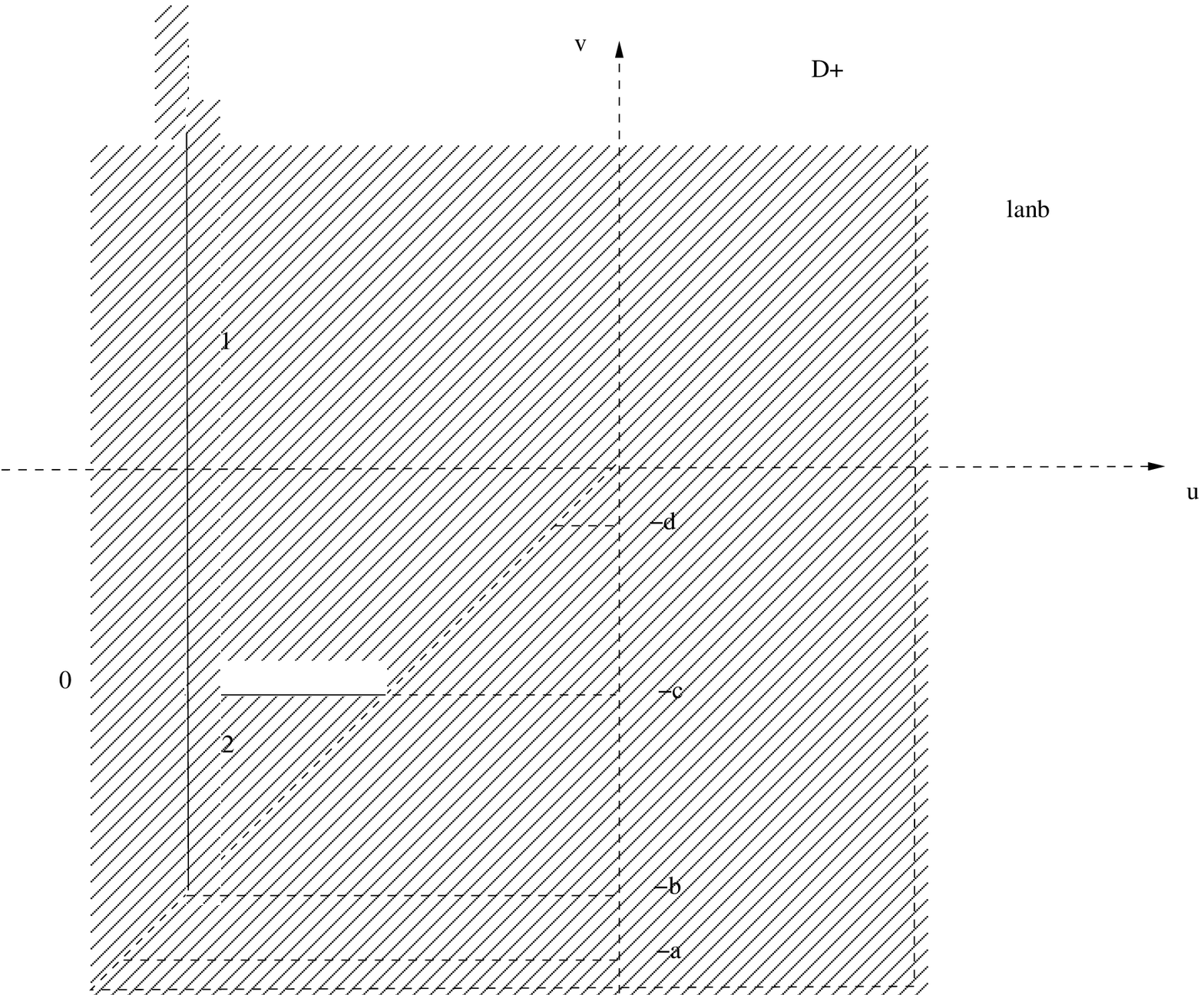}\\
$(c)$&$(d)$&$(e)$\\
\end{tabular}
\caption{\footnotesize{In $(a)$ a ``double open-end wrench'' shape
$A$ is occluded by another shape $B$. In $(b)$, $(c)$, $(d)$ and
$(e)$ we show the size functions of $(A \cup B, \p)$,
$(A,\p_{|A})$, $(B,\p_{|B})$, and $(A\cap B,\p_{|A\cap B})$,
respectively, computed taking $\p:X\rightarrow \R$,
$\p(P)=-\|P-H\|$. In this example the relation $\ell_{(X, \p)}=
\ell_{(A, \varphi_{|A})}+ \ell_{(B, \varphi_{|B})} - \ell_{(A\cap
B, \varphi_{|A\cap B})}$ of Corollary \ref{szrelation} holds
everywhere in $\Delta^+$.}} \label{figura:osso1}
\end{center}
\end{figure}
In the first example, shown in Figure \ref{figura:osso1}, the
relation $\ell_{(X, \p)}(u, v) = \ell_{(A, \varphi_{|A})}(u, v) +
\ell_{(B, \varphi_{|B})}(u, v) - \ell_{(A\cap B, \varphi_{|A\cap
B})}(u, v)$, given in Corollary \ref{szrelation}, holds for every
$(u, v) \in \Delta^+$.
\begin{figure}[htbp]
\begin{center}
\psfrag{A}{$A$} \psfrag{B}{$B$} \psfrag{AuB}{$A\cup B$}
\psfrag{AnB}{$A\cap B$} \psfrag{u}{$u$} \psfrag{v}{$v$}
\psfrag{0}{$0$} \psfrag{1}{$1$} \psfrag{2}{$2$} \psfrag{3}{$3$}
\psfrag{4}{$4$} \psfrag{-a}{$-a$} \psfrag{-b}{$-b$}
\psfrag{-c}{$-c$} \psfrag{-d}{$-d$} \psfrag{a}{$a$}
\psfrag{b}{$b$} \psfrag{c}{$c$} \psfrag{d}{$d$} \psfrag{H}{$H$}
\psfrag{D+}{$\Delta^+$} \psfrag{la}{$\ell_{(A,\p_{|A})}$}
\psfrag{lb}{$\ell_{(B,\p_{|B})}$} \psfrag{laub}{$\ell_{(A\cup
B,\p)}$} \psfrag{lanb}{$\ell_{(A\cap B,\p_{|A\cap B})}$}
\begin{tabular}{ccc}
\includegraphics[height=4.2cm]{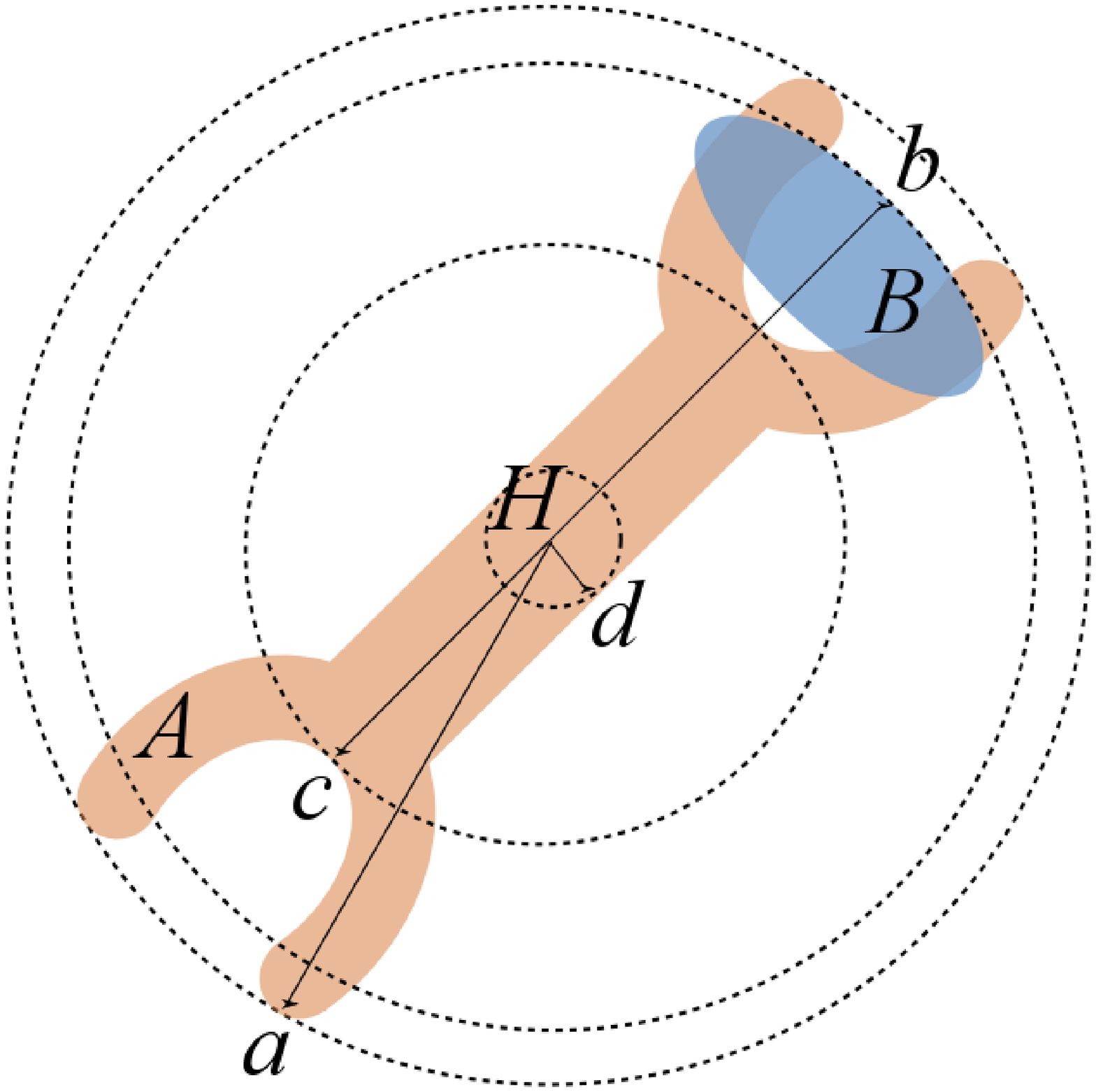}&
\includegraphics[height=4cm]{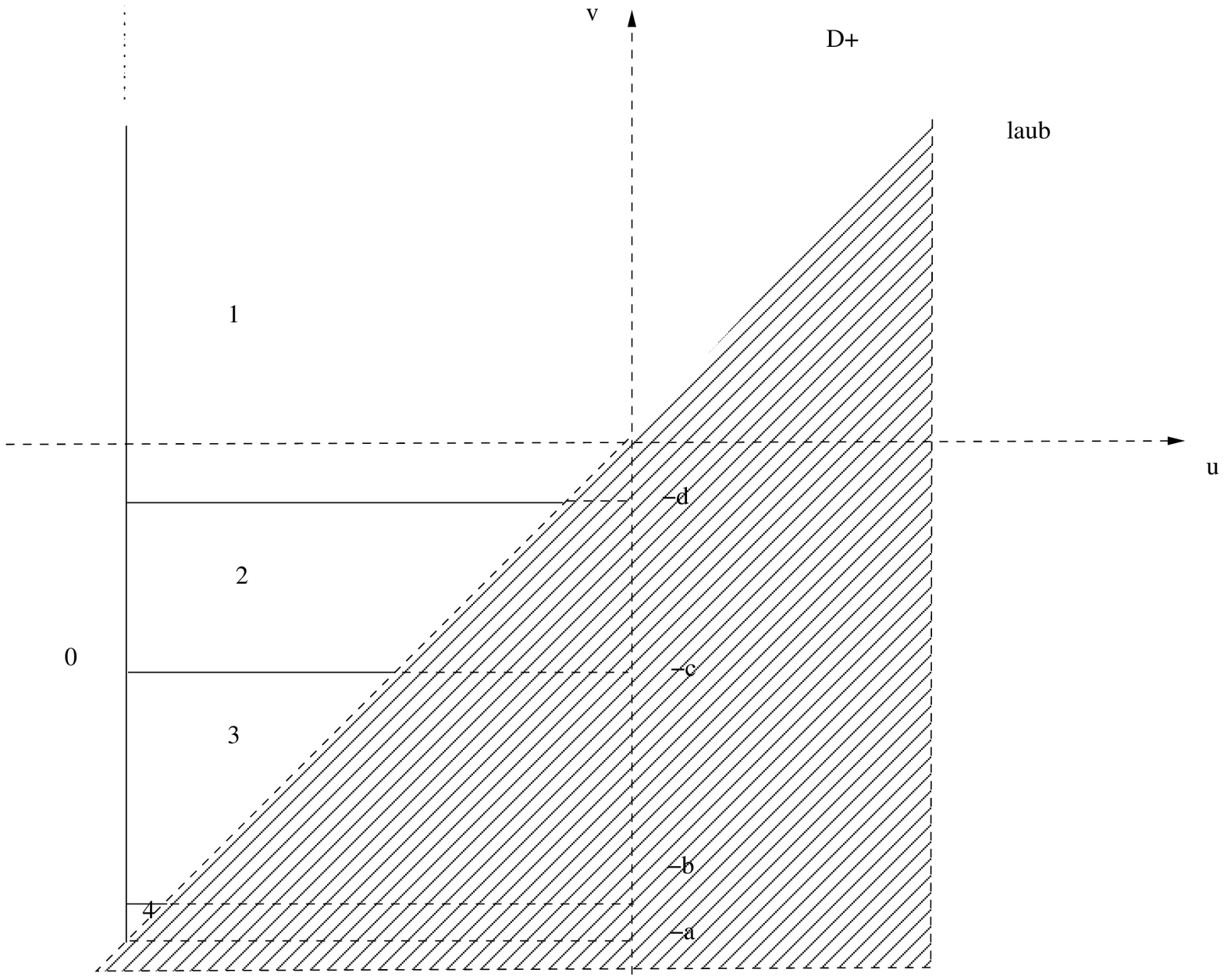}&
\includegraphics[height=4cm]{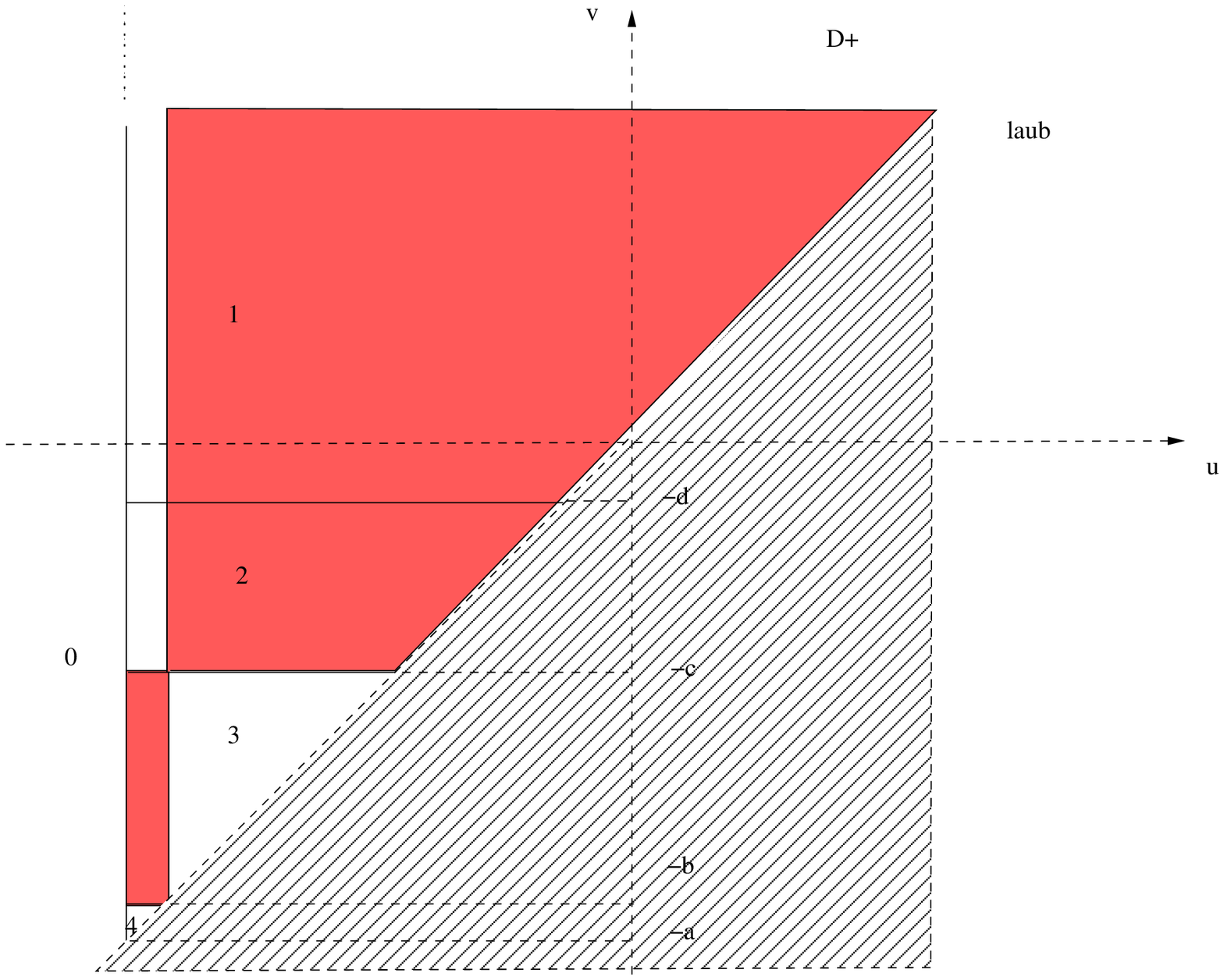}\\
$(a)$&$(b)$&$(c)$\\
\end{tabular}
\begin{tabular}{ccc}
\includegraphics[height=4cm]{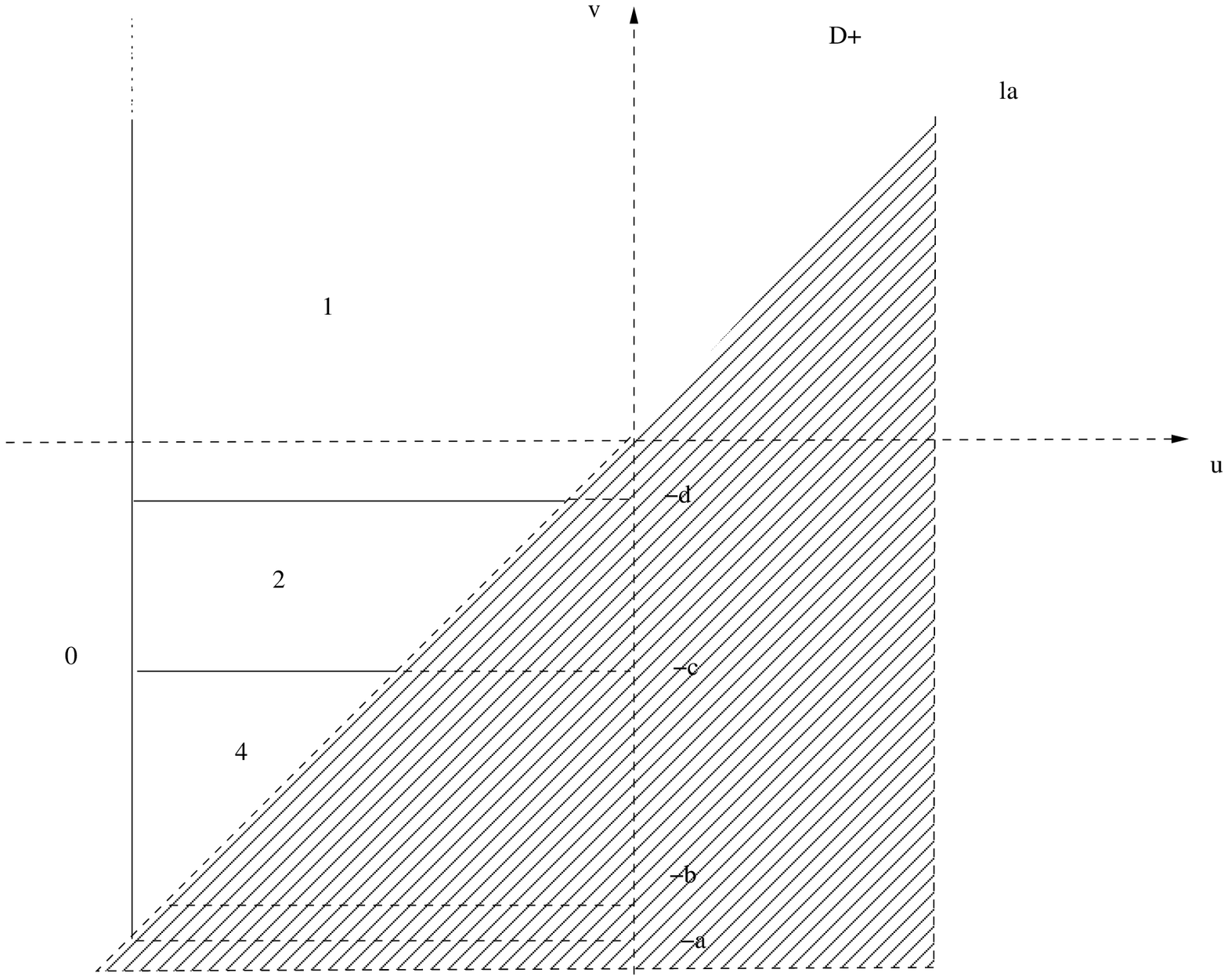}&
\includegraphics[height=4cm]{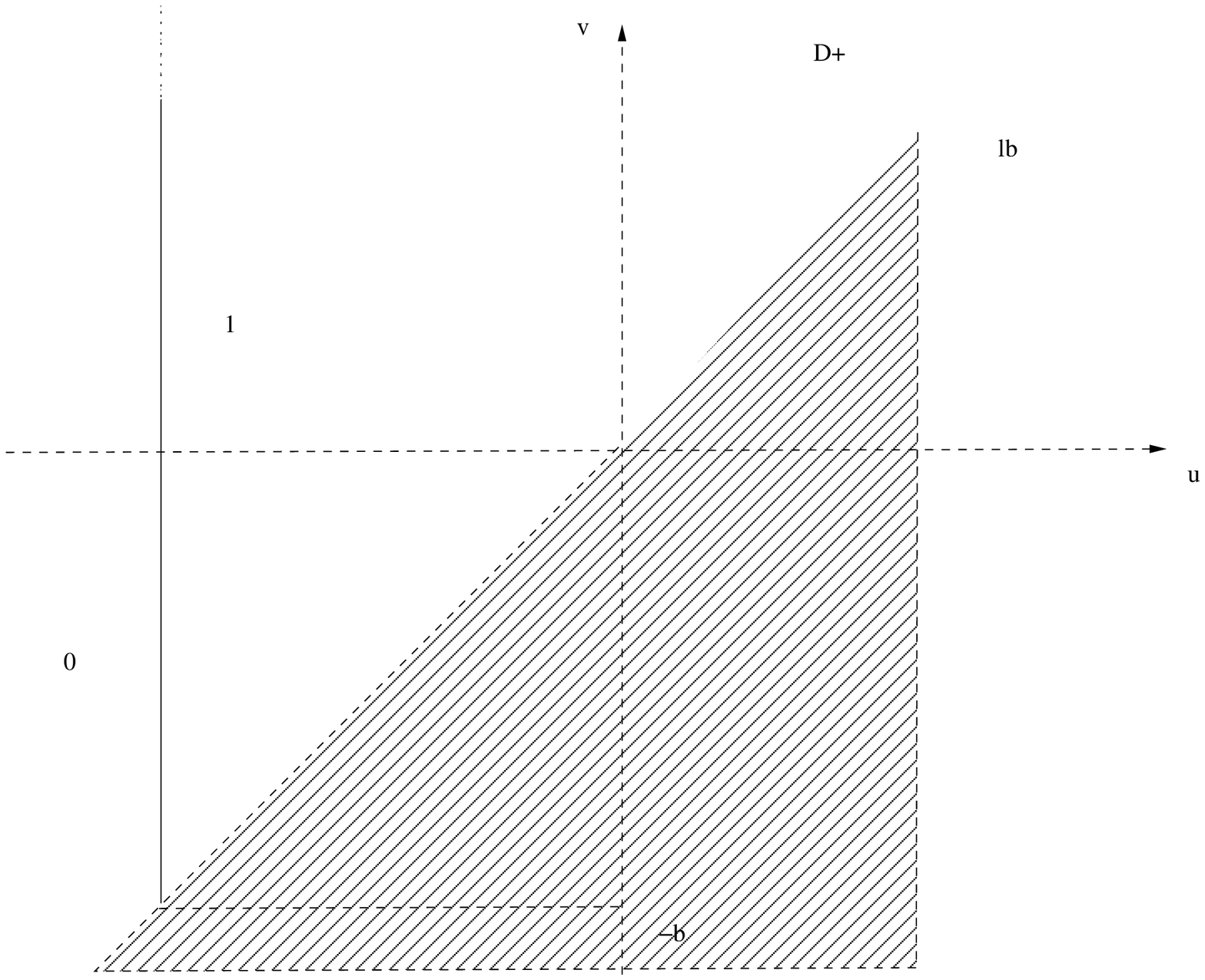}&
\includegraphics[height=4cm]{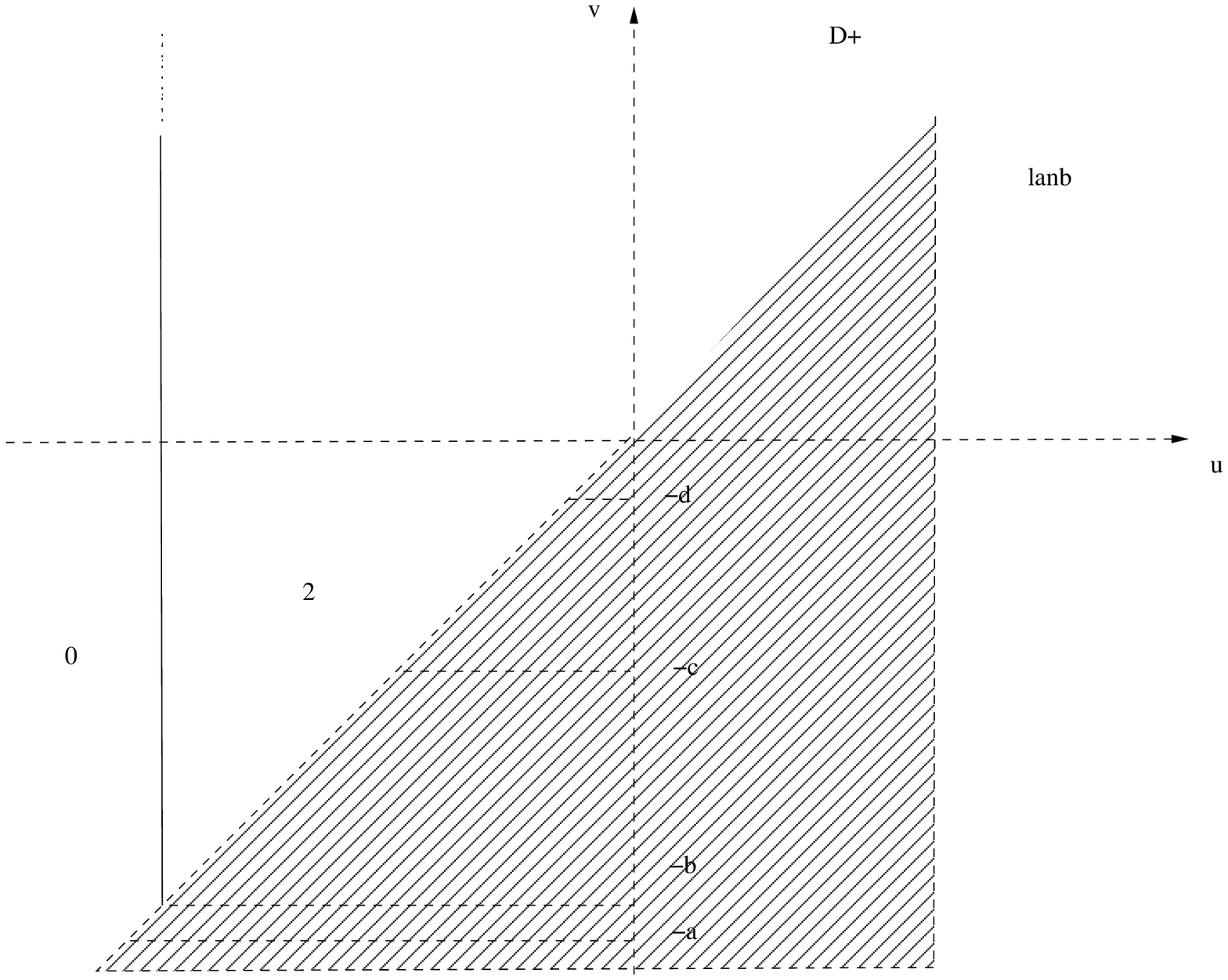}\\
$(d)$&$(e)$&$(f)$\\
\end{tabular}
\caption{\footnotesize{In $(a)$ the same ``double open-end
wrench'' shape $A$ as in Figure \ref{figura:osso1} is considered
together with  a different occluding shape $B$.  In $(b)$, $(d)$,
$(e)$, $(f)$ we display the size functions of $(A \cup B, \p)$,
$(A,\p_{|A})$, $(B,\p_{|B})$, and $(A\cap B,\p_{|A\cap B})$,
respectively, computed taking $\p:X\rightarrow \R$,
$\p(P)=-\|P-H\|$. In this case the relation $\ell_{(X, \p)}=
\ell_{(A, \varphi_{|A})}+ \ell_{(B, \varphi_{|B})} - \ell_{(A\cap
B, \varphi_{|A\cap B})}$ of Corollary \ref{szrelation} does not
hold everywhere in $\Delta^+$. In $(c)$ we underline the regions
of $\Delta^+$ where the equality is not valid by coloring them.}}
\label{figura:osso4}
\end{center}
\end{figure}
In the second example, shown in Figure \ref{figura:osso4}, a
deformation of occluding shape $B$ in Figure \ref{figura:osso1}
makes the relation given in Corollary \ref{szrelation} not valid
everywhere in $\Delta^+$.

More precisely, the condition
$\rank\ker\alpha_v=\rank\ker\alpha_{v,u}=1$ holds for every $(u,
v) \in  \Delta^+$, with $-a\le u < -b$ and $-c\le v$, whereas the
condition $\rank\ker\alpha_v=\rank\ker\alpha_{v,u}=0$ holds for
every $(u, v) \in \Delta^+$ with $u<-a$, for every $(u, v) \in
\Delta^+$ with $-a\le u<v<-b$, and  for every $(u, v) \in
\Delta^+$ with $-b\le u<v<-c$. Therefore, in these regions,
$\ell_{(X, \p)}(u,v)= \ell_{(A, \varphi_{|A})}(u,v)+ \ell_{(B,
\varphi_{|B})}(u,v) - \ell_{(A\cap B, \varphi_{|A\cap B})}(u,v)$.
In the remaining regions of $\Delta^+$,  this relation does not
hold.  In particular, for every $(u, v) \in  \Delta^+$ with $-a\le
u<-b$ and $-b\le v<-c$ we have $\rank \ker \alpha_v=0$ and
$\rank\ker\alpha_{v,u}=1$, yielding $\ell_{(X, \p)}(u,v)<
\ell_{(A, \varphi_{|A})}(u,v)+ \ell_{(B, \varphi_{|B})}(u,v) -
\ell_{(A\cap B, \varphi_{|A\cap B})}(u,v)$, while for every $(u,
v) \in \Delta^+$ with $-b\le u$ and $-c\le v$ we have $\rank
\ker\alpha_v=1$ and $\rank\ker\alpha_{v,u}=0$, yielding $\ell_{(X,
\p)}(u,v)> \ell_{(A, \varphi_{|A})}(u,v)+ \ell_{(B,
\varphi_{|B})}(u,v) - \ell_{(A\cap B, \varphi_{|A\cap B})}(u,v)$.
To simplify the visualization of the regions of $\Delta^+$ in
which the equality holds, the reader can refer to Figure
\ref{figura:osso4} $(c)$, where $\szf$ is displayed using white
for points $(u,v) \in \Delta^+$ that verify $\ell_{(X, \p)}(u,v)=
\ell_{(A, \varphi_{|A})}(u,v)+ \ell_{(B, \varphi_{|B})}(u,v) -
\ell_{(A\cap B, \varphi_{|A\cap B})}(u,v)$ and red for the other
ones.

\subsection{Conditions for the exactness of   $0 \rightarrow \check H_0^{u,v}(A\cap B)
\rightarrow \check H_0^{u,v}(A)\oplus\check H_0^{u,v}(B)
\rightarrow \check H_0^{u,v}(X)\rightarrow 0$}\label{suffCond}

In this section we look for sufficient conditions in order that
$ \alpha_v$ and $\alpha_{v,u}$ are injective, so that  the sequence
$$\begin{array}{ccccccc}
0& {\rightarrow}& \check H_0^{u,v}(A\cap B)&
\stackrel{\alpha}{\rightarrow} & \check H_0^{u,v}(A)\oplus\check
H_0^{u,v}(B)& \stackrel{\beta}{\rightarrow} & \check
H_0^{u,v}(X)\rightarrow 0
\end{array}$$
is exact (cf. Proposition \ref{persist}), and the relation
$\ell_{(X, \p)}(u,v)= \ell_{(A, \varphi_{|A})}(u,v)+ \ell_{(B,
\varphi_{|B})}(u,v) - \ell_{(A\cap B, \varphi_{|A\cap B})}(u,v)$
of Corollary \ref{szrelation} is satisfied.

The reason for identifying these conditions lies in the fact that
they can be used as a guidance in choosing the most appropriate
measuring function in order to study the shape of a partially
occluded object.

The first condition we exhibit (Theorem \ref{suff0}), relates the
exactness of the above sequence to the values taken by the size
function $\ell_{(A\cap B,\p_{|A\cap B})}$. Roughly speaking, it
indicates that the fewer the number of cornerpoints  in the size
function of $A\cap B$, the larger the region of $\Delta^+$ where
the above sequence is necessarily exact. We underline that this is
only a sufficient condition, as the examples in Section
\ref{examples} easily show.

The sketch of proof is the following. We begin by showing that the
surjectivity of $f_0$ is a sufficient condition, ensuring that
$\alpha_{v,u}$ is injective. Then we note that, for points $(
u,v)\in\Delta^+$ where the size function of $A\cap B$  has  no
cornerpoints  in the upper right region  $\{(u',v')\in \Delta^+:\
u\le u'\le v, v'>v\}$,  $f_0$ is necessarily surjective. So we
obtain a condition on $\ell_{(A\cap B,\p_{|A\cap B})}(u, v)$ such
that $\alpha_{v,u}$ is injective.  Finally, showing that if
$\ell_{(A\cap B,\p_{|A\cap B})}(u, v)\le 1$, then $\alpha_v$ is
injective, we prove the claim of Theorem \ref{suff0}.

\begin{lem}\label{f_0}
 Let $ \alpha = {\alpha_v}_{|\I f_0}$ and $ \beta =
{\beta_v}_{|\I g_0}$. If $f_0$ is surjective, then
  $\I \alpha = \ker
\beta$ and $\alpha_{v,u}=0$.
\end{lem}

\begin{proof}
By Proposition \ref{order2}{\em (ii)}, $\I\alpha\subseteq
\ker\beta$, so we need to prove that $\ker
\beta\subseteq\I\alpha$. Let $c \in \ker \beta \subseteq \ker
\beta_v$. Since $\I \alpha_v = \ker \beta_v$, there exists $d \in
\check{H}_0((A\cap B)_v)$ such that $\alpha_v (d) = c$. By
hypothesis, $f_0$ is surjective, so $\check{H}_0((A\cap B)_v) = \I
f_0$. Hence $d \in \I f_0$, implying $\alpha(d) = c$. Thus, $c \in
\I \alpha$, and hence $\I \alpha = \ker \beta$.

Let us now show that  $\alpha_{v,u}$ is trivial. By observing
diagram (\ref{MVend}), we see that $f_0$ is surjective if and only
if $f_0'$ is trivial. Since $f_0'$ is surjective, it holds that
$f_0$ is surjective if and only if $\check H_0((A\cap B)_v,(A\cap
B)_u)=0$. Therefore, if $f_0$ is surjective, then
$\alpha_{v,u}=0$.
\end{proof}

\begin{thm}\label{suff0}
Let $(u, v) \in \Delta^+$. If $\ell_{(A\cap B,\p_{|A\cap B})}(u, v') = \ell_{(A\cap B,\p_{|A\cap
B})}(v, v')\le 1$, for every $(v,v')\in \Delta^+$, then $ \ker \alpha_v=\ker\alpha_{v,u}=0$.
\end{thm}

\begin{proof}
From  $\ell_{(A\cap B,\p_{|A\cap B})}(u, v') = \ell_{(A\cap
B,\p_{|A\cap B})}(v, v')$, applying Proposition \ref{surj} with
$A\cap B$ in place of $X$ and $f_0$ in place of $\iota_0^{u,v}$,
it follows that $f_0$ is surjective. Hence, by Lemma \ref{f_0}, we
have that $\alpha_{v,u}$ is trivial.

Let us now prove that $\alpha_v$ is injective. From the assumption
$\ell_{(A\cap B,\p_{|A\cap B})}(v, v')\le 1$, for every $(v,v')\in
\Delta^+$, we deduce that either $(A\cap B)_v$ is empty or $(A\cap
B)_v$ is non-empty and connected. If $(A\cap B)_v$ is empty, then
$\check H_0((A\cap B)_v)$ is trivial and the claim is proved. Let
us consider the case when $(A\cap B)_v$ is non-empty and
connected. Let $z_0=\{z_0({\mathcal U}_{(A\cap B)_v})\}\in \check
H_0((A\cap B)_v)$. If $z_0\in \ker\alpha_v = \I \Delta_v$, for
each $z_0({\mathcal U}_{(A\cap B)_v})\in H_0({\mathcal U}_{(A\cap
B)_v})$ there is a $1$-chain $c_1({\mathcal U}_{A_v})$ on $A_v$
and a $1$-chain $c_1({\mathcal U}_{B_v})$ on $B_v$, such that the
homology class of $\partial c_1({\mathcal U}_{A_v})=-\partial
c_1({\mathcal U}_{B_v})$ is  equal to $z_0({\mathcal U}_{(A\cap
B)_v})$, up to  homomorphisms induced by inclusion. We now show
that   $\partial c_1({\mathcal U}_{A_v})$ is a boundary on $(A\cap
B)_v$. This will prove that $z_0({\mathcal U}_{(A\cap B)_v})$ is
trivial, yielding the injectivity of $\alpha_v$. If $c_1({\mathcal
U}_{A_v})=\sum_{i=1}^n a_i \cdot <U_i^0,U_i^1>$, then $\partial
c_1({\mathcal U}_{A_v})=\sum_{i=1}^n a_i\cdot U_i^1-\sum_{i=1}^n
a_i \cdot U_i^0$. From  $\partial c_1({\mathcal
U}_{A_v})=-\partial c_1({\mathcal U}_{B_v})$, we deduce that, for
$i=1,\ldots, n$, both $U_i^0$ and $U_i^1$ intersect $(A\cap B)_v$.
By Proposition \ref{simplechain}, the connectedness of $(A\cap
B)_v$ implies that there is a simple chain on $(A\cap B)_v$
connecting  $U_i^0$ and $U_i^1$, for $i=1,\ldots, n$. Therefore
$\partial c_1({\mathcal U}_{A_v})$ is a boundary on $(A\cap B)_v$.
\end{proof}

We conclude by observing that other sufficient conditions exist,
implying that both $\alpha_v$ and $\alpha_{v,u}$ are injective. An
example is given by the following result.

\begin{prop}\label{suff3}
If $\rank \check H_1(X_v)=0$ and $\rank \check H_0(X_u)= \ell_{(X,
\varphi)}(u, v)$, then $\ker \alpha_v=\ker \alpha_{v,u}=0$.
\end{prop}

\begin{proof}
The condition $\rank \check H_1(X_v)=0$ trivially implies $\ker
\alpha_v=0$. On the other hand, it implies the injectivity of the
homomorphism $h$ in the following exact sequence:
$$\begin{array}{ccccccccccc}
\cdots& {\rightarrow}& \check H_1(X_v)&
\stackrel{h'_1}{\rightarrow} & \check H_1(X_v, X_u)&
\stackrel{h}{\rightarrow} & \check
H_0(X_u)&\stackrel{h_0}{\rightarrow} & \check H_0(X_v)&
\stackrel{h'_0}{\rightarrow} & \check H_0(X_v, X_u)\rightarrow 0,
\end{array}$$
which is the leftmost vertical sequence in diagram (\ref{MVend}).
Therefore, by the assumption $\rank \check H_0(X_u)= \ell_{(X,
\varphi)}(u, v)$, it follows that
$$\rank \check H_1(X_v, X_u)= \rank \I h = \rank \ker h_0 = \rank \check H_0(X_u)- \ell_{(X,
\varphi)}(u, v) = 0,$$ and, consequently, the triviality of $\ker
\alpha_{v,u}$ has been proved.
\end{proof}

\section{Partial matching of cornerpoints in size functions}

As recalled in Section \ref{background}, in an earlier paper
\cite{FrLa01}, it was shown that size functions can be concisely
represented by collections of points, called cornerpoints, with
multiplicities.

This representation by cornerpoints has the important property of
being stable against shape continuous deformations. For this
reason, in dealing with the shape comparison problem, via size
functions, one actually compares the sets of cornerpoints using
either the Hausdorff distance or the matching distance (see e.g.
\cite{CoEdHa05,DAFrLa,DAFrLa06,FrLa97}). The Hausdorff distance
and the matching distance differ in that the former does not take
into account the multiplicities of cornerpoints, while the latter
does.

The aim of this section is to show what happens to cornerpoints in
the presence of occlusions. We prove that each cornerpoint for the
size function of an occluded shape $X$ is a cornerpoint for the
size function of the original shape $A$, or the occluding shape
$B$, or their intersection $A \cap B$, providing that one
condition holds (Corollary \ref{mult}). However, even when this
condition is not verified, it holds that the coordinates of
cornerpoints of $\szf$ are always related to the cornerpoints of
$\sza$ or $\szb$ or $\szab$ (Theorems \ref{uabscissa} and
\ref{vordinate}).

We begin by proving a relation between multiplicities of points
for the size functions associated with $X$, $A$ and $B$. Since
cornerpoints are points with positive multiplicity (cf.
Definitions \ref{cornerpt} and \ref{cornerptinfty}), we obtain
conditions  for  cornerpoints of the size functions of $A$ and $B$
to persist in $A\cup B$. This fact suggests that in size theory
the partial matching of an occluded shape with the original shape
can be translated into the partial matching of cornerpoints of the
corresponding size functions. This intuition will be developed in
the experimental Section \ref{exp}.

In the next proposition we obtain a relation involving the
multiplicities of points in the size functions associated with
$X$, $A$ and $B$.

\begin{prop}\label{multiplicityp}
For every  $p = (\overline{u},\overline{v}) \in \Delta^+$, it holds that
$$\mu_X(p) - \mu_A(p) - \mu_B(p) + \mu_{A \cap B}(p)= \lim_{\epsilon \rightarrow 0^+}\left(\rank\ker\alpha_{\bar v-\eps,\bar u-\eps}-\rank\ker\alpha_{\bar v-\eps,\bar u+\eps}+\rank\ker\alpha_{\bar v+\eps,\bar u+\eps}-\rank\ker\alpha_{\bar v+\eps,\bar u-\eps}\right).$$
\end{prop}

\begin{proof}
Applying Theorem \ref{szrelationcompl} four times with
$(u,v)=(\overline{u}+\epsilon ,\overline{v}-\epsilon)$,
$(u,v)=(\overline{u}-\epsilon ,\overline{v}-\epsilon)$,
$(u,v)=(\overline{u}+\epsilon ,\overline{v}+\epsilon)$,
$(u,v)=(\overline{u}-\epsilon ,\overline{v}+\epsilon)$, $\eps$
being a positive real number so small that $\overline{u}+\epsilon
<\overline{v}-\epsilon$, we get

\begin{eqnarray*}
\lefteqn{\ell _{(X,\varphi
)}(\overline{u}+\epsilon ,\overline{v}-\epsilon)-\ell _{(X,\varphi
)}(\overline{u}-\epsilon ,\overline{v}-\epsilon )-\ell
_{(X,\varphi)} (\overline{u}+\epsilon ,\overline{v}+\epsilon
)+\ell _{(X,\varphi )}(\overline{u}-\epsilon
,\overline{v}+\epsilon )}\\
&=&\ell_{(A,
\varphi_{|A})}(\overline{u}+\epsilon ,\overline{v}-\epsilon) +
\ell_{(B, \varphi_{|B})}(\overline{u}+\epsilon
,\overline{v}-\epsilon) - \ell_{(A\cap B,
\varphi_{|A\cap B})}(\overline{u}+\epsilon ,\overline{v}-\epsilon)+\rank\ker\alpha_{\bar v-\eps}-\rank\ker\alpha_{\bar v-\eps,\bar u+\eps}\\
&& -\left(\ell_{(A,
\varphi_{|A})}(\overline{u}-\epsilon ,\overline{v}-\epsilon) +
\ell_{(B, \varphi_{|B})}(\overline{u}-\epsilon
,\overline{v}-\epsilon) - \ell_{(A\cap B,
\varphi_{|A\cap B})}(\overline{u}-\epsilon ,\overline{v}-\epsilon)+\rank\ker\alpha_{\bar v-\eps}-\rank\ker\alpha_{\bar v-\eps,\bar u-\eps}\right)\\
&&  -\left(\ell_{(A,
\varphi_{|A})}(\overline{u}+\epsilon ,\overline{v}+\epsilon) +
\ell_{(B, \varphi_{|B})}(\overline{u}+\epsilon
,\overline{v}+\epsilon) - \ell_{(A\cap B,
\varphi_{|A\cap B})}(\overline{u}+\epsilon ,\overline{v}+\epsilon)+\rank\ker\alpha_{\bar v+\eps}-\rank\ker\alpha_{\bar v+\eps,\bar u+\eps}\right)\\
&&  + \ell_{(A,
\varphi_{|A})}(\overline{u}-\epsilon ,\overline{v}+\epsilon) +
\ell_{(B, \varphi_{|B})}(\overline{u}-\epsilon
,\overline{v}+\epsilon) - \ell_{(A\cap B,
\varphi_{|A\cap B})}(\overline{u}-\epsilon ,\overline{v}+\epsilon)+\rank\ker\alpha_{\bar v+\eps}-\rank\ker\alpha_{\bar v+\eps,\bar u-\eps}\\
&=&\ell _{(A,\varphi
)}(\overline{u}+\epsilon ,\overline{v}-\epsilon)-\ell _{(A,\varphi
)}(\overline{u}-\epsilon ,\overline{v}-\epsilon )-\ell
_{(A,\varphi)} (\overline{u}+\epsilon ,\overline{v}+\epsilon
)+\ell
_{(A,\varphi )}(\overline{u}-\epsilon ,\overline{v}+\epsilon )\\
&&  +\ell _{(B,\varphi
)}(\overline{u}+\epsilon ,\overline{v}-\epsilon)-\ell _{(B,\varphi
)}(\overline{u}-\epsilon ,\overline{v}-\epsilon )-\ell
_{(B,\varphi)} (\overline{u}+\epsilon ,\overline{v}+\epsilon
)+\ell
_{(B,\varphi )}(\overline{u}-\epsilon ,\overline{v}+\epsilon )\\
&& -\ell _{(A\cap B,\varphi
)}(\overline{u}+\epsilon ,\overline{v}-\epsilon)+\ell _{(A\cap
B,\varphi )}(\overline{u}-\epsilon ,\overline{v}-\epsilon )+\ell
_{(A\cap B,\varphi)} (\overline{u}+\epsilon ,\overline{v}+\epsilon
)-\ell _{(A\cap B,\varphi )}(\overline{u}-\epsilon ,\overline{v}+\epsilon )\\
&& +\rank\ker\alpha_{\bar v-\eps,\bar u-\eps}-\rank\ker\alpha_{\bar v-\eps,\bar u+\eps}+\rank\ker\alpha_{\bar v+\eps,\bar u+\eps}-\rank\ker\alpha_{\bar v+\eps,\bar u-\eps}.
\end{eqnarray*}
Hence, by definition of multiplicity of a point of $\Delta^+$
(Definition \ref{cornerpt}), we have that
\begin{eqnarray*}
\lefteqn{\lim_{\epsilon \rightarrow 0^+}\left(\rank\ker\alpha_{\bar v-\eps,\bar u-\eps}-\rank\ker\alpha_{\bar v-\eps,\bar u+\eps}+\rank\ker\alpha_{\bar v+\eps,\bar u+\eps}-\rank\ker\alpha_{\bar v+\eps,\bar u-\eps}\right)}\\
&=&\lim_{\epsilon \rightarrow 0^+}\left(\ell _{(X,\varphi
)}(\overline{u}+\epsilon ,\overline{v}-\epsilon)-\ell _{(X,\varphi
)}(\overline{u}-\epsilon ,\overline{v}-\epsilon )-\ell
_{(X,\varphi)} (\overline{u}+\epsilon ,\overline{v}+\epsilon
)+\ell _{(X,\varphi )}(\overline{u}-\epsilon
,\overline{v}+\epsilon )\right)\\
&&-\lim_{\epsilon \rightarrow 0^+}\left(\ell _{(A,\varphi
)}(\overline{u}+\epsilon ,\overline{v}-\epsilon)-\ell _{(A,\varphi
)}(\overline{u}-\epsilon ,\overline{v}-\epsilon )-\ell
_{(A,\varphi)} (\overline{u}+\epsilon ,\overline{v}+\epsilon
)+\ell
_{(A,\varphi )}(\overline{u}-\epsilon ,\overline{v}+\epsilon )\right)\\
&& -\lim_{\epsilon \rightarrow 0^+}\left(\ell _{(B,\varphi
)}(\overline{u}+\epsilon ,\overline{v}-\epsilon)-\ell _{(B,\varphi
)}(\overline{u}-\epsilon ,\overline{v}-\epsilon )-\ell
_{(B,\varphi)} (\overline{u}+\epsilon ,\overline{v}+\epsilon
)+\ell
_{(B,\varphi )}(\overline{u}-\epsilon ,\overline{v}+\epsilon )\right)\\
&&+\lim_{\epsilon \rightarrow 0^+}\left(\ell _{(A\cap B,\varphi
)}(\overline{u}+\epsilon ,\overline{v}-\epsilon)-\ell _{(A\cap
B,\varphi )}(\overline{u}-\epsilon ,\overline{v}-\epsilon )-\ell
_{(A\cap B,\varphi)} (\overline{u}+\epsilon ,\overline{v}+\epsilon
)+\ell _{(A\cap B,\varphi )}(\overline{u}-\epsilon ,\overline{v}+\epsilon )\right)\\
&=&\mu_X(p)-\mu_A(p) - \mu_B(p) + \mu_{A \cap B}(p).
\end{eqnarray*}
\end{proof}
Using the previous Proposition \ref{multiplicityp}, we find a
condition ensuring that proper cornerpoints for the size function
of $X$ are also proper cornerpoints for the size function of $A$
or $B$.

\begin{cor}\label{mult}
Let $p=(\overline u,\overline v)$ be a proper cornerpoint  for
$\ell _{(X,\varphi)}$ and $$\lim_{\epsilon \rightarrow
0^+}\left(\rank\ker\alpha_{\bar v-\eps,\bar
u-\eps}-\rank\ker\alpha_{\bar v-\eps,\bar
u+\eps}+\rank\ker\alpha_{\bar v+\eps,\bar
u+\eps}-\rank\ker\alpha_{\bar v+\eps,\bar u-\eps}\right)\le 0.$$
Then $p$ is a proper cornerpoint  for either $\ell_{(A,\p|_A)}$ or
$\ell_{(B,\p|_B)}$ or both.
\end{cor}

\begin{proof}
Let $\underset{\eps\rightarrow0^+}\lim\left(\rank\ker\alpha_{\bar
v-\eps,\bar u-\eps}-\rank\ker\alpha_{\bar v-\eps,\bar
u+\eps}+\rank\ker\alpha_{\bar v+\eps,\bar
u+\eps}\right.-\left.\rank\ker\alpha_{\bar v+\eps,\bar
u-\eps}\right)\le 0$. From Proposition \ref{multiplicityp}, we
deduce that $\mu_X(p)\le \mu_A(p) + \mu_B(p) - \mu_{A \cap B}(p)$.
Since $p$ is a cornerpoint for $\ell _{(X,\varphi)}$, it holds
that $\mu_X(p)>0$. Since multiplicities are always non-negative,
this easily implies that either $\mu_A(p)>0$ or $\mu_B(p)>0$ (or
both), proving the statement.
\end{proof}

\begin{rem}
If $p=(\overline u,\overline v)$ is  a proper cornerpoint  for
$\ell _{(X,\varphi)}$ and $\ell_{(A\cap B),\p_{|A\cap
B}}(\overline v,v')\le 1$ for every $v'>\overline v$, then it is a
proper cornerpoint  for either $\ell_{(A,\p|_A)}$ or
$\ell_{(B,\p|_B)}$ or both.

This is easily seen by combining Lemma \ref{f_0} with Proposition
\ref{surj} so that, by the right-continuity of size functions and
the fact that they are non-decreasing in the first variable, for a
sufficiently small $\eps$ it holds that $\ker\alpha_{\bar
v-\eps,\bar u-\eps}=0$, $\ker\alpha_{\bar v-\eps,\bar u+\eps}=0$,
$\ker\alpha_{\bar v+\eps,\bar u+\eps}=0$, $\ker\alpha_{\bar
v+\eps,\bar u-\eps}=0$.
\end{rem}

The following two theorems state that the abscissas of the
cornerpoints for $\szf$ are abscissas of cornerpoints for $\sza$
or $\szb$ or $\szab$; the ordinates of the cornerpoints for $\szf$
are, in general, homological 0-critical values for $(A, \p_{|A})$
or $(B, \p_{|B})$ or $(A \cap B, \p_{|A\cap B})$, and, under
restrictive conditions, abscissas or ordinates of cornerpoints for
$\sza$ or $\szb$ or $\szab$, respectively.

These facts can easily be seen in the examples illustrated in
Figures \ref{figura:osso1}-\ref{figura:osso4}. In particular, in
Figure \ref{figura:osso4}, the size function $\szf$ presents the
proper cornerpoint $(-a, -b)$, which is neither a cornerpoint for
$\sza$ nor $\szb$ nor $\szab$. Nevertheless, its abscissa $-a$ is
the abscissa of all cornerpoints for $\sza$, while its ordinate
$-b$ is the abscissa of the cornerpoint at infinity for both
$\szb$ and $\szab$.
\begin{thm}\label{uabscissa}
If $p = (\overline{u},\overline{v}) \in \Delta^+$ is a proper
cornerpoint for $\szf$, then there exists at least one proper
cornerpoint for $\sza$ or $\szb$ or $\szab$ having $\overline{u}$
as abscissa. Moreover, if $(\overline{u}, \infty) \in \Delta^*$ is
a cornerpoint at infinity for $\szf$, then it is a cornerpoint at
infinity for $\sza$ or $\szb$.
\end{thm}
\begin{proof}
As for the first assertion, we prove the contrapositive statement.

Let $\overline{u} \in \R$, and let us suppose that there are no
proper cornerpoints for $\sza$, $\szb$ and $\szab$ having
$\overline{u}$ as abscissa. Then it follows that, for every $v >
\overline{u}$:
$$
\begin{array}{l}
\underset{\eps\rightarrow0^+}\lim\left(\szab(\overline{u}+\eps,v)-\szab(\overline{u}-\eps,v)\right)=0,\\
\underset{\eps\rightarrow0^+}\lim\left(\sza(\overline{u}+\eps,v)-\sza(\overline{u}-\eps,v)\right)=0,\\
\underset{\eps\rightarrow0^+}\lim\left(\szb(\overline{u}+\eps,v)-\szb(\overline{u}-\eps,v)\right)=0.
\end{array}
$$
Indeed, if there exists $v > \overline{u}$, such that
$$\underset{\eps\rightarrow0^+}\lim\left(\szab(\overline{u}+\eps,v)-\szab(\overline{u}-\eps,v)\right)\neq0,$$
then $\overline{u}$ is a discontinuity point for $\szab(\cdot,v)$,
implying the presence of at least one proper cornerpoint having
$\overline{u}$ as abscissa (\cite{FrLa01}, Lemma 3). Analogously
for $\sza$ and $\szb$.

Moreover, since size functions are natural valued functions and
are non-decreasing in the first variable, for every $v >
\overline{u}$, there exists $\overline{\eps}
> 0$ small enough such that $v - \overline{\eps}> \overline{u}+
\overline{\eps}$, and
$$0=\underset{\eps\rightarrow0^+}\lim\left(\szab(\overline{u}+\eps,v)-\szab(\overline{u}-\eps,v)\right)=\szab(\overline{u}+\overline{\eps},v)-\szab(\overline{u}-\overline{\eps},v).$$
So, for every $\eta< \overline{\eps}$, we have
$\szab(\overline{u}+\eta,v)=\szab(\overline{u}-\eta,v)$. This is
equivalent to saying that $\rank \check{H}_0((A\cap B)_v) - \rank
\check{H}_0((A\cap B)_v, (A\cap B)_{\overline{u}+\eta})=\rank
\check{H}_0((A\cap B)_v) - \rank \check{H}_0((A\cap B)_v, (A\cap
B)_{\overline{u}-\eta})$, that is, $\rank \check{H}_0((A\cap B)_v,
(A\cap B)_{\overline{u}+\eta})=\rank \check{H}_0((A\cap B)_v,
(A\cap B)_{\overline{u}-\eta}).$ Thus, in a similar way, for
$\sza$ and $\szb$, we obtain $ \rank \check{H}_0(A_v,
A_{\overline{u}+\eta})=\rank \check{H}_0(A_v,
A_{\overline{u}-\eta})$ and $ \rank \check{H}_0(B_v,
B_{\overline{u}+\eta})=\rank \check{H}_0(B_v,
B_{\overline{u}-\eta}).$ Now, let us consider the following
diagram:

\begin{centering}
\hfill \xymatrix {\check{H}_0((A\cap B)_{v-\eta},(A\cap
B)_{\overline{u}-\eta})\ar[rr]^-{\alpha_{v-\eta,\overline{u}-\eta}}
\ar[d]^{j_1}&&\check{H}_0(A_{v-\eta},A_{\overline{u}-\eta})\oplus
\check{H}_0(B_{v-\eta},B_{\overline{u}-\eta})\ar[d]_{j_2}
\\ \check{H}_0((A\cap B)_{v-\eta},(A\cap
B)_{\overline{u}+\eta})\ar[rr]^-{\alpha_{v-\eta,\overline{u}+\eta}}&&\check{H}_0(A_{v-\eta},A_{\overline{u}+\eta})\oplus
\check{H}_0(B_{v-\eta},B_{\overline{u}+\eta}),}\hfill
\end{centering}
\\where the homomorphisms $j_1$ and $j_2$ are induced by
inclusions. Since they are surjective and their respective domain
and codomain have the same rank, we deduce that $j_1$ and $j_2$
are isomorphisms. So, we obtain that $\ker
\alpha_{v-\eta,\overline{u}-\eta}\simeq \ker
\alpha_{v-\eta,\overline{u}+\eta}.$

Analogously, from the diagram

\begin{centering}
\hfill \xymatrix {\check{H}_0((A\cap B)_{v+\eta},(A\cap
B)_{\overline{u}-\eta})\ar[rr]^-{\alpha_{v+\eta,\overline{u}-\eta}}
\ar[d]^{k_1}&&\check{H}_0(A_{v+\eta},A_{\overline{u}-\eta})\oplus
\check{H}_0(B_{v+\eta},B_{\overline{u}-\eta})\ar[d]_{k_2}
\\ \check{H}_0((A\cap B)_{v+\eta},(A\cap
B)_{\overline{u}+\eta})\ar[rr]^-{\alpha_{v+\eta,\overline{u}+\eta}}&&\check{H}_0(A_{v+\eta},A_{\overline{u}+\eta})\oplus
\check{H}_0(B_{v+\eta},B_{\overline{u}+\eta}),}\hfill
\end{centering}
\\we can deduce that $\ker
\alpha_{v+\eta,\overline{u}-\eta}\simeq \ker
\alpha_{v+\eta,\overline{u}+\eta}$. Thus, since $\eta$ can be
chosen arbitrarily small, it holds that
$$
\begin{array}{l}
\underset{\eta\rightarrow0^+}\lim(\ker
\alpha_{v-\eta,\overline{u}-\eta}- \ker
\alpha_{v-\eta,\overline{u}+\eta})=0,\\
\underset{\eta\rightarrow0^+}\lim(\ker
\alpha_{v+\eta,\overline{u}-\eta}- \ker
\alpha_{v+\eta,\overline{u}+\eta})=0.
\end{array}
$$
Therefore, applying Proposition \ref{multiplicityp}, it follows
that
$$\mu_X(p) - \mu_A(p) - \mu_B(p) + \mu_{A \cap B}(p)=0$$
and, in particular, by the hypothesis that $p = (\overline{u}, v)$
is not a proper cornerpoint for $\szab$, or $\sza$, or $\szb$, for
any $v
> \overline{u}$, it holds that $\mu_X(p)=0.$

In the case of cornerpoints at infinity, we observe that, if
$(\overline{u}, \infty)$ is a cornerpoint at infinity for $\szf$,
then $\overline{u}= \underset{p \in C}\min \p(p)$, for at least
one connected component $C$ of $X$ (\cite{FrLa01}, Prop. 9).
Furthermore, since $X = A \cup B$, it follows that $\overline{u}=
\underset{p \in C \cap A}\min \p_{|A}(p)$ or $\overline{u}=
\underset{p \in C \cap B}\min \p_{|B}(p)$, from which (by
\cite{FrLa01}, Prop. 9), $(\overline{u}, \infty)$ is shown to be a
cornerpoint at infinity for $\sza$ or $\szb$.
\end{proof}

\begin{thm}\label{vordinate}
If $p = (\overline{u},\overline{v}) \in \Delta^+$ is a proper
cornerpoint for $\szf$, then $\overline{v}$ is a homological
0-critical value for $(A, \p_{|A})$ or $(B, \p_{|B})$ or $(A \cap
B, \p_{|A\cap B})$. Furthermore, if there exists at most a finite
number of homological 0-critical values for $(A, \p_{|A})$, $(B,
\p_{|B})$, and $(A \cap B, \p_{|A\cap B})$, then $\overline{v}$ is
the abscissa of a cornerpoint (proper or at infinity) or the
ordinate of a proper cornerpoint for $\sza$ or $\szb$ or $\szab$.
\end{thm}

\begin{proof}
Regarding the first assertion, we prove the contrapositive
statement.

Let $\overline{v} \in \R$, and let us suppose that $\overline{v}$
is not a homological 0-critical value for the size pairs $(A,
\p_{|A})$, $(B, \p_{|B})$ and $(A \cap B, \p_{|A\cap B})$. Then,
by Definition \ref{p-critValue}, for every $\overline{\eps}>0$,
there exists $\eps$ with $0< \eps < \overline{\eps}$, such that
the vertical homomorphisms $h$ and $k$ induced by inclusions in
the following commutative diagram

\begin{centering}
\hfill \xymatrix {\cdots\ar[r]&\check{H}_0((A\cap
B)_{\overline{v}-\eps})\ar[r]
\ar[d]^{h}&\check{H}_0(A_{\overline{v}-\eps})\oplus
\check{H}_0(B_{\overline{v}-\eps})\ar[r]
\ar[d]^{k}&\check{H}_0(X_{\overline{v}-\eps})\ar[r]
\ar[d]^{\iota_0^{\overline{v}-\eps,
\overline{v}+\eps}}&0\ar[d]_{0}
\\\cdots \ar[r]&\check{H}_0((A\cap B)_{\overline{v}+\eps})\ar[r]&\check{H}_0(A_{\overline{v}+\eps})\oplus
\check{H}_0(B_{\overline{v}+\eps})\ar[r]&\check{H}_0(X_{\overline{v}+\eps})\ar[r]&0}\hfill
\end{centering}
\\are isomorphisms. Hence, using the Five Lemma, we can deduce
that $\iota_0^{\overline{v}-\eps, \overline{v}+\eps}$ is an
isomorphism, implying that $\overline{v}$ is not a homological
0-critical value for $(X, \p)$. Consequently, applying Proposition
\ref{0-critValueSF}, it holds that, for every $u < \overline{v}$,
$\underset{\eps \rightarrow 0^+}\lim\left(\szf(u, \overline{v} -
\eps)- \szf(u, \overline{v} + \eps)\right)=0.$ Hence, it follows
that $\underset{\eps \rightarrow 0^+}\lim\left(\szf(\overline{u}-
\eps, \overline{v} - \eps)- \szf(\overline{u}- \eps, \overline{v}
+ \eps)\right)=0$, choosing $u = \overline{u}- \eps$ and
$\underset{\eps \rightarrow 0^+}\lim[\szf(\overline{u}+ \eps,
\overline{v} - \eps)- \szf(\overline{u}+ \eps, \overline{v} +
\eps)]=0$, choosing $u = \overline{u}+ \eps$. Therefore, by
Definition \ref{cornerpt}, we obtain $\mu_X(p)=0$.

Now, let us proceed with the proof of the second statement,
assuming that $\overline{v}$ is a homological 0-critical value for
$(A, \p_{|A})$. It is analogous for $(B, \p_{|B})$ and $(A \cap B,
\p_{|A\cap B})$. For such a $\overline{v}$, by Definition
\ref{p-critValue}, it holds that, for every sufficiently small
$\eps > 0$, ${\iota}_0^{\overline{v} - \eps, \overline{v} + \eps}:
\check{H}_0(A_{\overline{v} - \eps})\rightarrow
\check{H}_0(A_{\overline{v} + \eps})$ is not an isomorphism. In
particular, by Proposition \ref{finitecrit} \emph{(i)}, if
${\iota}_0^{\overline{v} - \eps, \overline{v} + \eps}$ is not
surjective for any sufficiently small $\eps>0$, then there exists
$v> \overline{v}$, such that $\overline{v}$ is a discontinuity
point for $\sza(\cdot, v)$. This condition necessarily implies the
existence of a cornerpoint (proper or at infinity) for $\sza$,
having $\overline{v}$ as abscissa (\cite{FrLa01}, Lemma 3).

On the other hand, by Proposition \ref{finitecrit} \emph{(ii)}, if
${\iota}_0^{\overline{v} - \eps, \overline{v} + \eps}$ is
surjective for every sufficiently small $\eps>0$, then there
exists $u< \overline{v}$ such that $\overline{v}$ is a
discontinuity point for $\sza(u, \cdot)$. This condition
necessarily implies the existence of a proper cornerpoint for
$\sza$, having $\overline{v}$ as ordinate (\cite{FrLa01}, Lemma
3).
\end{proof}

\section{Experimental results}\label{exp}

In this section we present two experiments demonstrating the
robustness of size functions under partial occlusions.

Psychophysical observations indicate that human and monkey
perception of partially occluded shapes changes according to
whether, or not, the occluding pattern is visible to the observer,
and whether the occluded shape is a filled figure or an outline
\cite{KoVoOr95}. In particular, discrimination performance is
higher for filled shapes than for outlines, and in both cases it
significantly improves when shapes are occluded by a visible
rather than invisible object.

In computer vision experiments, researcher usually work with
invisible  occluding patterns, both on  outlines (see, e.g.,
\cite{ChCh05,GhPe05,MoBeMa01,SuSu05,TaVe05}) and on filled shapes
(see, e.g., \cite{HoOh03}).

To test size function performance under occlusions, we work with
70 filled images, each chosen from a different class of the MPEG-7
dataset \cite{MPEG-7}. The two experiments differ in  the
visibility of the occluding pattern. Since in the first
experiments the occluding pattern is visible, we aim at finding a
fingerprint of the original shape in the size function of the
occluded shape. In the second experiment, where the occluding
pattern is invisible, we perform a direct comparison between the
occluded shape and the original shape. In both experiments, the
occluding pattern is a  rectangular shape occluding from the top,
or the left,
 by an area we increasingly vary from
$20\%$ to $60\%$ of the height or width of the bounding box of the
original shape. We compute  size functions for both the original
shapes and the occluded ones, choosing a family of eight measuring
functions having only the set of black pixels as domain. They are
defined as follows: four of them as the distance from the line
passing through the origin (top left point of the bounding box),
rotated by an angle of $0$, $\frac{\pi}{4}$, $\frac{\pi}{2}$ and
$\frac{3\pi}{4}$ radians, respectively, with respect to the
horizontal position; the other
 four  as minus the distance from the same lines,
respectively. This family of measuring functions is chosen only
for demonstrative purposes, since the associated size functions
are simple  in terms of the number of cornerpoints, but, at the
same time, non-trivial in terms of shape information.

The first experiment aims to show how a trace of the size function
describing the shape of an object is contained in the size
function related to the occluded shape when the occluding pattern
is visible  (see first column of Tables
\ref{camelHV}--\ref{pocketHV}). With reference to the notation
used in our theoretical setting, we are considering $A$ as the
original shape, $B$ as the black rectangle, and $X$ as the
occluded shape generated by their union.

In Table \ref{cprecognition}, for some different levels of
occlusion, each 3D bar chart displays, along the z-axis, the
percentage of common cornerpoints between the set of size
functions associated with the 70 occluded shapes (x-axis), and the
set of size functions associated with the 70 original ones
(y-axis). We see that, for each occluded shape, the highest bar is
always on the diagonal, that is, where the occluded object is
compared with the corresponding original one.

\begin{table}[htbp]
\begin{tabular}{ccc}
\includegraphics[width=0.32\linewidth]{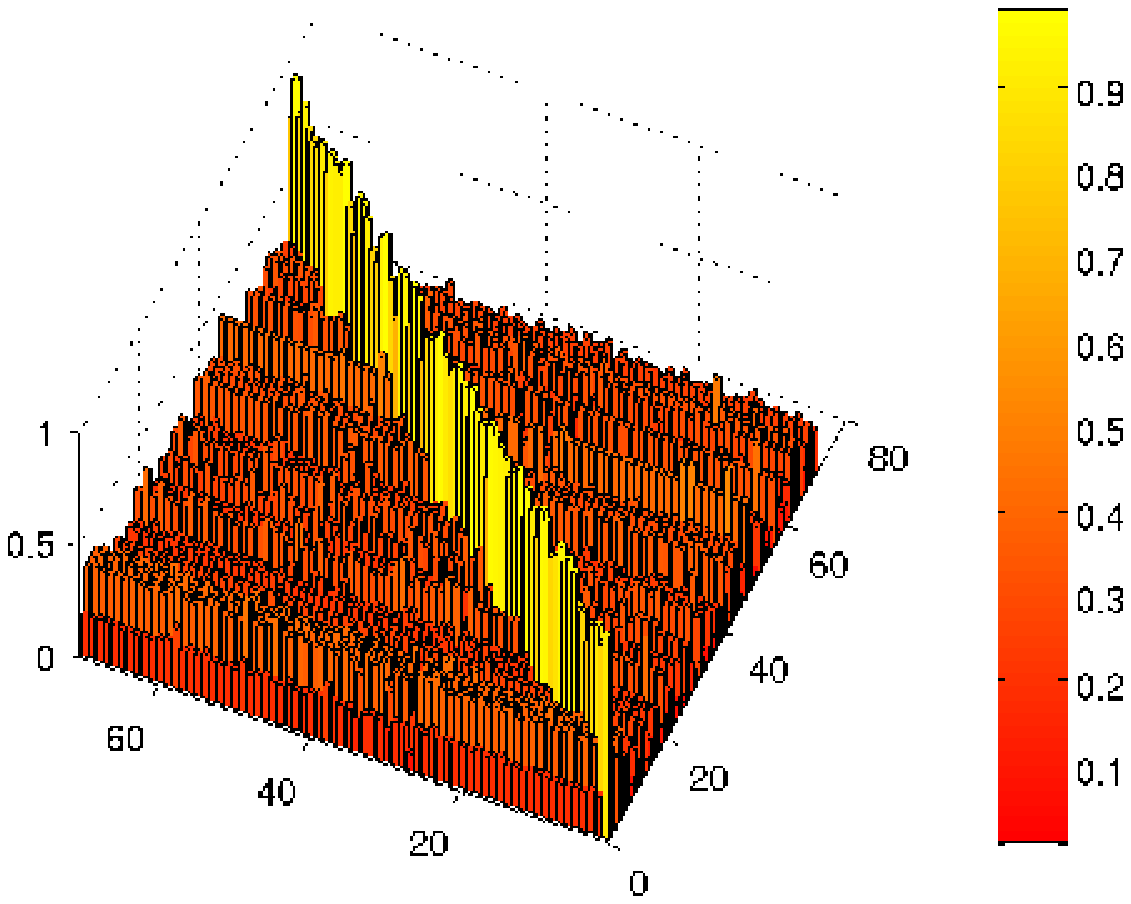} &
\includegraphics[width=0.31\linewidth]{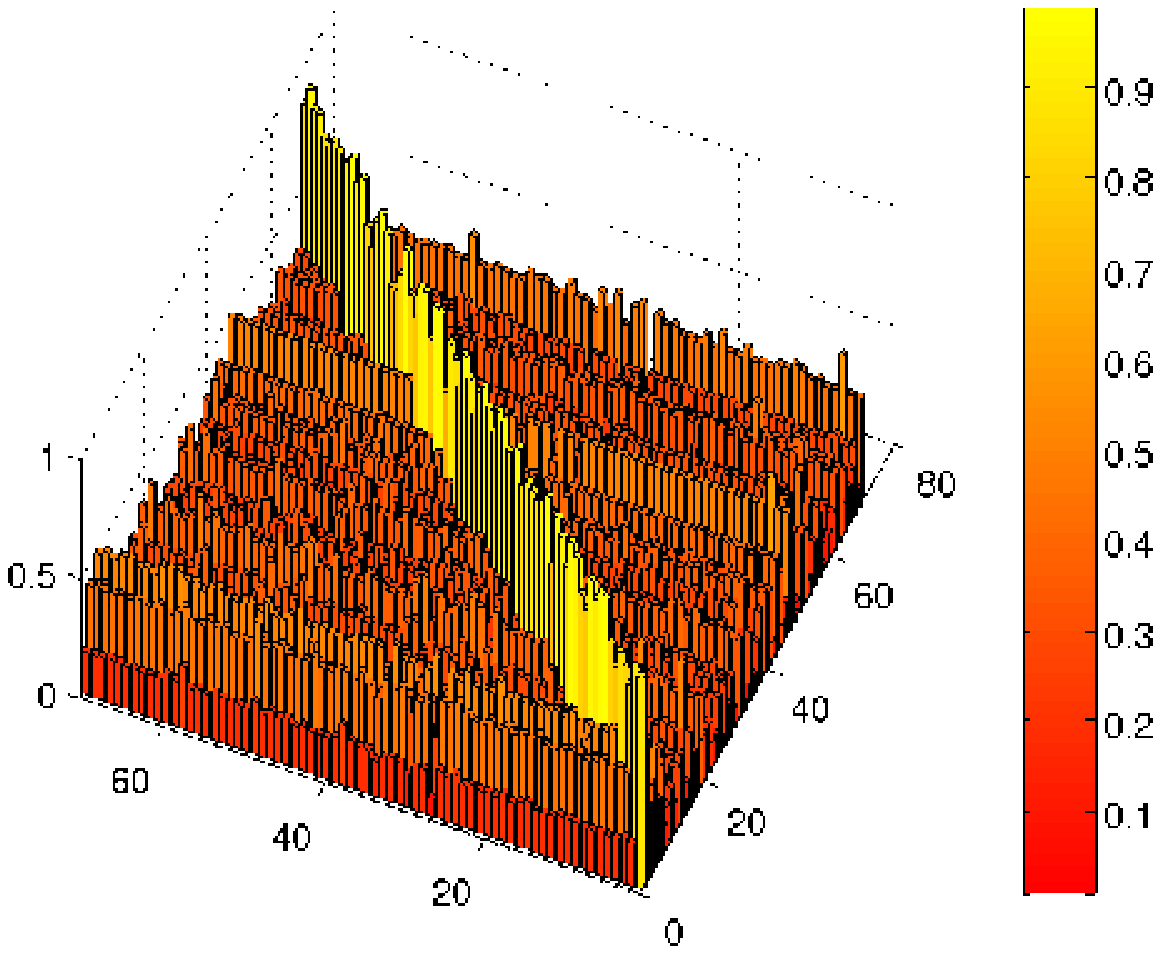} &
\includegraphics[width=0.31\linewidth]{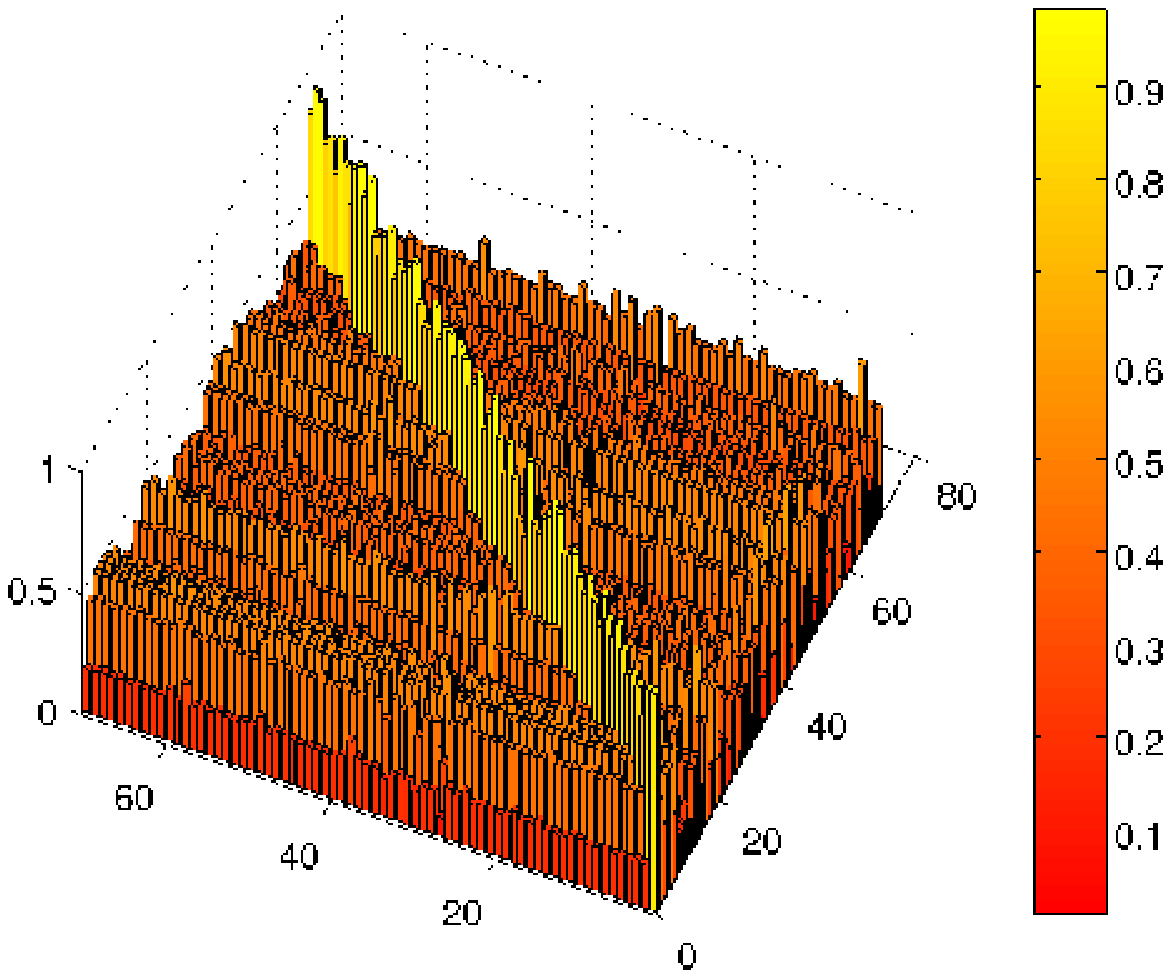} \\
\includegraphics[width=0.32\linewidth]{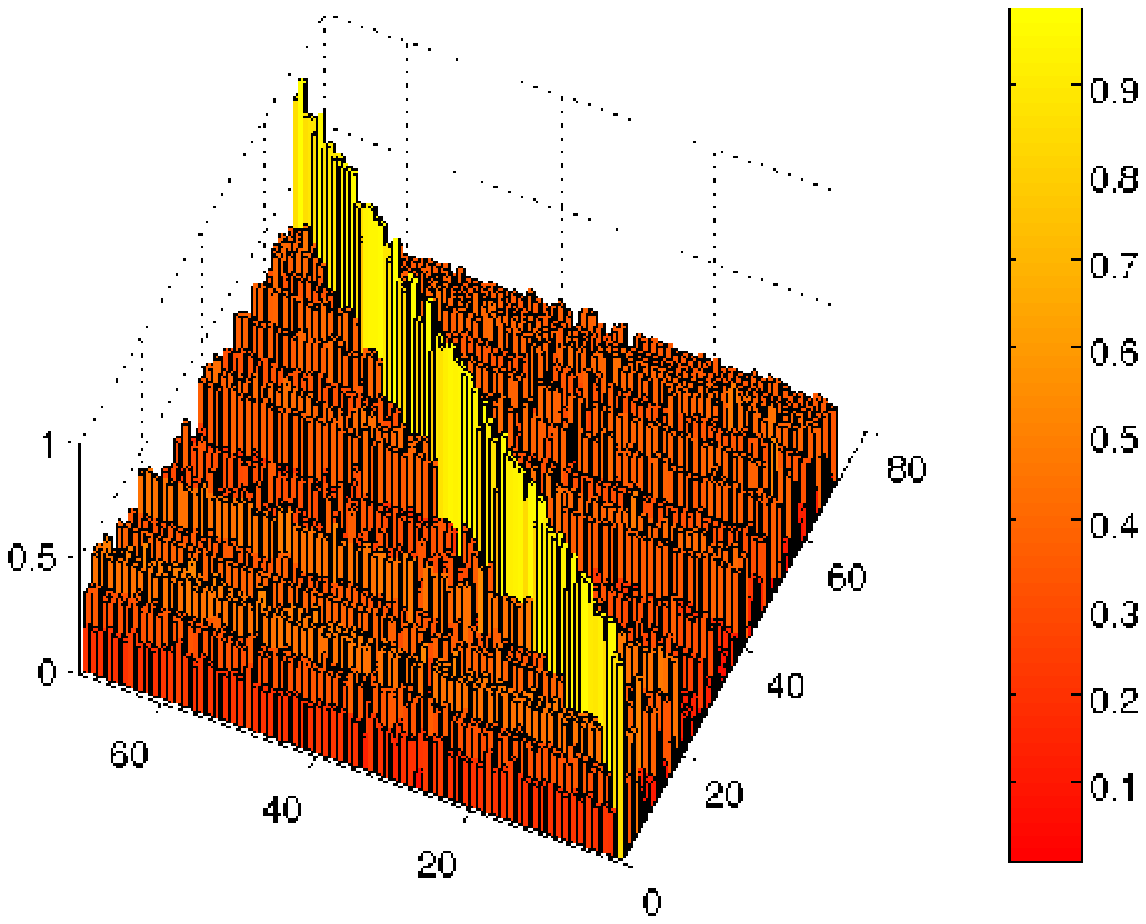} &
\includegraphics[width=0.31\linewidth]{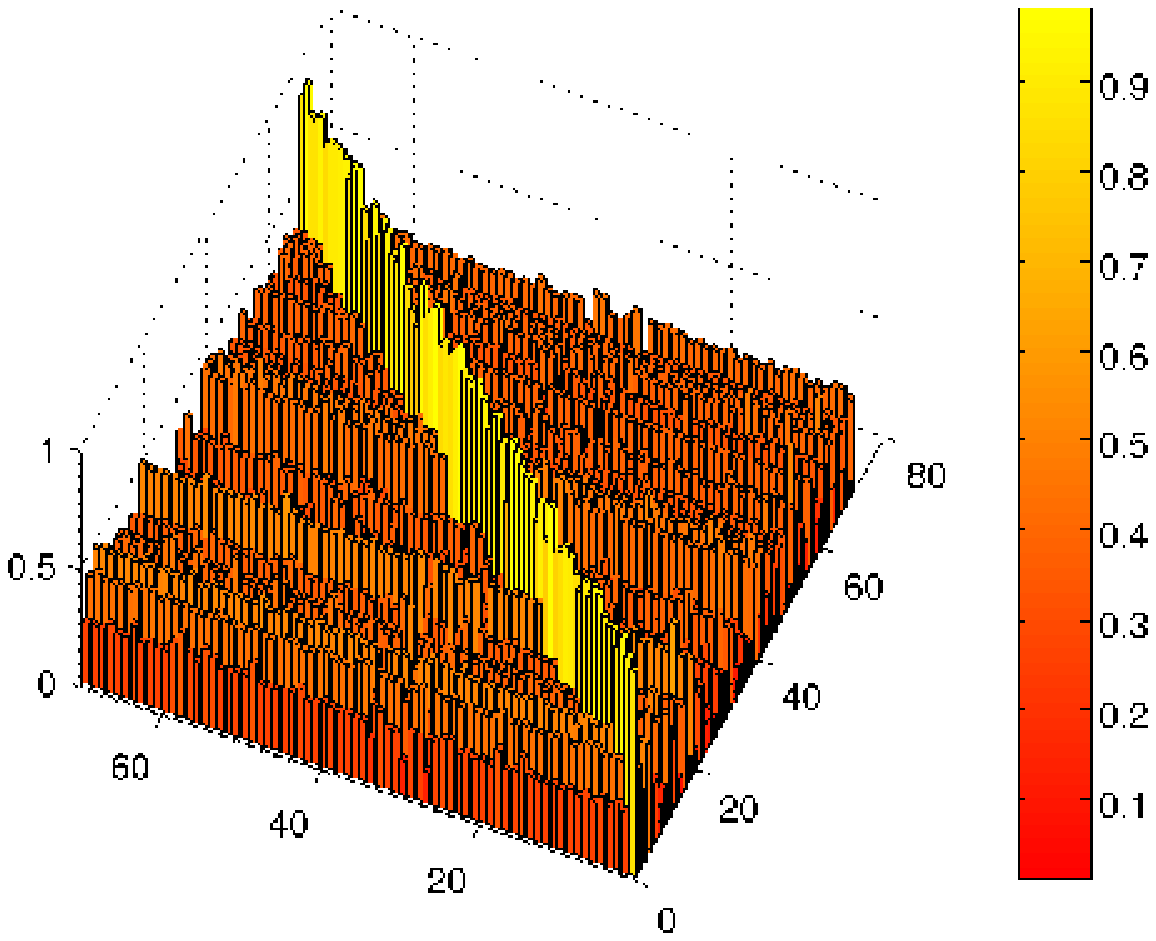} &
\includegraphics[width=0.31\linewidth]{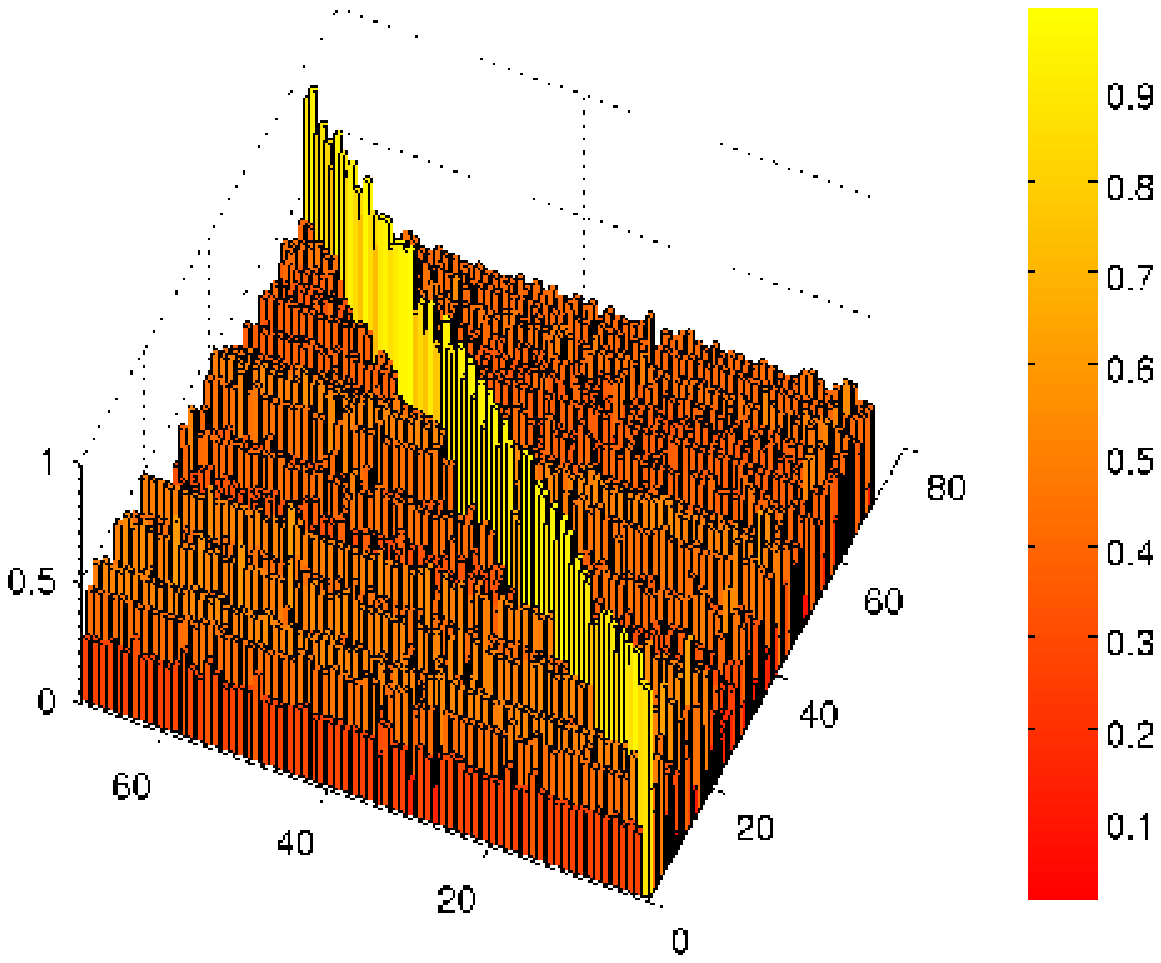} \\
\end{tabular}
\caption{\footnotesize{3D bar charts displaying the percentage of
common cornerpoints (z-axis) between the 70 occluded shapes
(x-axis) and the 70 original ones (y-axis) correspondingly ordered. First row: Shapes are occluded from top  by $20\%$ (column 1),  by $40\%$ (column 2),
 by $60\%$ (column 3). Second row: Shapes are occluded from the left  by $20\%$ (column 1),  by $40\%$ (column 2),
 by $60\%$ (column 3).}}\label{cprecognition}
\end{table}

Moreover, to display the robustness of cornerpoints under
occlusion, three particular instances of our dataset images are
shown in Tables \ref{camelHV}--\ref{pocketHV} (first column) with
their size functions with respect to the second group of four
measuring functions (the next-to-last column). The chosen images
are characterized by different homological features, which will be
changed in presence of occlusion. For example, the ``camel'' in
Table \ref{camelHV} is a connected shape without holes, but it may
happen that the occlusion makes the first homological group
non-trivial (see second row, first column).
\begin{table}[htbp]
\begin{tabular}{cccccc}
\includegraphics[width=0.145\linewidth]{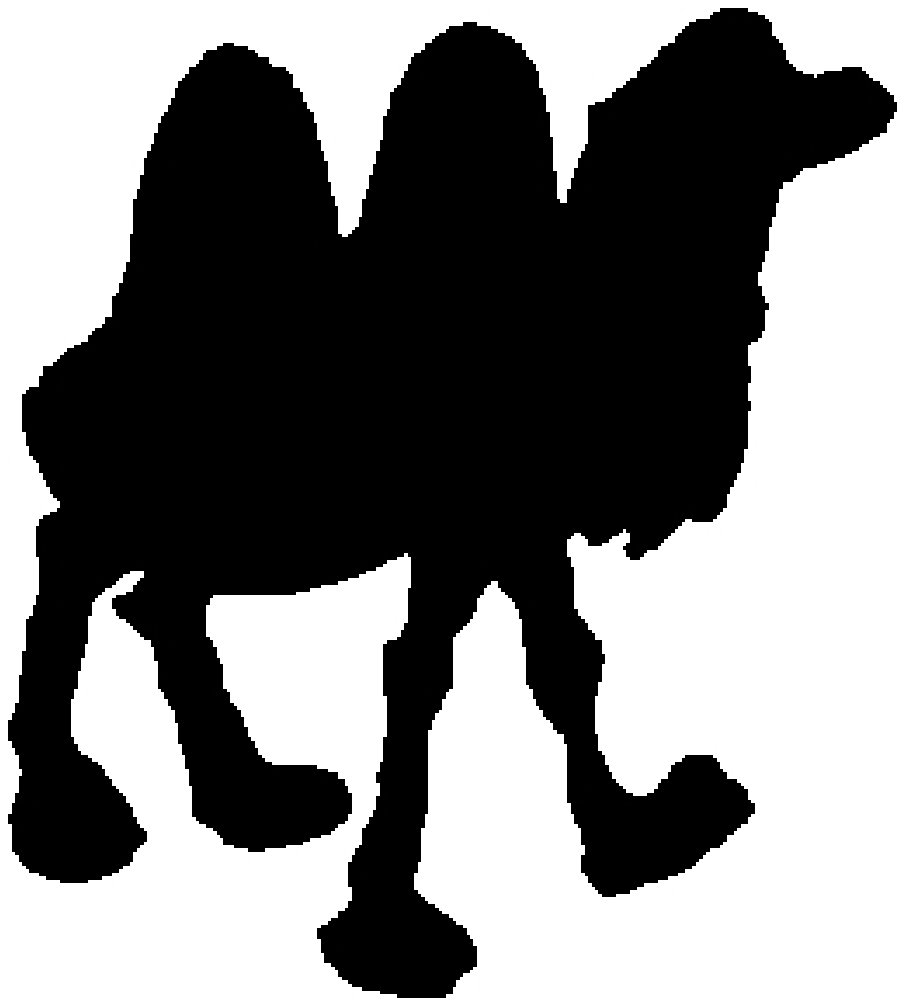} &
\includegraphics[width=0.15\linewidth]{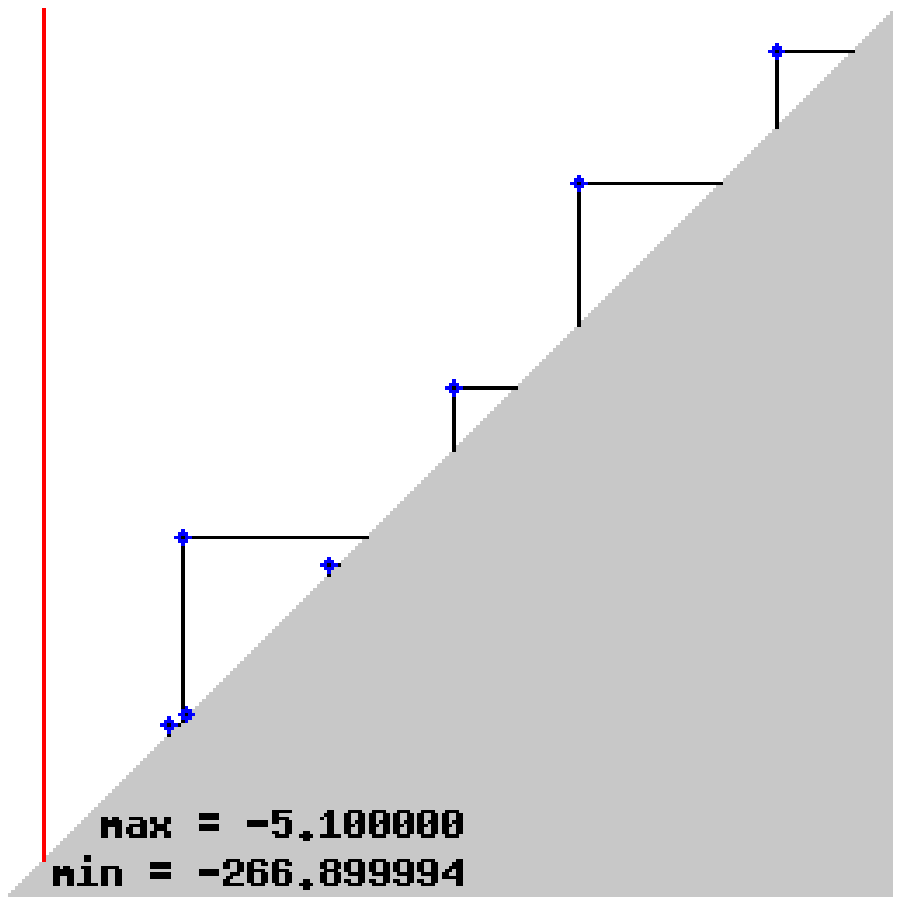} &
\includegraphics[width=0.15\linewidth]{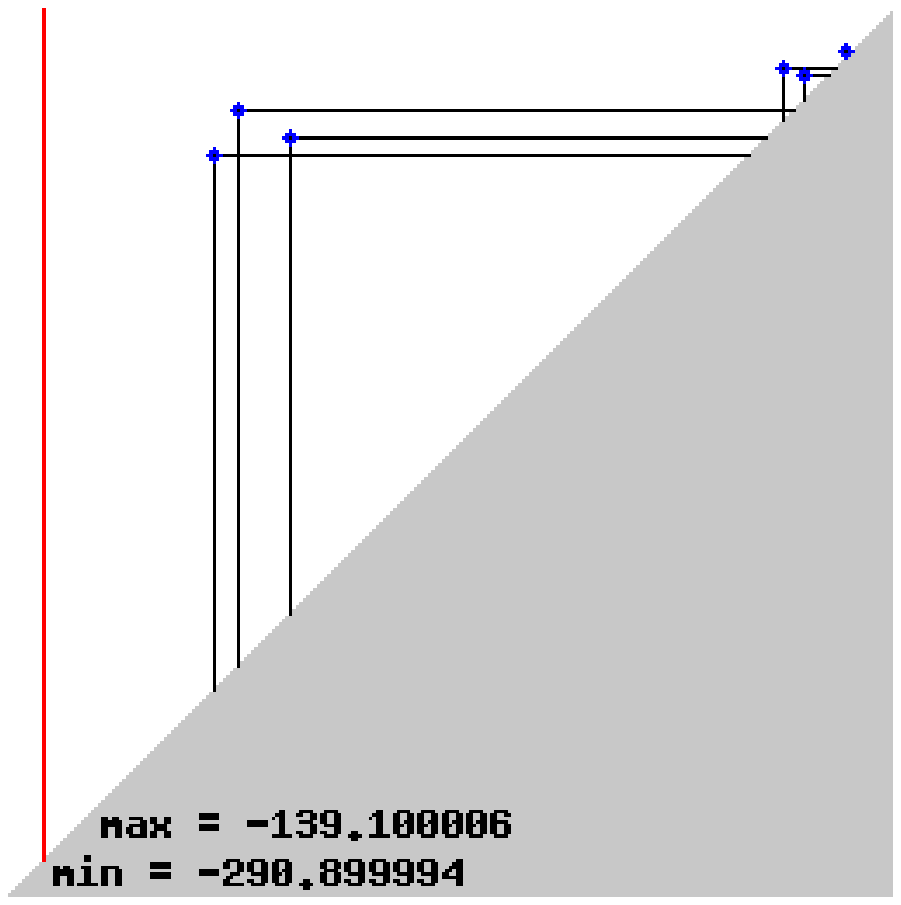} &
\includegraphics[width=0.15\linewidth]{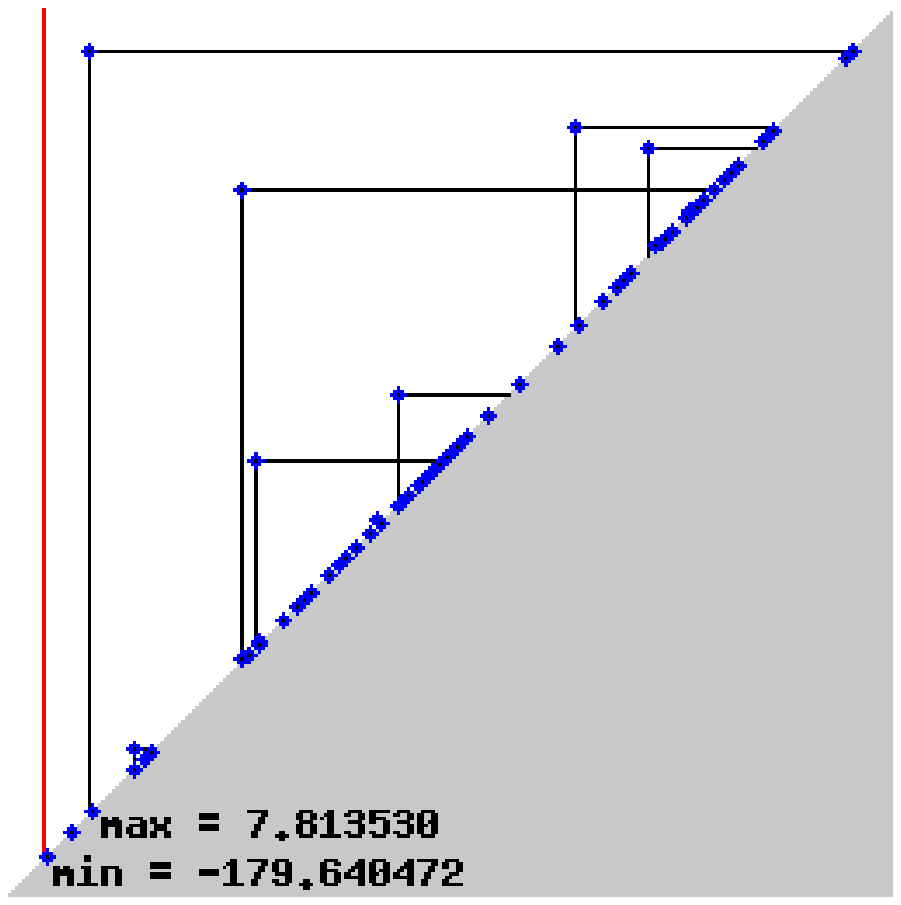} &
\includegraphics[width=0.15\linewidth]{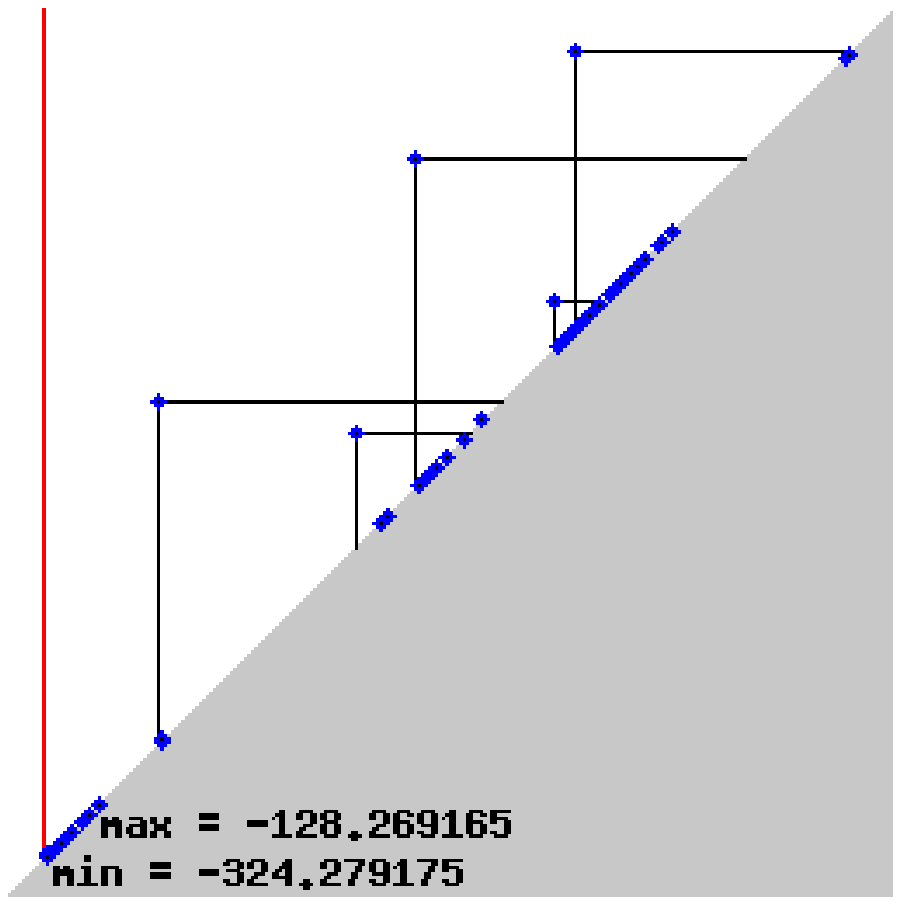} \\
\includegraphics[width=0.145\linewidth]{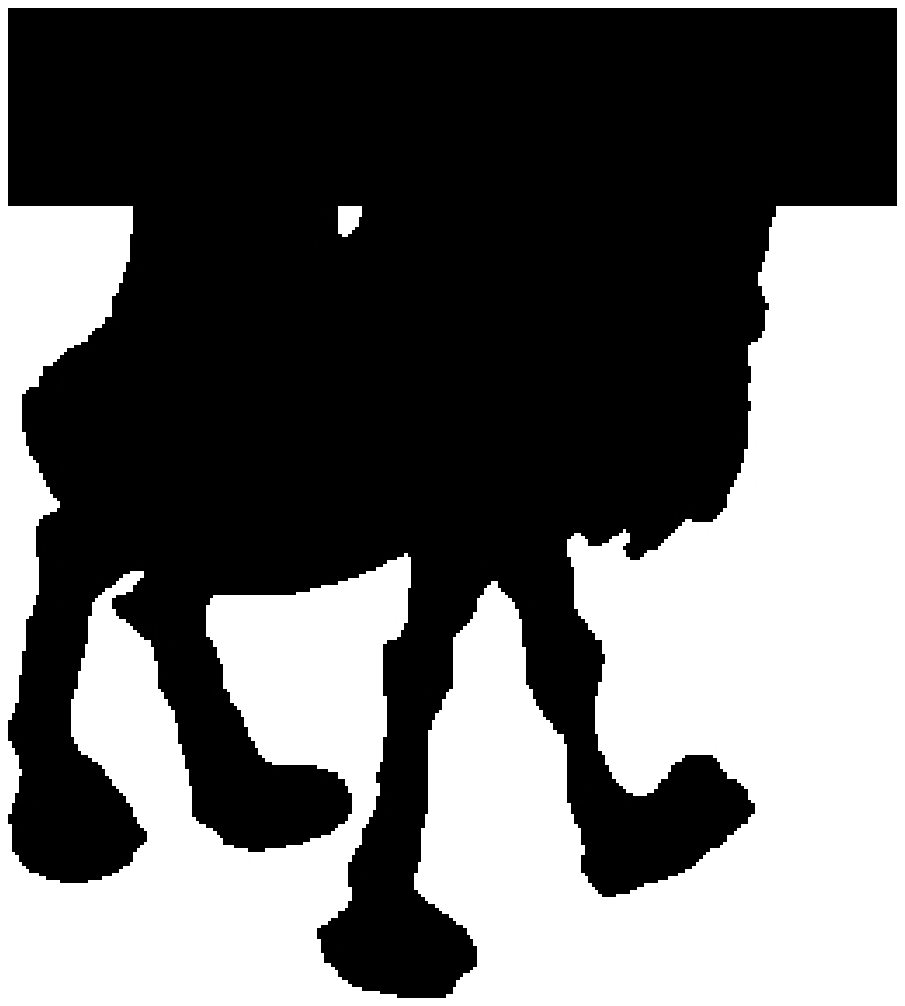} &
\includegraphics[width=0.15\linewidth]{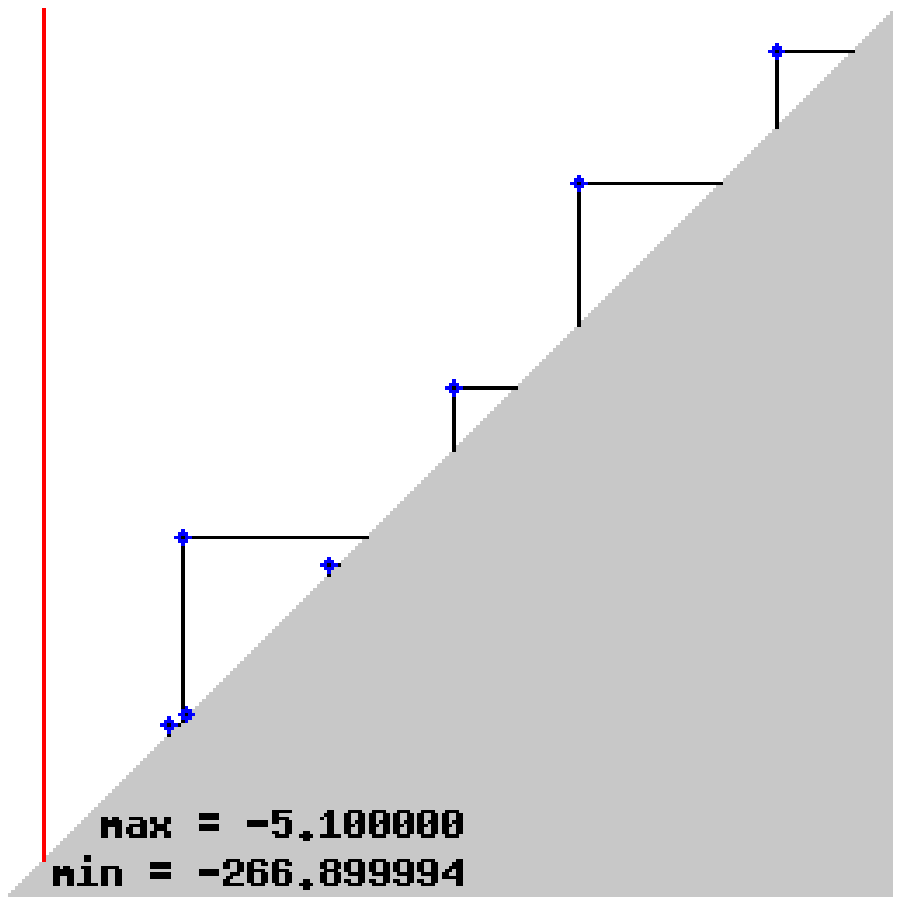} &
\includegraphics[width=0.15\linewidth]{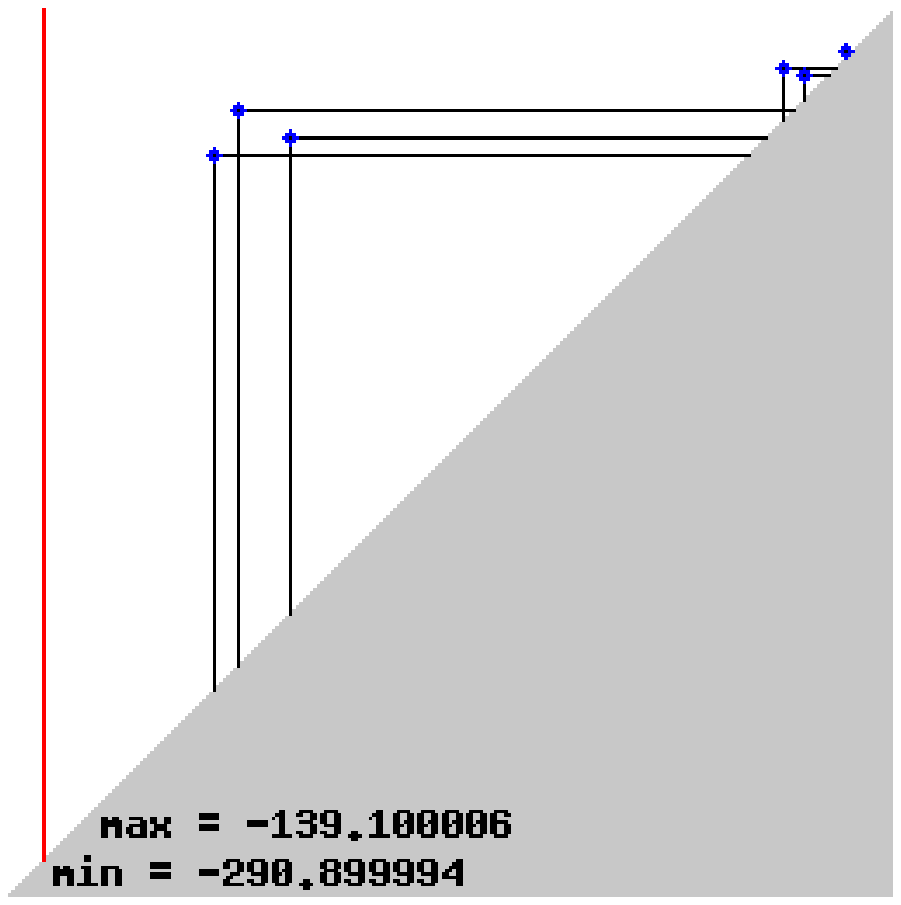} &
\includegraphics[width=0.15\linewidth]{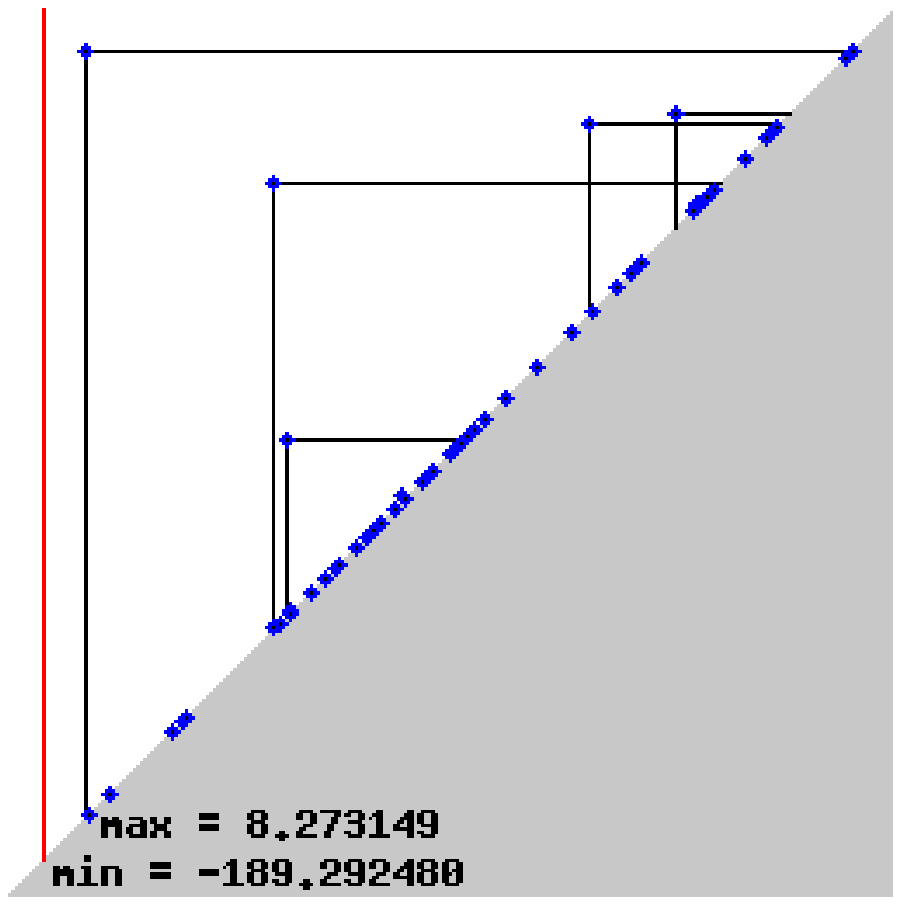} &
\includegraphics[width=0.15\linewidth]{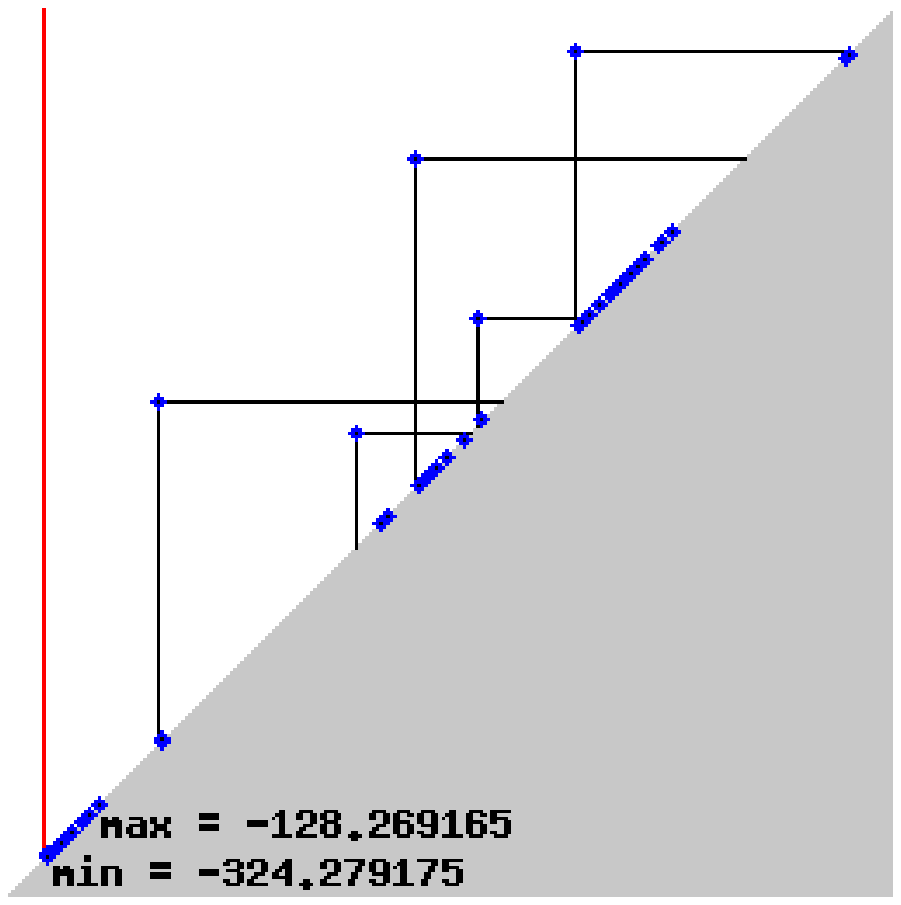} \\
\includegraphics[width=0.145\linewidth]{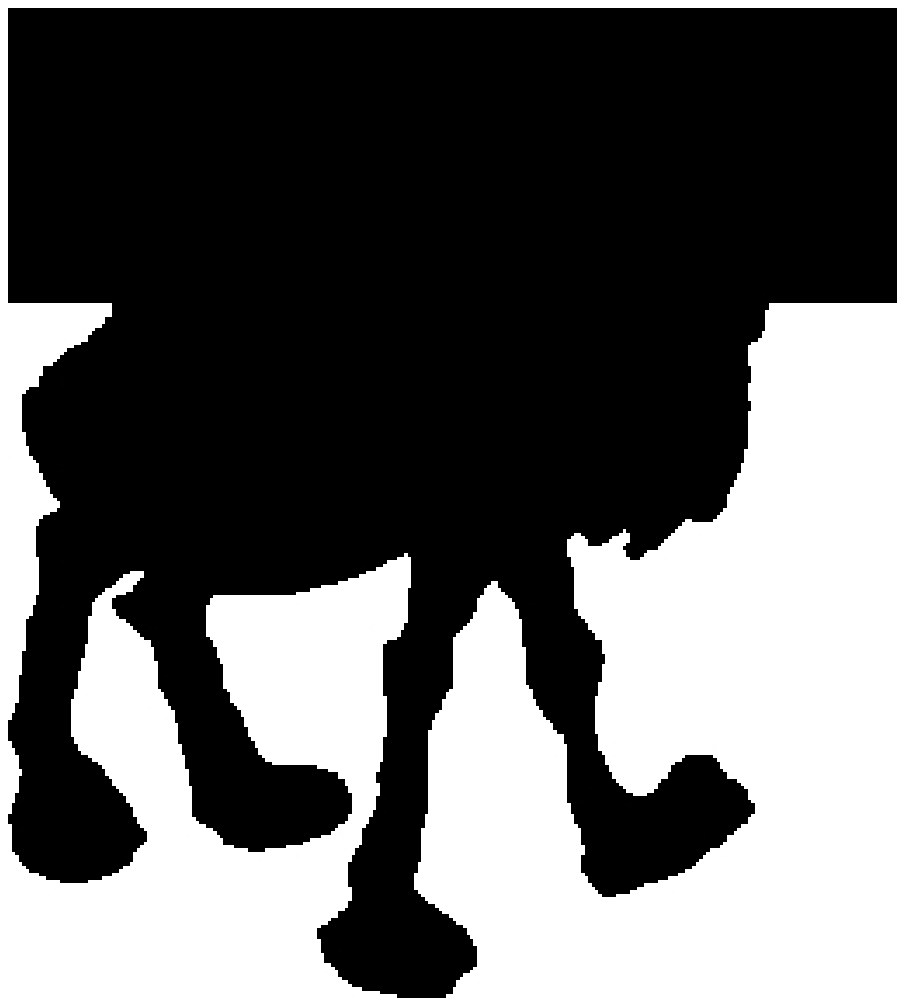} &
\includegraphics[width=0.15\linewidth]{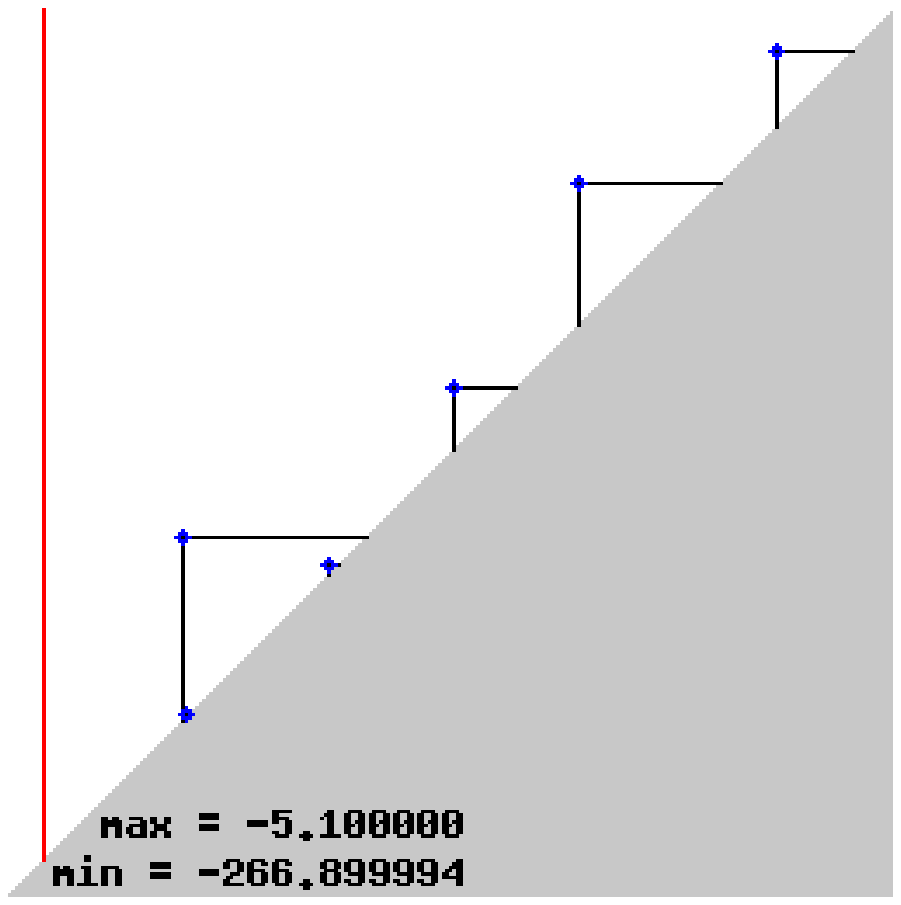} &
\includegraphics[width=0.15\linewidth]{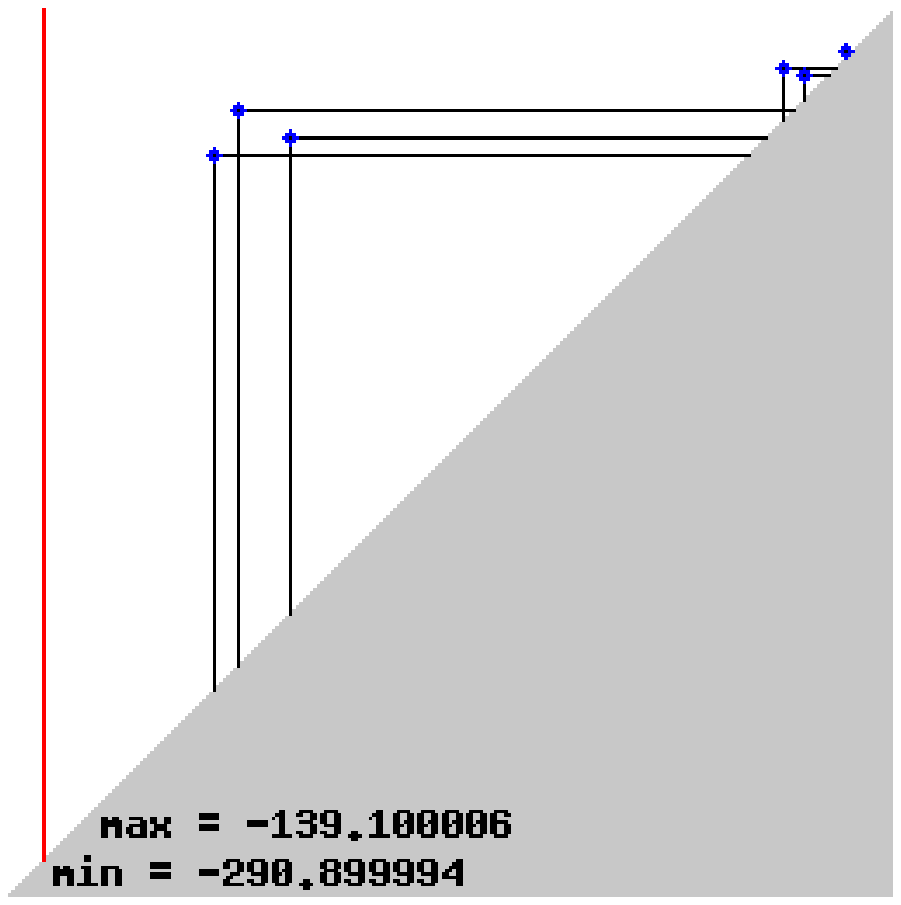} &
\includegraphics[width=0.15\linewidth]{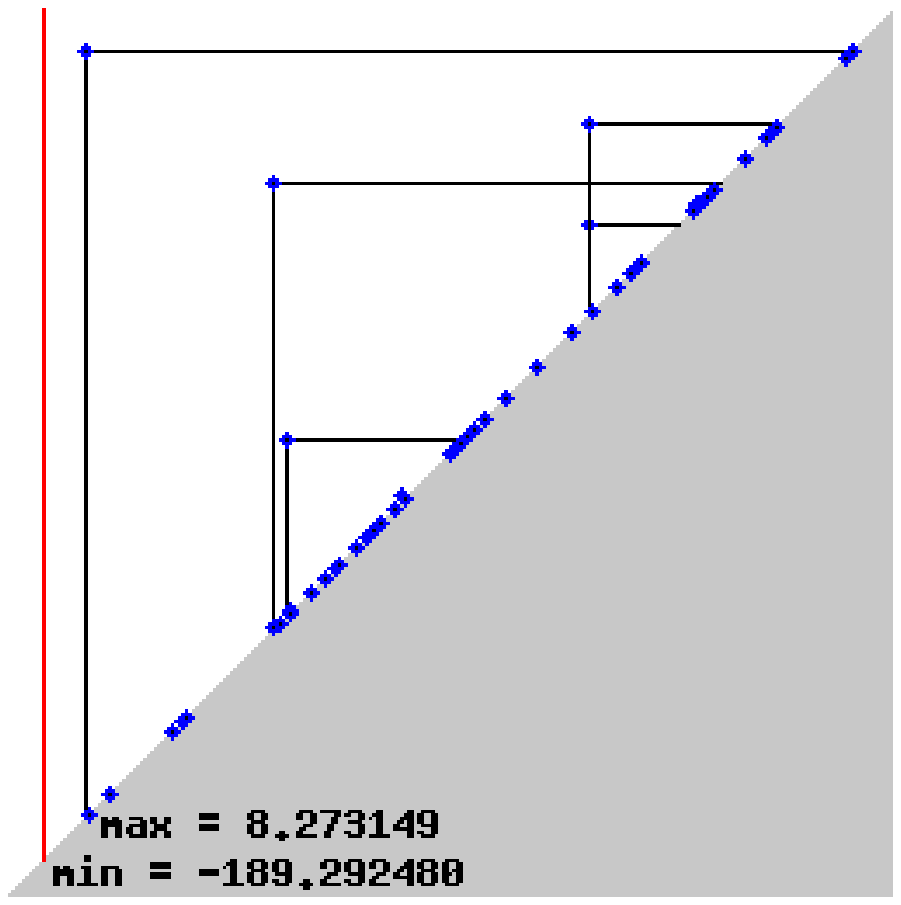} &
\includegraphics[width=0.15\linewidth]{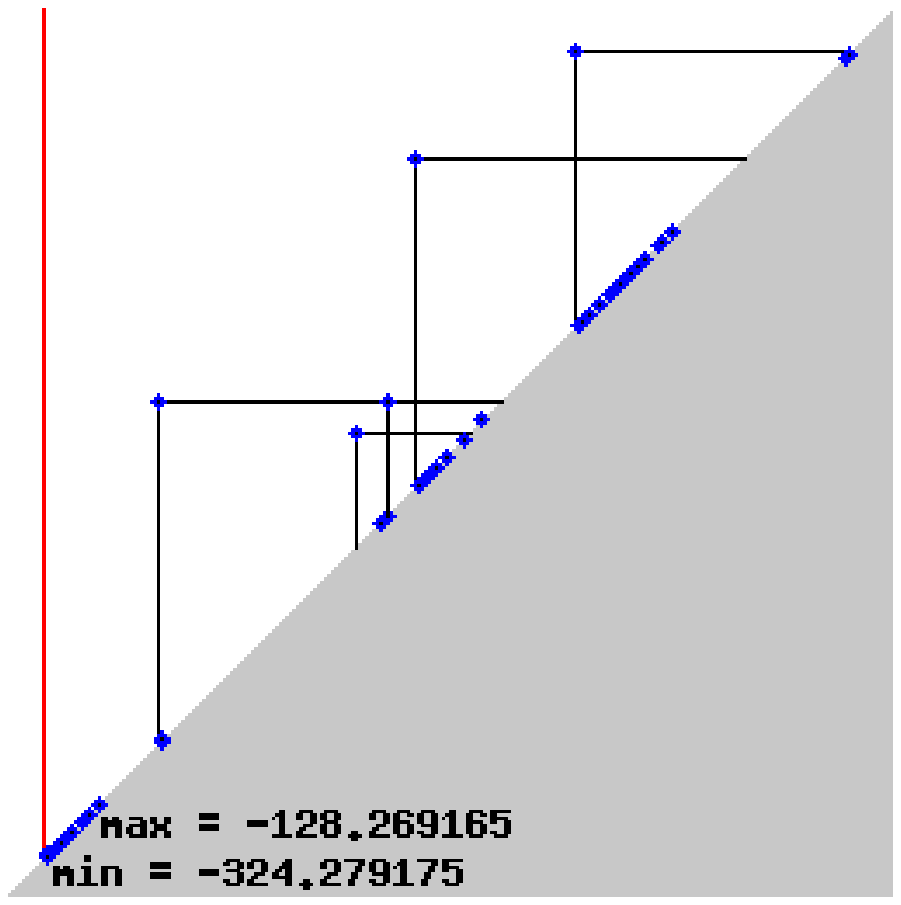} \\
\includegraphics[width=0.145\linewidth]{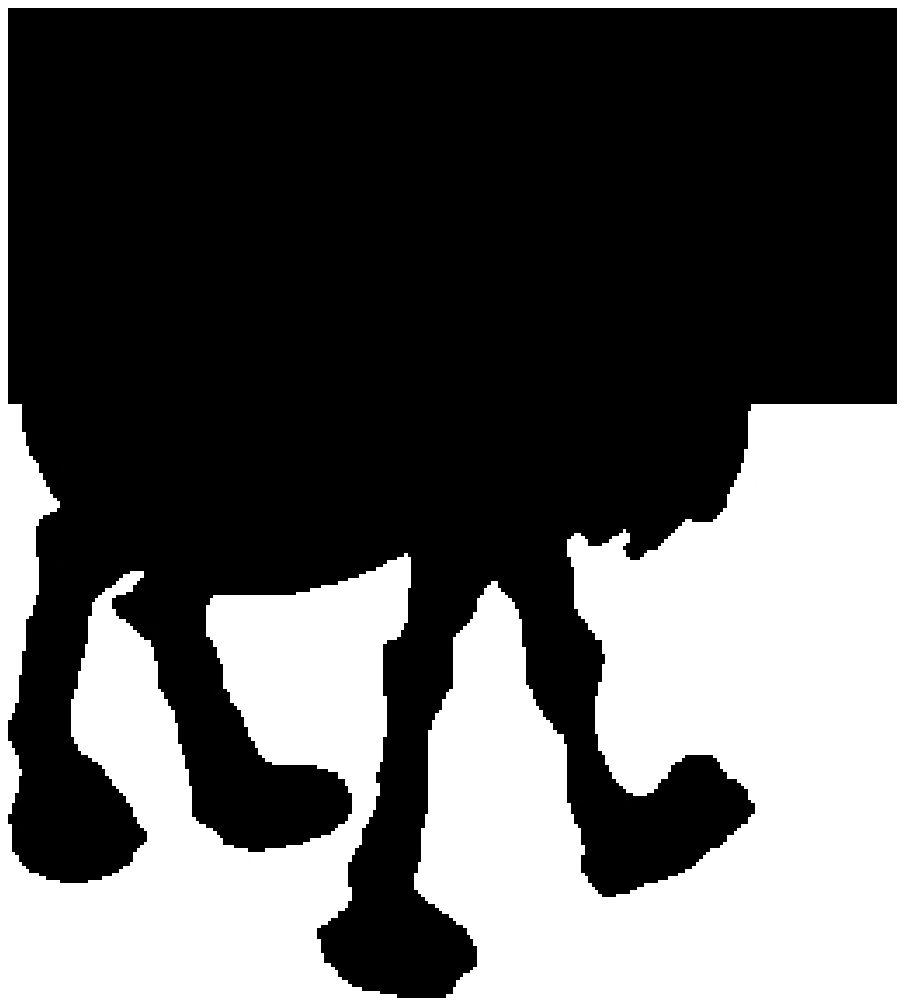} &
\includegraphics[width=0.15\linewidth]{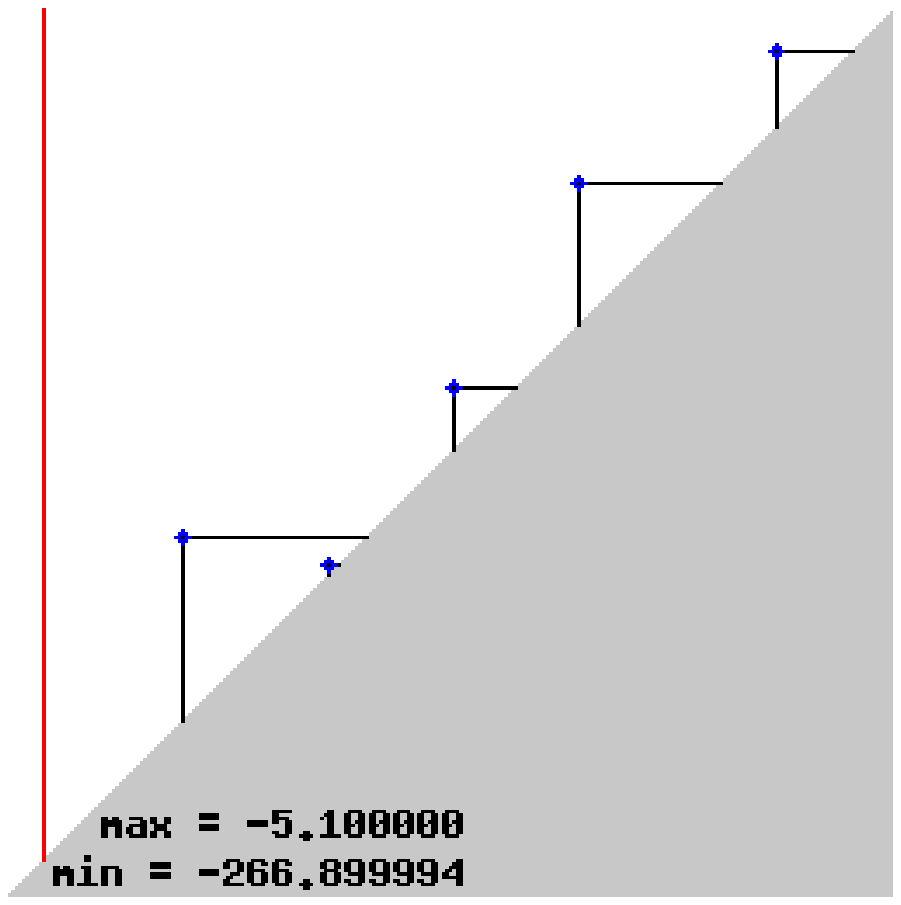} &
\includegraphics[width=0.15\linewidth]{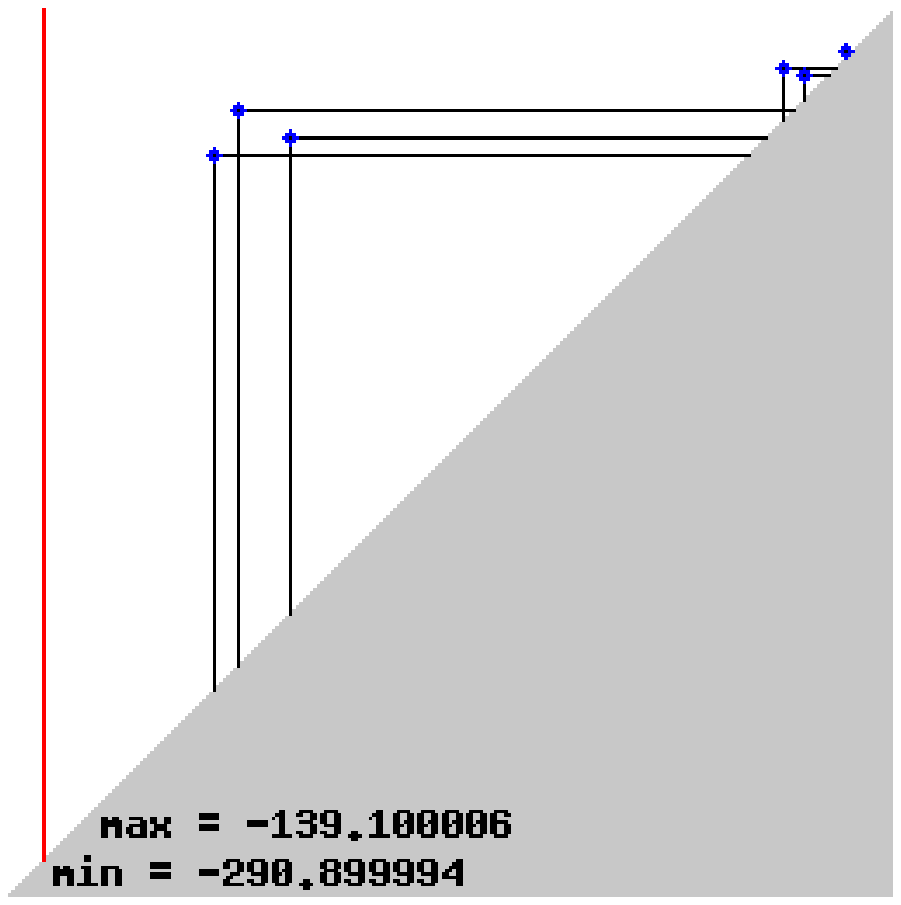} &
\includegraphics[width=0.15\linewidth]{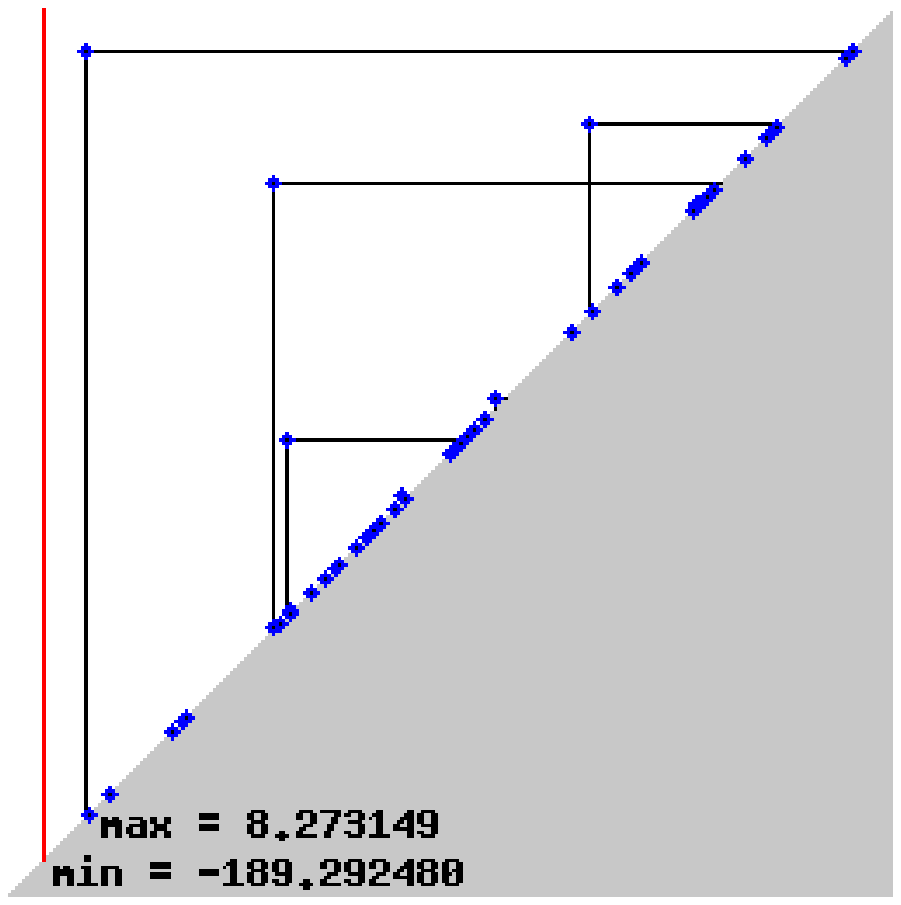} &
\includegraphics[width=0.15\linewidth]{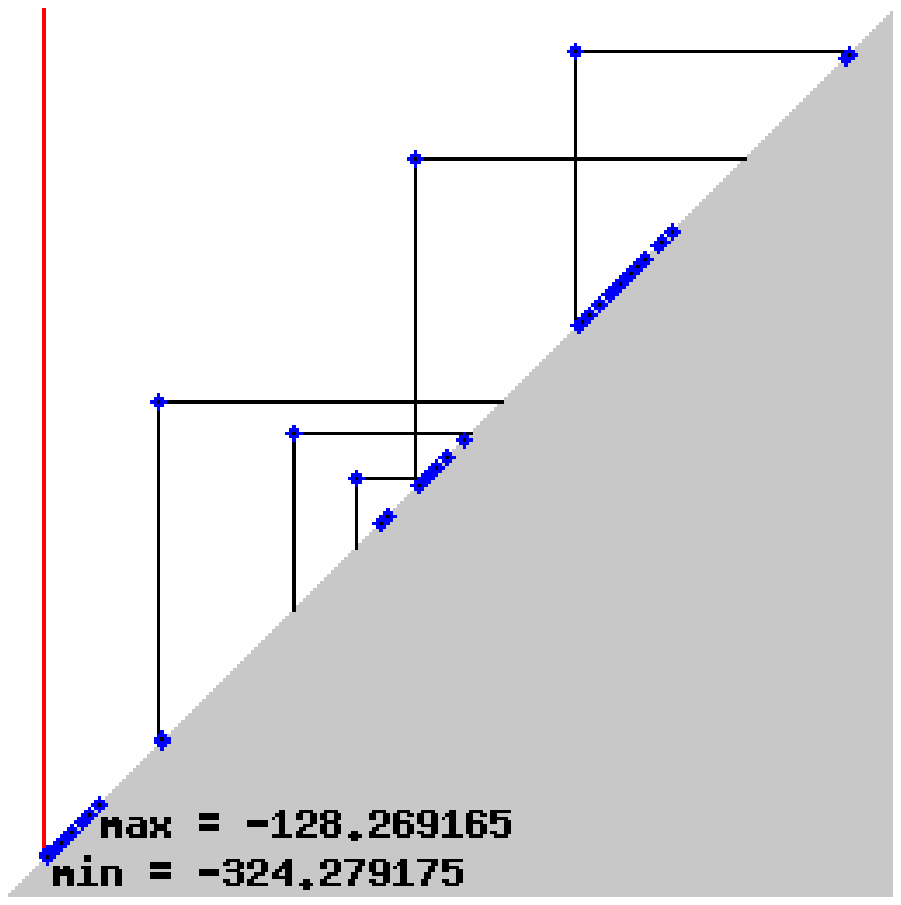} \\
\includegraphics[width=0.145\linewidth]{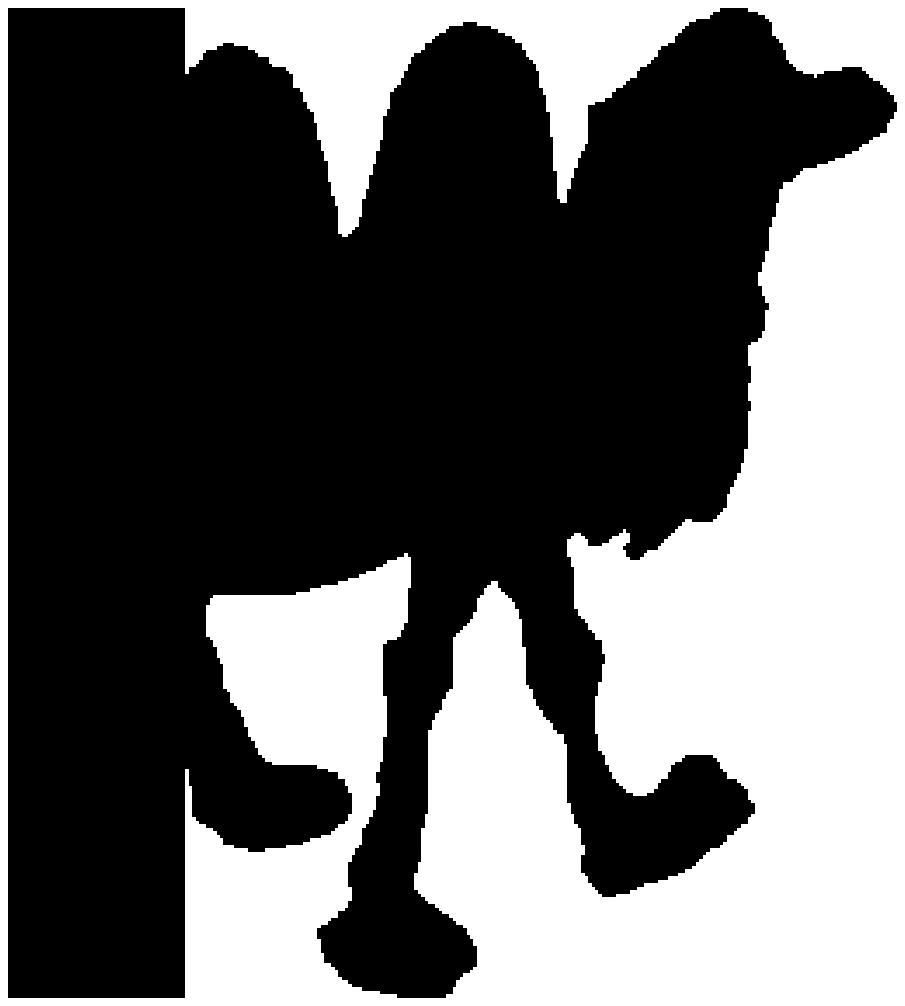} &
\includegraphics[width=0.15\linewidth]{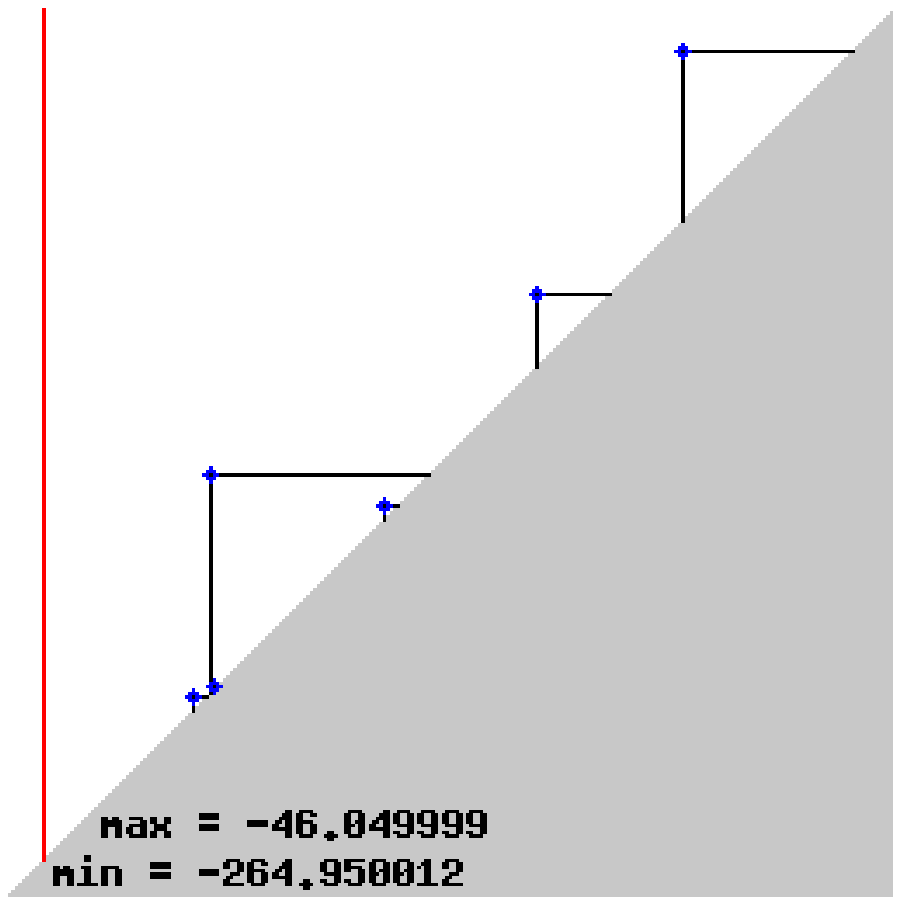} &
\includegraphics[width=0.15\linewidth]{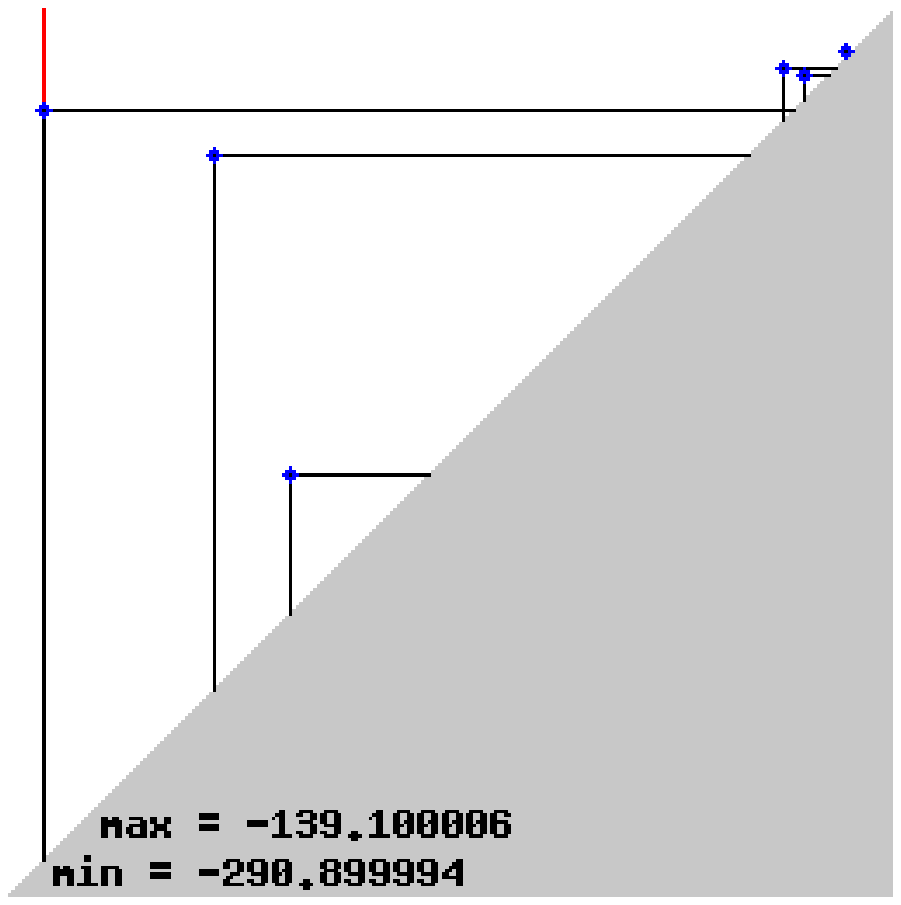} &
\includegraphics[width=0.15\linewidth]{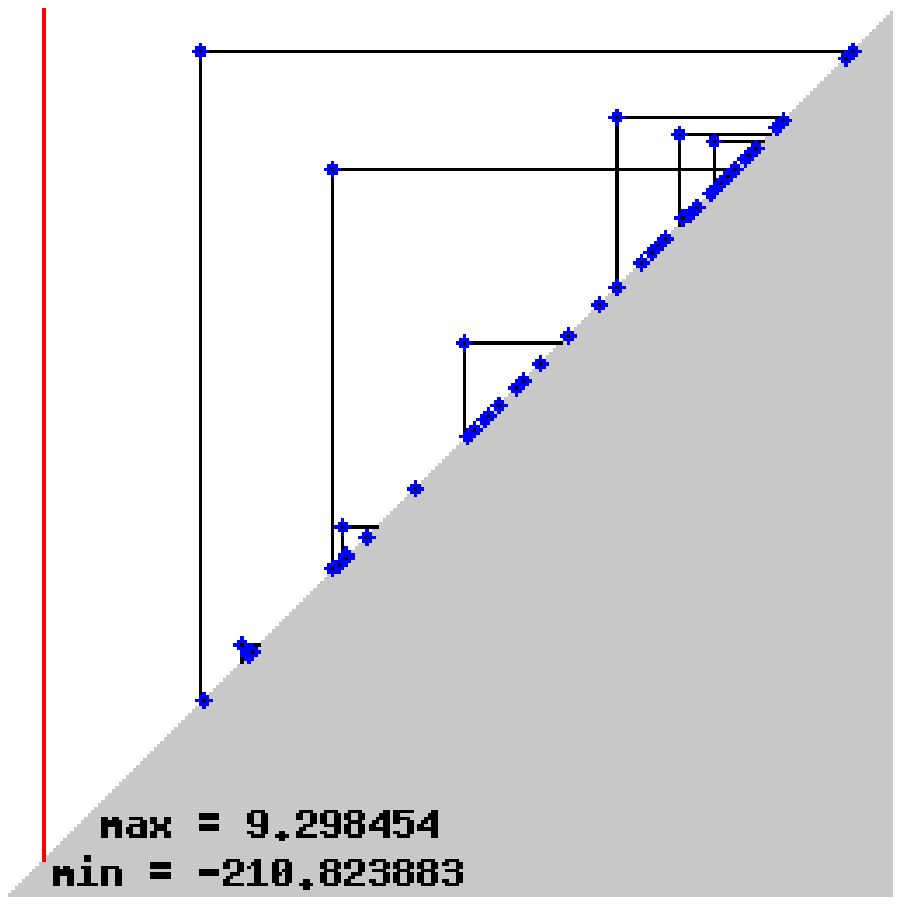} &
\includegraphics[width=0.15\linewidth]{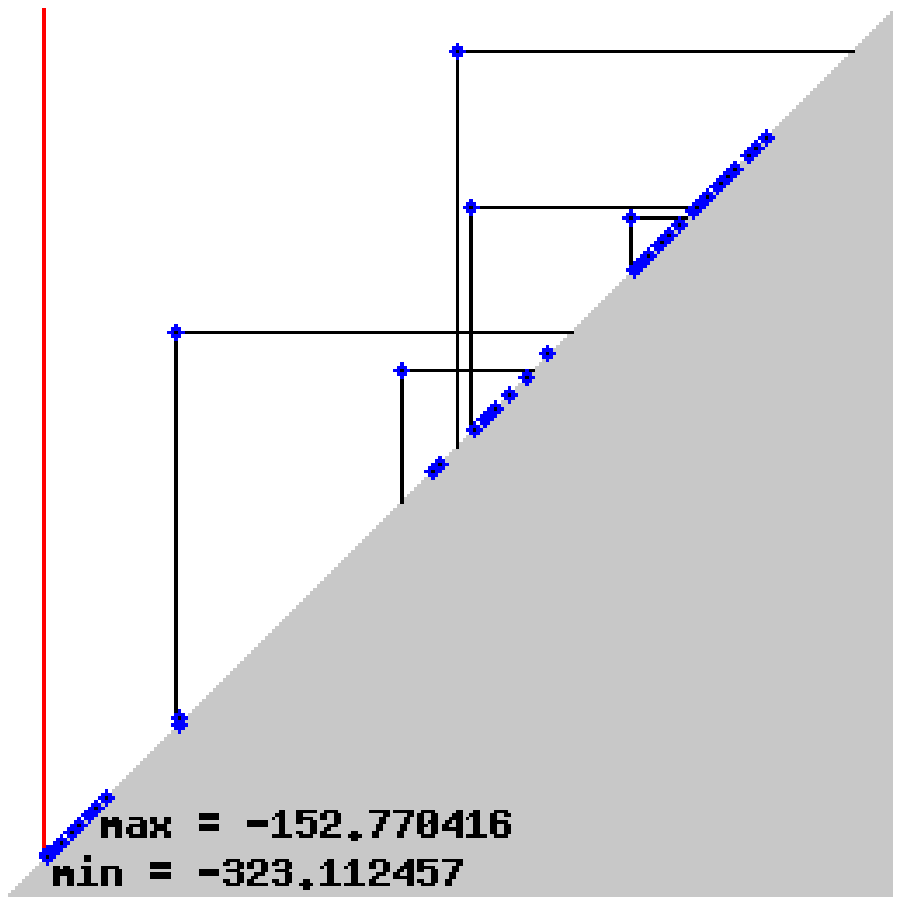} \\
\includegraphics[width=0.145\linewidth]{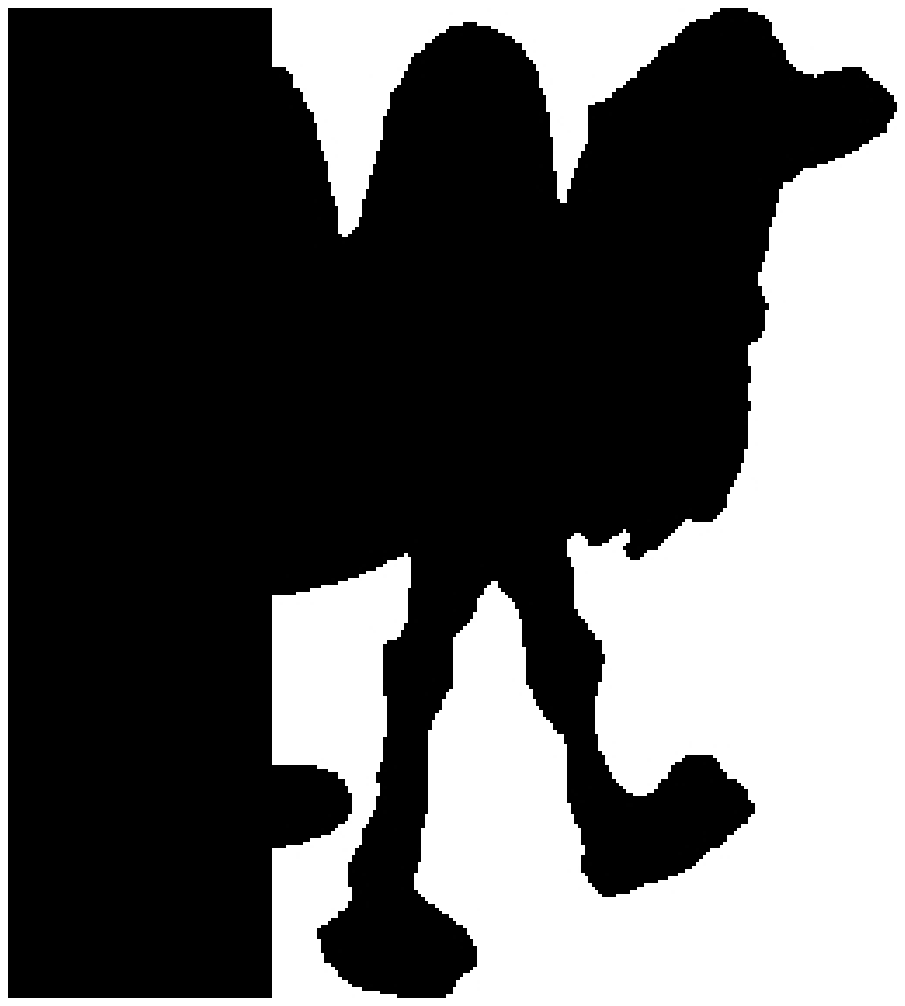} &
\includegraphics[width=0.15\linewidth]{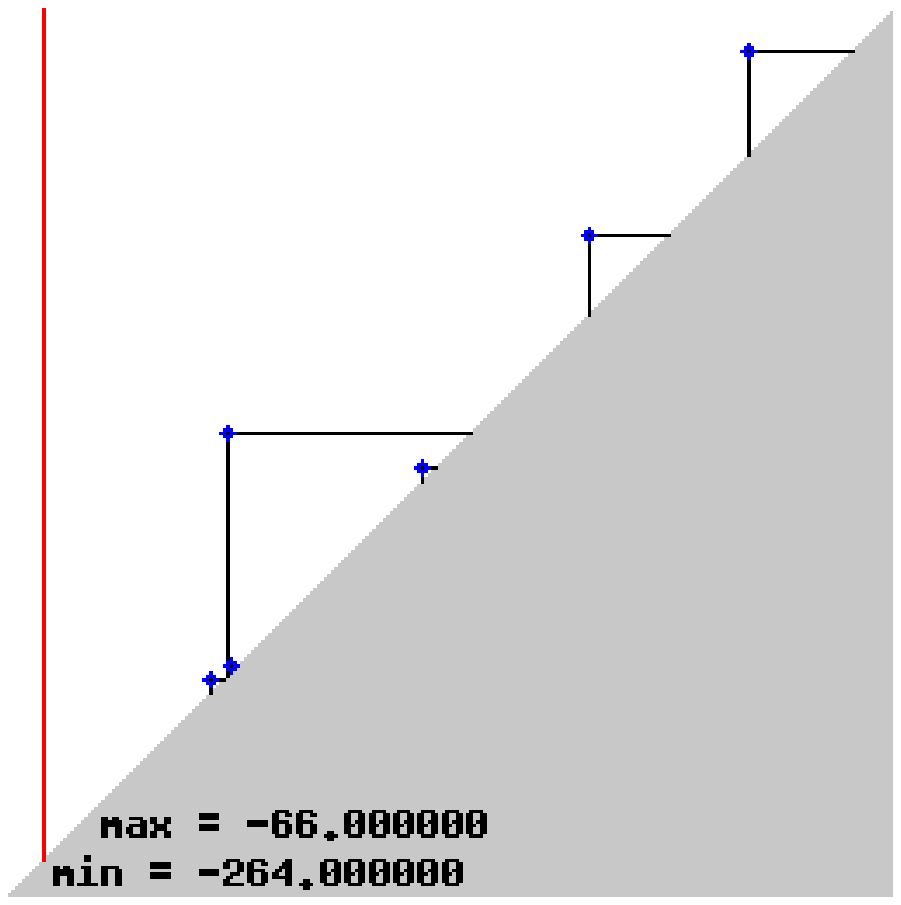} &
\includegraphics[width=0.15\linewidth]{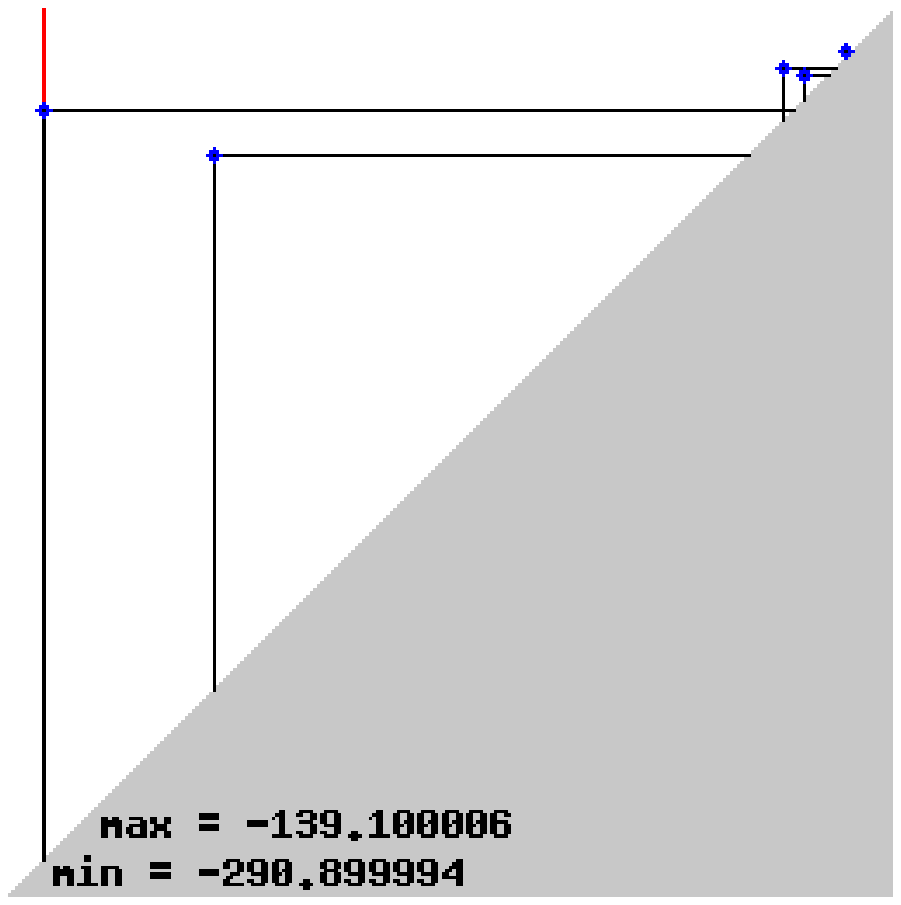} &
\includegraphics[width=0.15\linewidth]{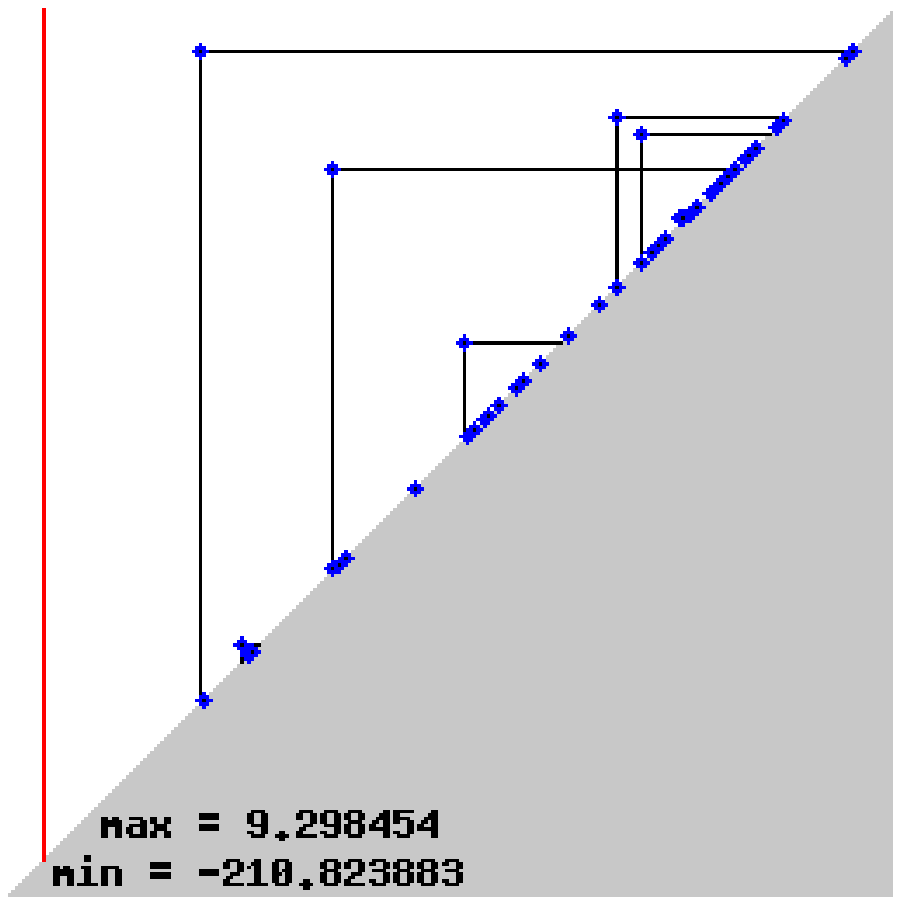} &
\includegraphics[width=0.15\linewidth]{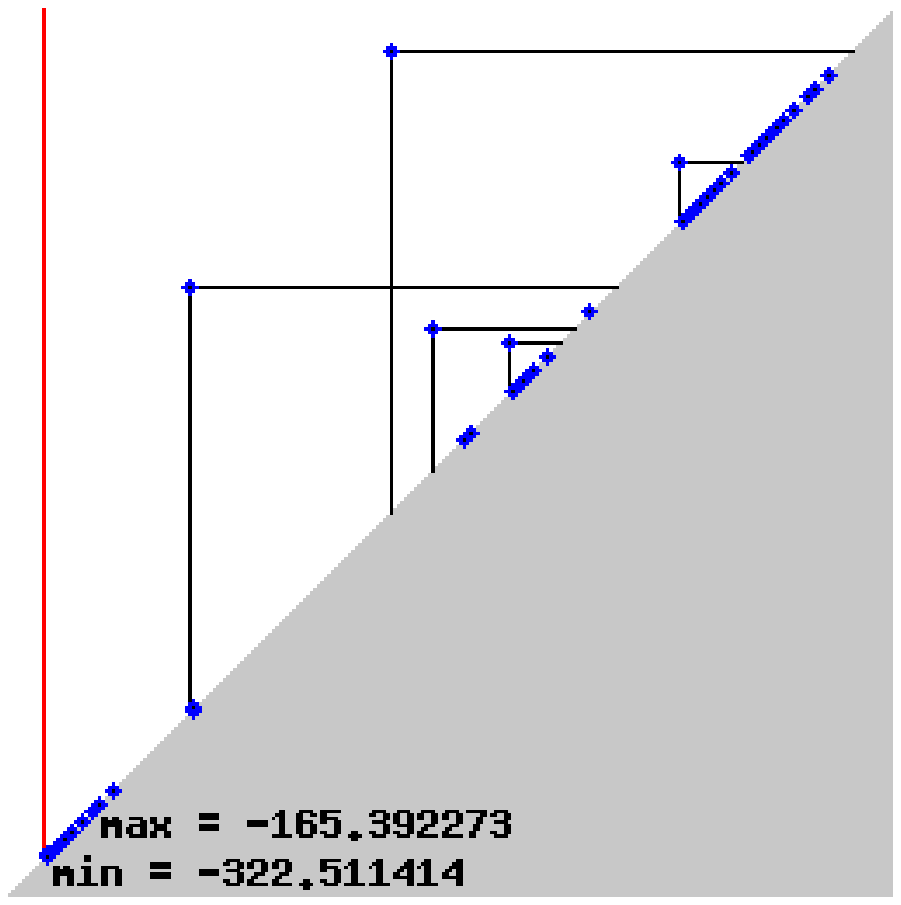} \\
\includegraphics[width=0.145\linewidth]{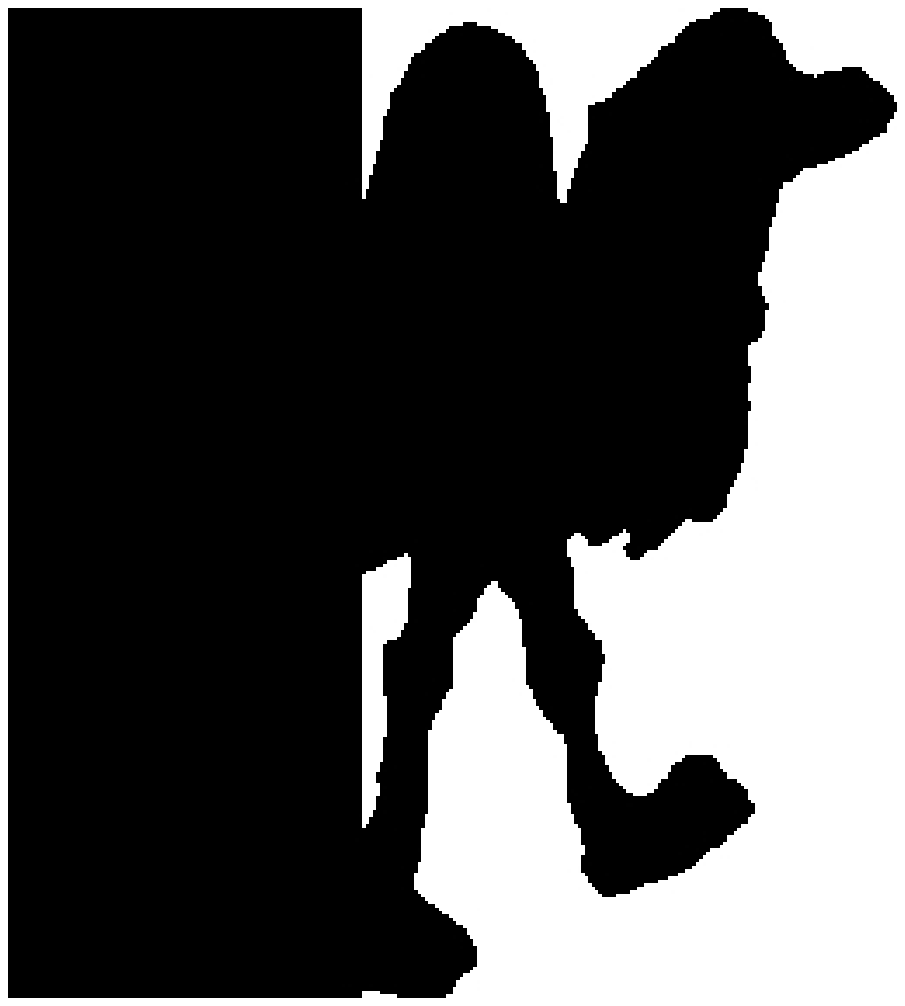} &
\includegraphics[width=0.15\linewidth]{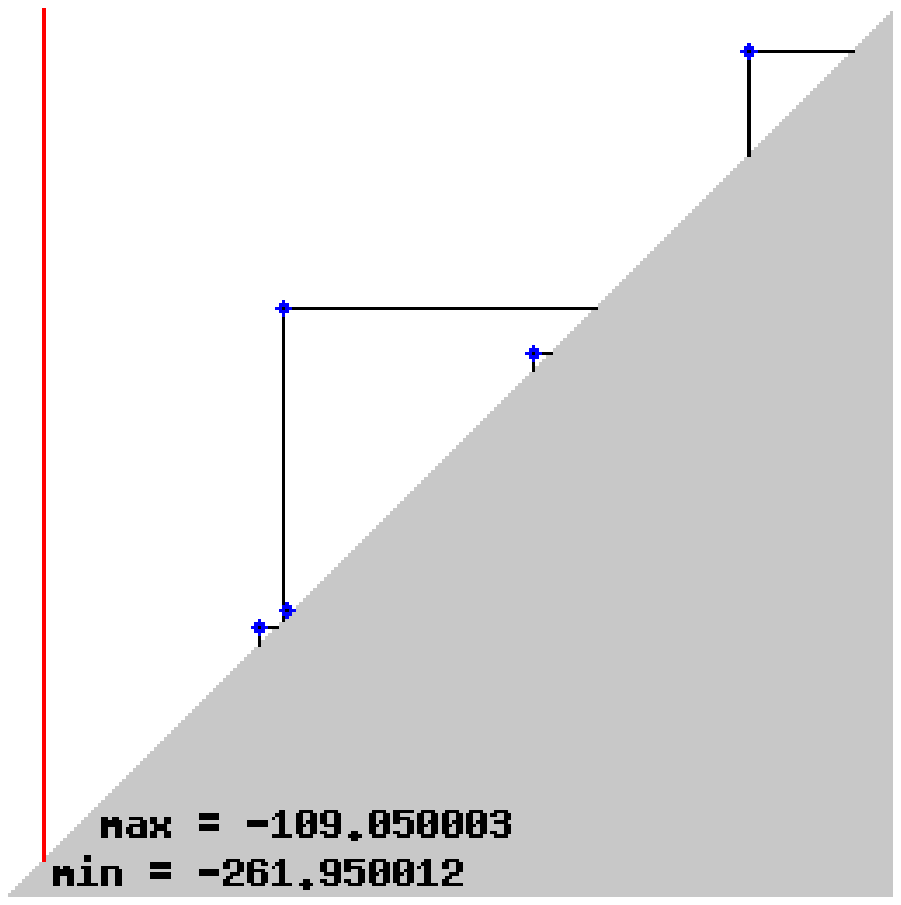} &
\includegraphics[width=0.15\linewidth]{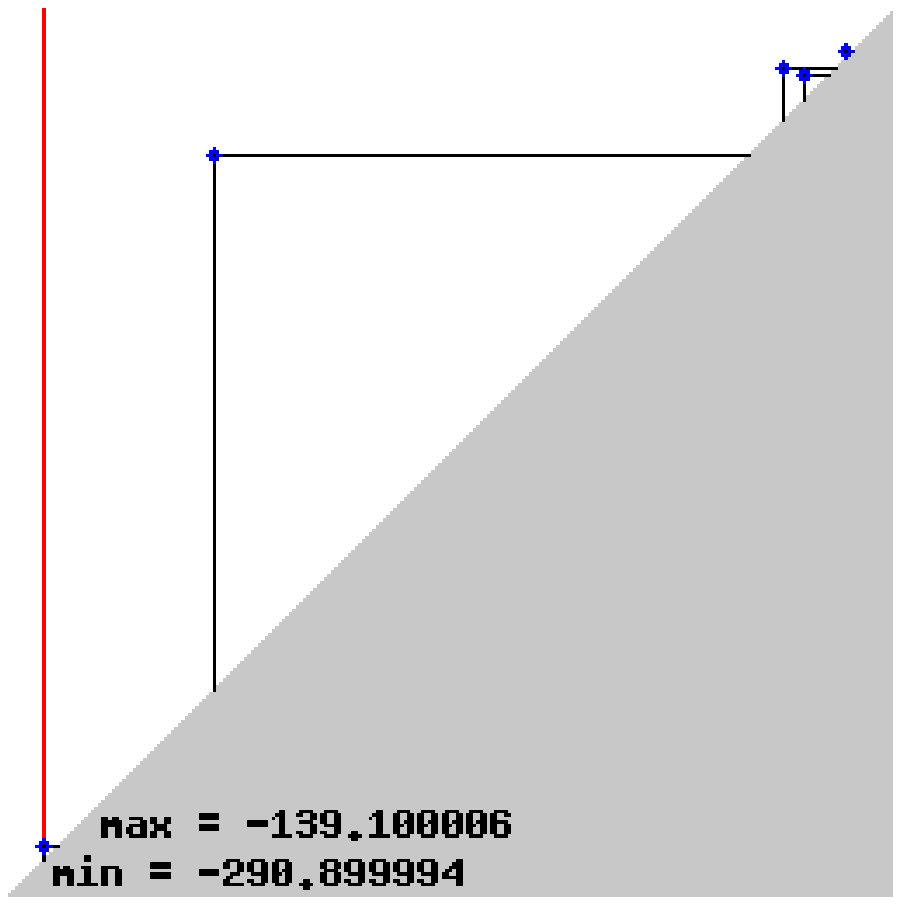} &
\includegraphics[width=0.15\linewidth]{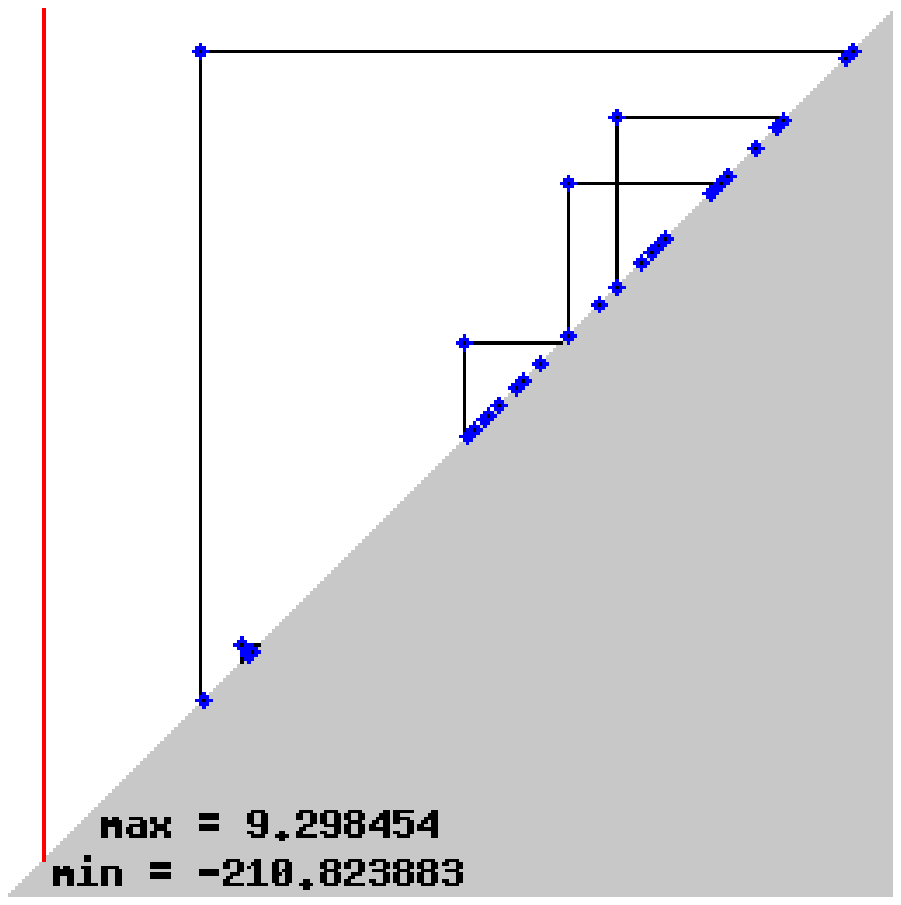} &
\includegraphics[width=0.15\linewidth]{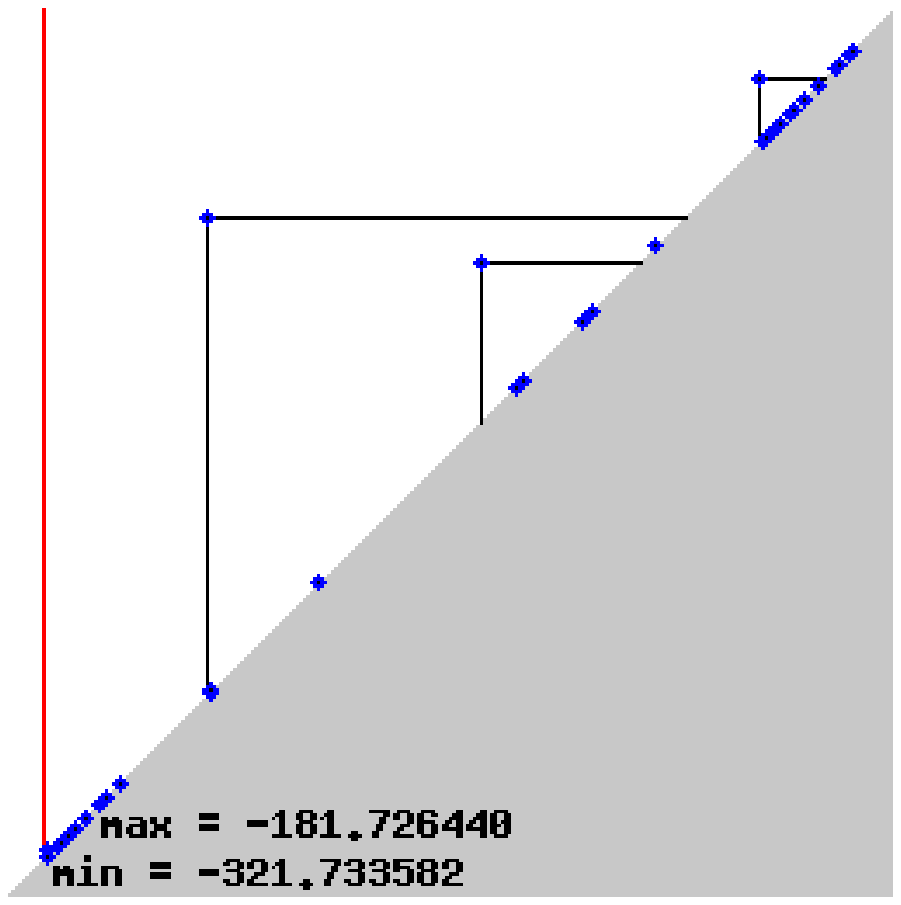} \\
\end{tabular}
\caption{\footnotesize{The first column: (row 1) original
``camel'' shape, (rows 2--4) occluded from top by $20\%$,  $30\%$,
$40\%$,  (row 5--7) occluded from left by $20\%$,  $30\%$,
$40\%$. From second column onwards: corresponding size functions
related to measuring functions defined as minus distances from
four lines rotated by $0$, $\pi/4$, $\pi/2$, $3\pi/4$, with
respect to the horizontal position.}} \label{camelHV}
\end{table}
On the other hand,  Table \ref{frogHV} shows a ``frog'', which is
a connected shape with several holes. The different percentages of
occlusion can create some new holes or destroy them (see rows
3--4).
\begin{table}[htbp]
\begin{tabular}{cccccc}
\includegraphics[width=0.15\linewidth]{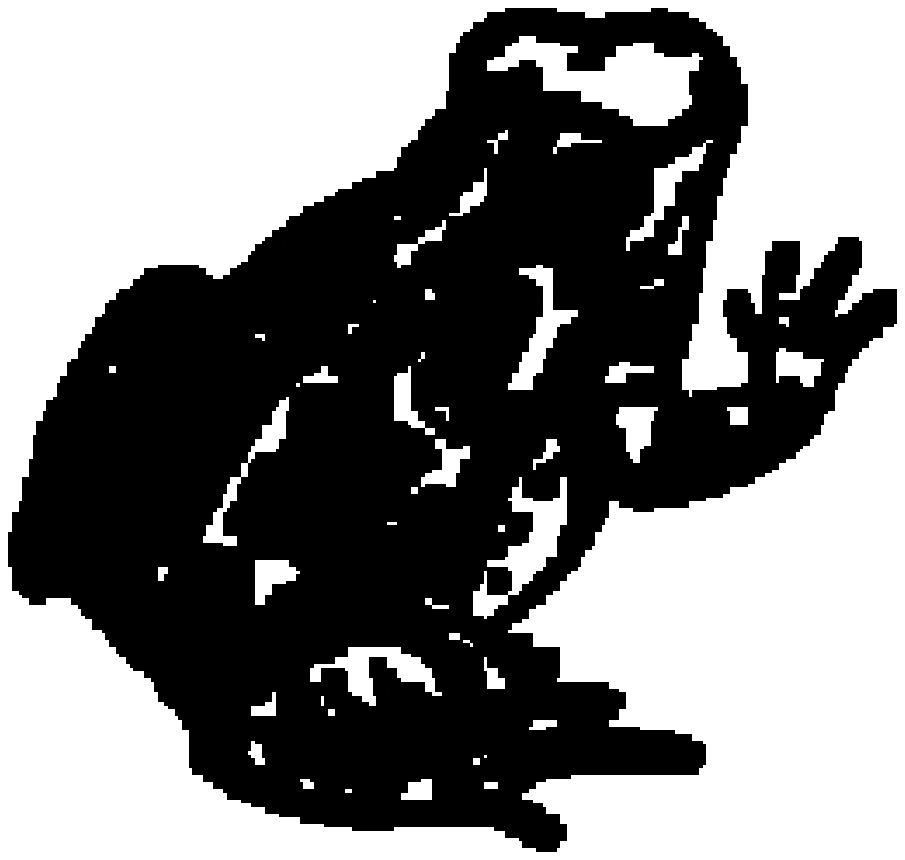} &
\includegraphics[width=0.15\linewidth]{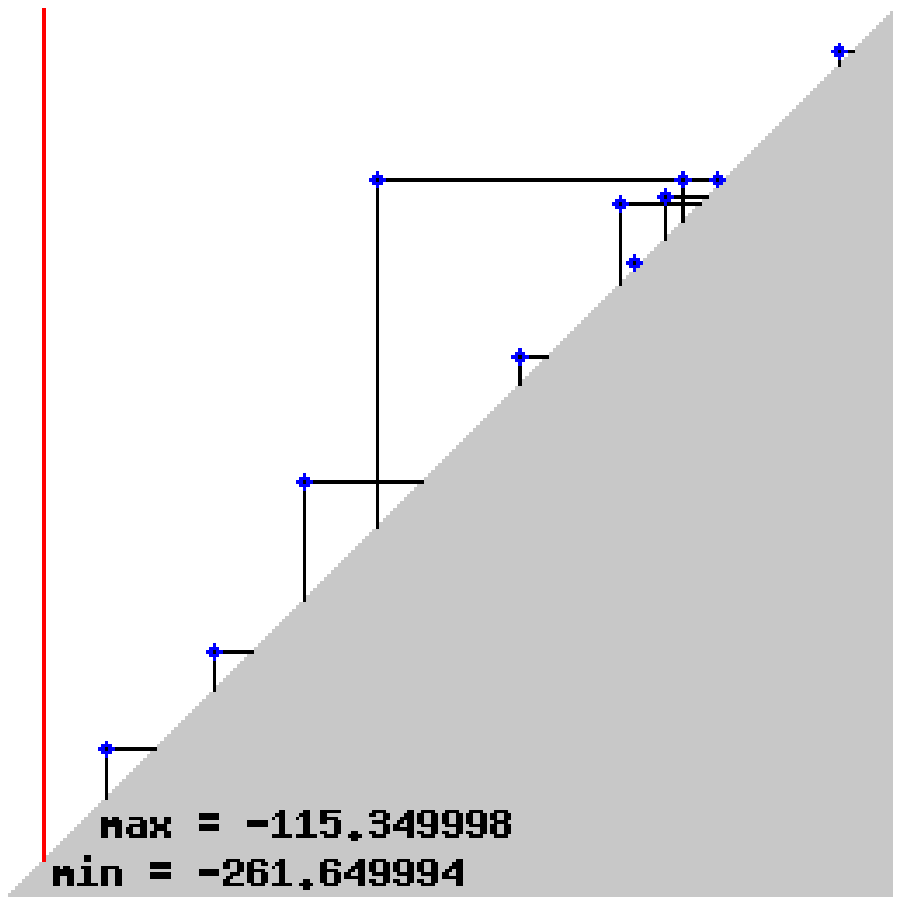} &
\includegraphics[width=0.15\linewidth]{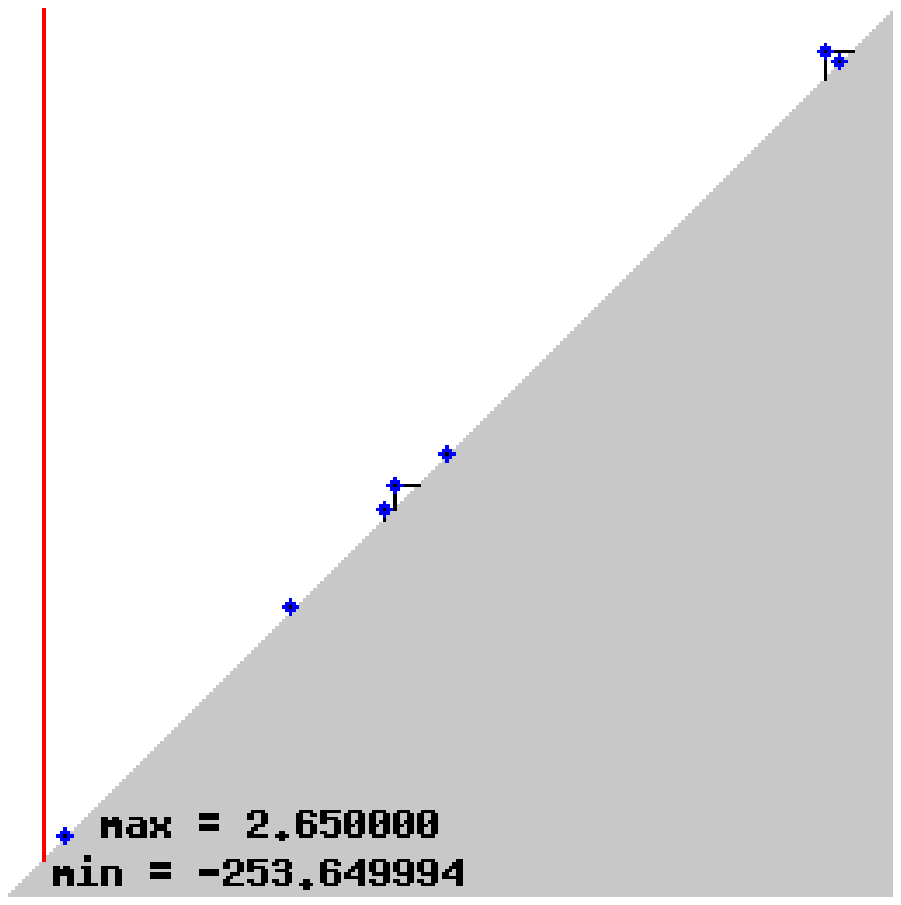} &
\includegraphics[width=0.15\linewidth]{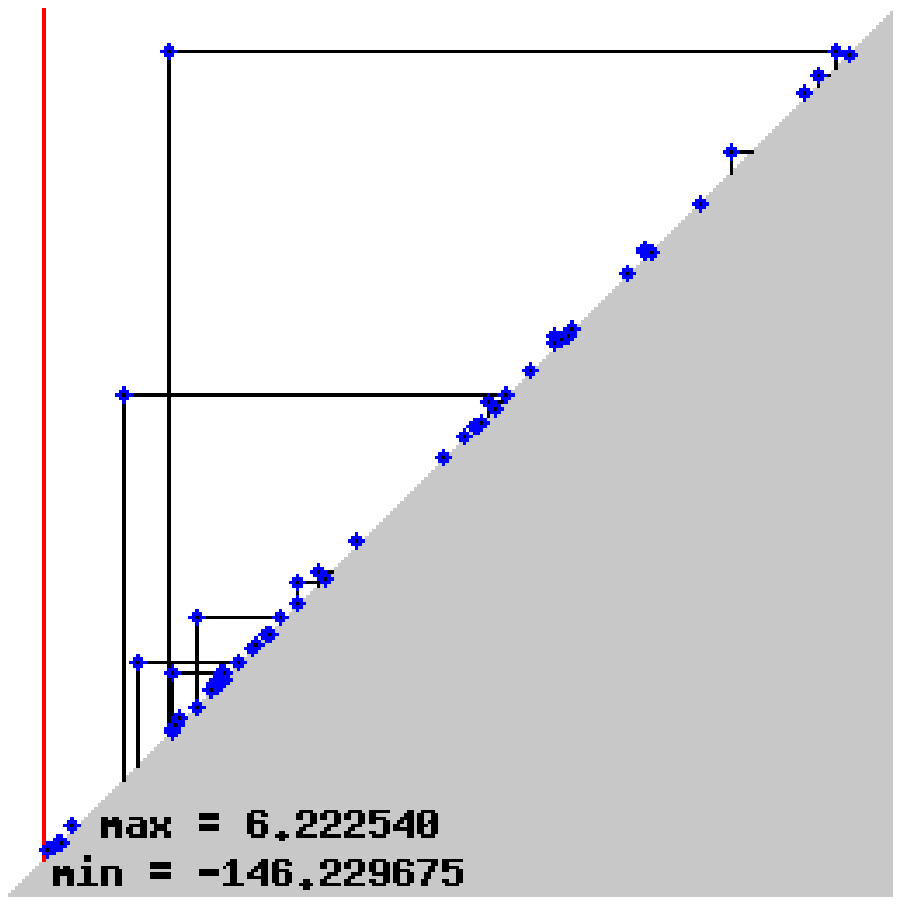} &
\includegraphics[width=0.15\linewidth]{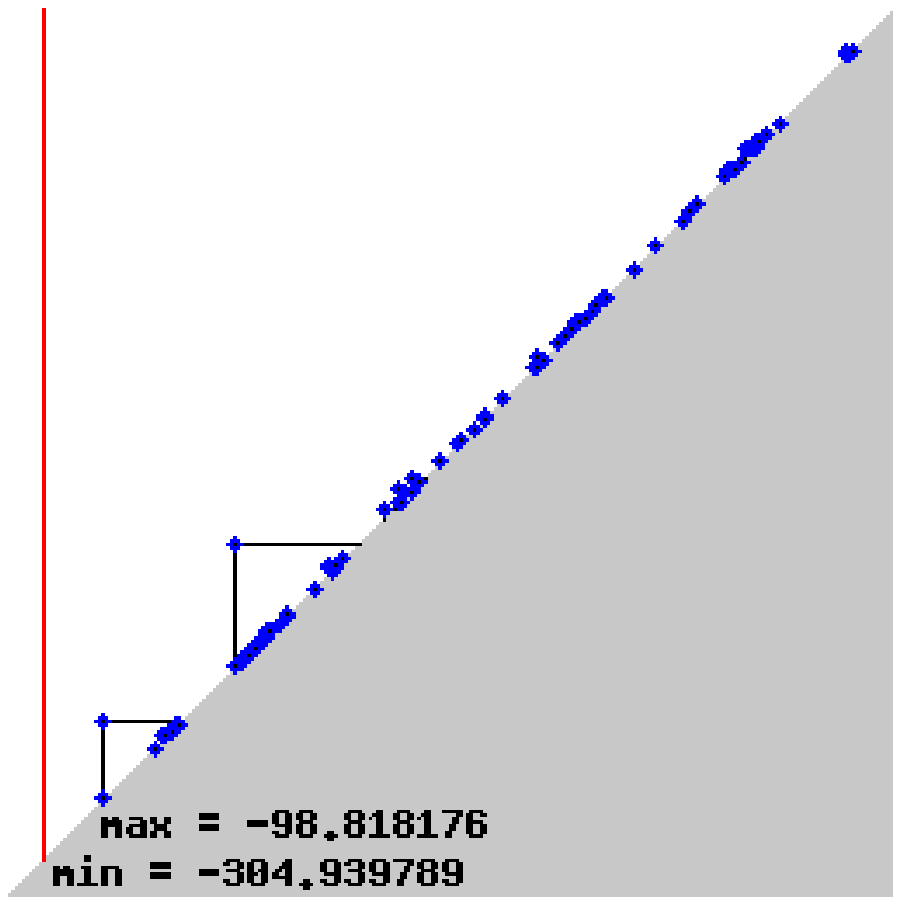} \\
\includegraphics[width=0.15\linewidth]{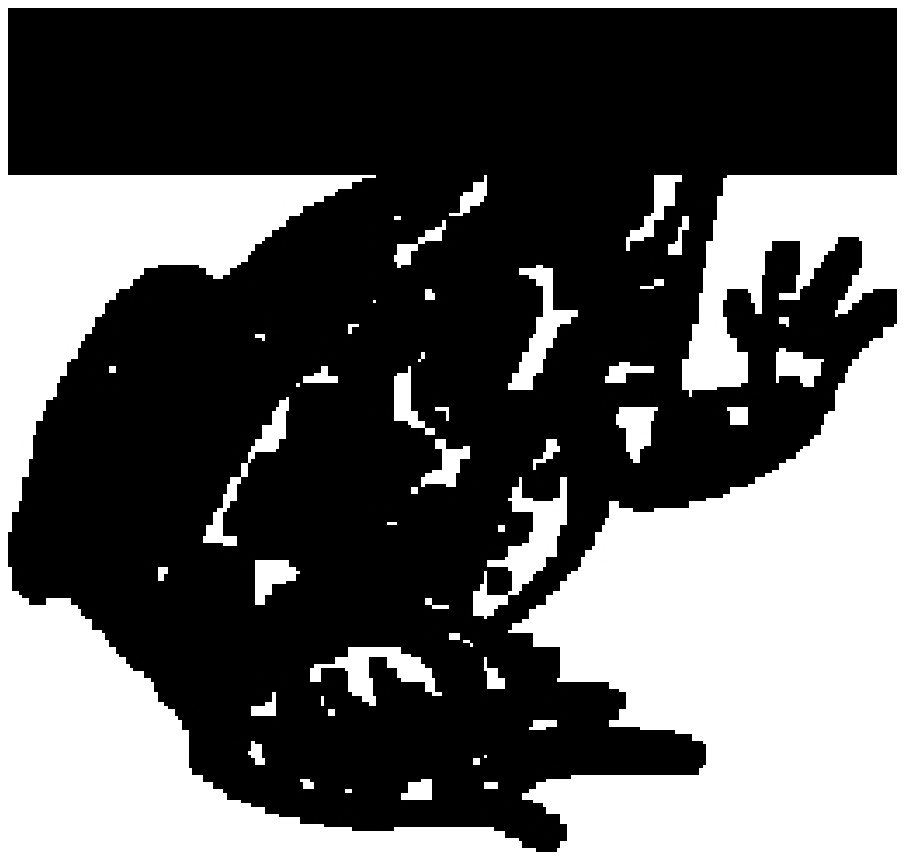} &
\includegraphics[width=0.15\linewidth]{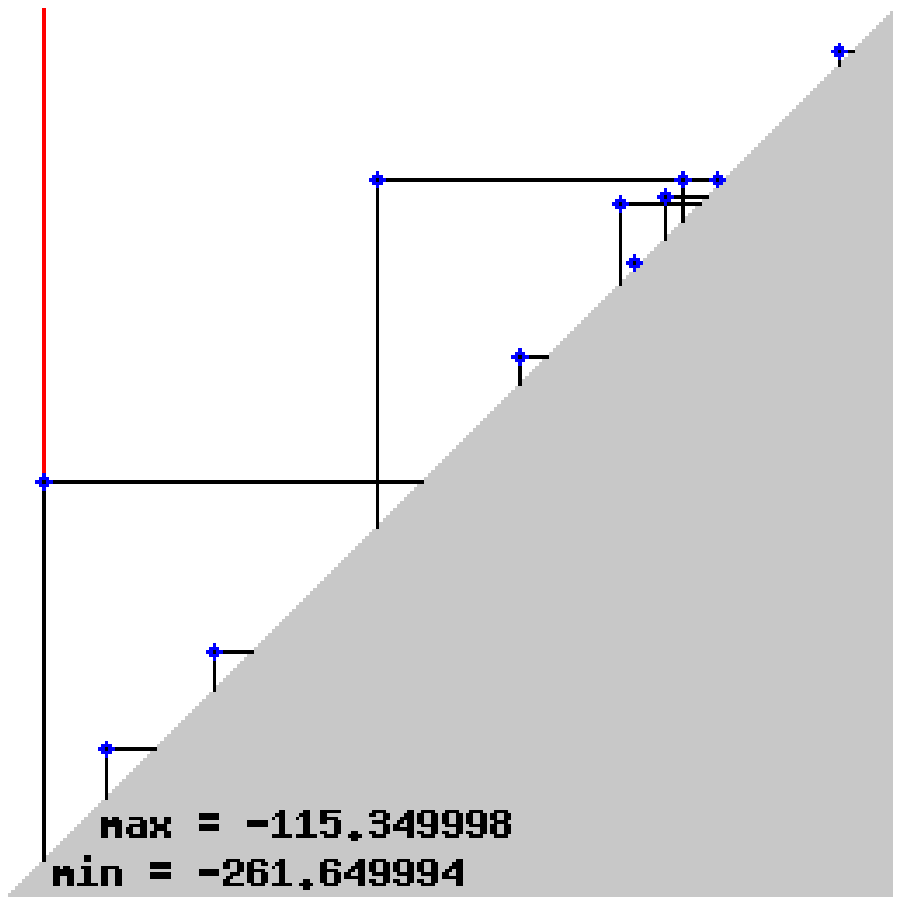} &
\includegraphics[width=0.15\linewidth]{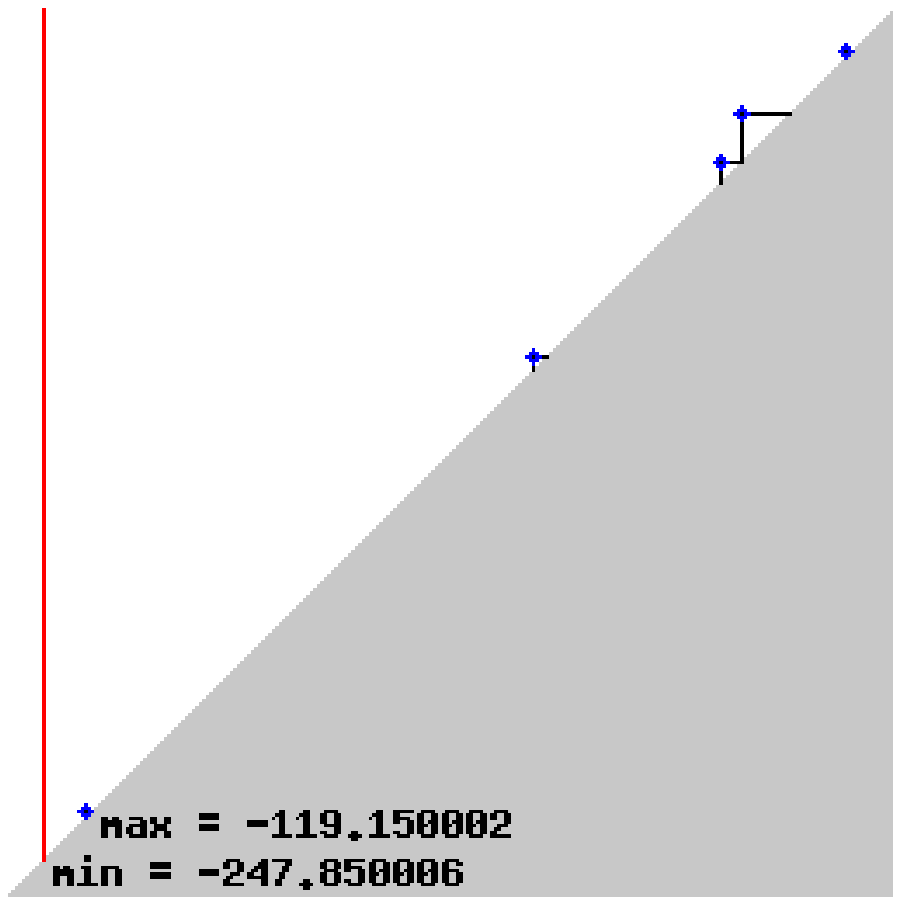} &
\includegraphics[width=0.15\linewidth]{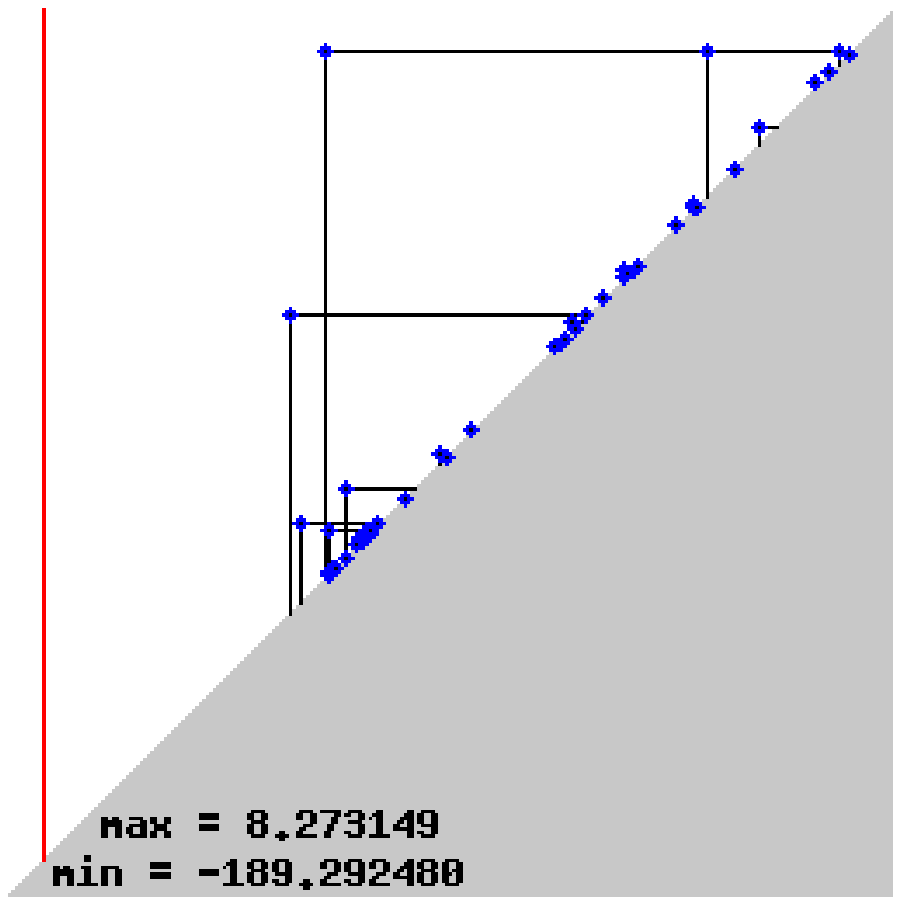} &
\includegraphics[width=0.15\linewidth]{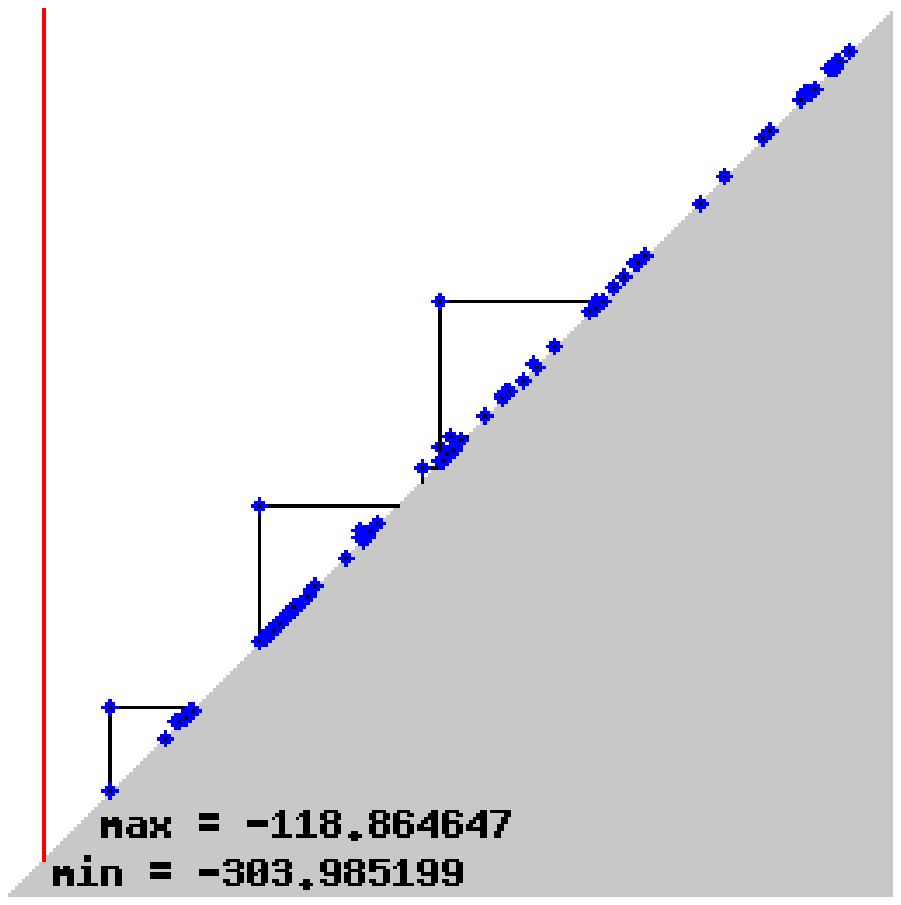} \\
\includegraphics[width=0.15\linewidth]{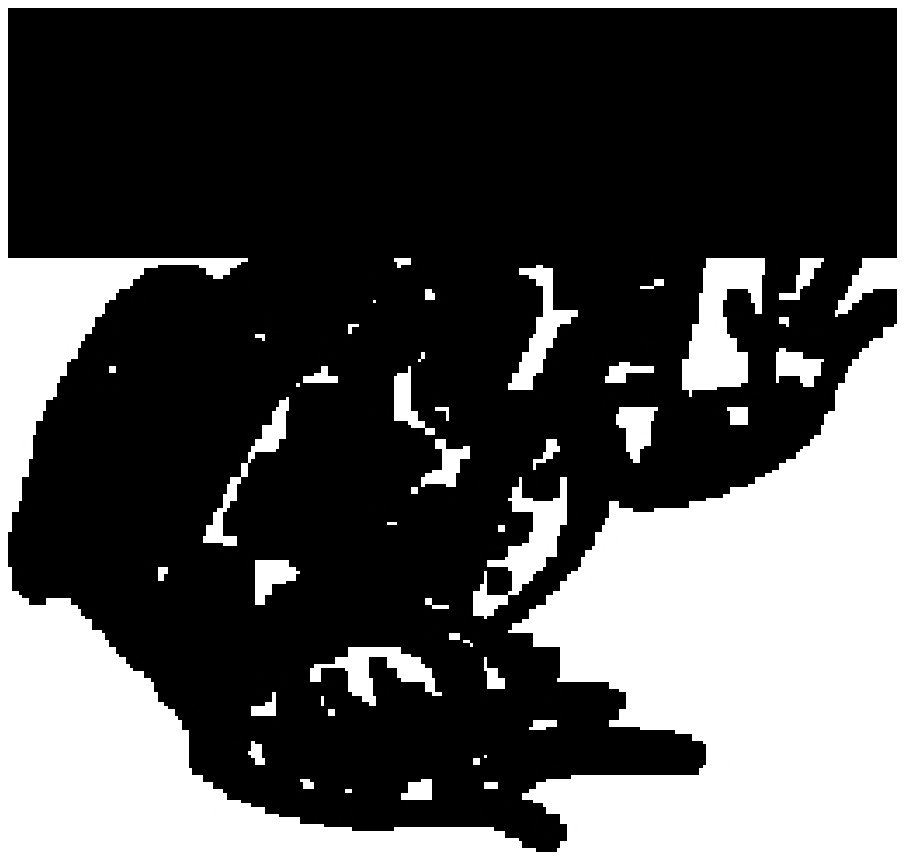} &
\includegraphics[width=0.15\linewidth]{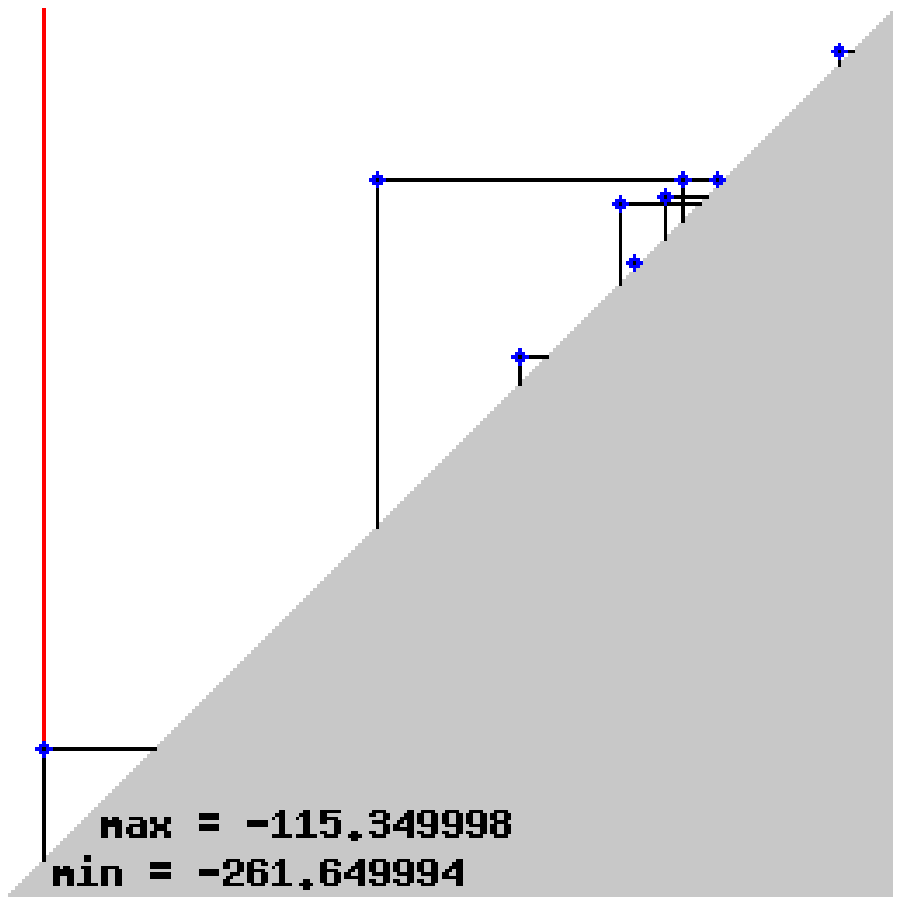} &
\includegraphics[width=0.15\linewidth]{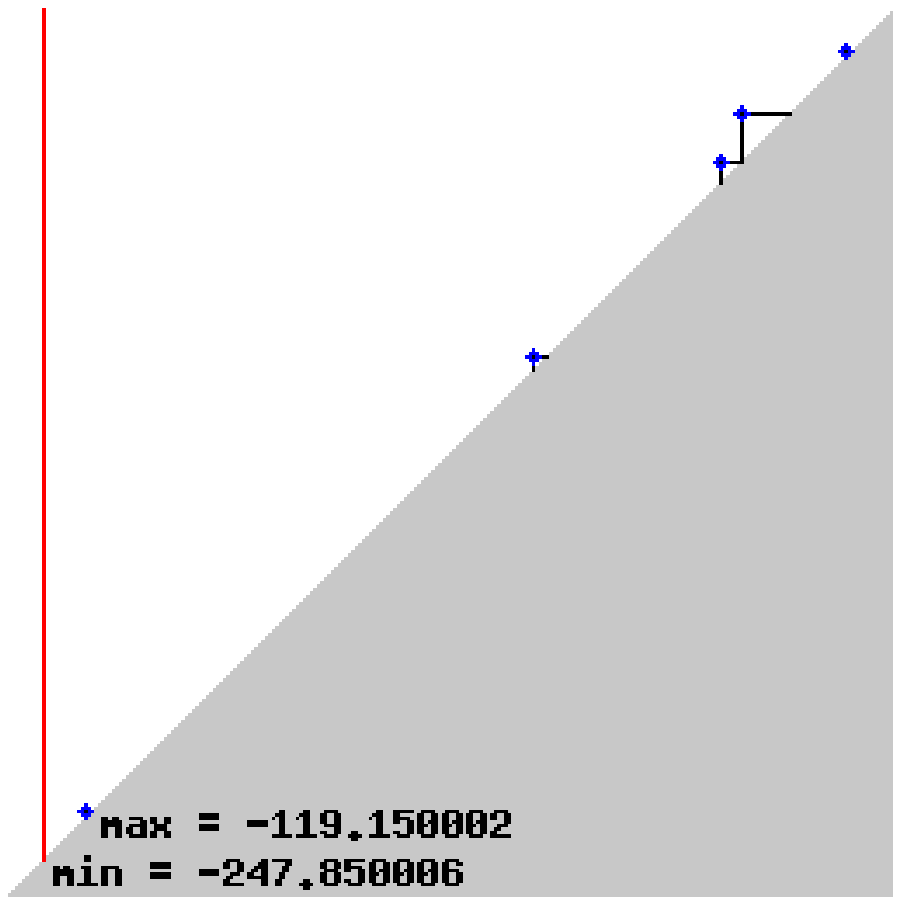} &
\includegraphics[width=0.15\linewidth]{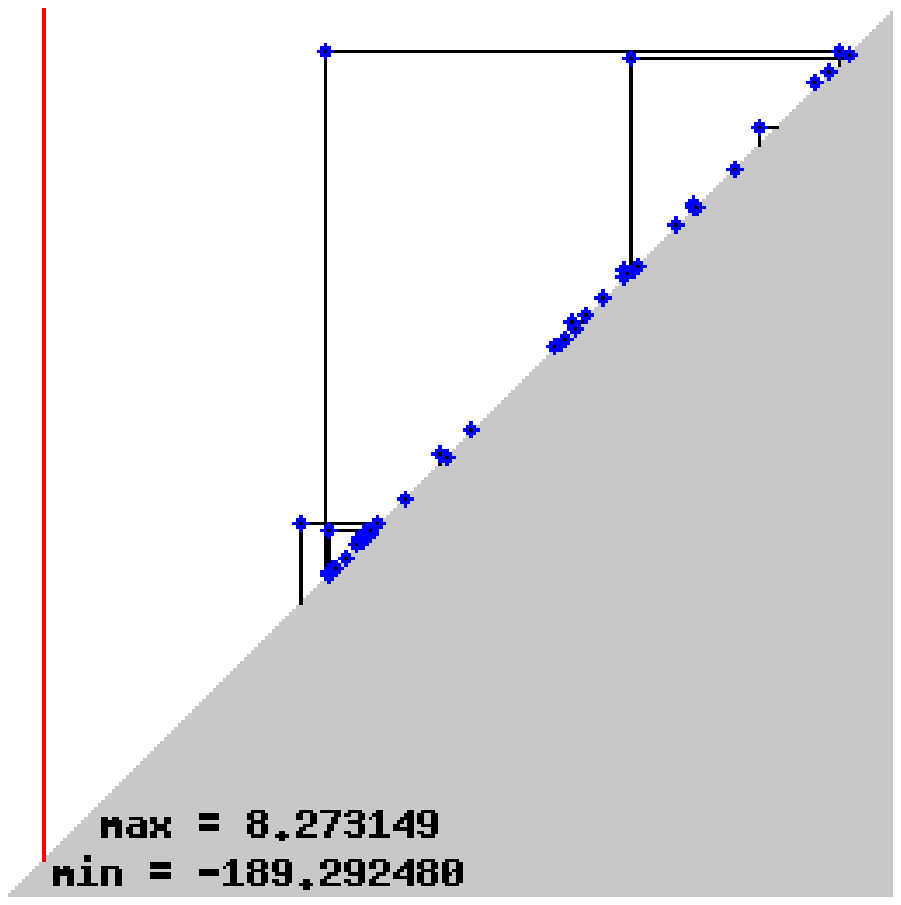} &
\includegraphics[width=0.15\linewidth]{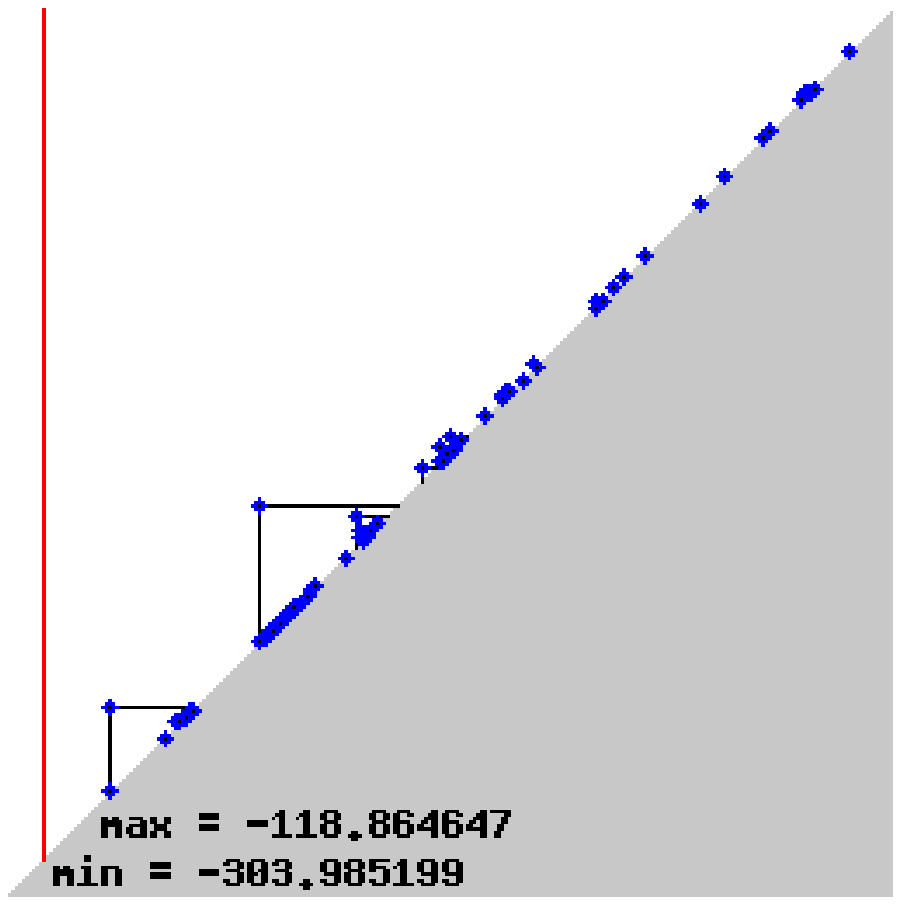} \\
\includegraphics[width=0.15\linewidth]{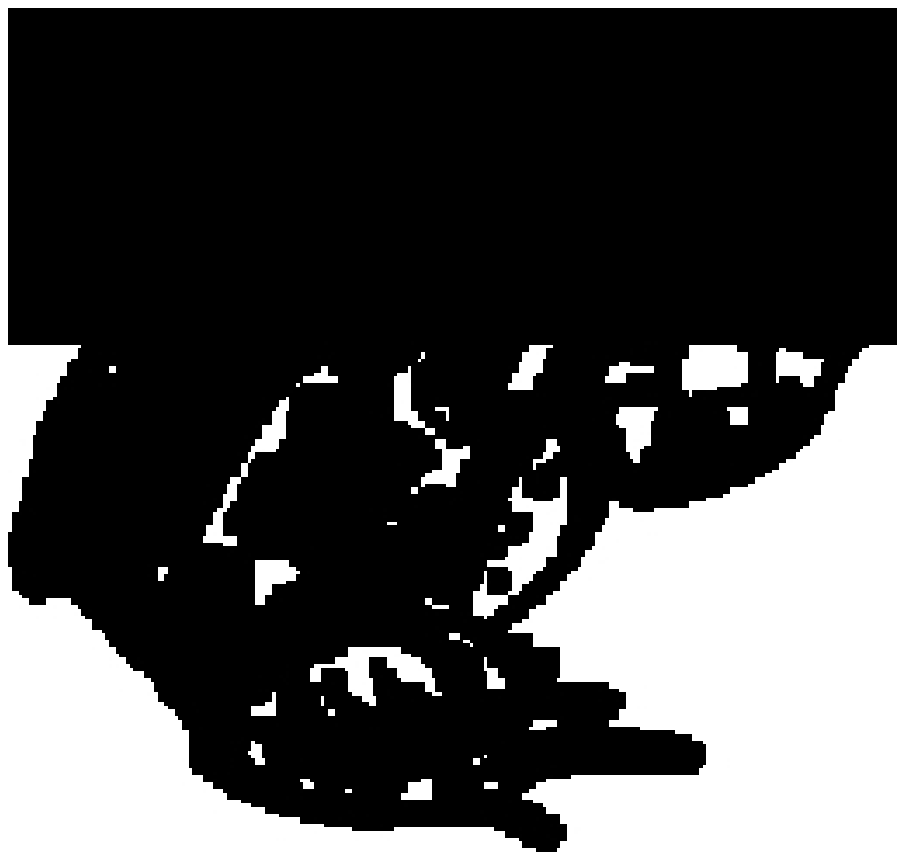} &
\includegraphics[width=0.15\linewidth]{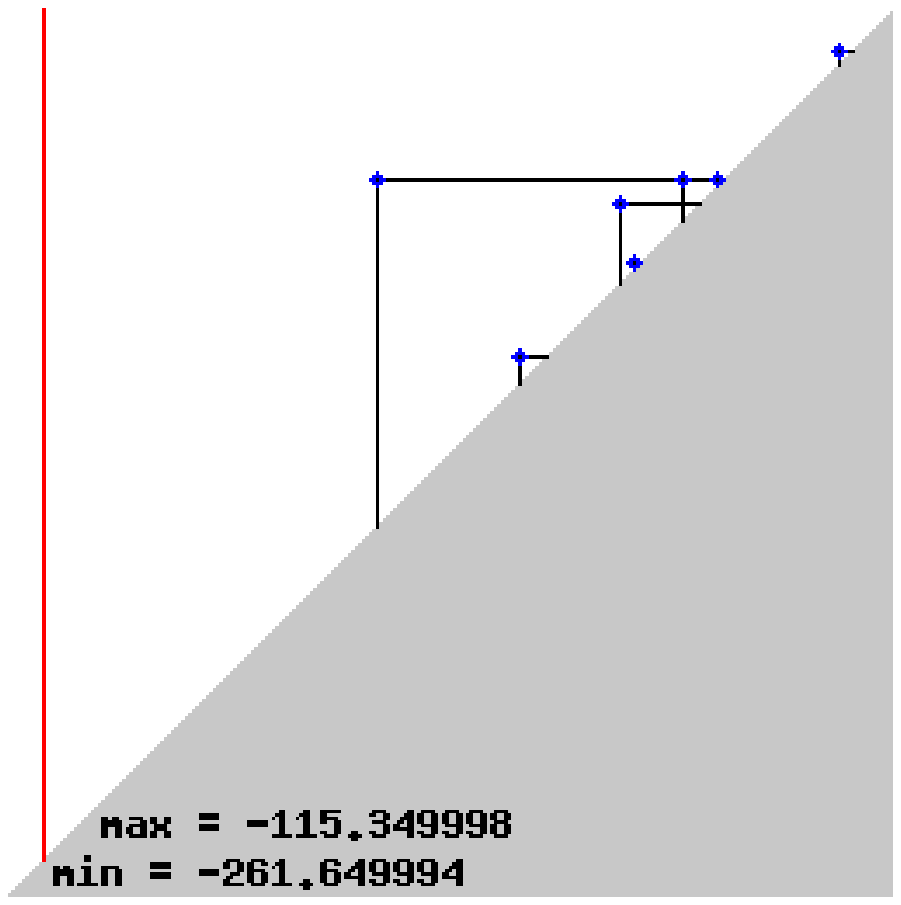} &
\includegraphics[width=0.15\linewidth]{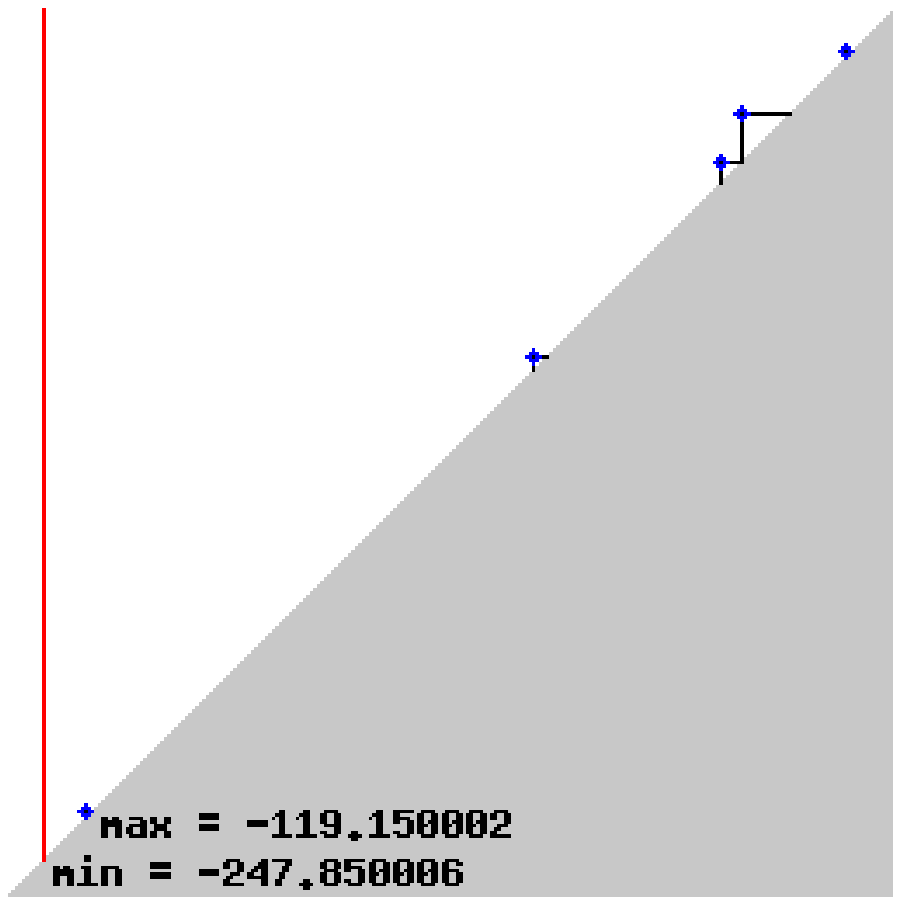} &
\includegraphics[width=0.15\linewidth]{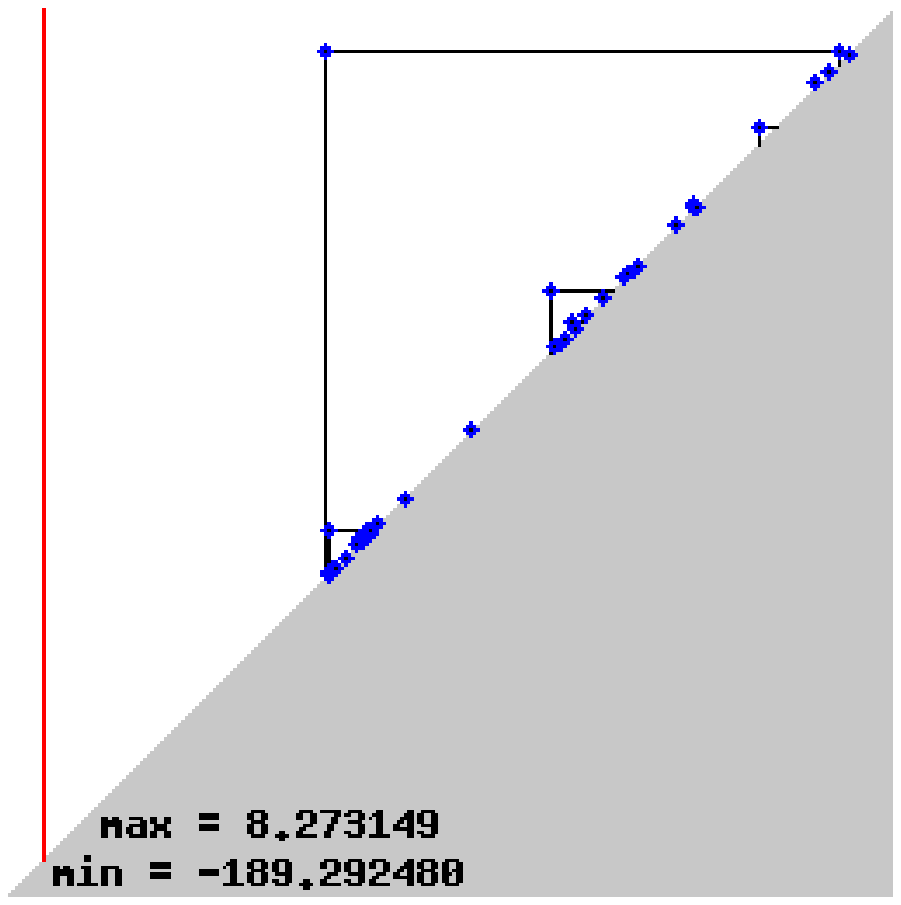} &
\includegraphics[width=0.15\linewidth]{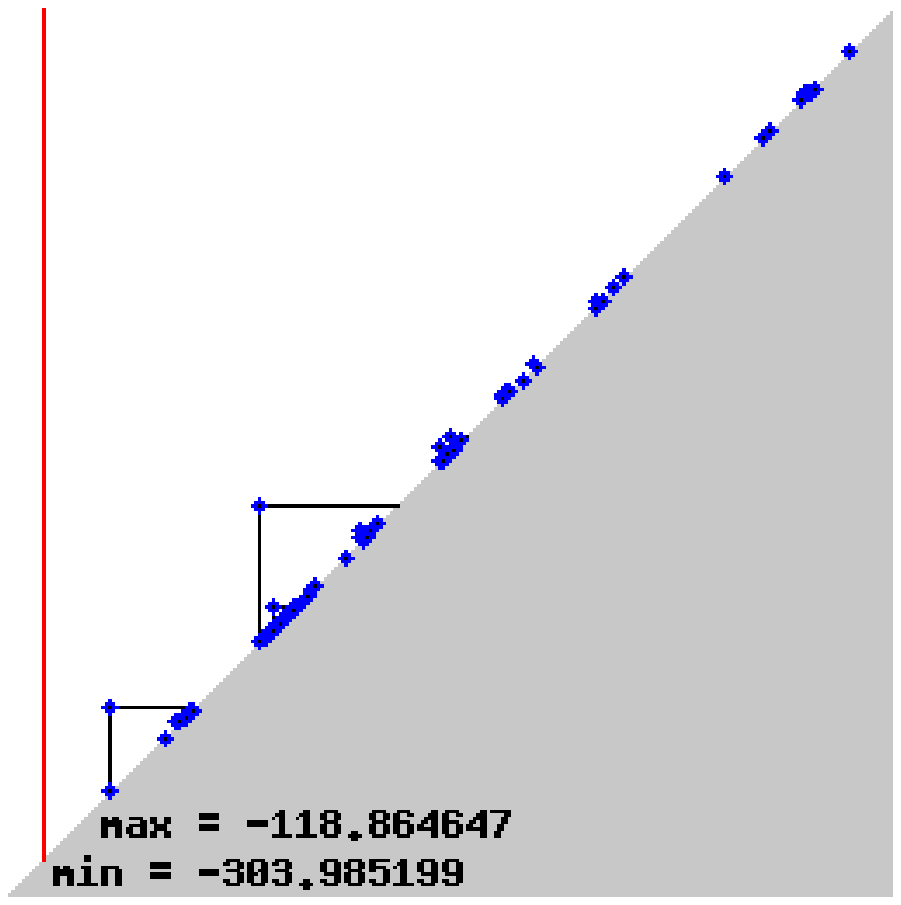} \\
\includegraphics[width=0.15\linewidth]{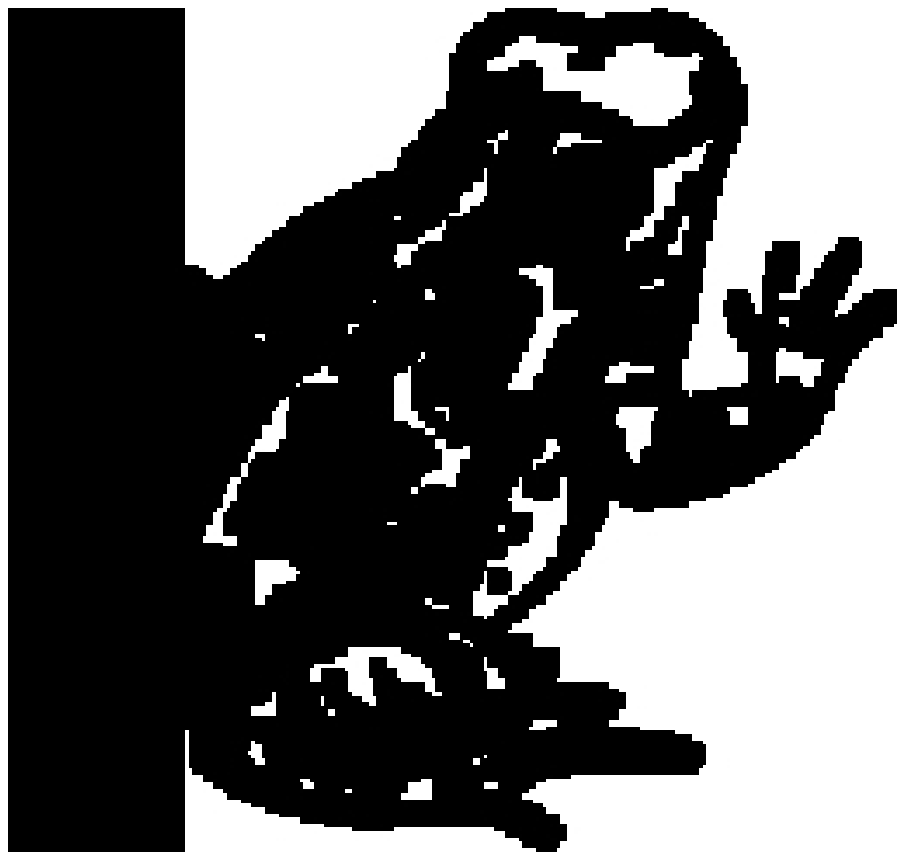} &
\includegraphics[width=0.15\linewidth]{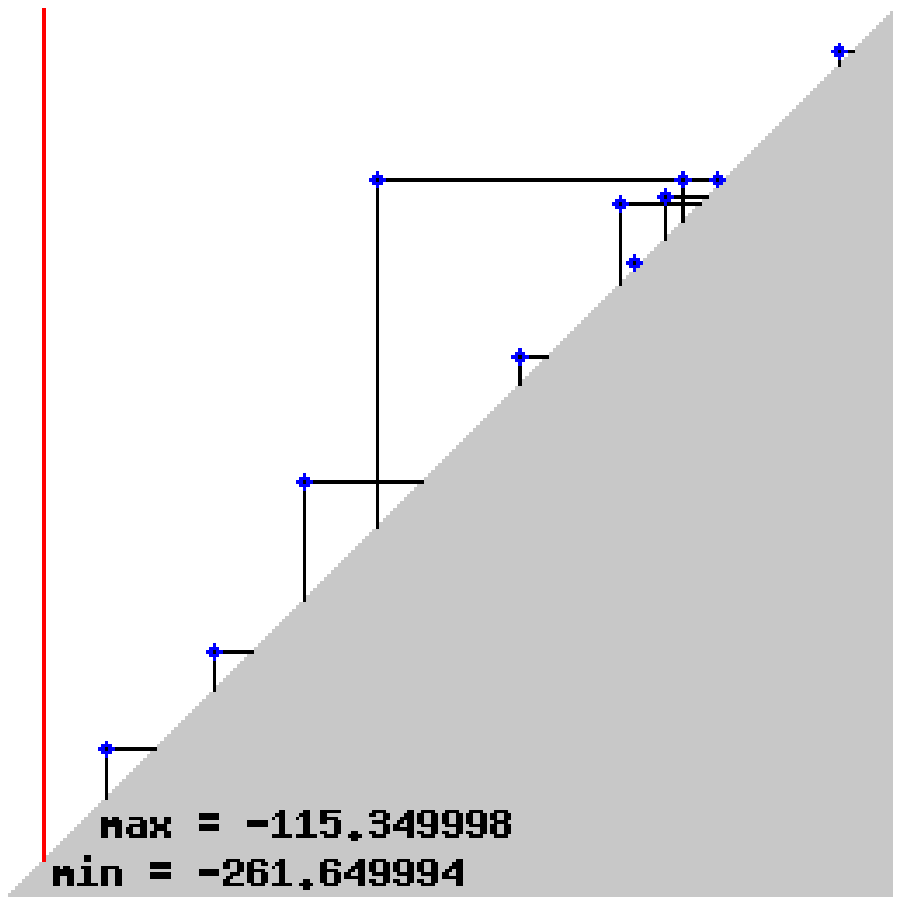} &
\includegraphics[width=0.15\linewidth]{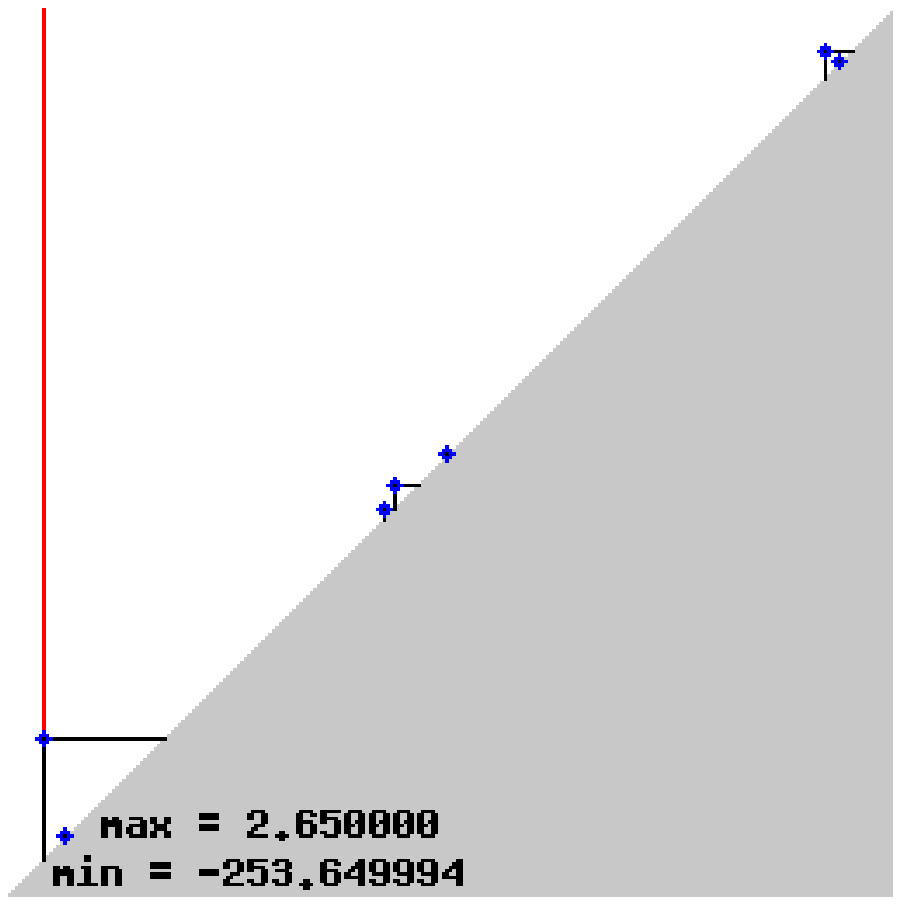} &
\includegraphics[width=0.15\linewidth]{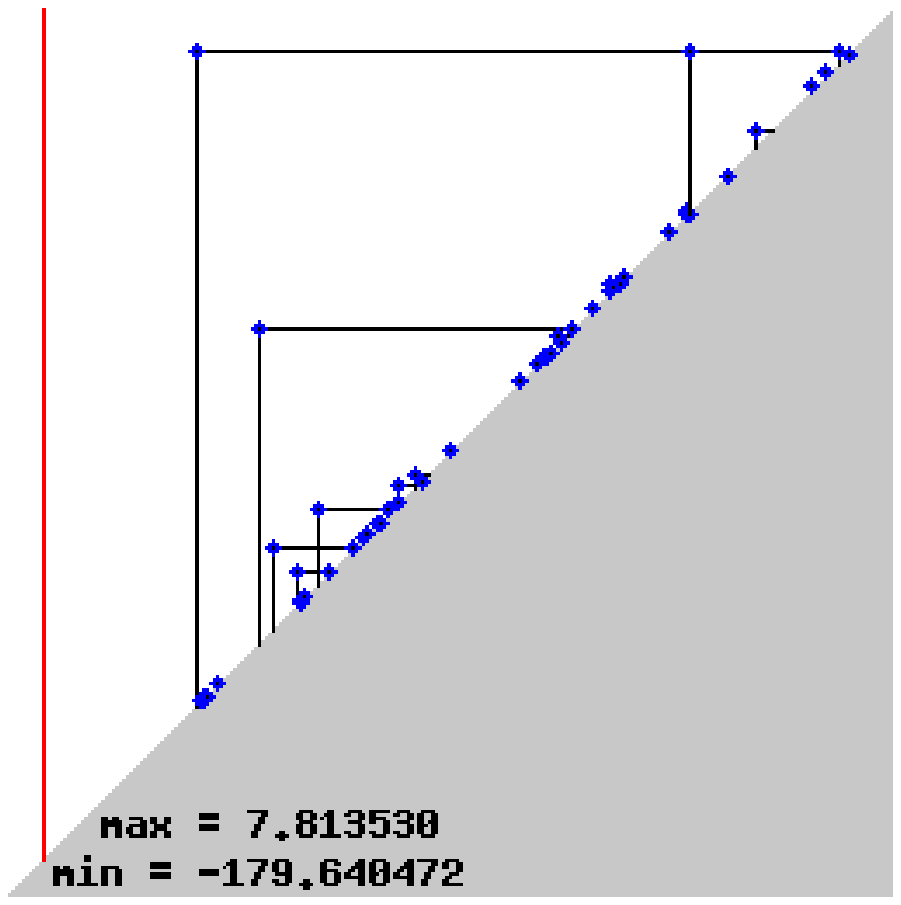} &
\includegraphics[width=0.15\linewidth]{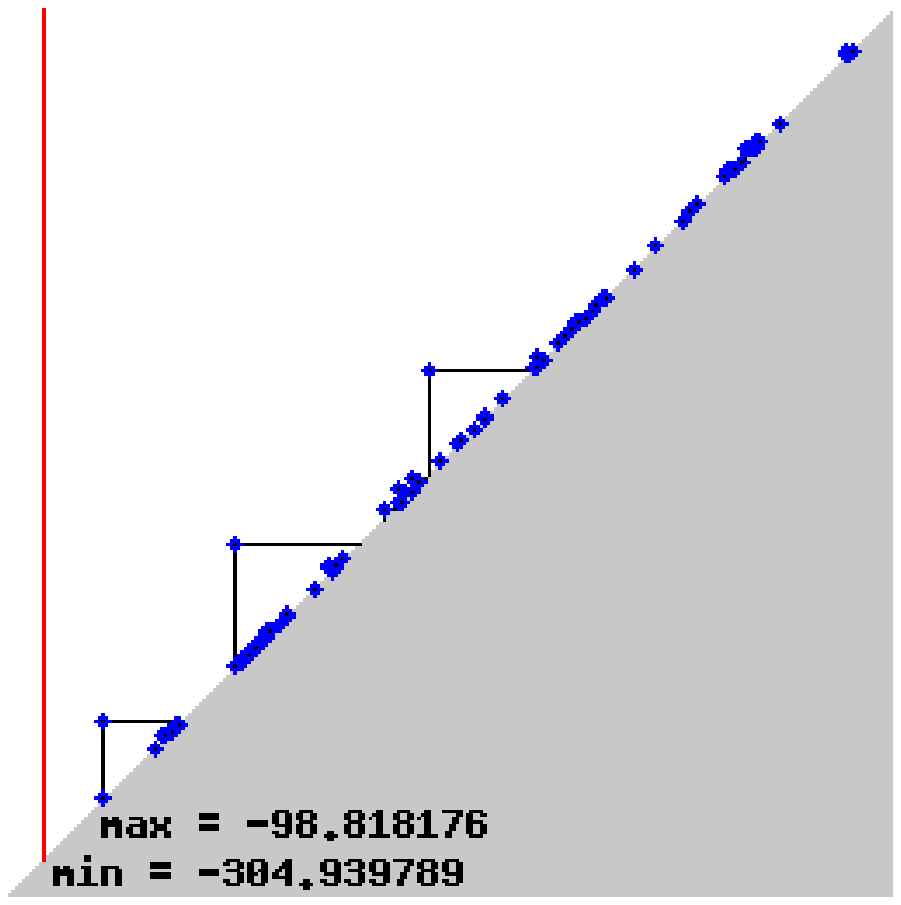} \\
\includegraphics[width=0.15\linewidth]{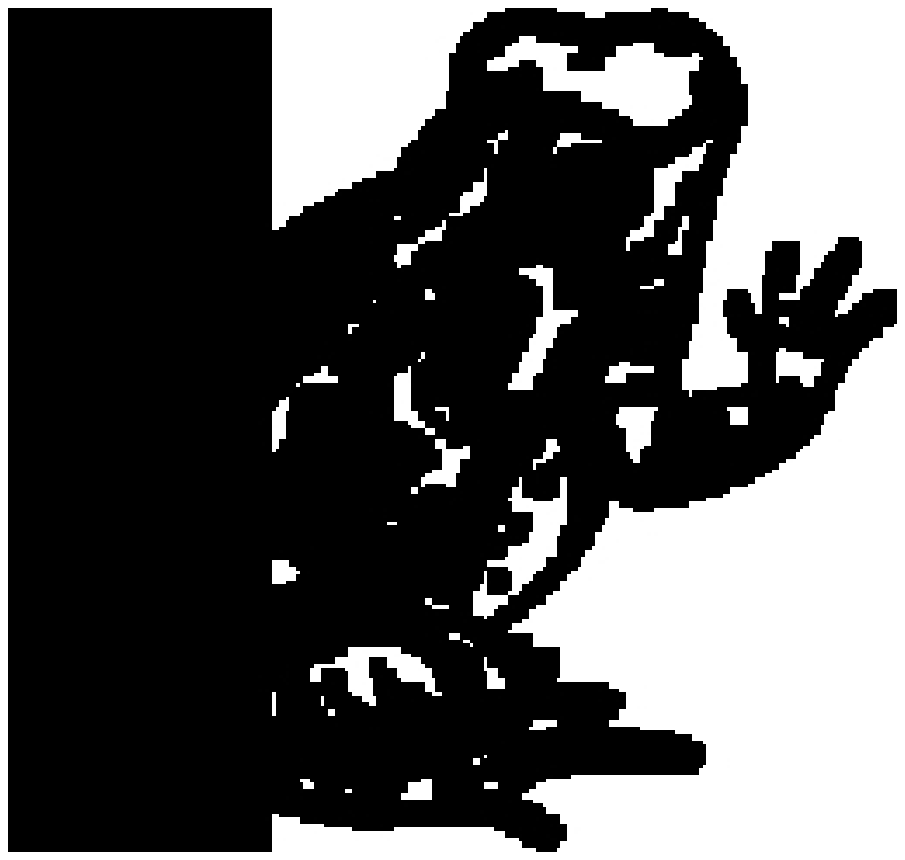} &
\includegraphics[width=0.15\linewidth]{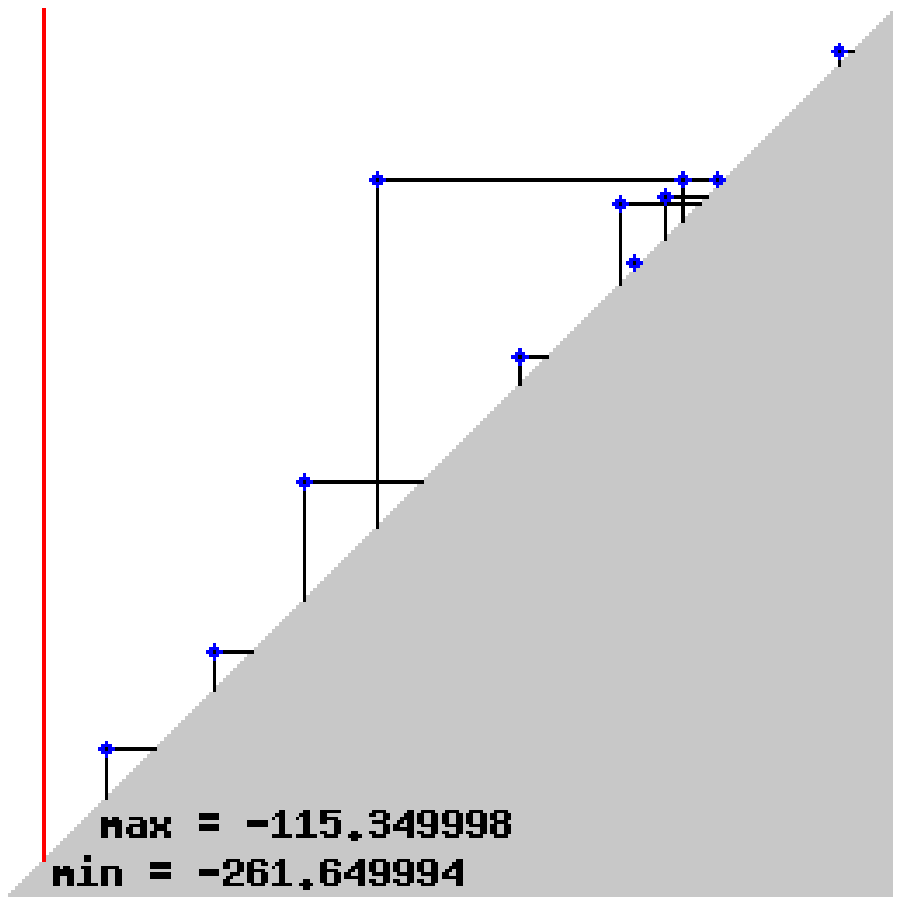} &
\includegraphics[width=0.15\linewidth]{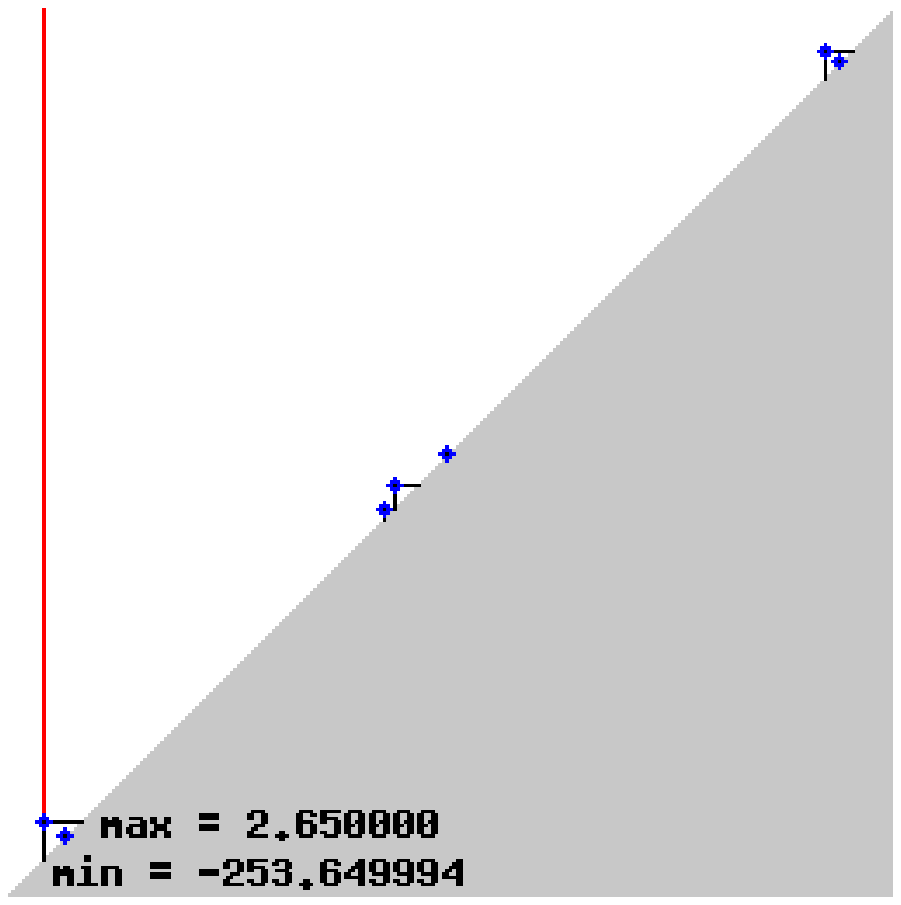} &
\includegraphics[width=0.15\linewidth]{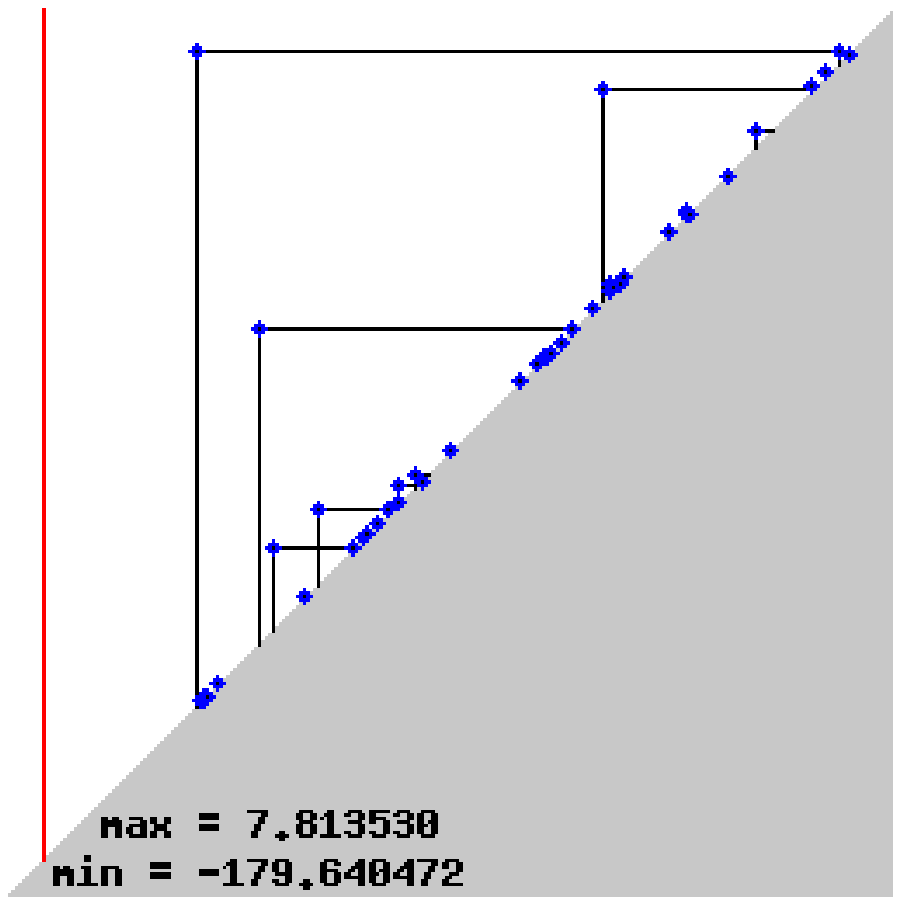} &
\includegraphics[width=0.15\linewidth]{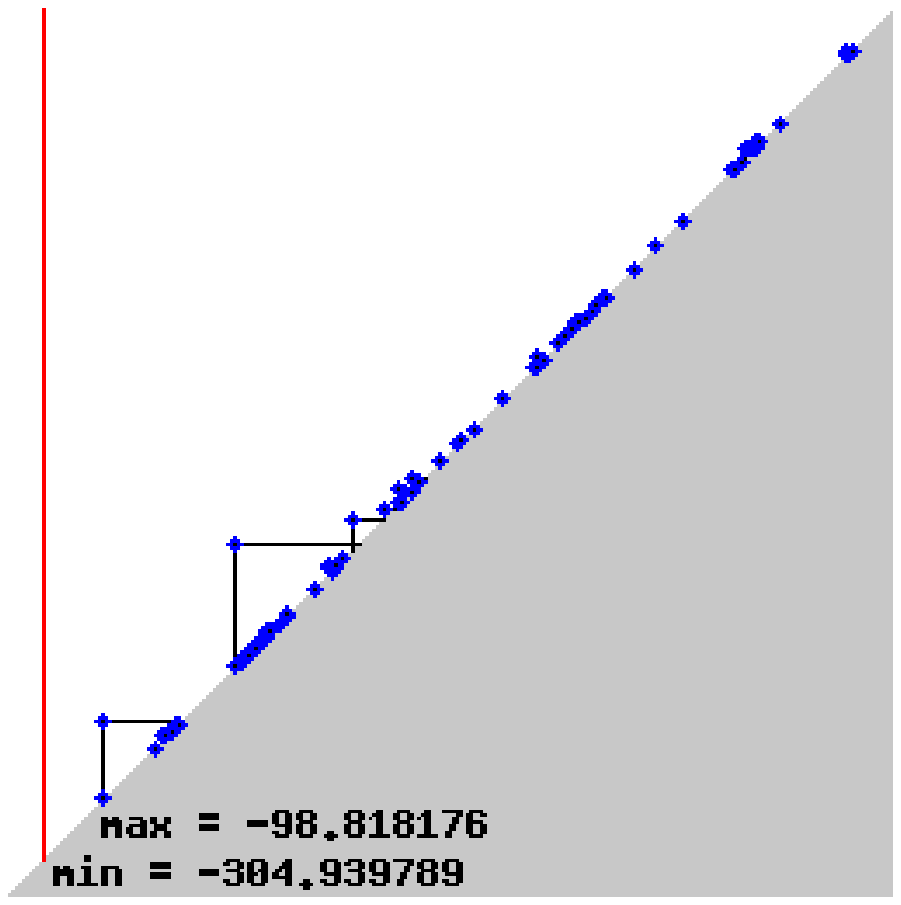} \\
\includegraphics[width=0.15\linewidth]{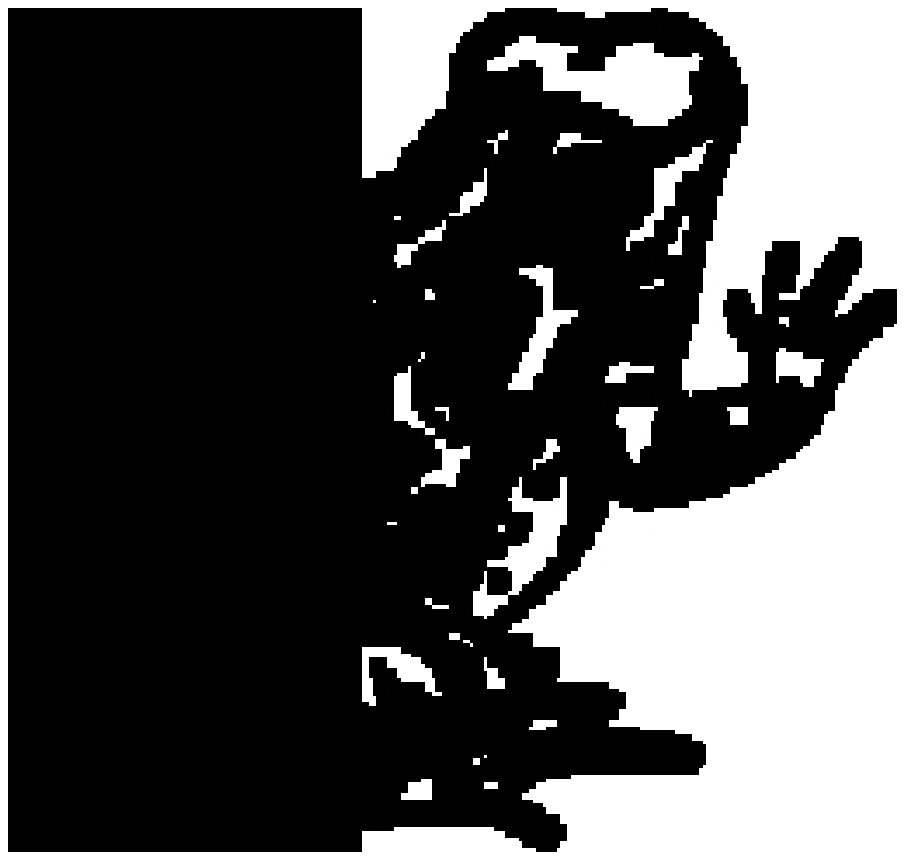} &
\includegraphics[width=0.15\linewidth]{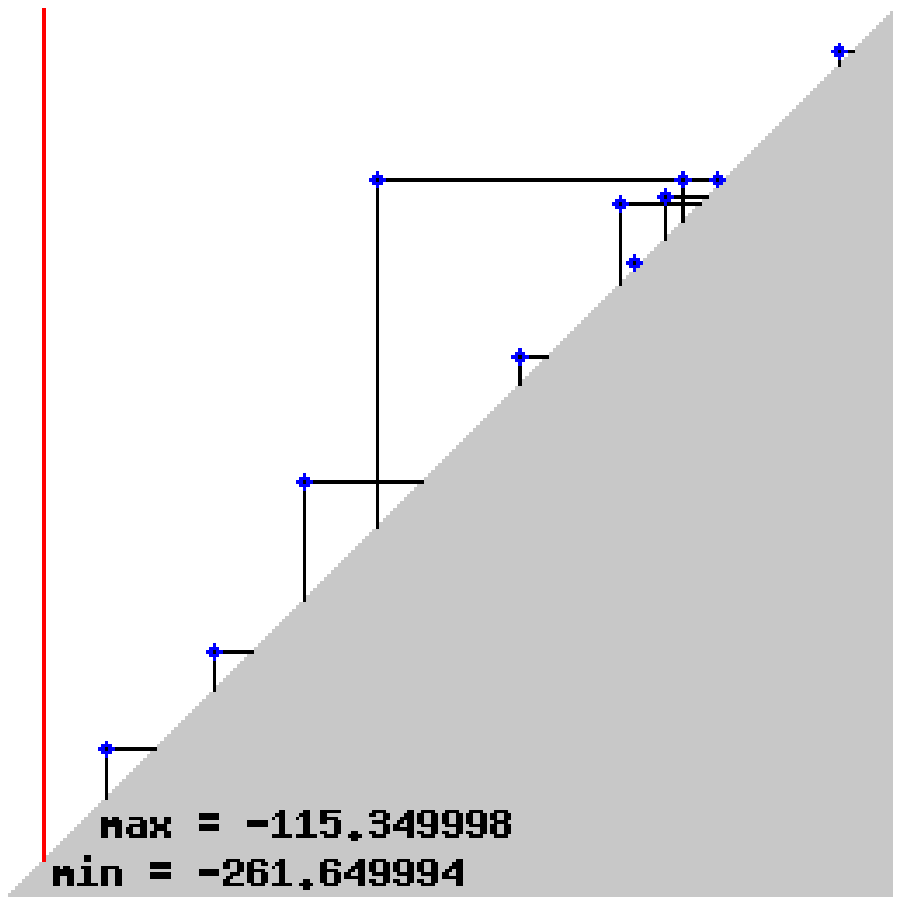} &
\includegraphics[width=0.15\linewidth]{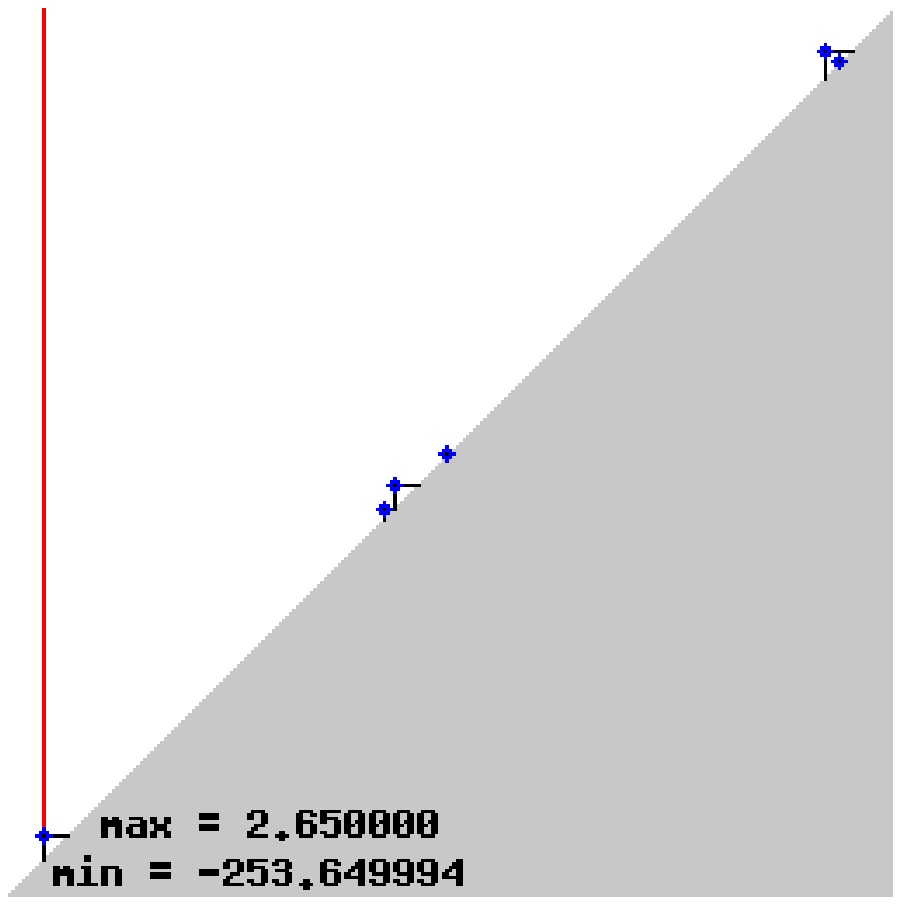} &
\includegraphics[width=0.15\linewidth]{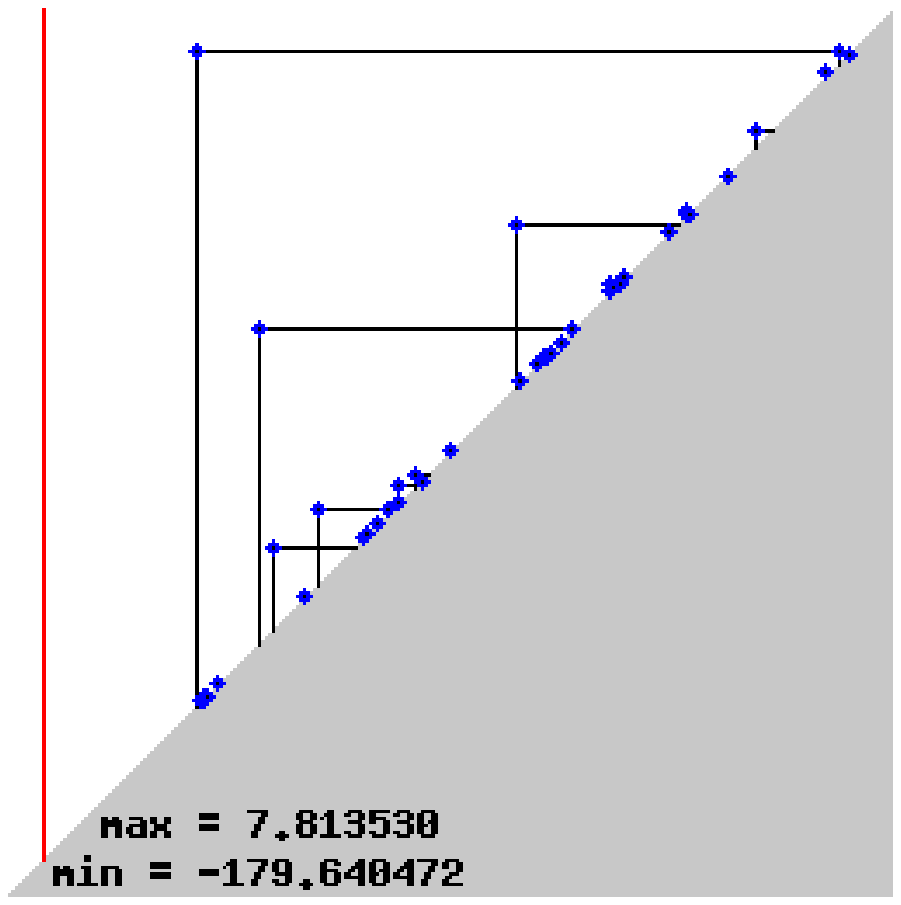} &
\includegraphics[width=0.15\linewidth]{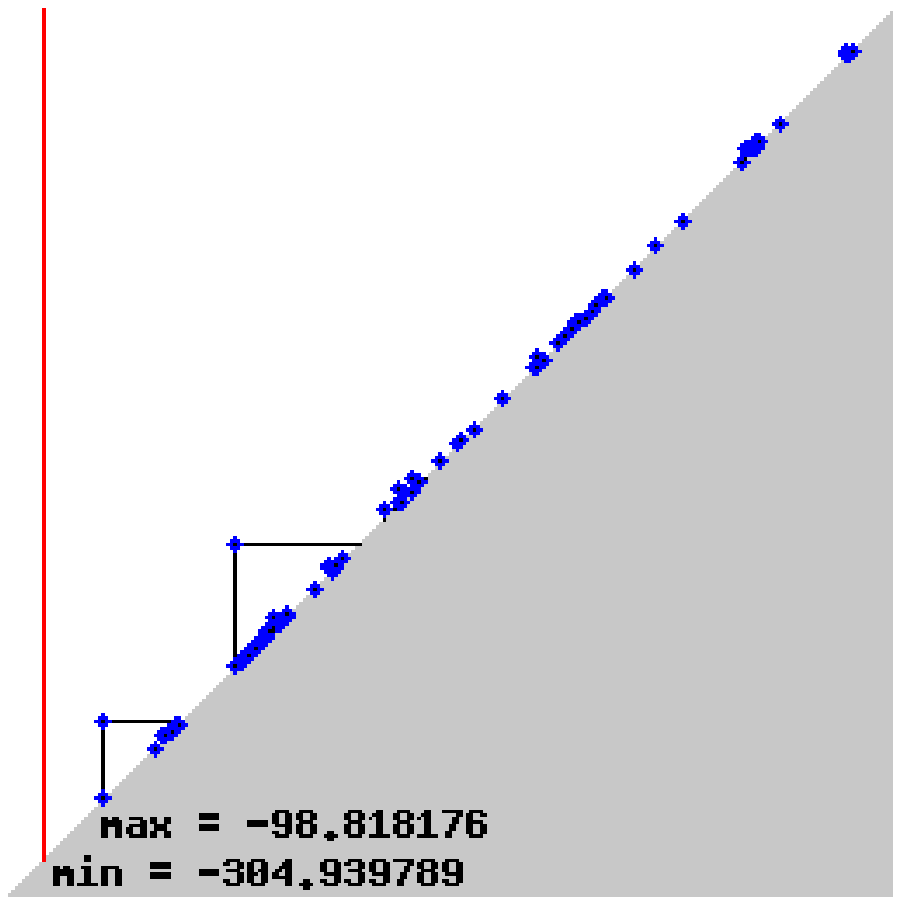} \\
\end{tabular}
\caption{\footnotesize{The first column: (row 1) original ``frog''
shape, (rows 2--4) occluded from top by $20\%$,  $30\%$, $40\%$,
(row 5--7) occluded from left by $20\%$,  $30\%$,  $40\%$. From
second column onwards: corresponding size functions related to
measuring functions defined as minus distances from four lines
rotated by $0$, $\pi/4$, $\pi/2$, $3\pi/4$, with respect to the
horizontal position.}} \label{frogHV}
\end{table}
Eventually, the ``pocket watch'', represented in Table
\ref{pocketHV}, is primarily characterized by several connected
components, whose number decreases as the occluding area
increases. This results in a reduction of the number of
cornerpoints at infinity in its size functions.
\begin{table}[htbp]
\begin{tabular}{cccccc}
\includegraphics[width=0.115\linewidth]{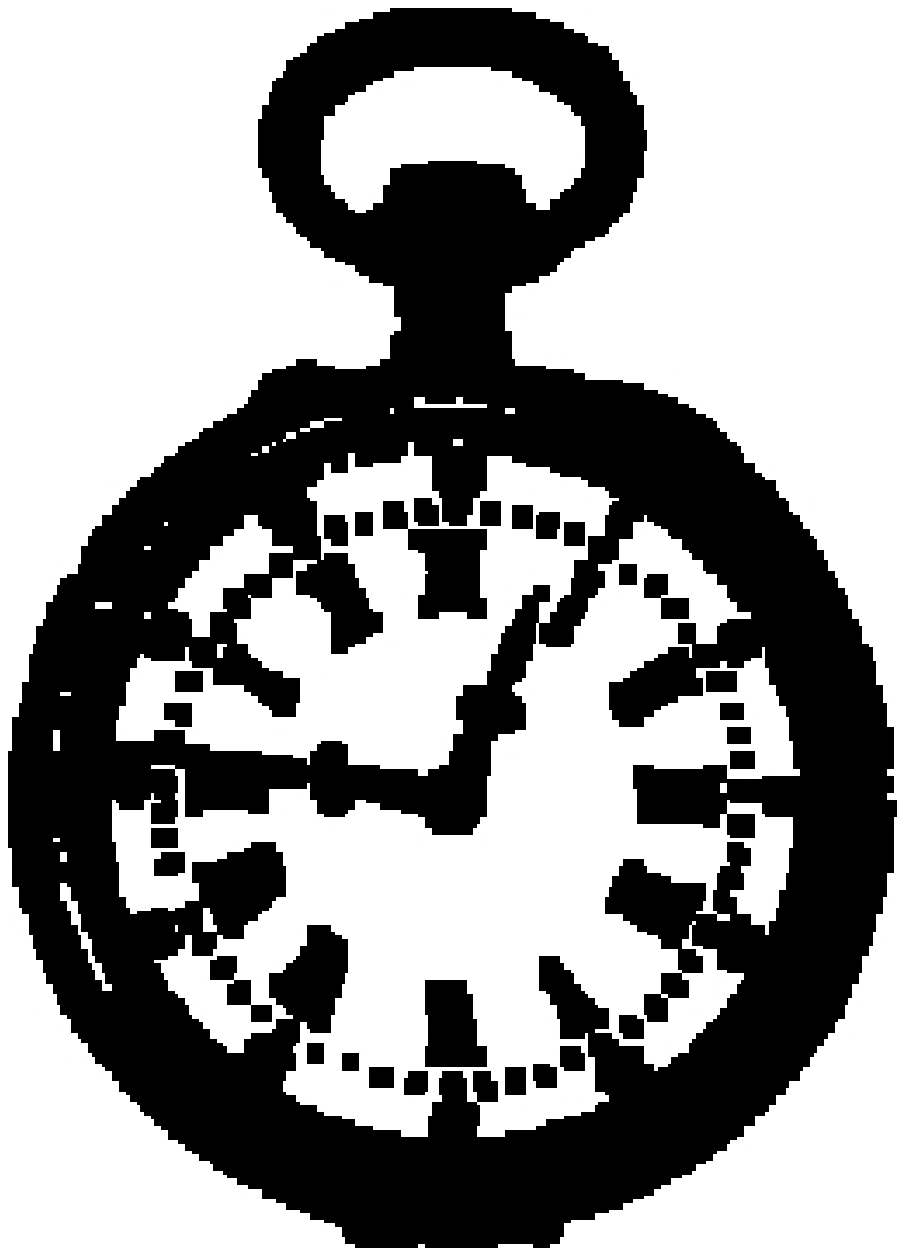} &
\includegraphics[width=0.15\linewidth]{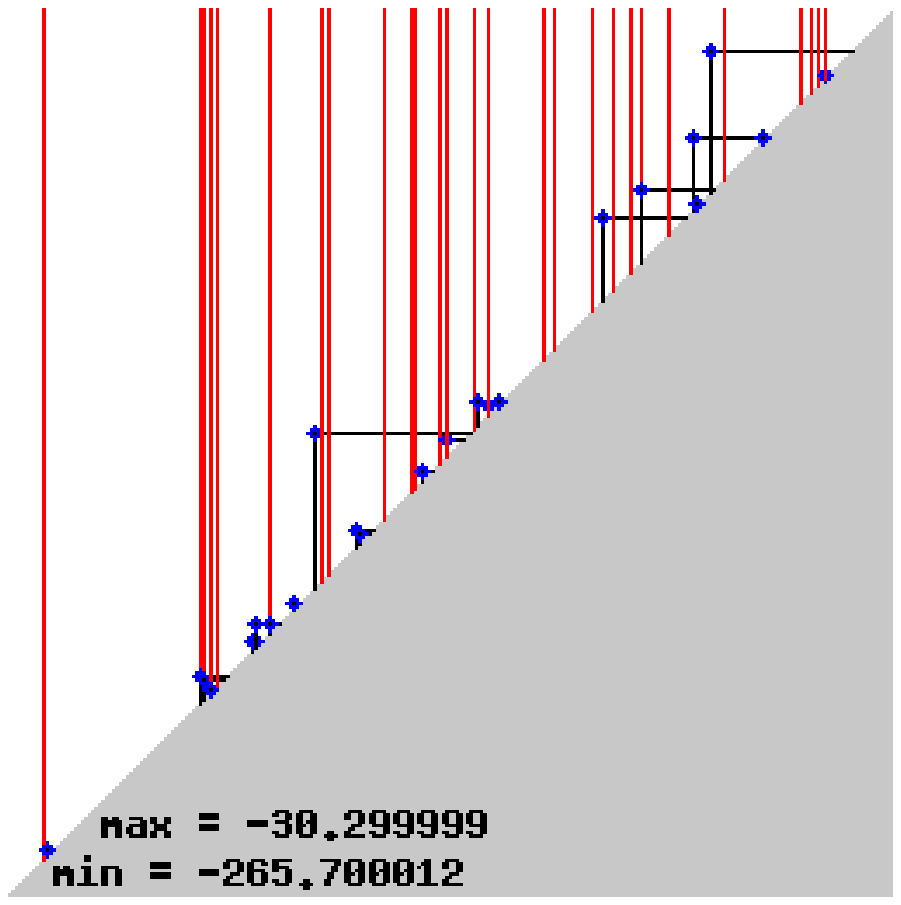} &
\includegraphics[width=0.15\linewidth]{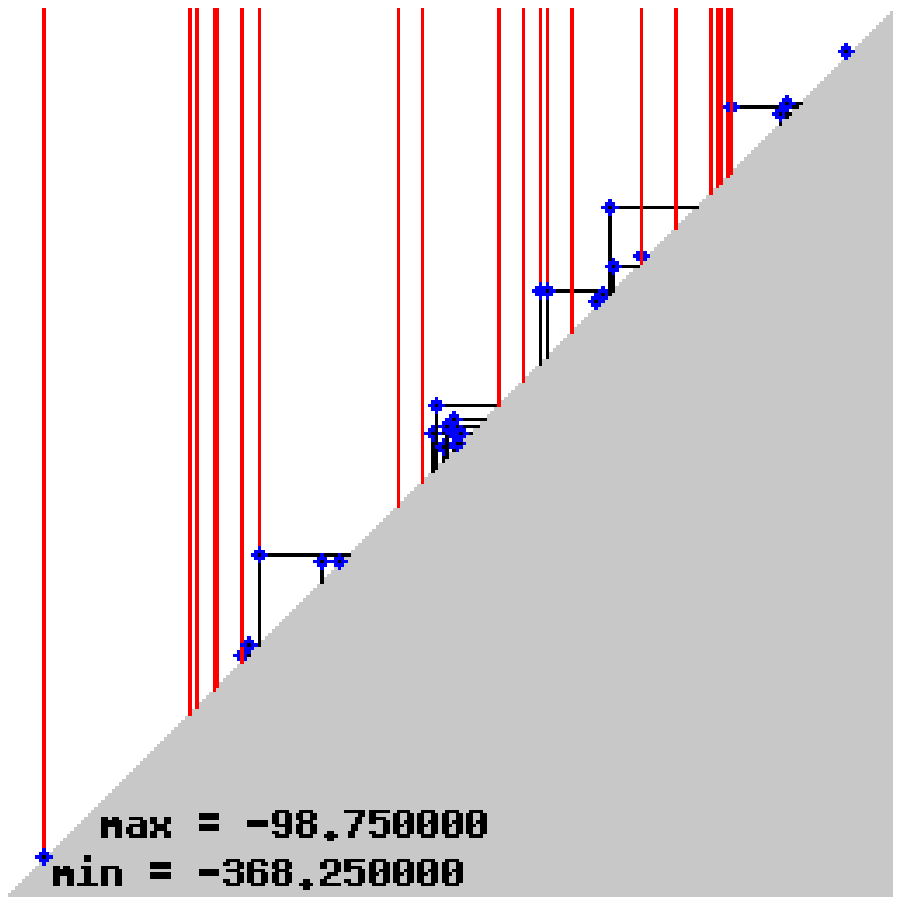} &
\includegraphics[width=0.15\linewidth]{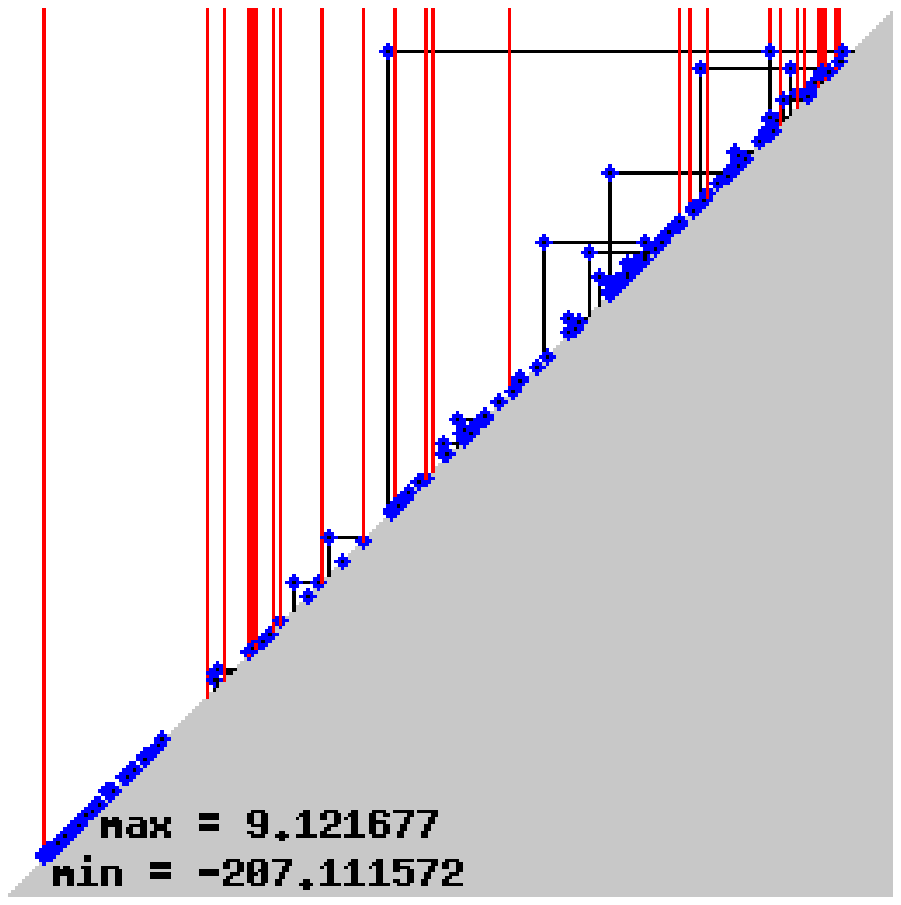} &
\includegraphics[width=0.15\linewidth]{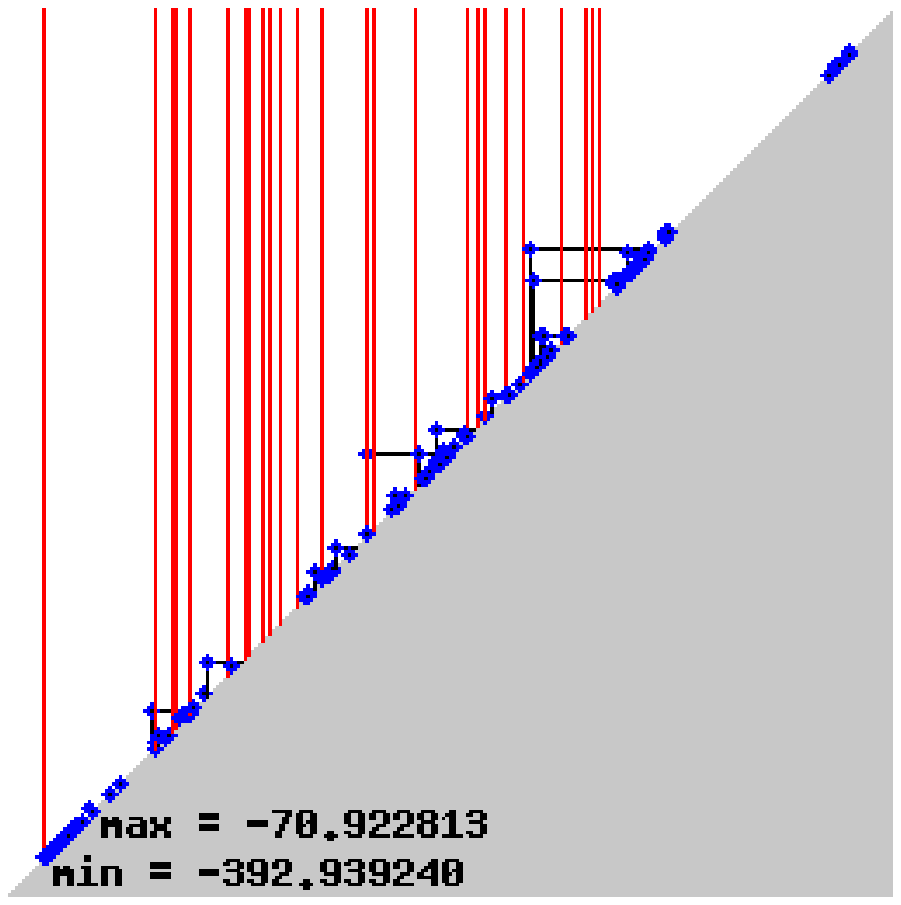} \\
\includegraphics[width=0.115\linewidth]{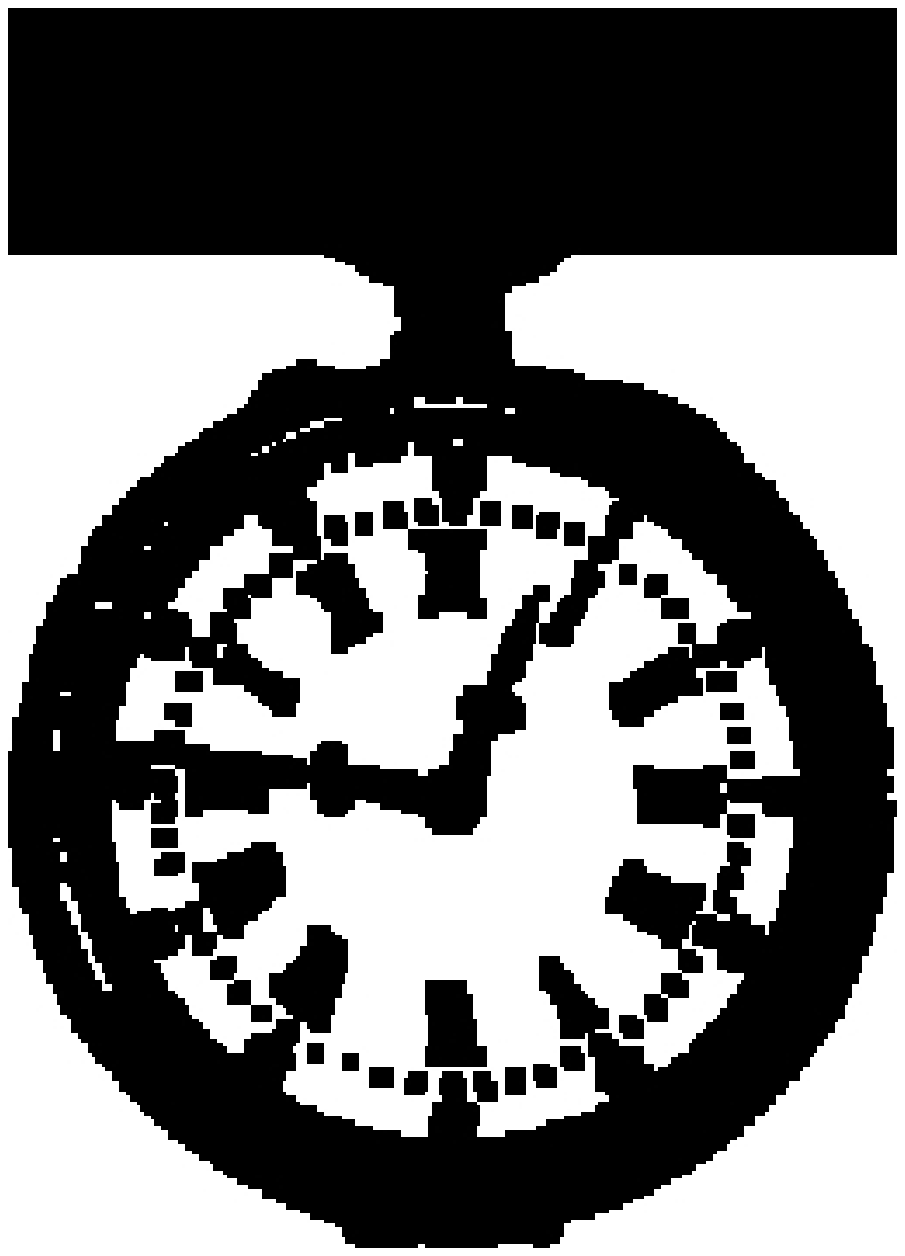} &
\includegraphics[width=0.15\linewidth]{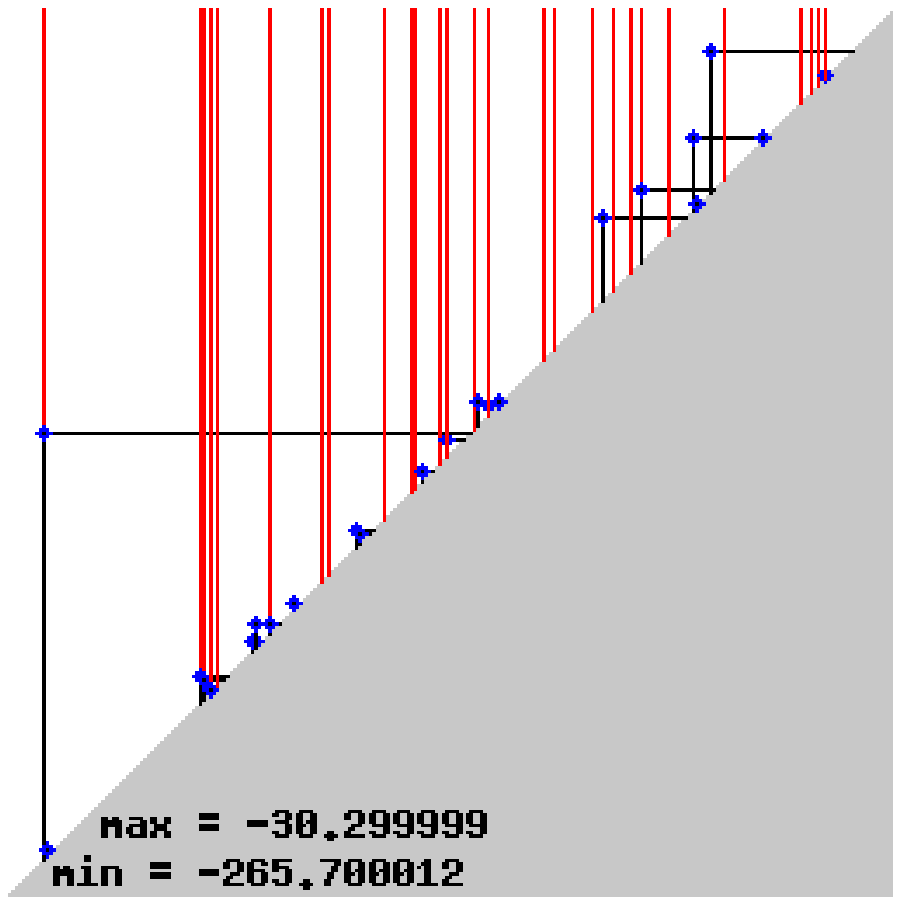} &
\includegraphics[width=0.15\linewidth]{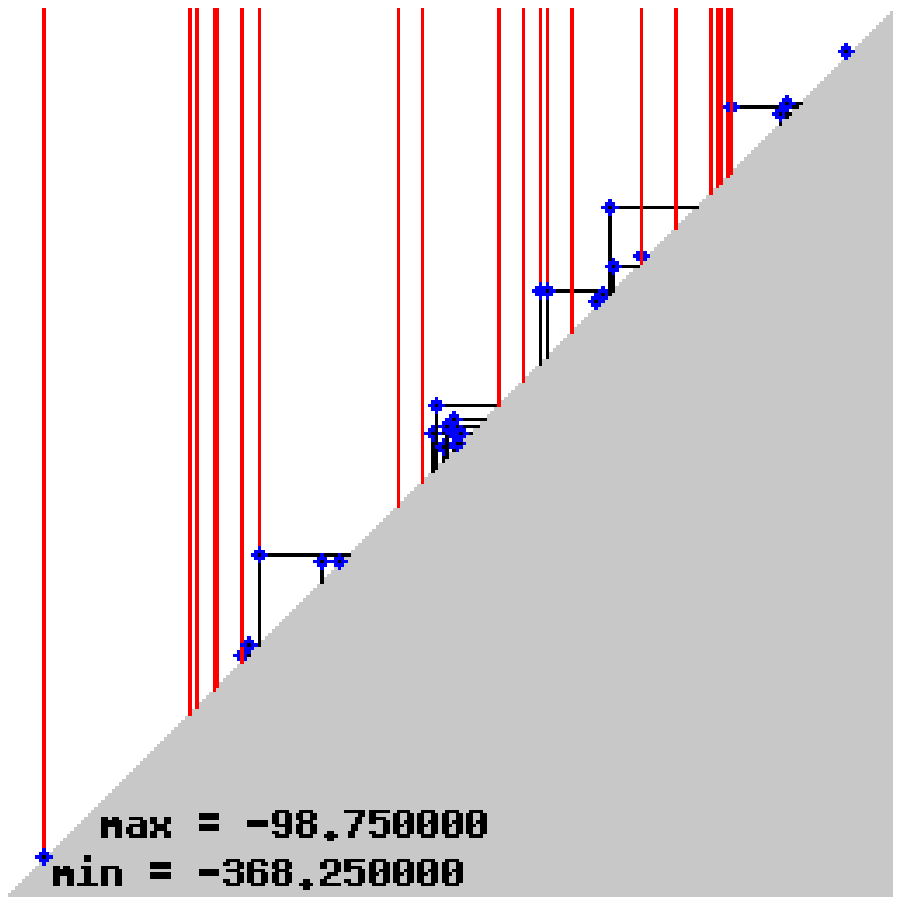} &
\includegraphics[width=0.15\linewidth]{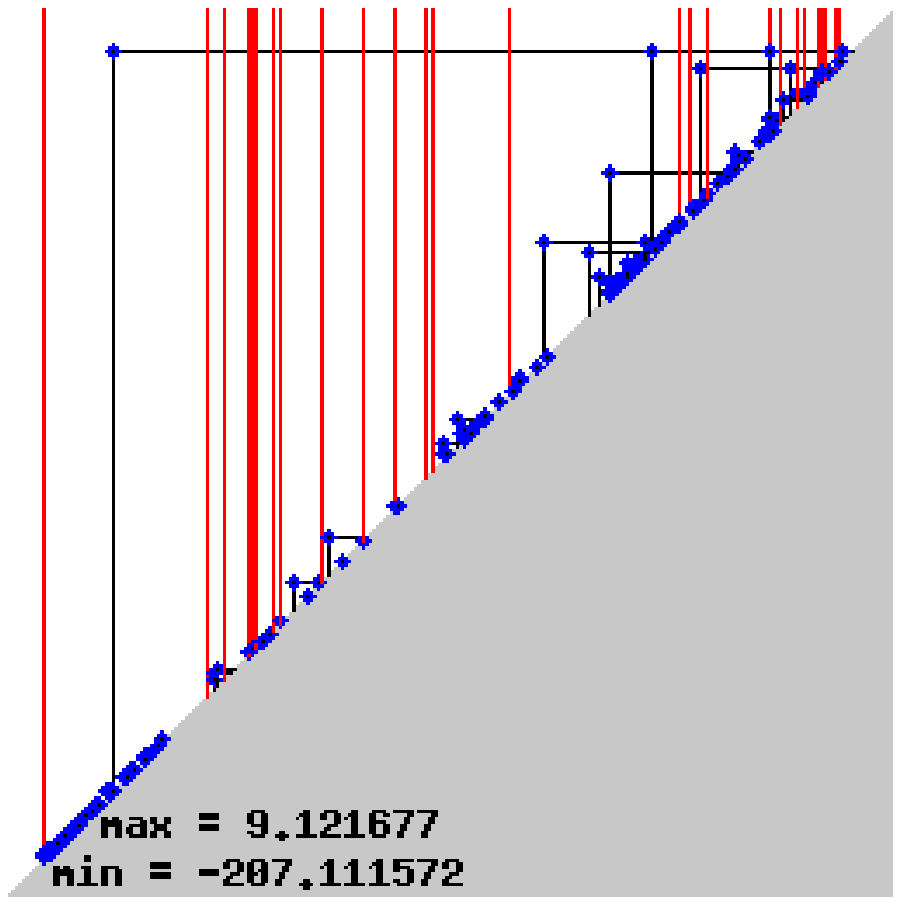} &
\includegraphics[width=0.15\linewidth]{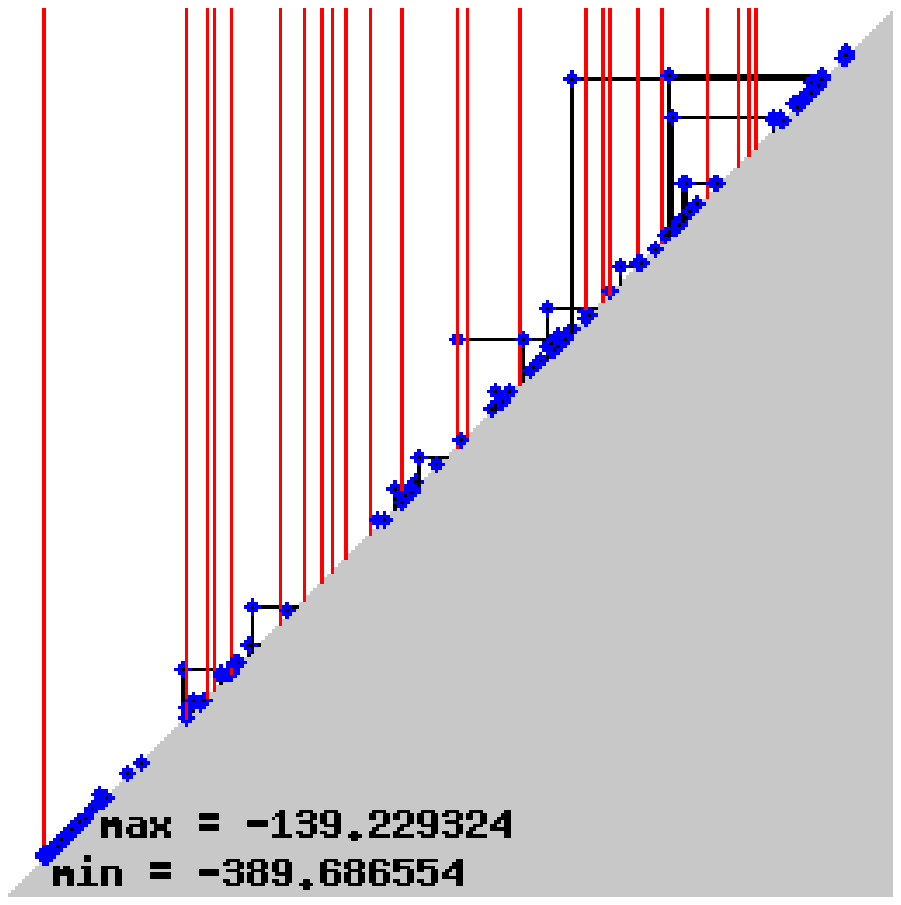} \\
\includegraphics[width=0.115\linewidth]{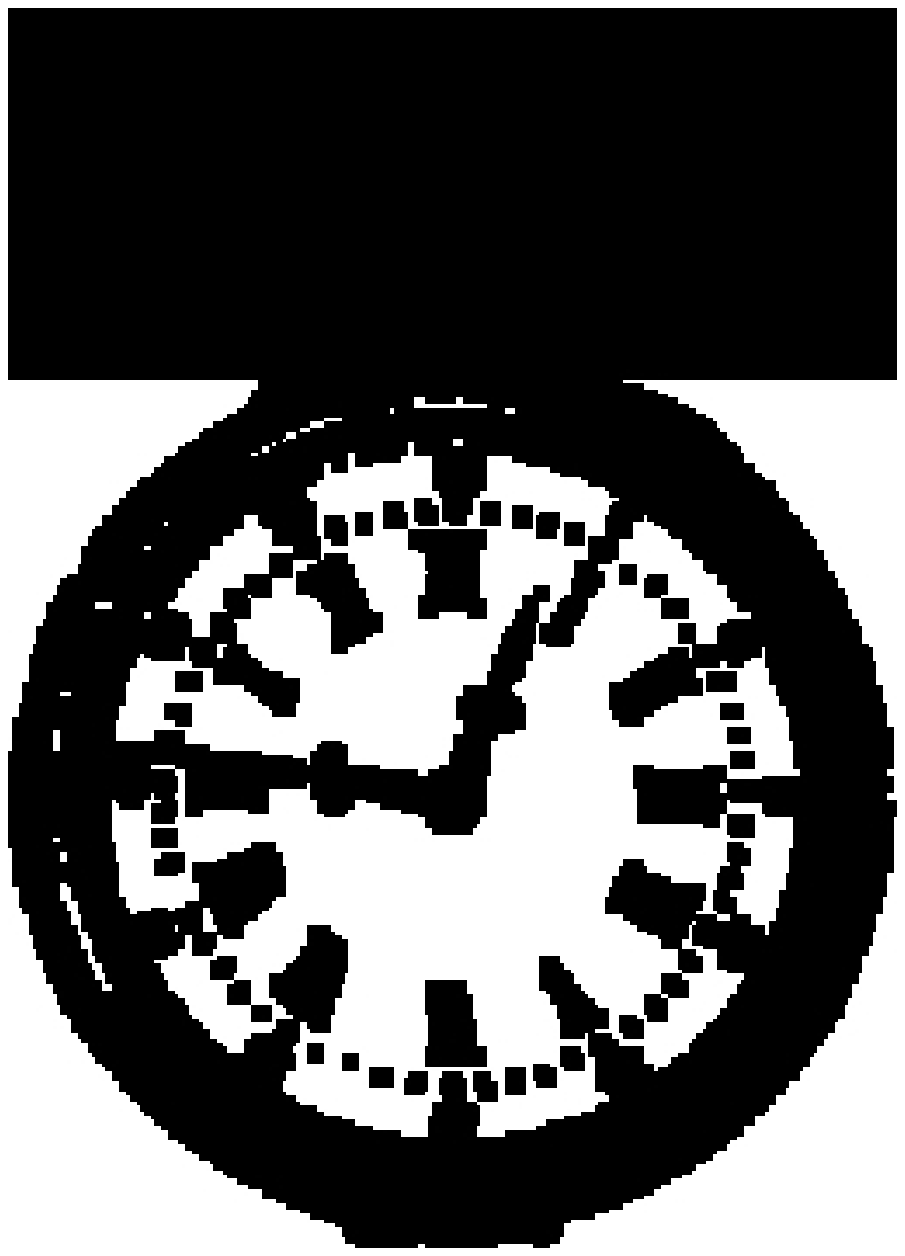} &
\includegraphics[width=0.15\linewidth]{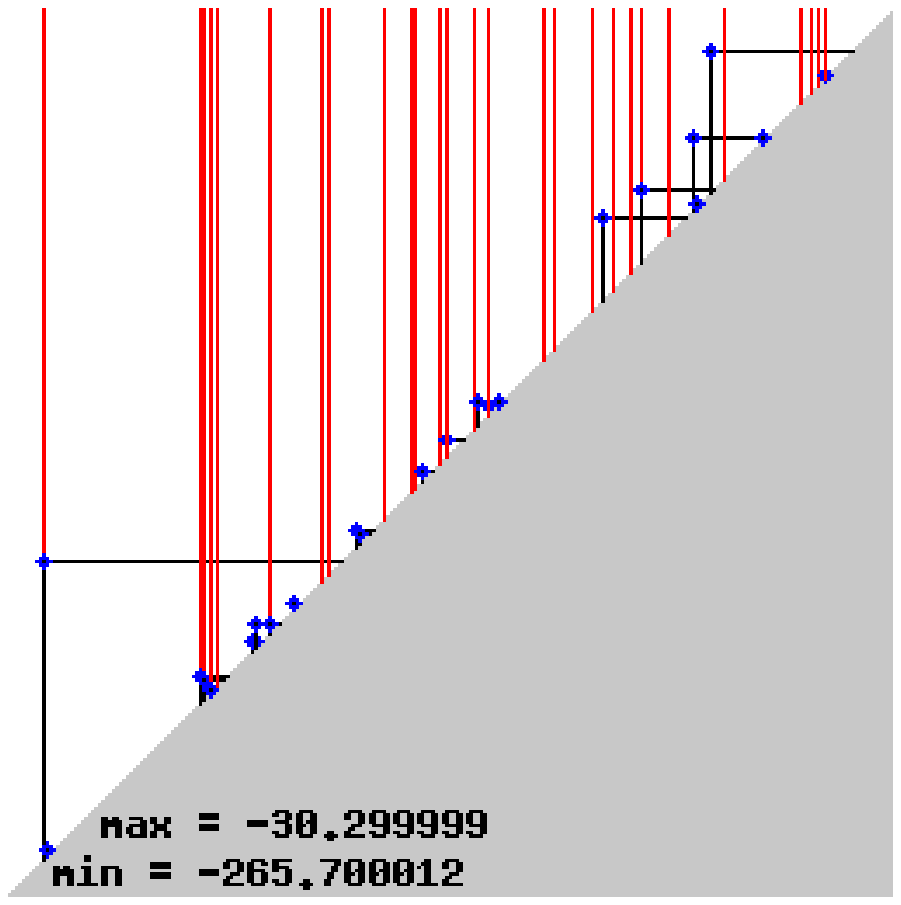} &
\includegraphics[width=0.15\linewidth]{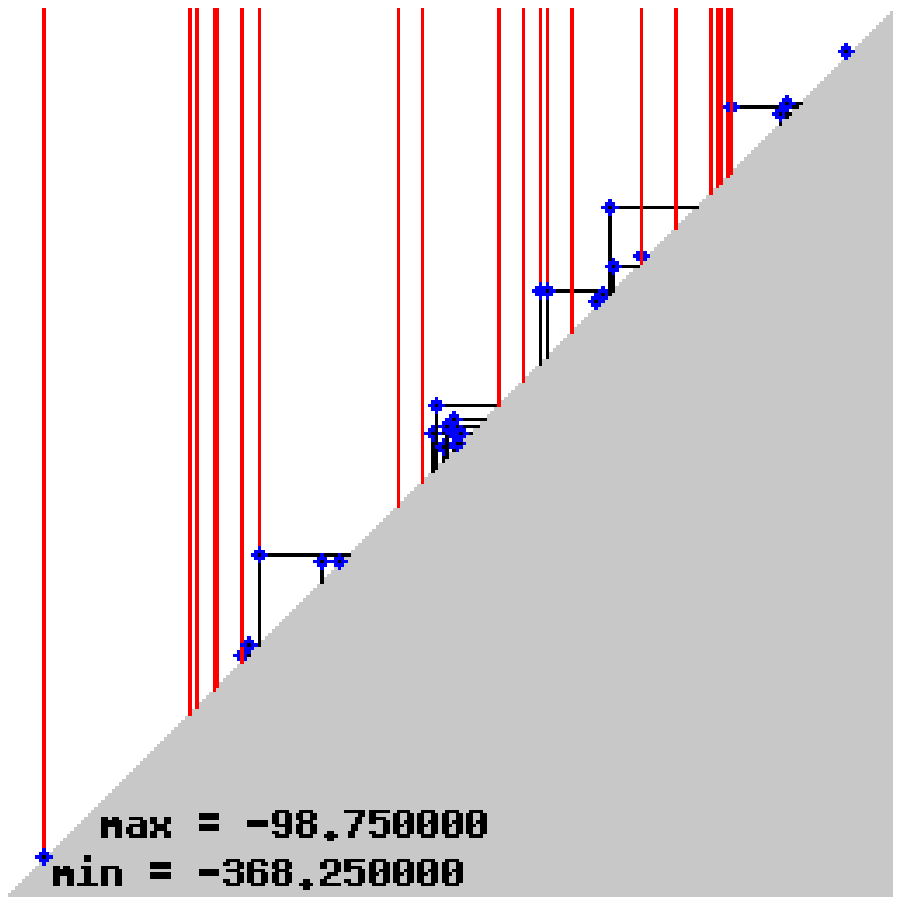} &
\includegraphics[width=0.15\linewidth]{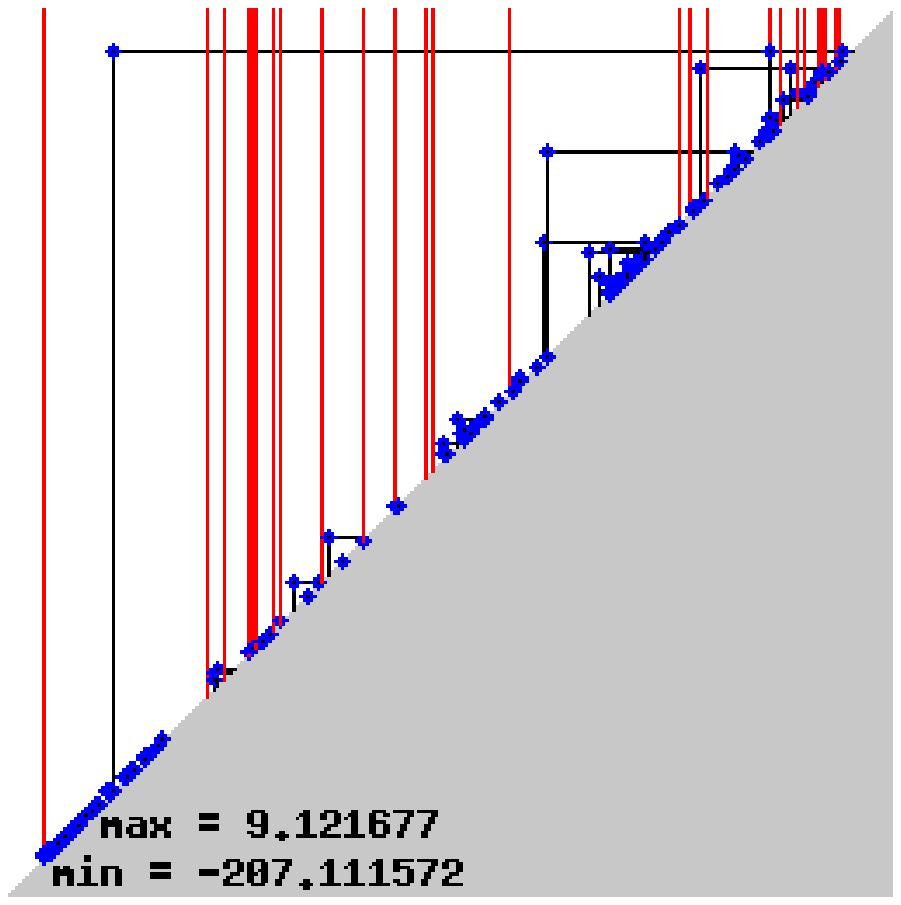} &
\includegraphics[width=0.15\linewidth]{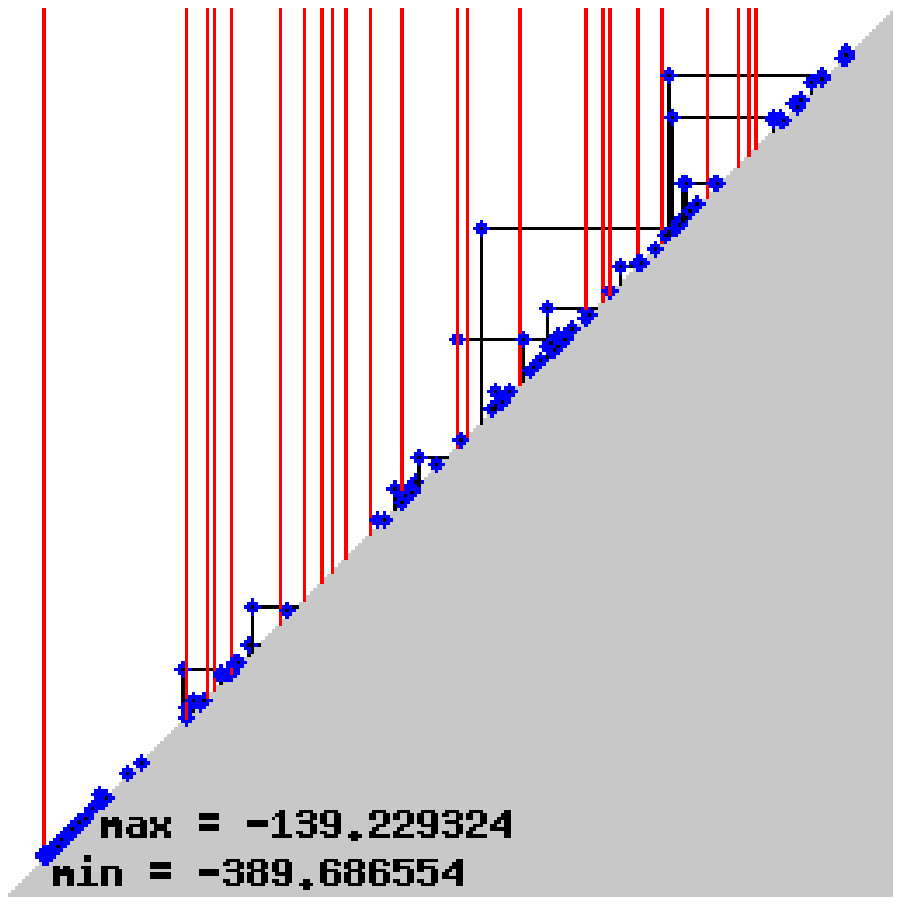} \\
\includegraphics[width=0.115\linewidth]{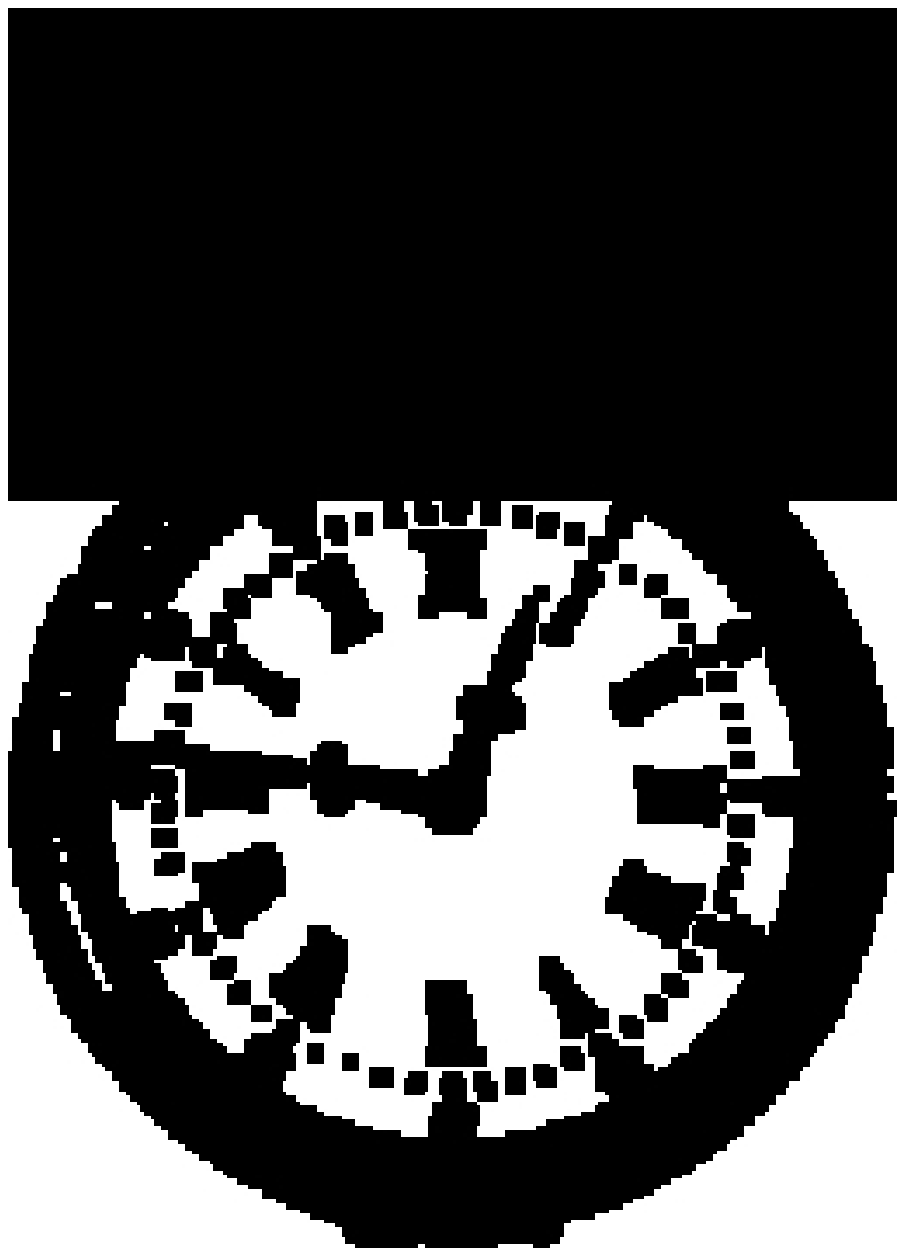} &
\includegraphics[width=0.15\linewidth]{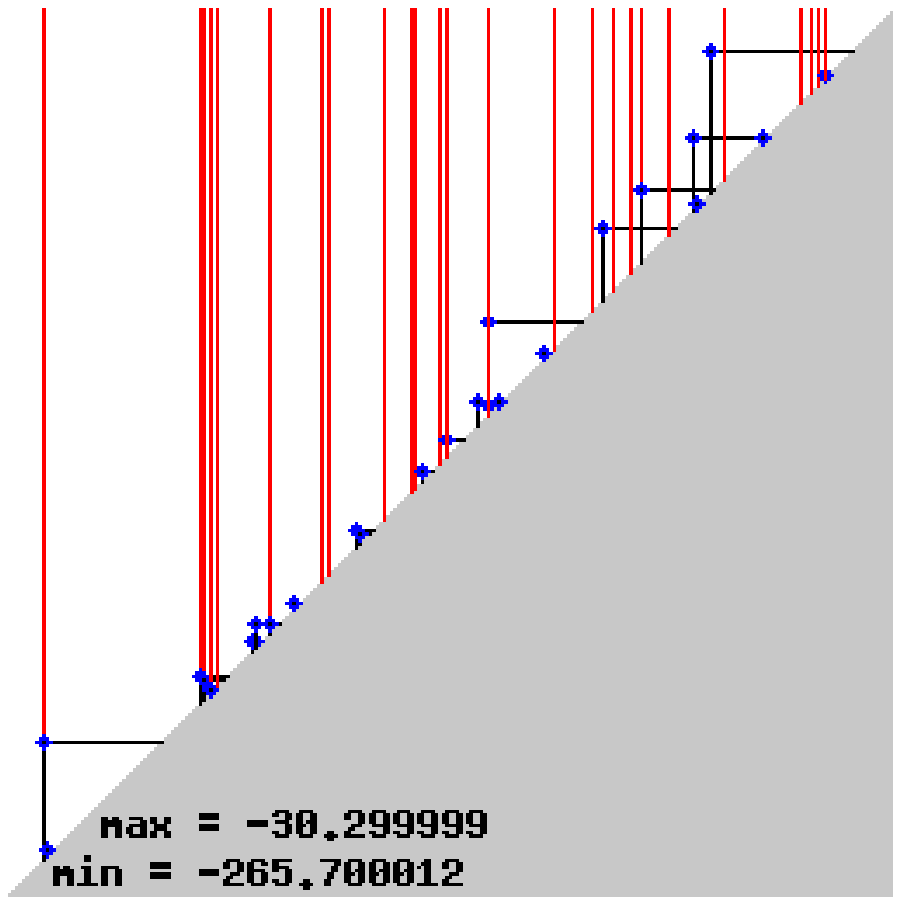} &
\includegraphics[width=0.15\linewidth]{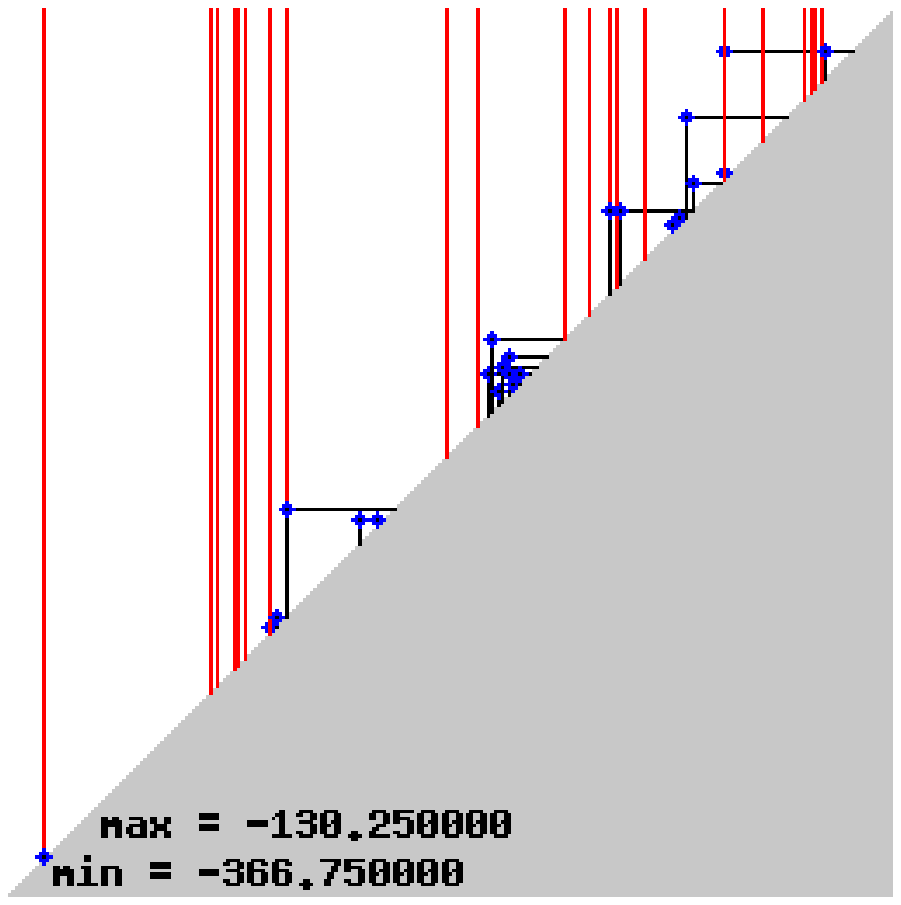} &
\includegraphics[width=0.15\linewidth]{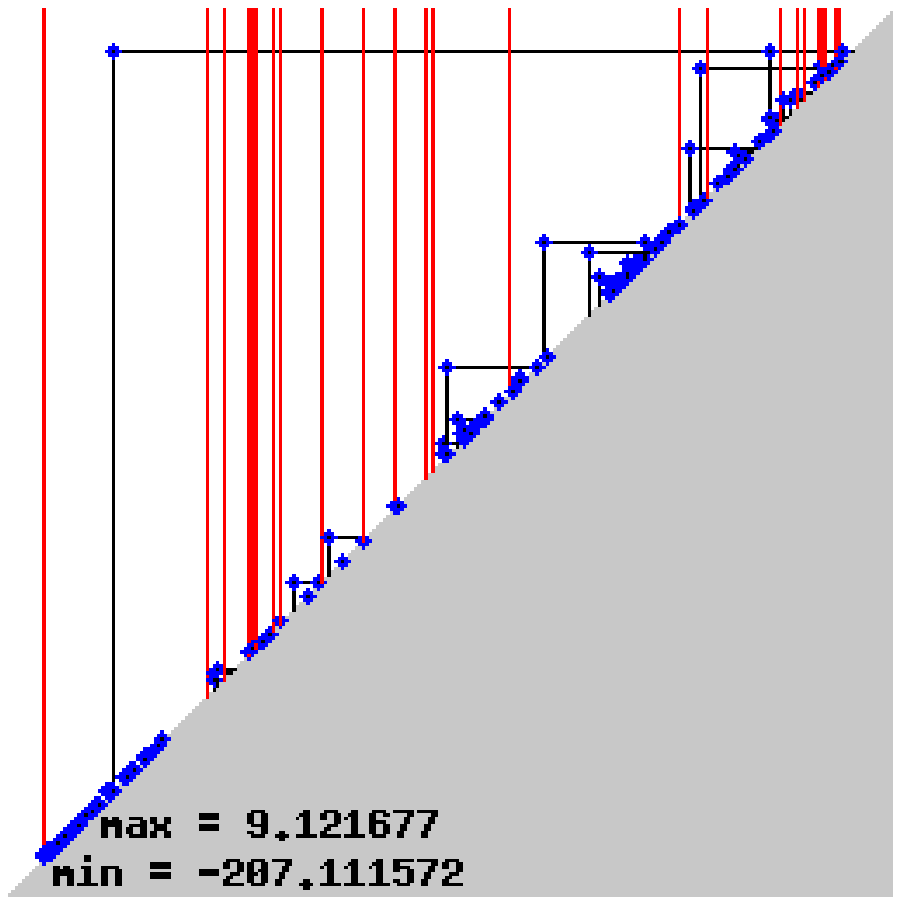} &
\includegraphics[width=0.15\linewidth]{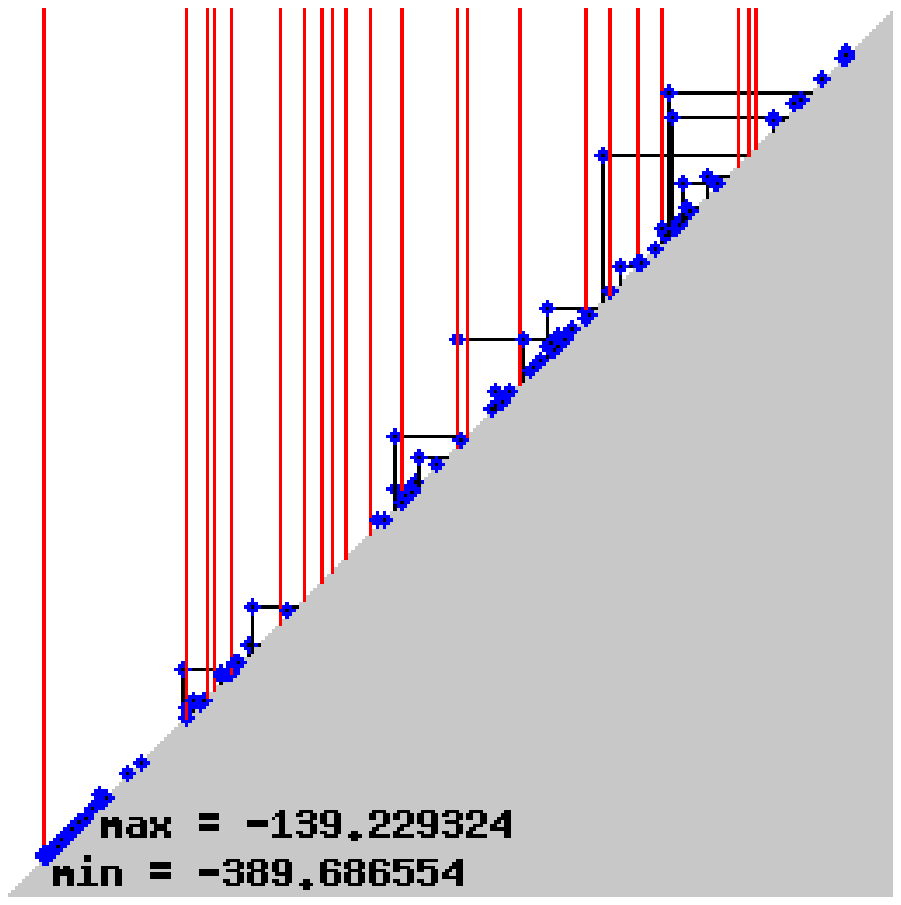} \\
\includegraphics[width=0.115\linewidth]{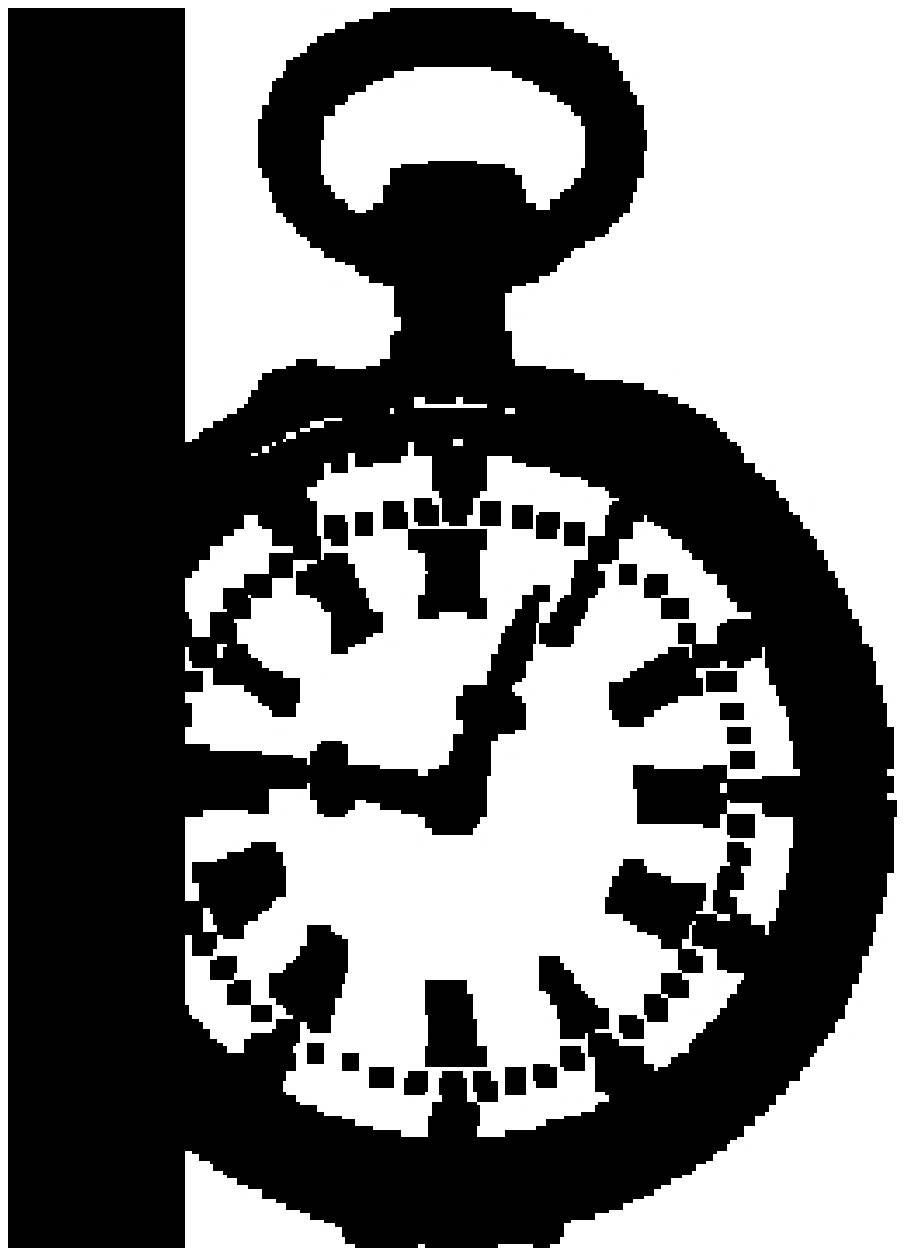} &
\includegraphics[width=0.15\linewidth]{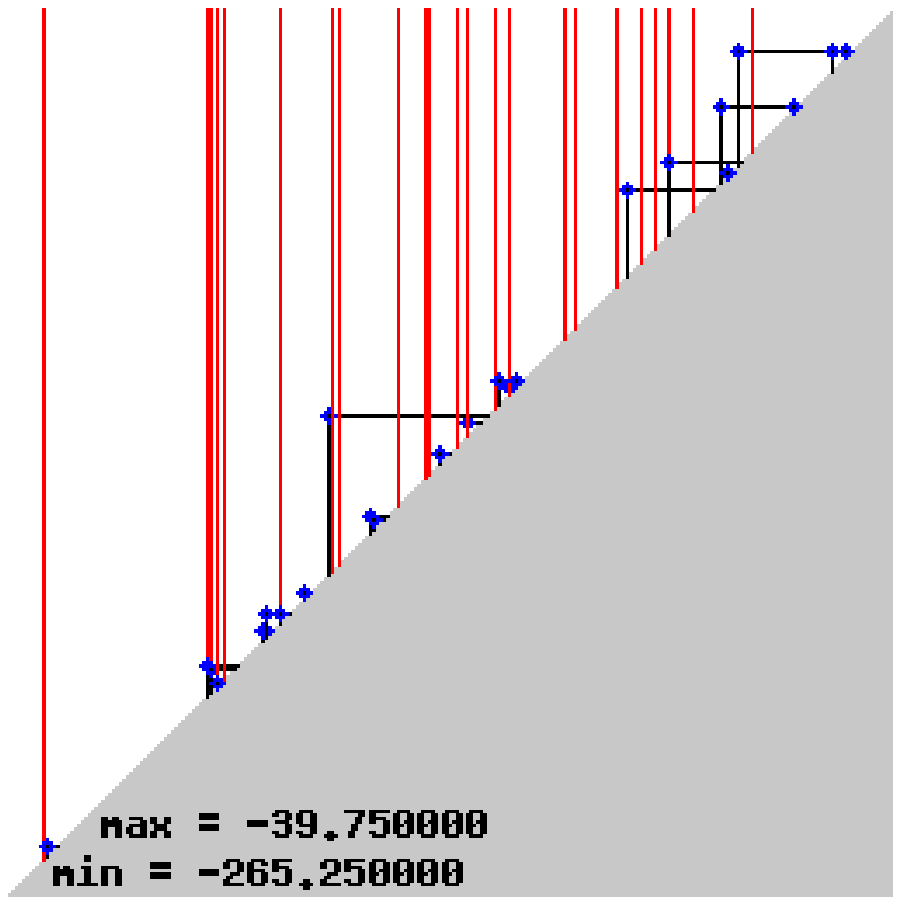} &
\includegraphics[width=0.15\linewidth]{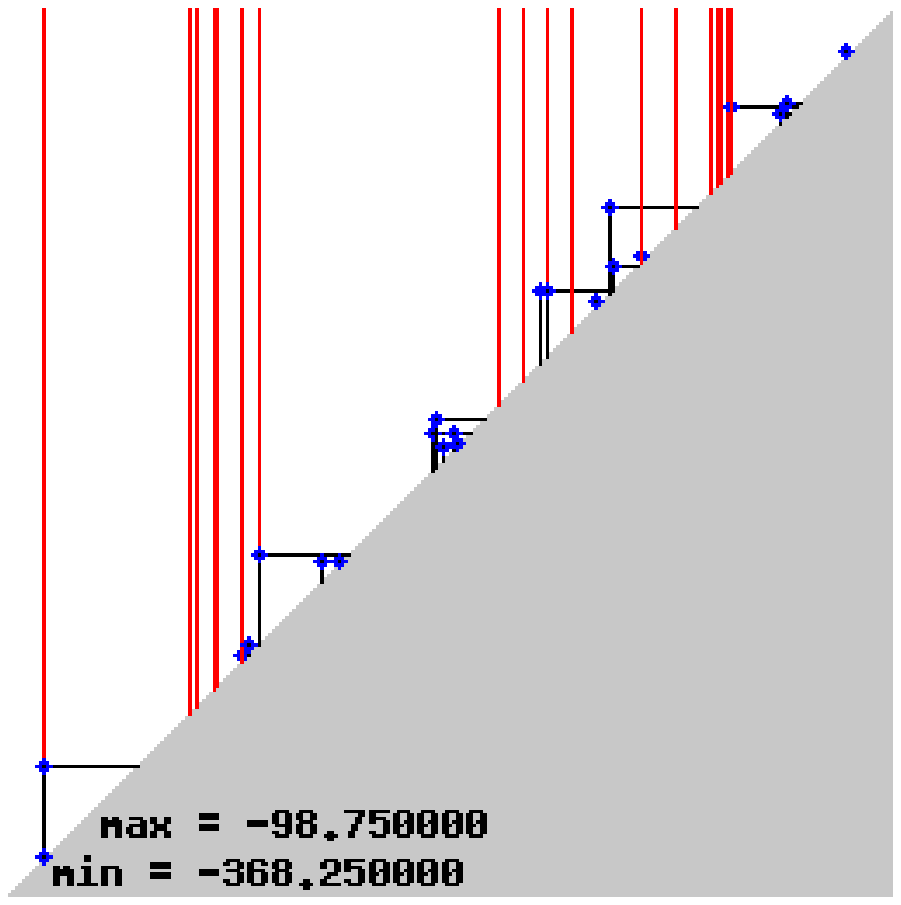} &
\includegraphics[width=0.15\linewidth]{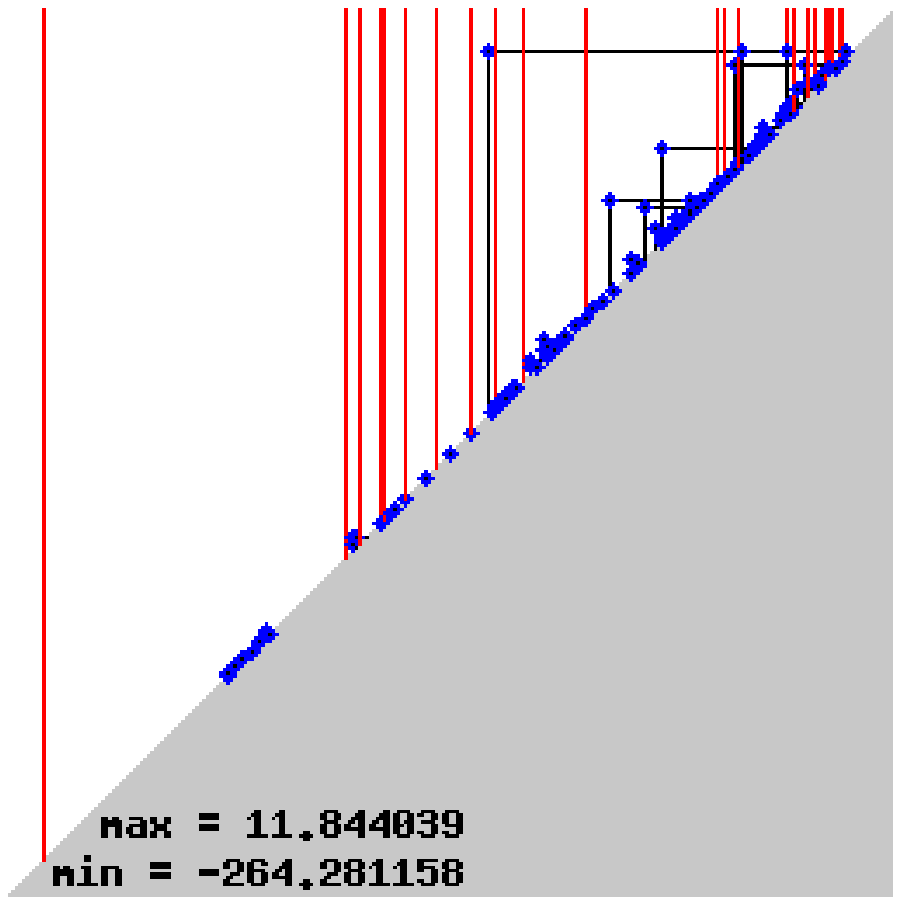} &
\includegraphics[width=0.15\linewidth]{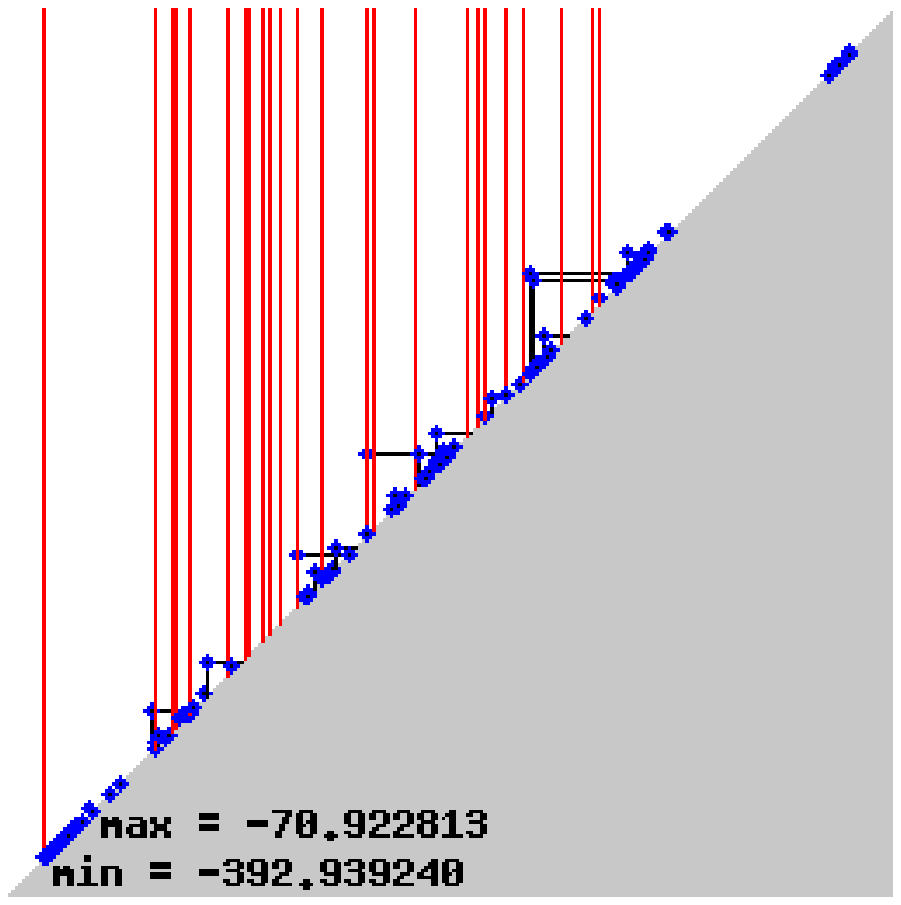} \\
\includegraphics[width=0.115\linewidth]{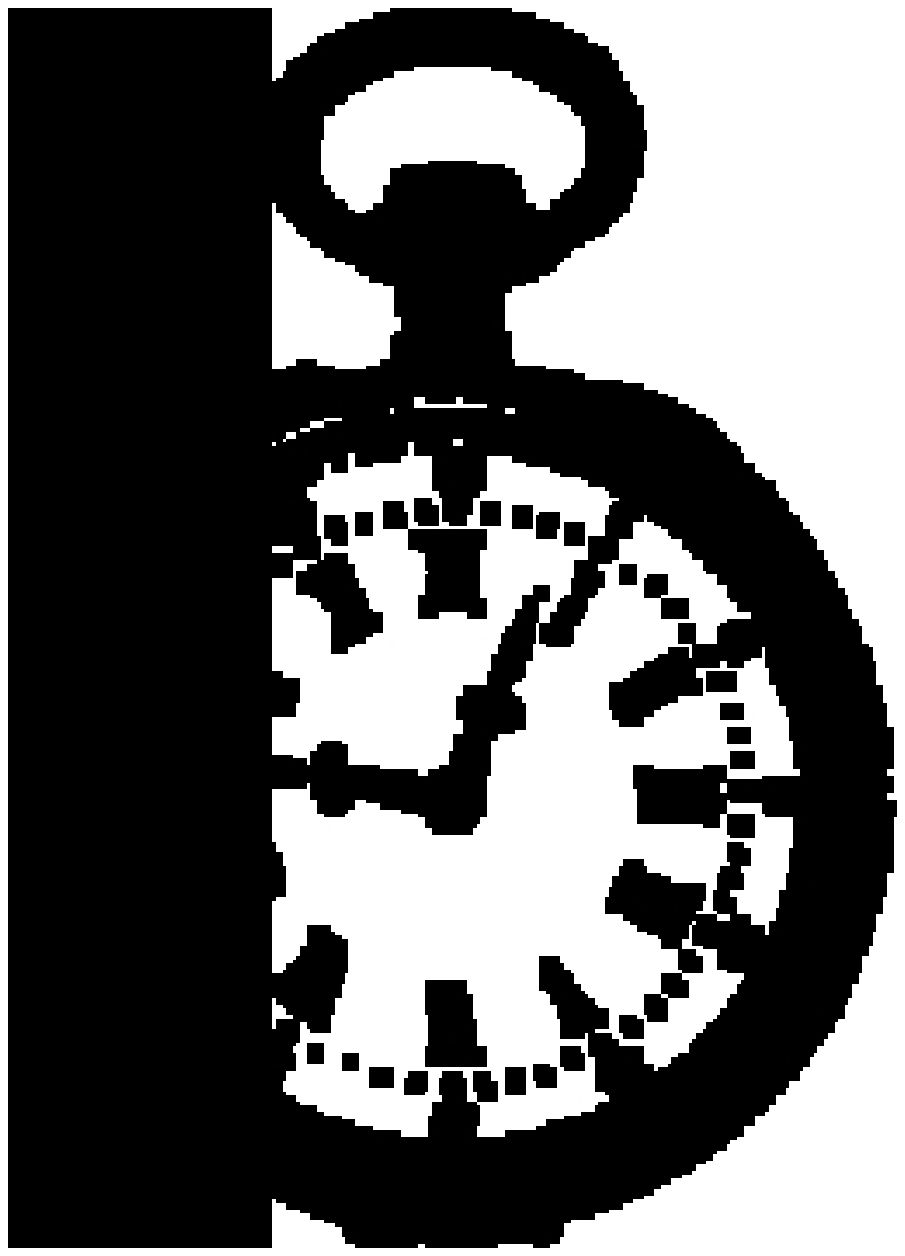} &
\includegraphics[width=0.15\linewidth]{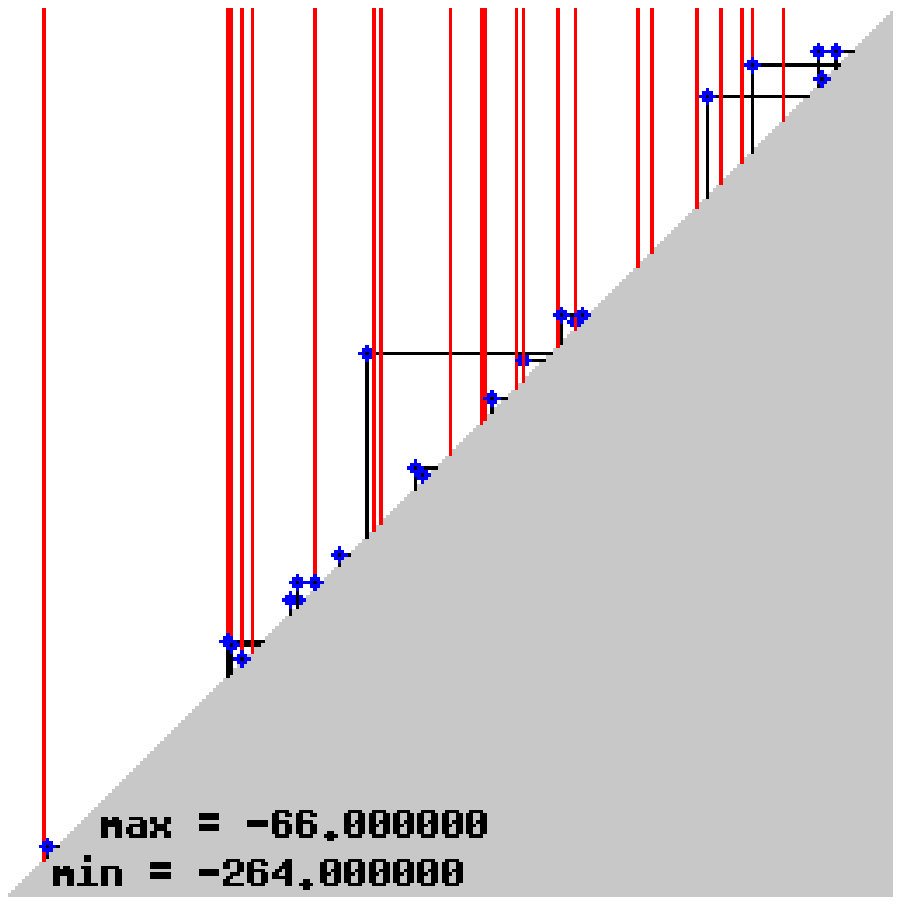} &
\includegraphics[width=0.15\linewidth]{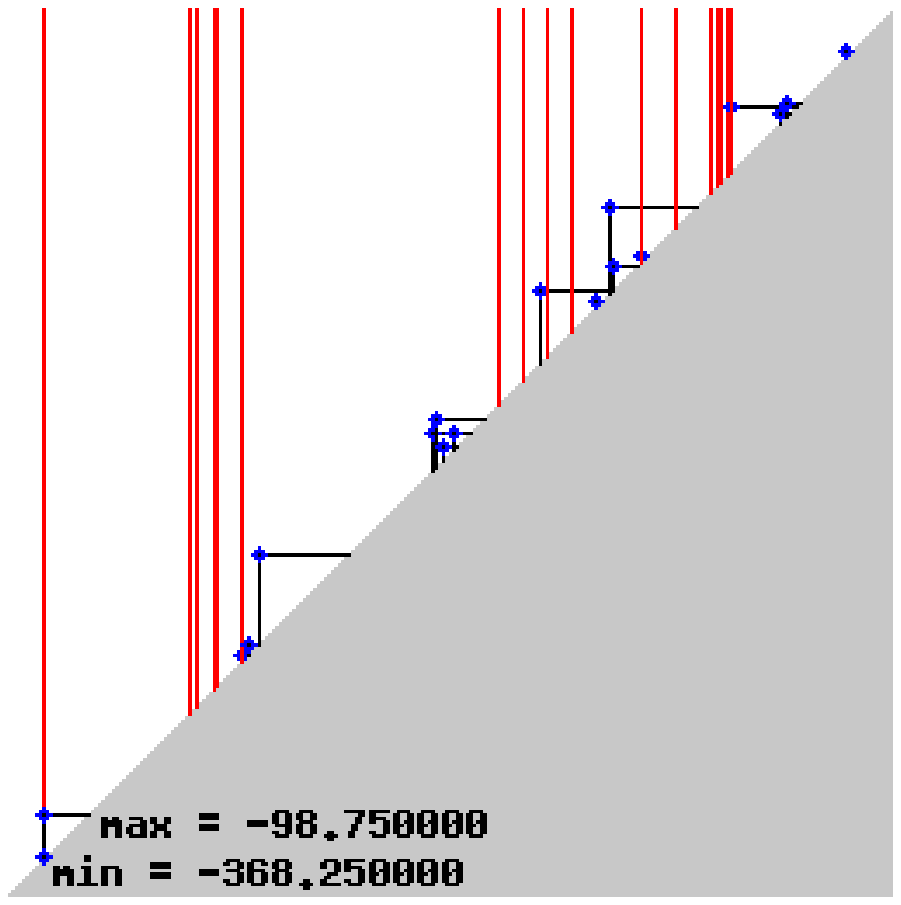} &
\includegraphics[width=0.15\linewidth]{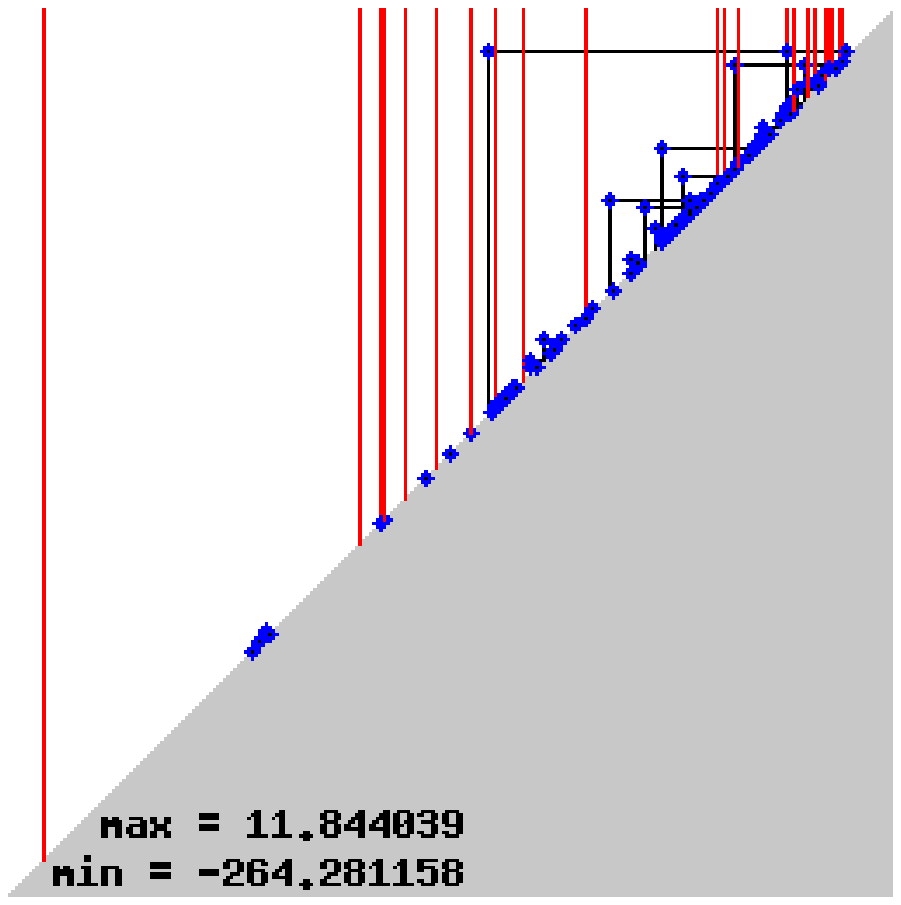} &
\includegraphics[width=0.15\linewidth]{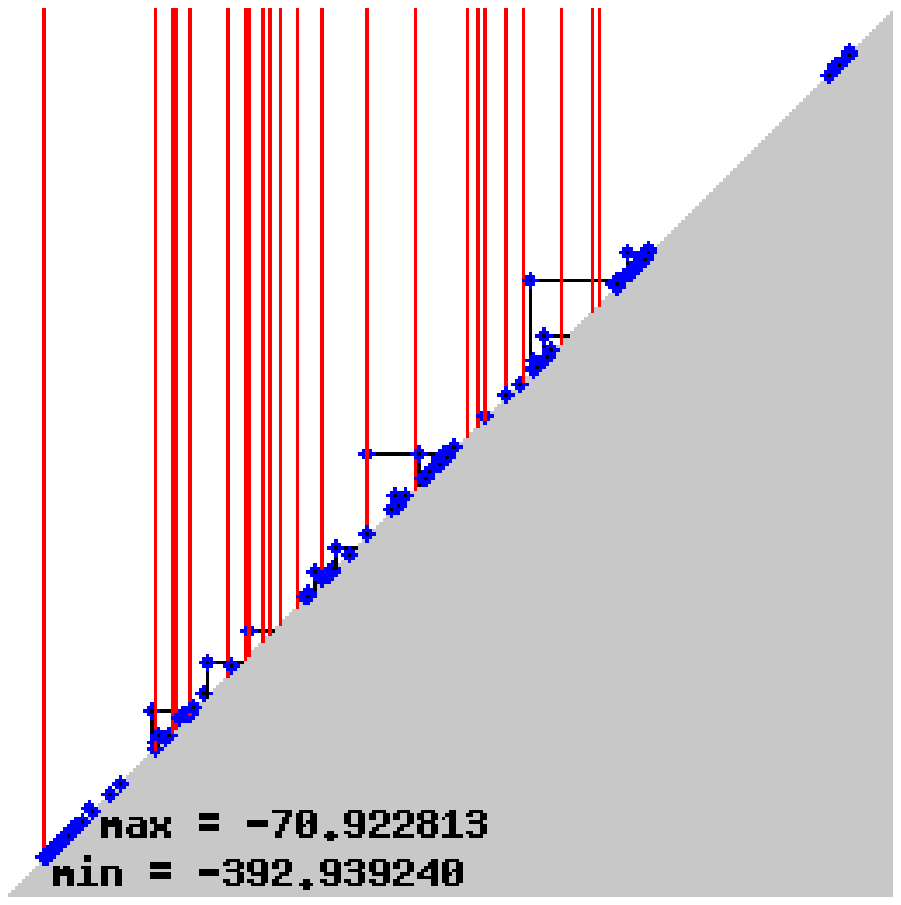} \\
\includegraphics[width=0.115\linewidth]{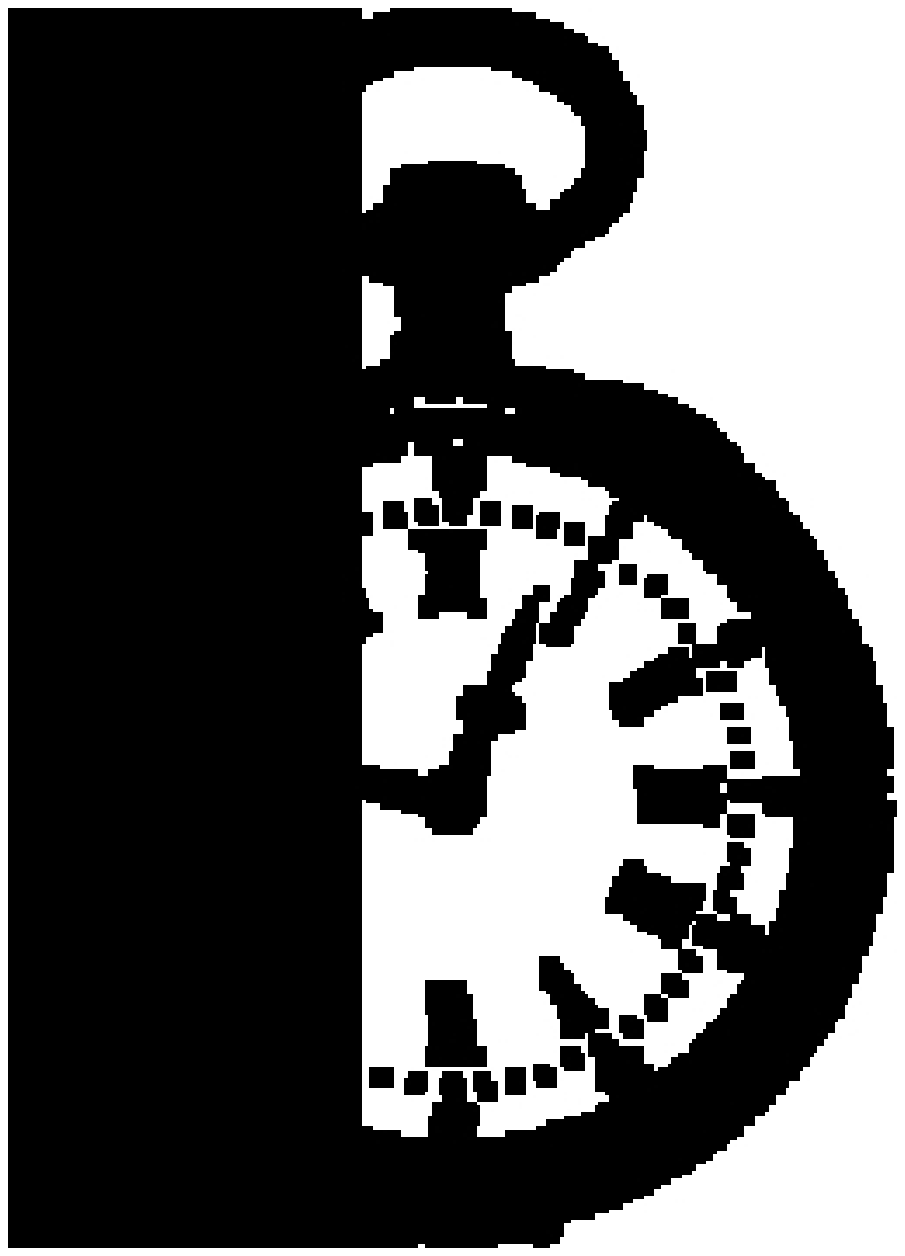} &
\includegraphics[width=0.15\linewidth]{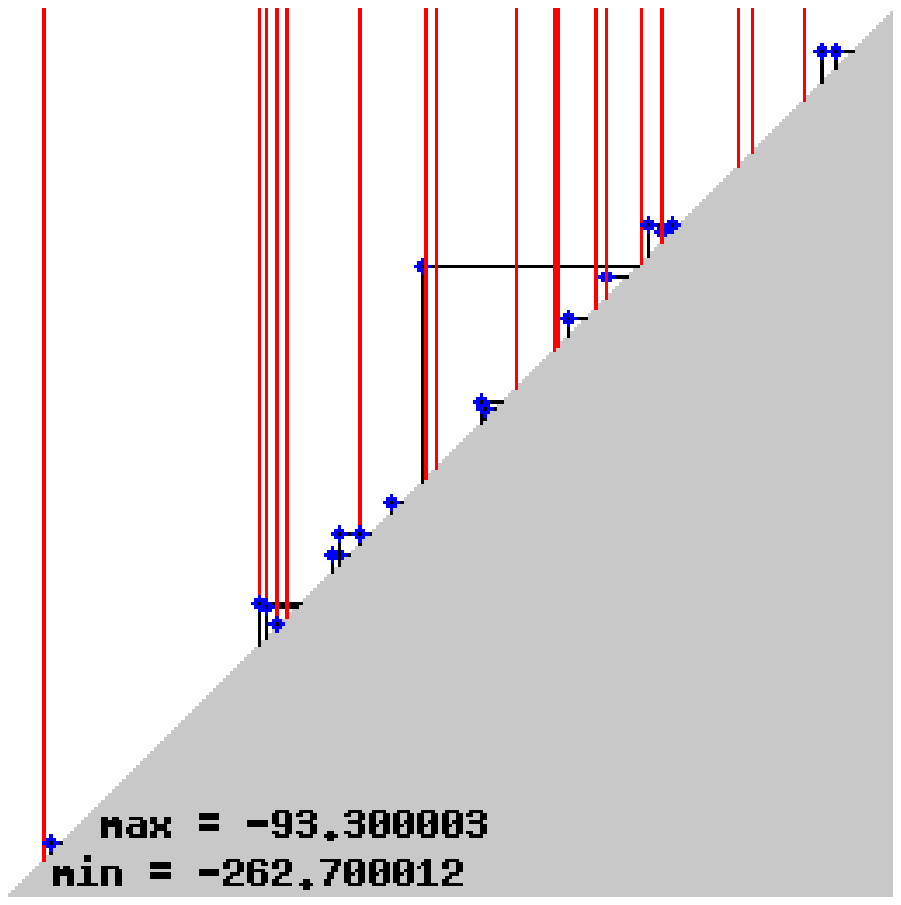} &
\includegraphics[width=0.15\linewidth]{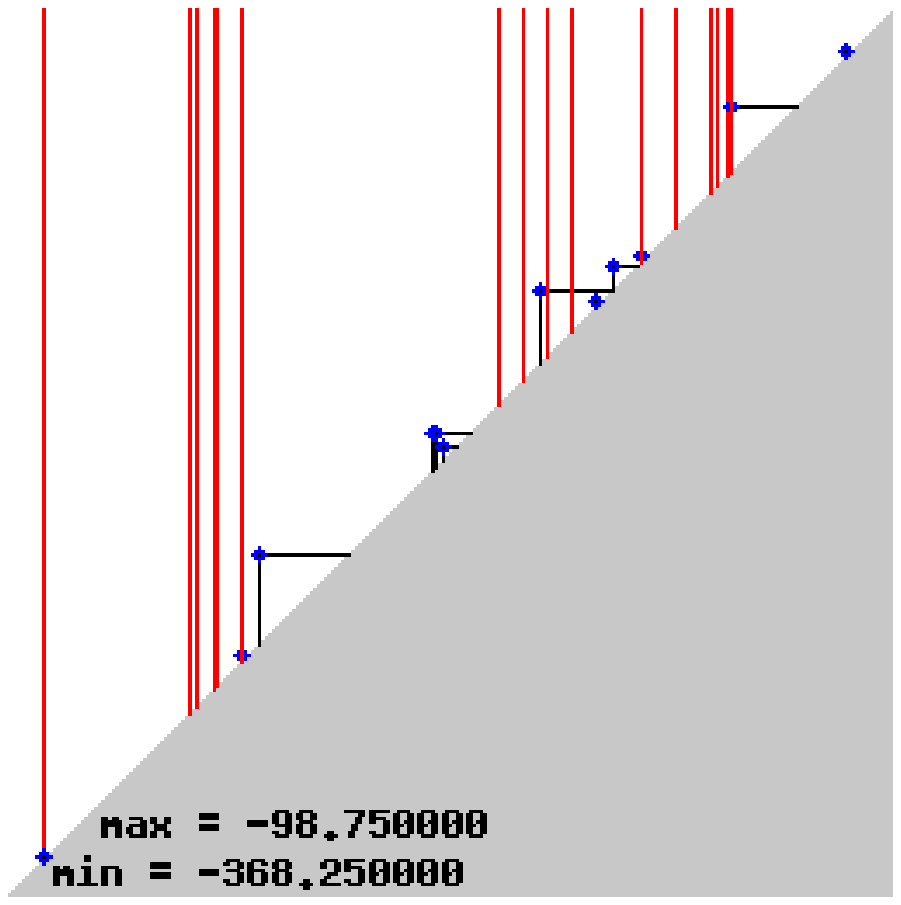} &
\includegraphics[width=0.15\linewidth]{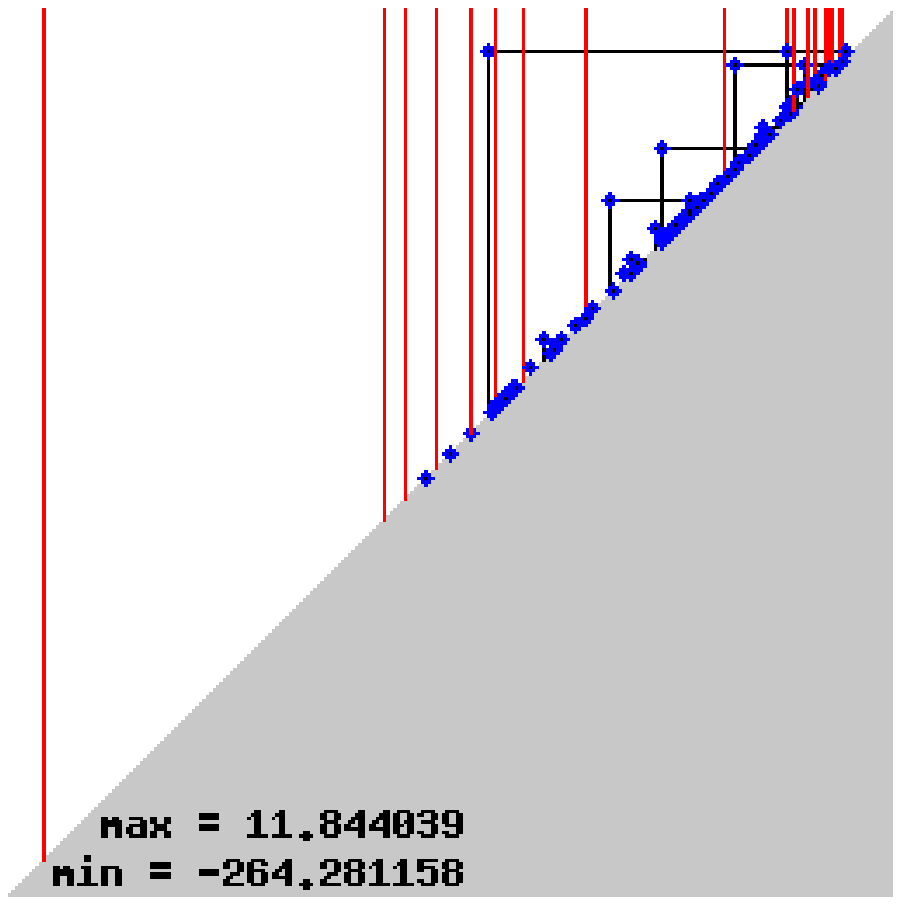} &
\includegraphics[width=0.15\linewidth]{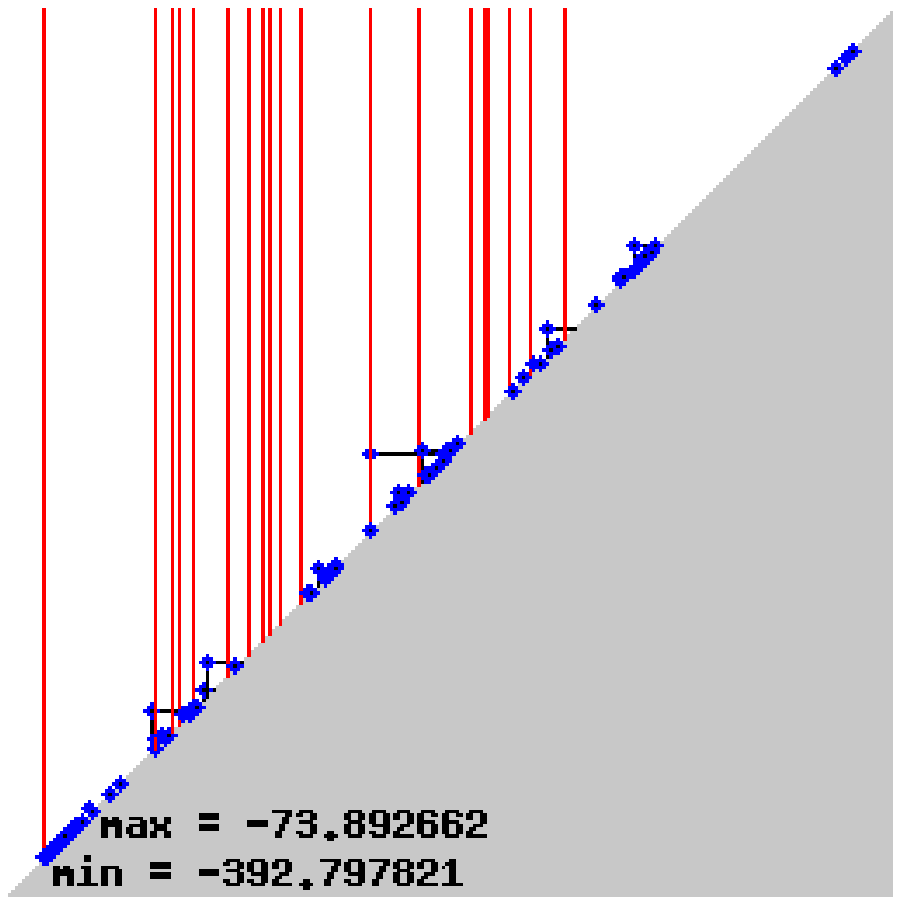} \\
\end{tabular}
\caption{\footnotesize{The first column: (row 1) original ``pocket
watch'' shape, (rows 2--4) occluded from top by $20\%$,  $30\%$,
$40\%$,  (row 5--7) occluded from left by $20\%$,  $30\%$,
$40\%$. From second column onwards: corresponding size functions
related to measuring functions defined as minus distances from
four lines rotated by $0$, $\pi/4$, $\pi/2$, $3\pi/4$, with
respect to the horizontal position.}} \label{pocketHV}
\end{table}
In spite of these topological changes, it can easily be seen that,
given a measuring function, even if the size function related to a
shape and the size function related to the occluded shape are
defined by  different cornerpoints, because of occlusion, a common
subset of these is present, making a partial matching possible
between them.

The second experiment is a recognition test for occluded shapes by
comparison of size functions. In this case the rectangular-shaped
occlusion is not visible (see Table \ref{W20}). When the original
shape is disconnected by the occlusion, we retain only the
connected component of greatest area. With reference to the
notation used in our theoretical setting, here we are considering
$X$ as the original shape, $A$ as the  the occluded shape, and $B$
as the invisible part of $X$.

\begin{table}[htbp]
\begin{tabular}{|c|c|c|c|c|c|}
\hline
\includegraphics[width=0.15\linewidth]{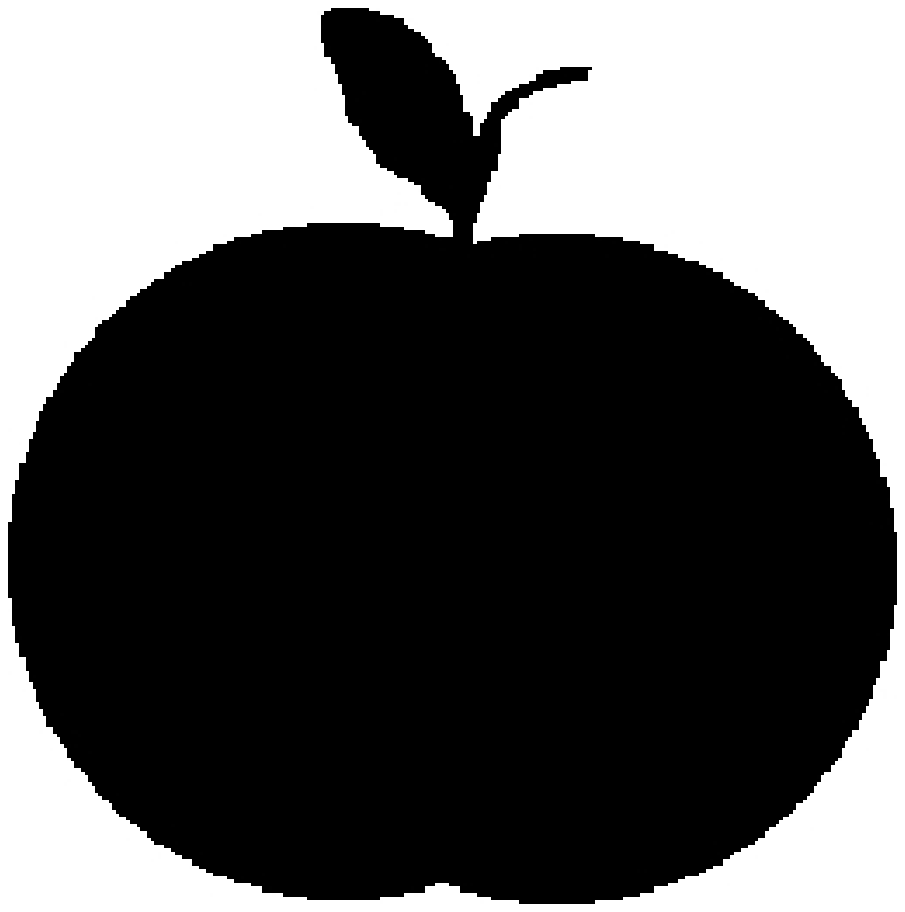} &
\includegraphics[width=0.15\linewidth]{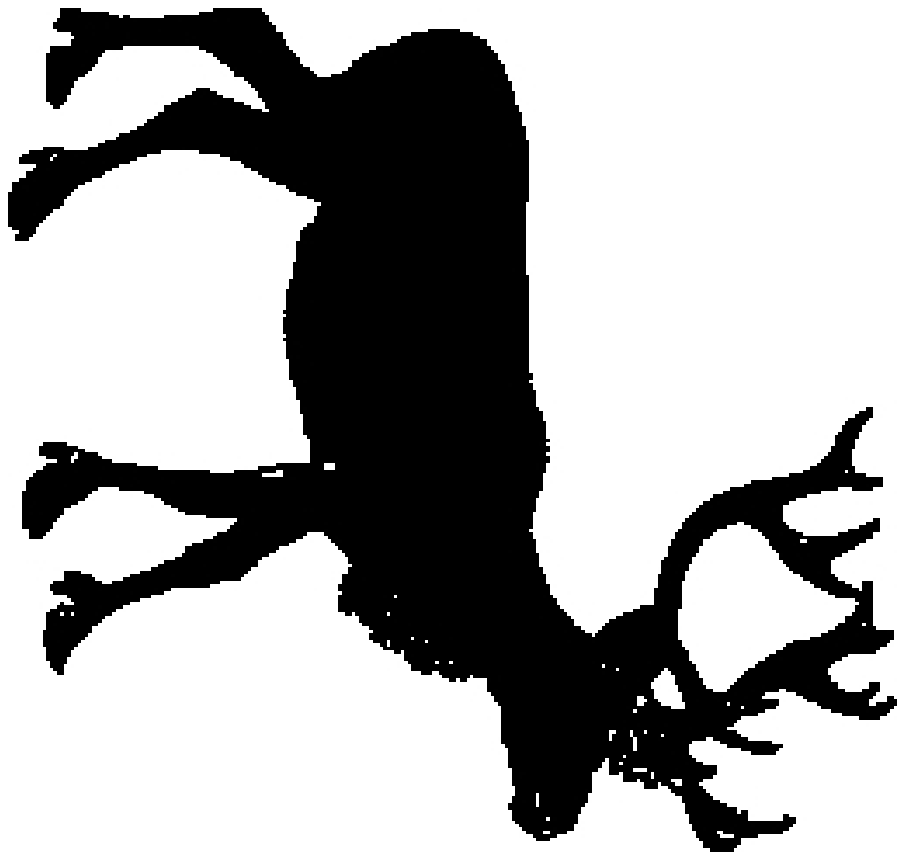} &
\includegraphics[width=0.13\linewidth]{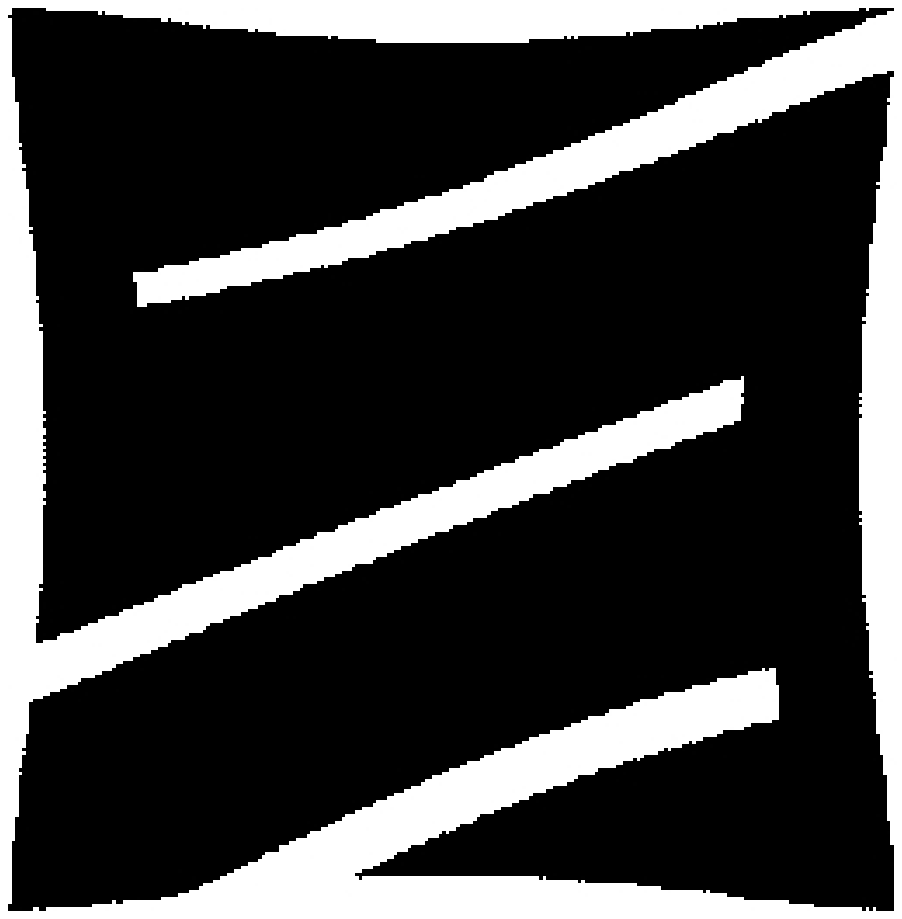} &
\includegraphics[width=0.08\linewidth]{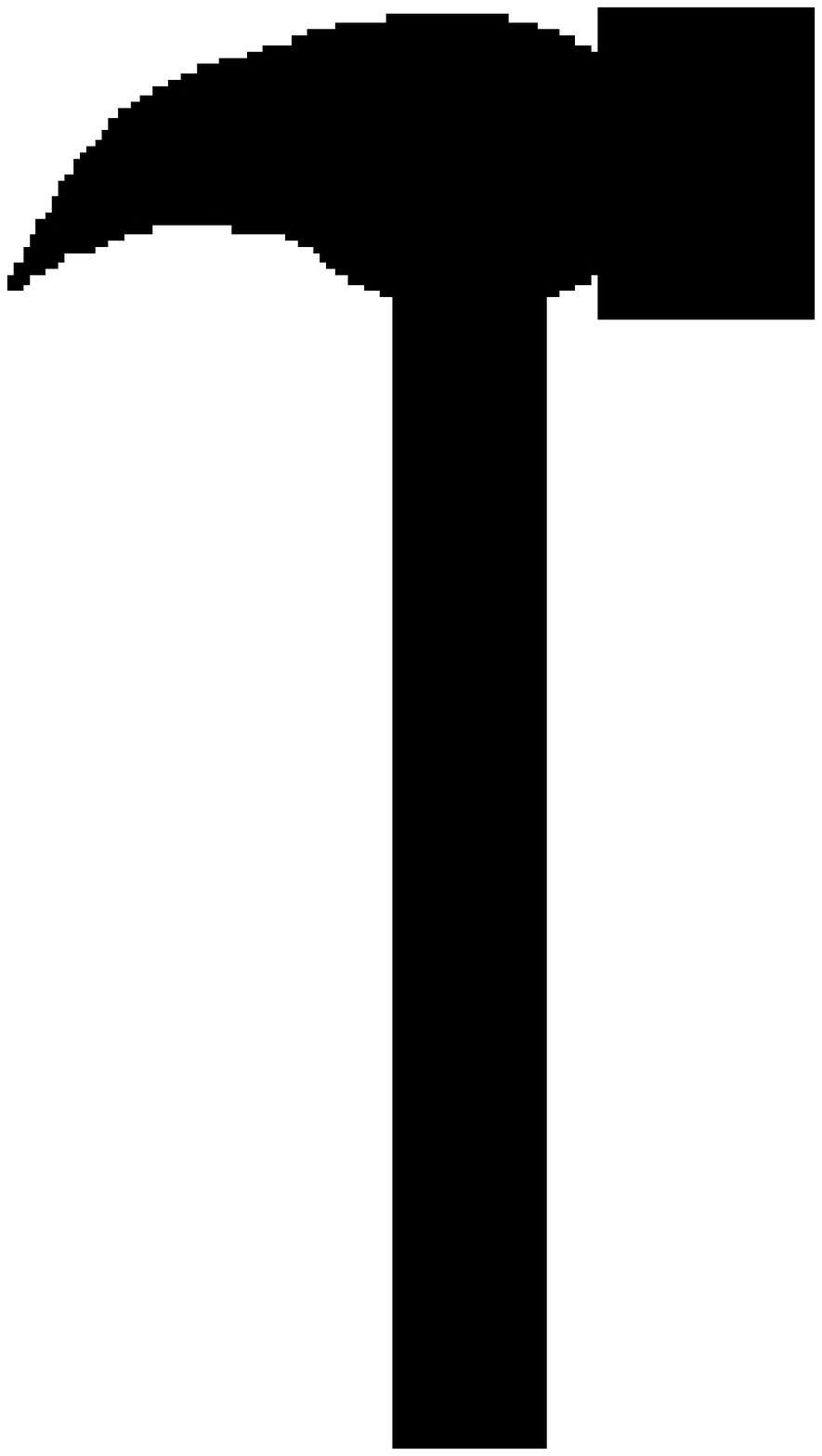} &
\includegraphics[width=0.12\linewidth]{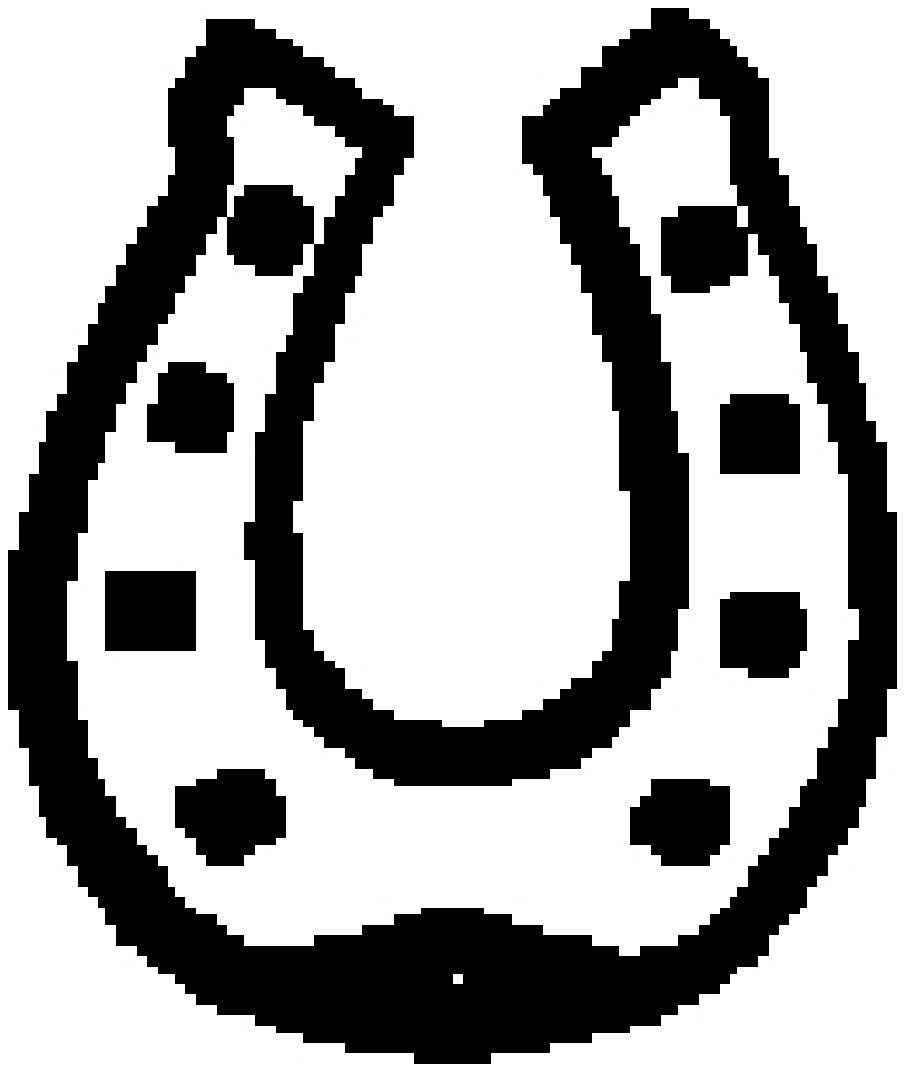} &
\includegraphics[width=0.15\linewidth]{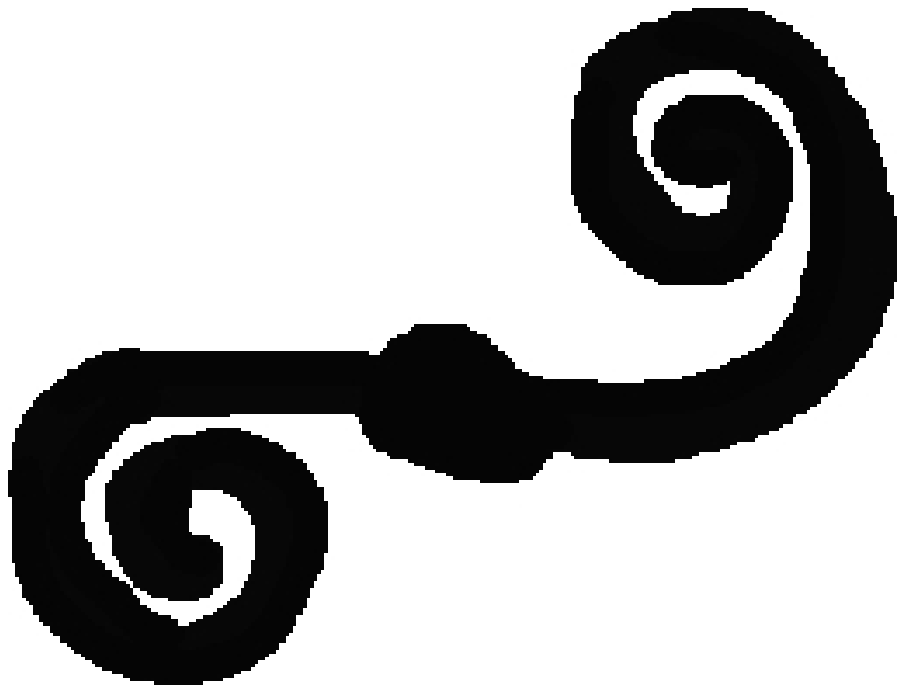} \\
\hline
\includegraphics[width=0.15\linewidth]{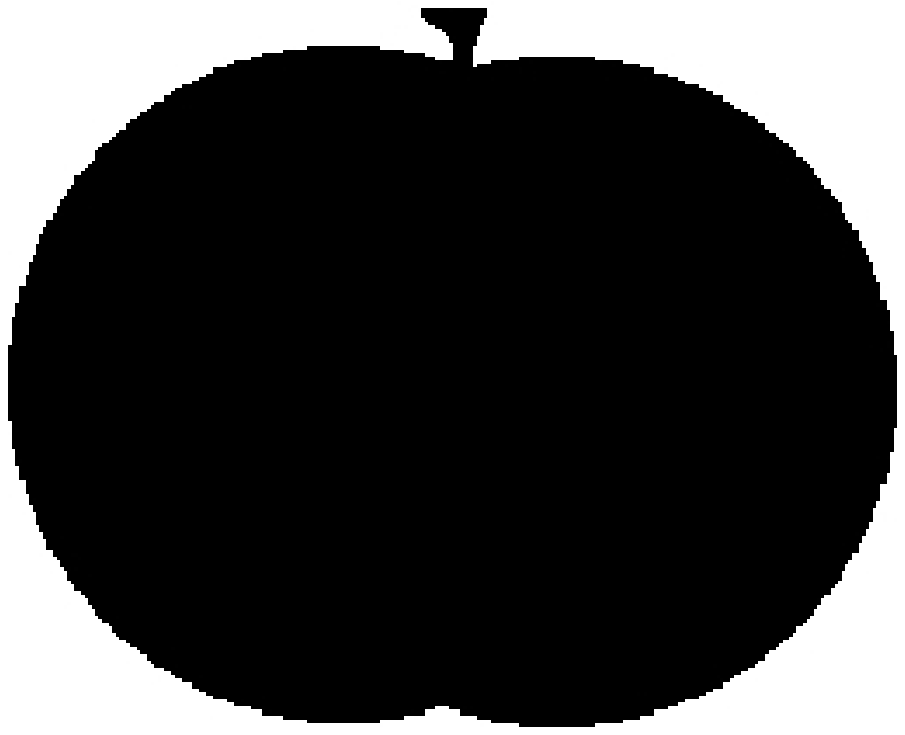} &
\includegraphics[width=0.15\linewidth]{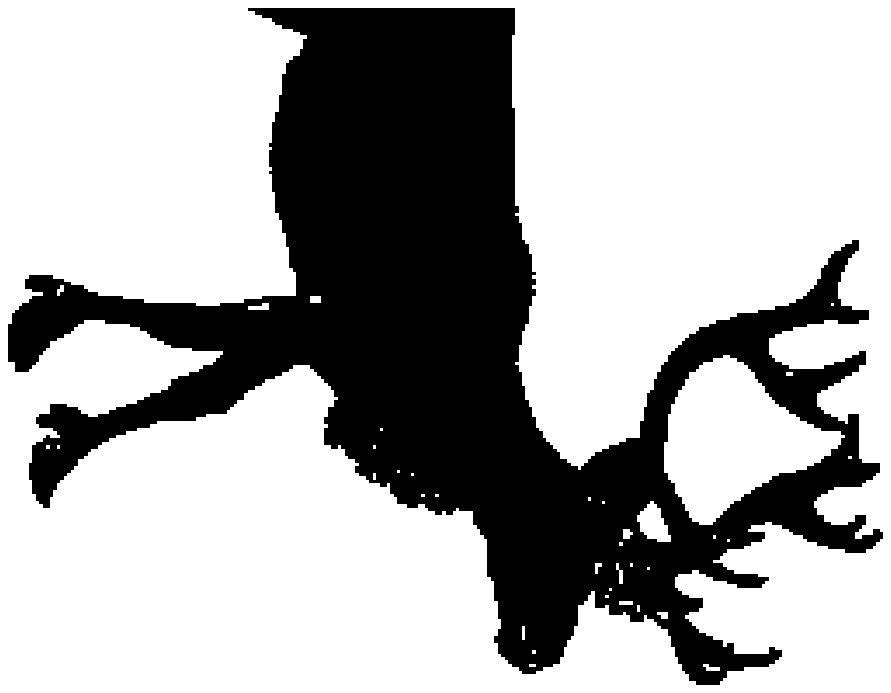} &
\includegraphics[width=0.13\linewidth]{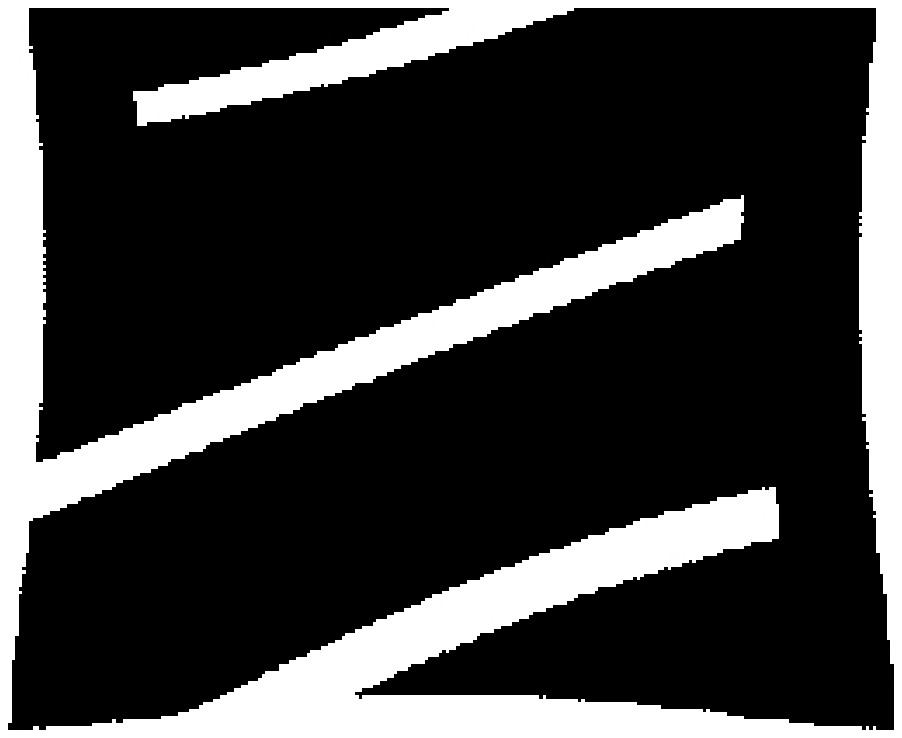} &
\includegraphics[width=0.08\linewidth]{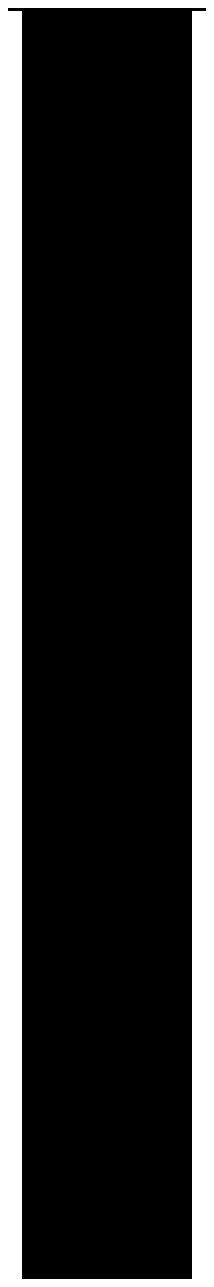} &
\includegraphics[width=0.12\linewidth]{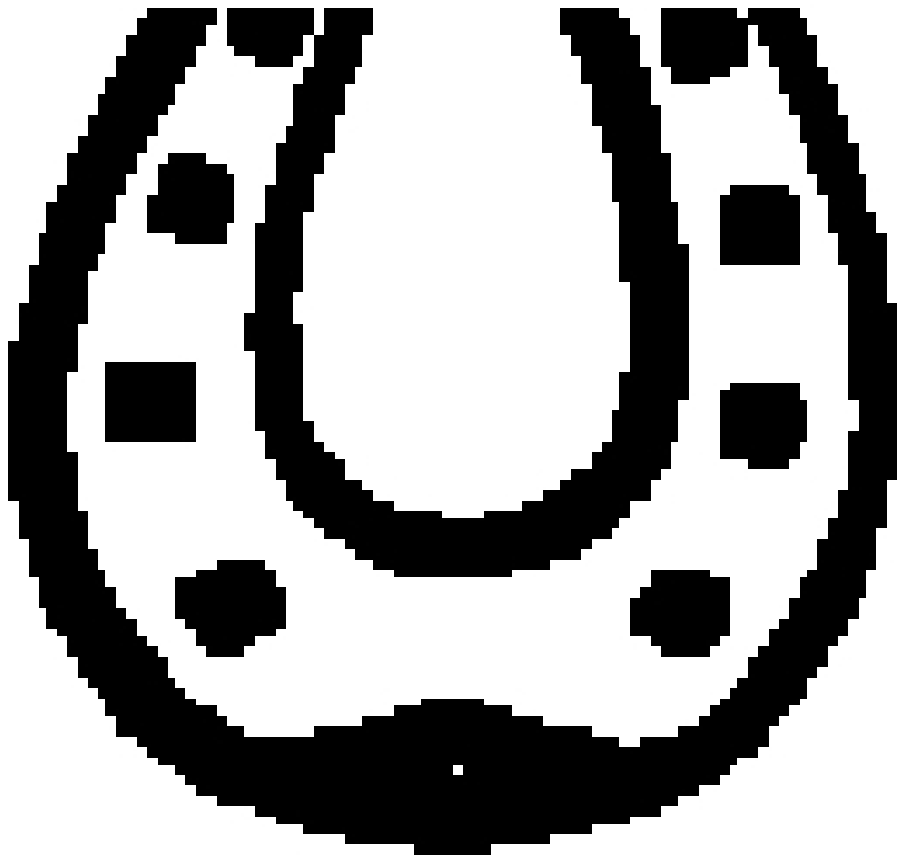} &
\includegraphics[width=0.15\linewidth]{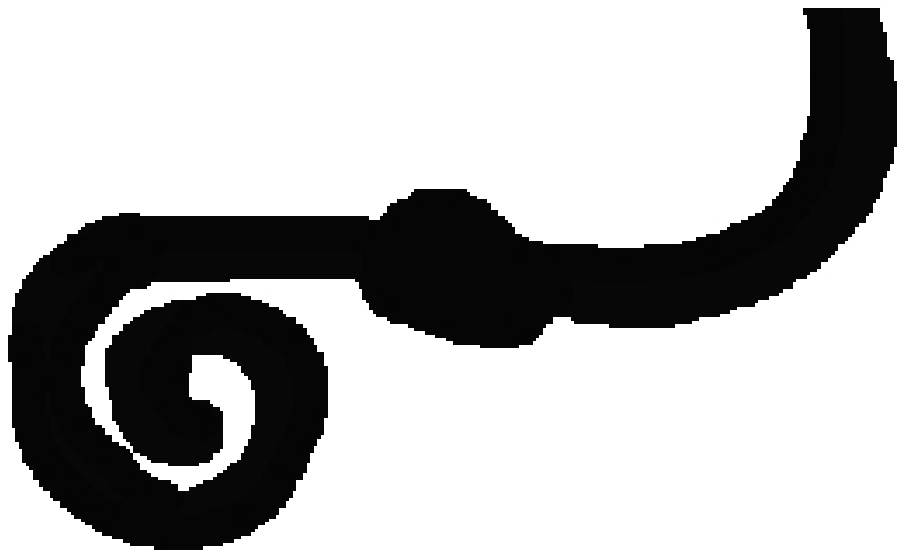} \\
\hline
\includegraphics[width=0.15\linewidth]{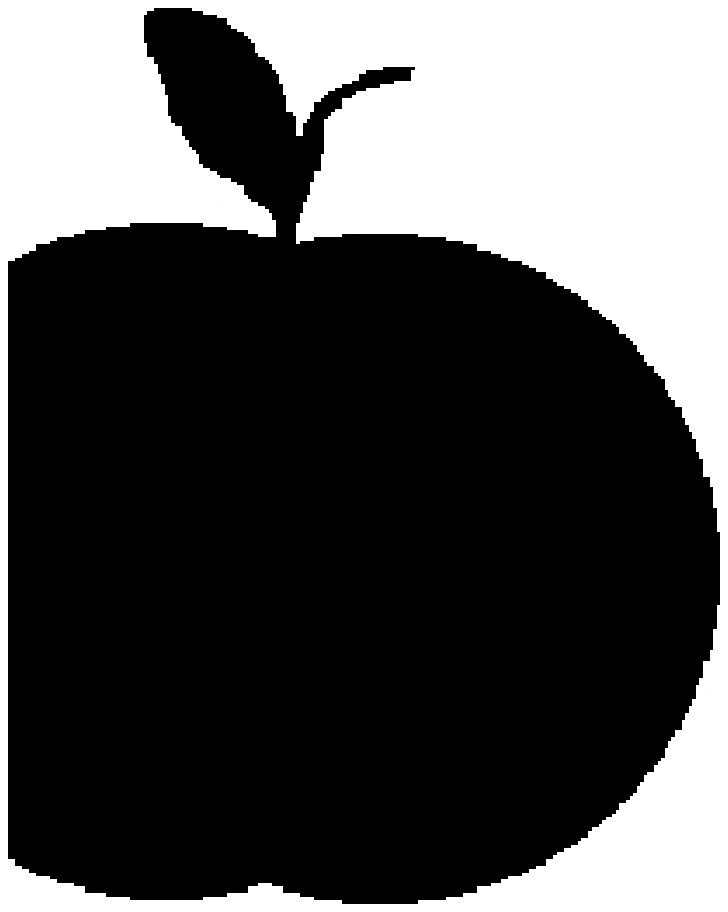} &
\includegraphics[width=0.15\linewidth]{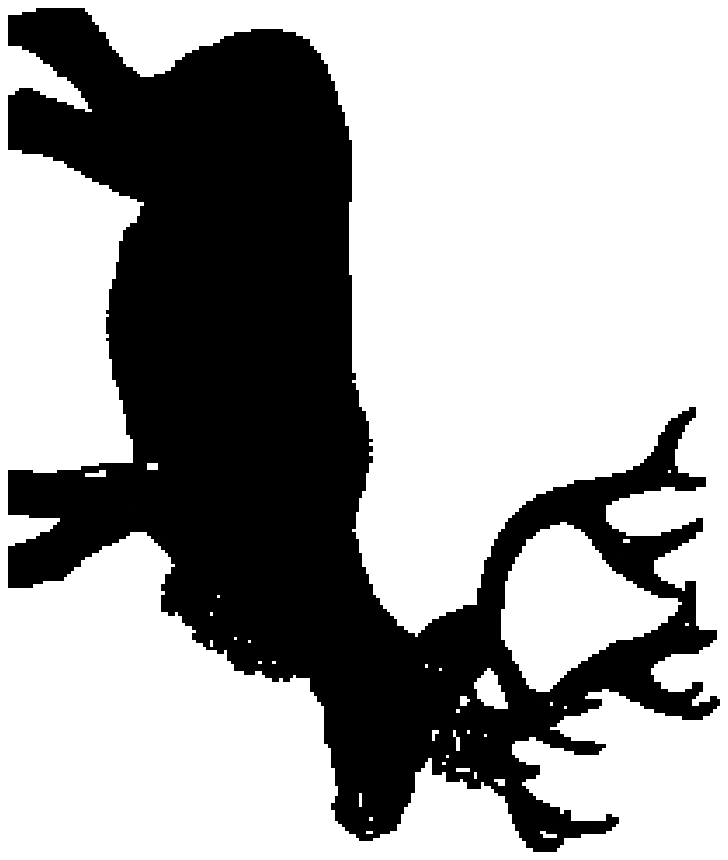} &
\includegraphics[width=0.13\linewidth]{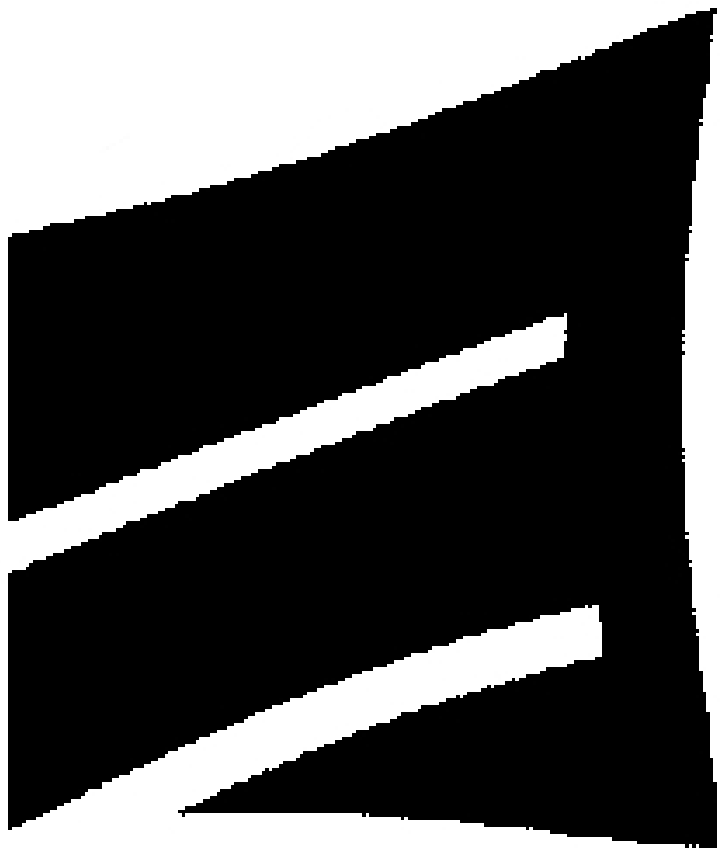} &
\includegraphics[width=0.08\linewidth]{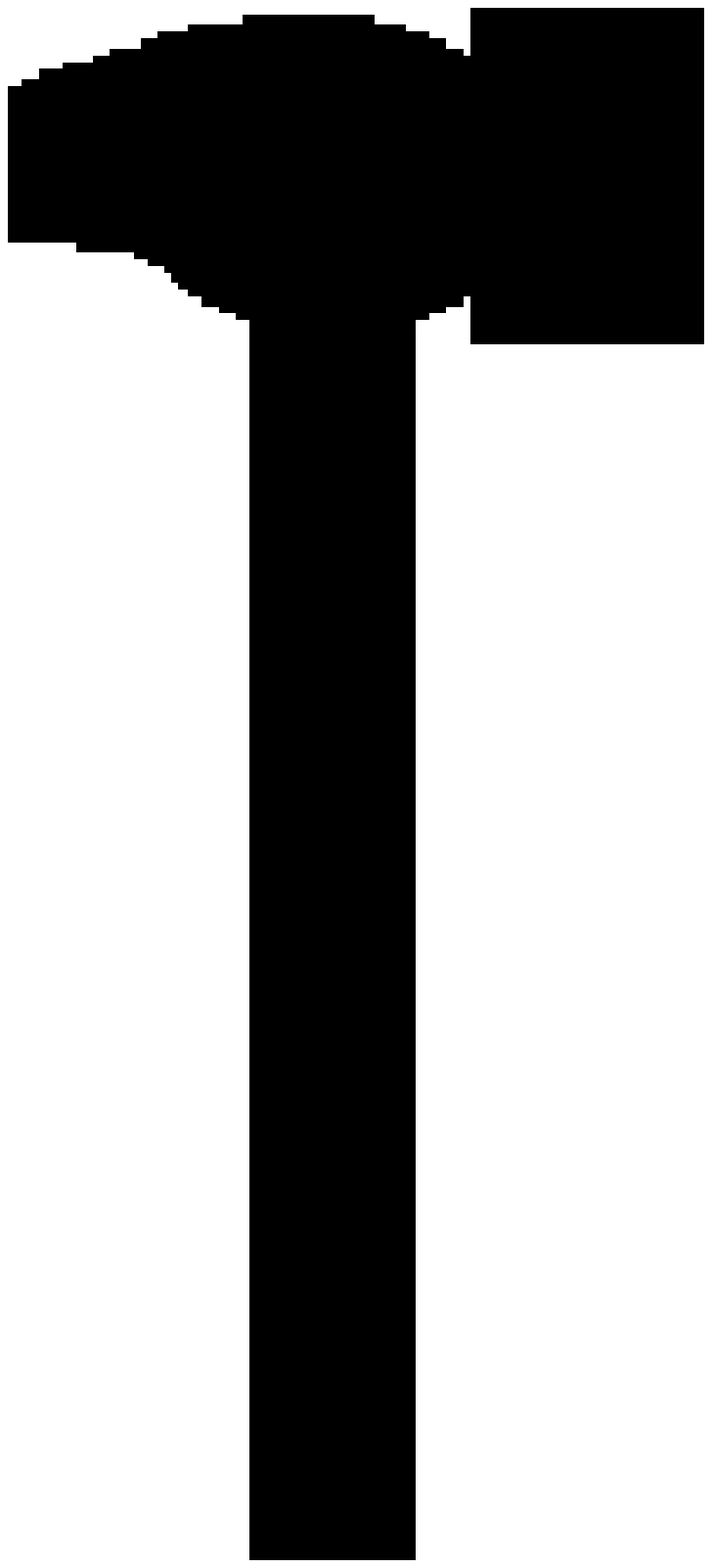} &
\includegraphics[width=0.12\linewidth]{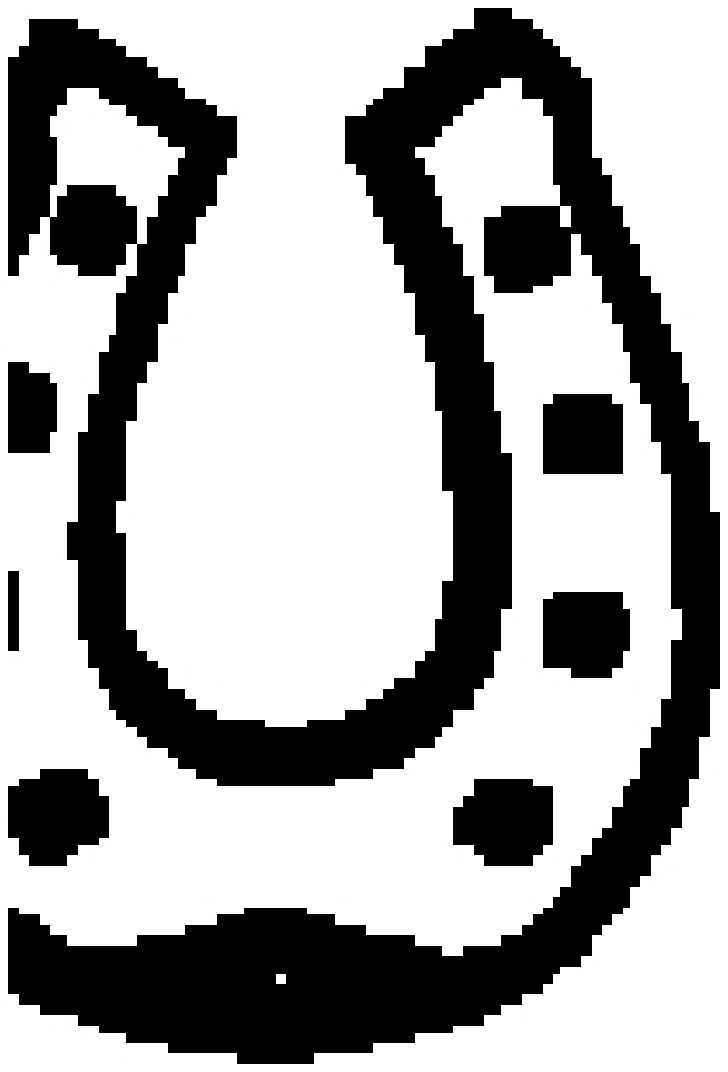} &
\includegraphics[width=0.15\linewidth]{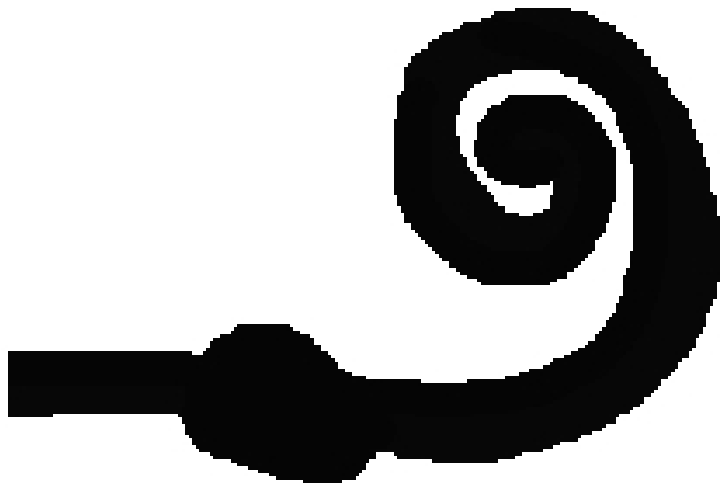} \\
\hline
\end{tabular}
\caption{\footnotesize{The first row: some instances from the
MPEG-7 dataset; the second and third rows: by $20\%$ occluded from
the top and from the left, respectively.}} \label{W20}
\end{table}

By varying the amount of occluded area, we compare  each occluded
shape with each of the $70$ original shapes. Comparison is
performed by calculating the sum of the eight Hausdorff distances
between the sets of cornerpoints for the   size functions
associated with the corresponding eight measuring functions. Then
each occluded shape is assigned to the class of its nearest
neighbor among the original shapes.

In Table \ref{graph}, two graphs describe the rate of correct
recognition in the presence of an increasing percentage of
occlusion. The leftmost graph is related to the occlusion from the
top, the rightmost one is related to the same occlusion from the
left.
\begin{table}[htbp]
\begin{tabular}{cc}
\includegraphics[width=0.5\linewidth]{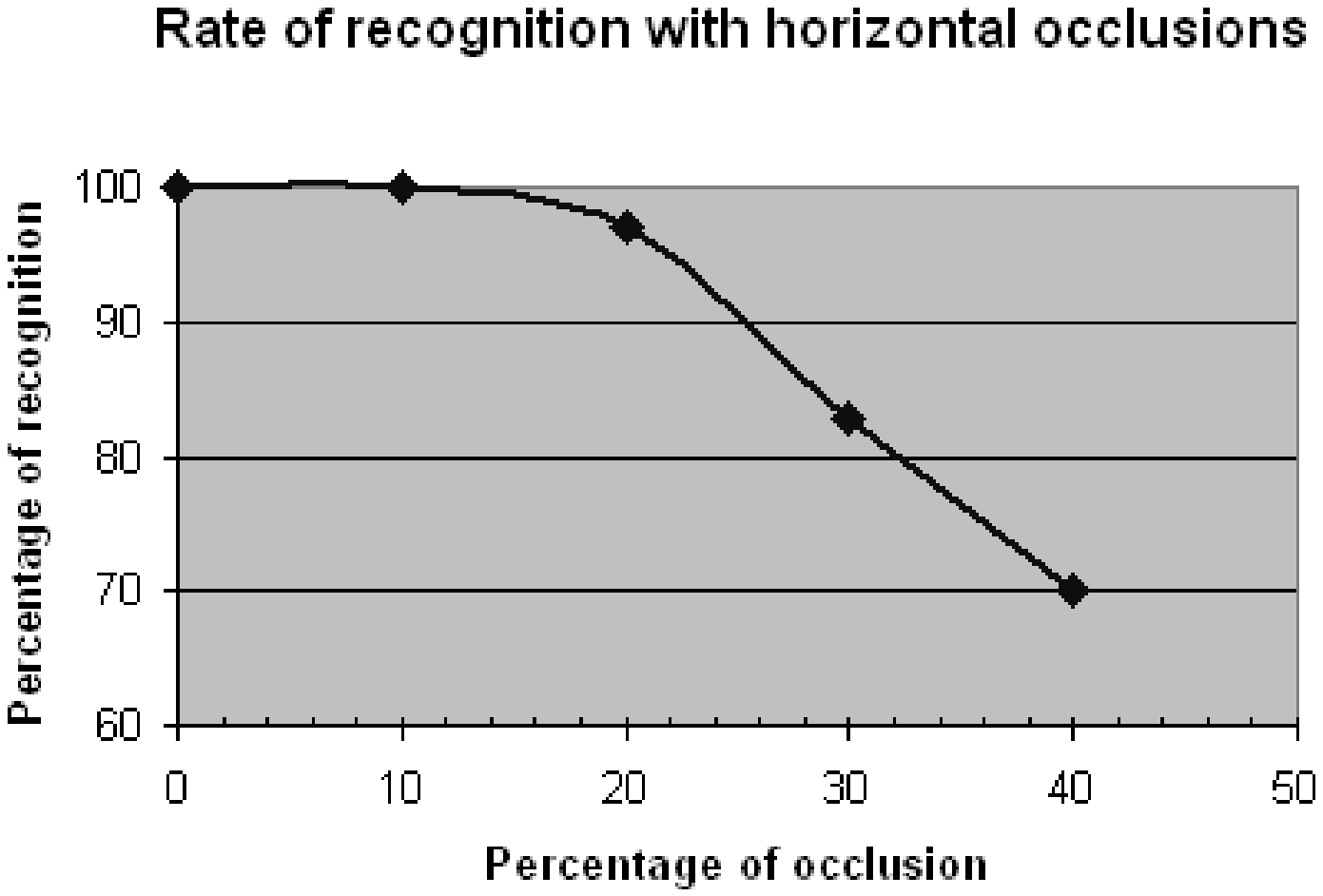} &
\includegraphics[width=0.5\linewidth]{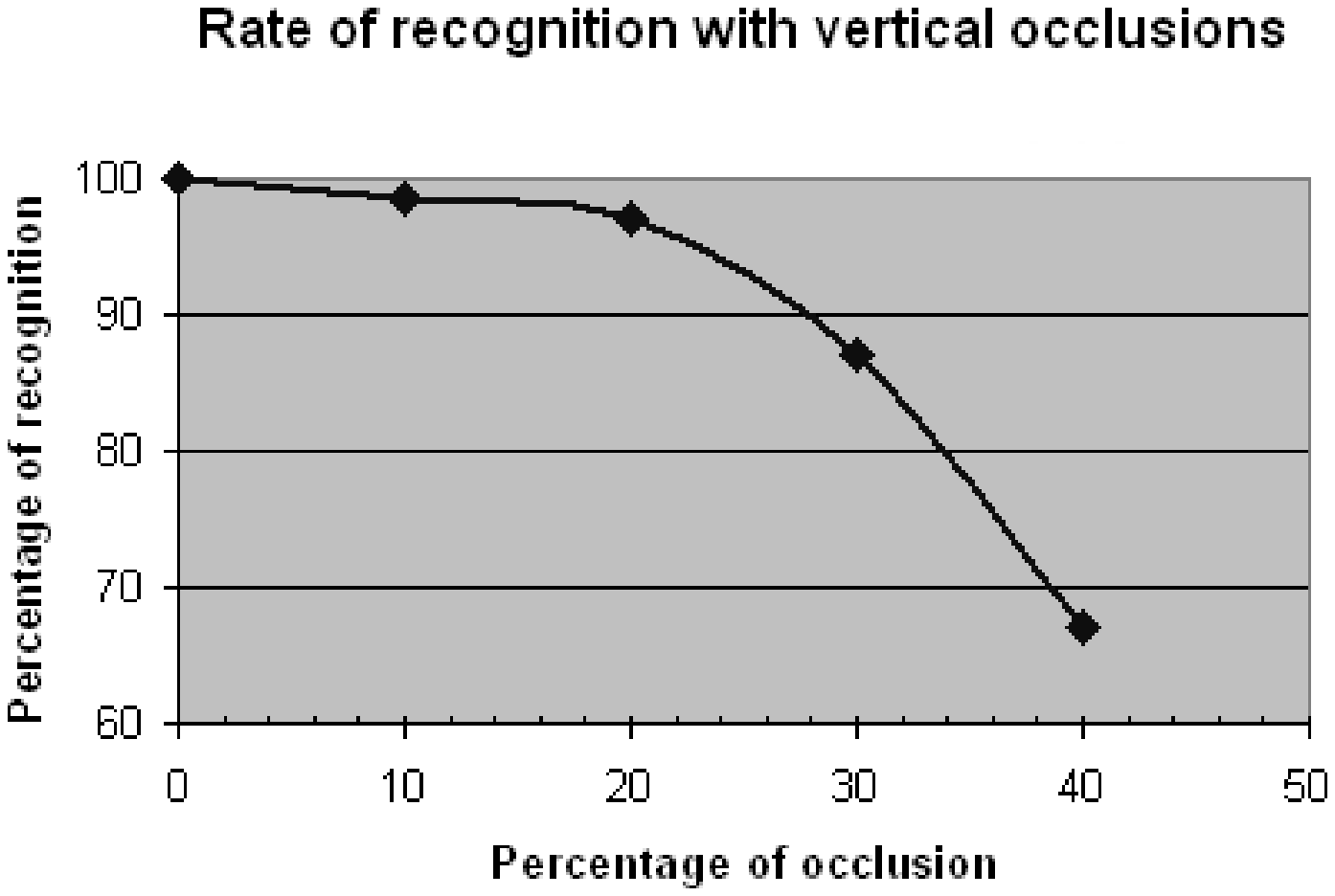} \\
\end{tabular}
\caption{\footnotesize{The leftmost (rightmost, respectively)
graph describes the recognition trend when the occluded area from
the top (left, respectively) increases.}} \label{graph}
\end{table}

\section{Discussion}

The main contribution of this paper is the analysis of the
behavior of size functions in the presence of occlusions.

Specifically we have proved that size functions assess a partial
matching between shapes by showing common subsets of cornerpoints.

Therefore, using size functions, recognition of a shape that is
partially occluded by a foreground shape becomes an easy task.
Indeed, recognition is achieved simply by associating with the
occluded shape that form whose size function presents the largest
common subset of cornerpoints (as in the experiment in Table
\ref{cprecognition}).

In practice, however, shapes may undergo other deformations due to
perspective, articulations, or noise, for instance. As a
consequence of these alterations, cornerpoints may move. Anyway,
small continuous changes in shape induce small displacements in
cornerpoint configuration.

It has to be expected that, when a shape is not only occluded but
also deformed, it will not be possible to find a common subset of
cornerpoints between the original shape and the occluded one,
since the deformation has slightly changed the cornerpoint
position. At the same time, however, the Hausdorff distance
between the size function of the original shape and the size
function of the occluded shape will not need to be small, because
it takes into account the total number of cornerpoints, including,
for example, those inherited from the occluding pattern (as in the
experiment in Table \ref{graph}).

The present work is a necessary  step, in view of the more general
goal of recognizing shapes in the presence of both occlusions and
deformations. The development  of a method to measure partial
matching of cornerpoints that do not exactly overlap but are
slightly shifted, would be desirable.

\renewcommand{\thesection}{A}
\setcounter{equation}{0}  

\section{Appendix}\label{cech}

{\bf \Cech homology.} In this description of \Cech homology
theory, we follow \cite{HY61}.

Given a compact Hausdorff space $X$, let $\Sigma (X)$ denote the
family of all finite coverings of $X$ by open sets. The coverings
in $\Sigma (X)$ will be denoted by script letters $\U$, $\V$,
$\ldots\ $ and the open sets in a covering by italic capitals $U$,
$V$, $\ldots\ $ An element $\U$ of $\Sigma (X)$  may be considered
as a simplicial complex if we define {\em vertex} to mean {\em
open set $U$ in $\U$} and agree that a subcollection $U_0,\ldots,
U_p$ of such vertices constitutes a $p$-simplex if and only if the
intersection $\cap_{i=0}^pU_i$ is not empty. The resulting complex
is known as the {\em nerve of the covering $\U$}.

Given a covering $\U$ in $\Sigma(X)$, we may define the chain
groups $C_p(\U,G)$, the cycle groups $Z_p(\U,G)$, the boundary
groups $B_p(\U,G)$, and the homology groups $H_p(\U,G)$.

The collection $\Sigma(X)$ of finite open coverings of a space $X$
may be partially ordered by refinement. A covering $\V$ refines
the covering $\U$, and we write $\U<\V$, if every element of $\V$
is contained in some element of $\U$. It turns out that
$\Sigma(X)$ is a direct set under refinement.

If $\U<\V$ in $\Sigma(X)$, then there is a simplicial mapping
$\pi_{\U\V}$ of $\V$ into $\U$ called a {\em projection}. This is
defined by taking $\pi_{\U\V}(V)$, $V\in\V$, to be any (fixed)
element $U$ of $\U$ such that $V$ is contained in $U$. There may
be  many projections of $\V$ into $\U$. Each projection
$\pi_{\U\V}$ induces a chain mapping of $C_p(\V,G)$ into
$C_p(\U,G)$, still denoted by $\pi_{\U\V}$, and this in turn
induces homomorphisms $_{*}\pi_{\U\V}$ of $H_p(\V,G)$ into
$H_p(\U,G)$. If $\U<\V$ in $\Sigma(X)$, then it can be proved that
any two projections of $\V$ into $\U$ induce the same homomorphism
of $H_p(\V,G)$ into $H_p(\U,G)$.

Now we are ready to define a \Cech cycle. A {\em $p$-dimensional
\Cech cycle} of the space $X$ is a collection $z_p=\{z_p(\U)\}$ of
$p$-cycles $z_p(\U)$, one for each and every cycle group
$Z_p(\U,G)$, $\U\in\Sigma(X)$, with the property that if $\U<\V$,
then $\pi_{\U\V}z_p(\V)$ is homologous to $z_p(\U)$. Each cycle
$z_p(\U)$ in the collection $z_p$ is called a {\em coordinate of
the \Cech cycle.} Hence a \Cech cycle has a coordinate on every
covering of the space $X$. The addition of \Cech cycles is defined
by setting $ \{z_p(\U)\}+\{z'_p(\U)\}=\{z_p(\U)+z'_p(\U)\}$. The
homology relation is defined as follows. A \Cech cycle
$z_p=\{z_p(\U)\}$ is homologous to zero (or is a {\em bounding
\Cech cycle}) if each coordinate $z_p(\U)$ is homologous to zero
on the covering $\U$, for all $\U$ in $\Sigma(X)$. Then two \Cech
cycles $z_p$ and $z'_p$ are {\em homologous \Cech cycles} if their
difference $z_p-z'_p$ is homologous to zero. This homology
relation is an equivalence relation. The corresponding equivalence
classes $[z_p]$ are the elements of the {\em $p$th \Cech homology
group $\check H(X,G)$}, where $[z_p]+[z'_p]=[z_p+z'_p]$.

Let us now see how continuous mappings between spaces induce
homomorphisms on \Cech homology groups. Let $f:X\rightarrow Y$ be
a continuous mapping of $X$ into $Y$, where both $X$ and $Y$ are
compact Hausdorff spaces. Then each open covering $\U\in
\Sigma(Y)$ can be associated with an open covering $f^{-1}(\U)\in
\Sigma(X)$. In particular, we may define a simplicial mapping
$f_\U$ of $f^{-1}(\U)$ into $\U$ by setting $f_\U(f^{-1}(U))=U$
for each non-empty set $f^{-1}(U)$, $U\in \U$. If $\U<\V$, then
the maps $f_\U$ and $f_\V$ commute with the projection of
$f^{-1}(\V)$ into $f^{-1}(\U)$ and the projection of $\V$ into
$\U$. Now we can define the {\em homomorphism induced by the
continuous mapping $f$} as the map $f_*:\check H_p(X,G)\rightarrow
\check H_p(Y,G)$ by setting, for every $z_p\in \check H_p(X,G)$,
$f_*(z_p)=\{f_\U(z_{p}(f^{-1}(\U))\}$.

It is also possible to define relative \Cech cycles in the
following way. If $A$ is a closed subset of $X$, we say that a
{\em simplex} $\langle U_0,\ldots ,U_p\rangle$ of $\U\in
\Sigma(X)$ is {\em on} $A$ if and only if the intersection
$\cap_{i=0}^pU_i$ meets $A$. The collection of all simplexes of
$\U$ on $A$ is a closed subcomplex $\U_A$ of $\U$. Therefore, we
may consider the relative simplicial groups $H_p(\U,\U_A, G)$ over
a coefficient group $G$. Since for $\V>\U$ in $\Sigma(X)$, the
projection $\pi_{\U\V}$ of $\V$ into $\U$ projects $\V_A$ into
$\U_A$, each projection $\pi_{\U\V}$ is a simplicial mapping of
the pair $(\V,\V_A)$ into the pair $(\U,\U_A)$. We may define a
{\em $p$-dimensional \Cech cycle} of the space $X$ {\em relative
to $A$} as a collection $z_p=\{z_p(\U)\}$ of $p$-chains $z_p(\U)$,
$\U\in \Sigma (X)$, with the property that $z_p(\U)$ is a
$p$-cycle on $\U$ relative to $\U_A$, and if $\U<\V$, then
$\pi_{\U\V}z_p(\V)$ is homologous to $z_p(\U)$ relative to $\U_A$.
Evidently, $\check H_p(X,\emptyset)= \check H_p(X)$ and $\check
H_p(X,X)= 0$, for each integer $p$.

\renewcommand{\thesection}{B}
\setcounter{equation}{0}  
\section{Appendix}\label{exactness}

{\bf Exactness axiom in \Cech homology and Mayer-Vietoris
sequence.}

\Cech homology theory has all the axioms of homology theories
except the exactness axiom. However, if some assumptions are made
on the considered spaces and coefficients, this axiom also holds.
Indeed, in \cite{EiSt52}, Chap. IX, Thm. 7.6 (see also
\cite{Ke61}), we read the following result concerning the sequence
of a pair $(X,A)$

\begin{eqnarray*}
\begin{array}{cccccccccccc}
\cdots\rightarrow\!\!\!&\!\!\!\check{H}_{p+1}(X,
A)\!\!\!&\!\!\!\stackrel{\partial}{\rightarrow}\!\!\!&\!\!\!\check{H}_p(A)\!\!\!&\!\!\!\stackrel{i_*}{\rightarrow}\!\!\!&\!\!\!\check{H}_p(X)\!\!\!&\!\!\!\stackrel{j_*}{\rightarrow}\!\!\!&\!\!\!\check{H}_{p}(X,
A)\!\!\!&\!\!\!\rightarrow\!\!\!&\!\!\!
\cdots\!\!\!\rightarrow\!\!\!&\!\!\!\check{H}_0(X,
A)\!\!\!\rightarrow\!\!\!&\!\!\!0
\end{array}
\end{eqnarray*}
which, in general, is only of order $2$ (this means that the
composition of any two successive homomorphisms of the sequence is
zero, i.e. $\I\subseteq \ker$).

\begin{thm}{\rm (\cite{EiSt52}, Chap. IX, Thm. 7.6)}\label{pair}
If $(X,A)$ is compact and $G$ is a vector space over a field, then
the homology sequence of the pair $(X,A)$ is exact.
\end{thm}

It follows that, if $(X,A)$ is compact and $G$ is a vector space
over a field, \Cech homology satisfies all the axioms of homology
theories, and therefore all the general theorems in Chap. I of
\cite{EiSt52} also hold for \Cech homology. In particular, using
\cite{EiSt52}, Chap. I, Thm. 15.3, we have the exactness of the
Mayer-Vietoris sequence in \Cech homology:

\begin{thm}\label{MV}
Let $(X,A, B)$ be a compact proper triad and $G$ be a vector space
over a field. The Mayer-Vietoris sequence of  $(X,A,B)$  with
$X=A\cup B$
\begin{eqnarray*}
\begin{array}{cccccccccccc}
\cdots\rightarrow\!\!\!&\!\!\!\check{H}_{p+1}(X)\!\!\!&\!\!\!\stackrel{\Delta}{\rightarrow}\!\!\!&\!\!\!\check{H}_p(A\cap
B)\!\!\!&\!\!\!\stackrel{\alpha}{\rightarrow}\!\!\!&\!\!\!\check{H}_p(A)\oplus
\check{H}_p(B)\!\!\!&\!\!\!\stackrel{\beta}{\rightarrow}\!\!\!&\!\!\!\check{H}_{p}(X)\!\!\!&\!\!\!\rightarrow\!\!\!&\!\!\!
\cdots\!\!\!\rightarrow\!\!\!&\!\!\!\check{H}_0(X)\!\!\!\rightarrow\!\!\!&\!\!\!0
\end{array}
\end{eqnarray*}
is exact.
\end{thm}

Concerning homomorphisms between Mayer-Vietoris sequences, from
\cite{EiSt52}, Chap. I, Thm. 15.4, we  deduce the following
result.

\begin{thm}\label{MVhom}
If  $(X,A, B)$ and $(Y,C,D)$ are  proper triads, $X=A\cup B$,
$Y=C\cup D$, and $f:(X,A,B)\rightarrow (Y,C,D)$ is a map of one
proper triad into another, then $f$ induces a homomorphism of the
Mayer-Vietoris sequence of  $(X, A, B)$ into that of  $(Y, C, D)$
such that commutativity holds in the diagram
\begin{eqnarray*}\label{MVcomm}
\begin{array}{cccccccccc}
\cdots\rightarrow\!\!\!&\!\!\!\check{H}_{p+1}(X)\!\!\!&\!\!\!{\rightarrow}\!\!\!&\!\!\!\check{H}_p(A\cap B)\!\!\!&\!\!\!{\rightarrow}\!\!\!&\!\!\!\check{H}_p(A)\oplus \check{H}_p(B)\!\!\!&\!\!\!{\rightarrow}\!\!\!&\!\!\!\check{H}_{p}(X)\!\!\!&\!\!\!\rightarrow\!\!\!&\!\!\! \cdots\\
&&&&&&&&&\\
&\downarrow &&\downarrow &&\downarrow &&\downarrow &&\\
&&&&&&&&&\\
 \cdots\rightarrow\!\!\!&\!\!\!\check{H}_{p+1}(Y)\!\!\!&\!\!\!
{\rightarrow}\!\!\!&\!\!\!\check{H}_p(C\cap D)\!\!\!&\!\!\!
{\rightarrow}\!\!\!&\!\!\!\check{H}_p(C)\oplus
\check{H}_p(D)\!\!\!&\!\!\!{\rightarrow}\!\!\!&\!\!\!\check{H}_p(Y)\!\!\!&\!\!\!\rightarrow\!\!\!&\!\!\!\cdots\\
&&&&&&&&&\\
\end{array}
\end{eqnarray*}
\end{thm}

A relative form of the Mayer-Vietoris sequence, different from the
one proposed in \cite{EiSt52}, is useful in the present paper. In
order to obtain this sequence, we can adapt the construction
explained in \cite{Ha02} to \Cech homology and obtain the
following result.

\begin{thm}\label{MVrel}
If $(X, A, B)$ and $(Y, C, D)$ are compact proper triads with
$X=A\cup B$, $Y=C\cup D$, $Y\subseteq X$, $C\subseteq A$,
$D\subseteq B$, then there is a relative Mayer-Vietoris sequence
of homology groups with coefficients in  a vector space $G$ over a
field
\begin{eqnarray*}
\begin{array}{cccccccccccc}
\cdots\rightarrow\!\!\!&\!\!\!\check{H}_{p+1}(X,Y)\!\!\!&\!\!\!{\rightarrow}\!\!\!&\!\!\!\check{H}_p(A\cap
B, C\cap D)\!\!\!&\!\!\!{\rightarrow}\!\!\!&\!\!\!\check{H}_p(A,
C)\oplus \check{H}_p(B,
D)\!\!\!&\!\!\!{\rightarrow}\!\!\!&\!\!\!\check{H}_{p}(X,
Y)\!\!\!&\!\!\!\rightarrow\!\!\!&\!\!\!
\cdots\!\!\!\rightarrow\!\!\!&\!\!\!\check{H}_0(X,
Y)\!\!\!\rightarrow\!\!\!&\!\!\!0
\end{array}
\end{eqnarray*}
that is exact.
\end{thm}

\begin{proof}
Given a covering $\mathcal U$ of $\Sigma(X)$, we may consider the
relative simplicial homology groups $H_p({\mathcal U},{\mathcal
U}_Y)$, $H_p({\mathcal U}_A,{\mathcal U}_C)$, $H_p({\mathcal
U}_B,{\mathcal U}_D)$, $H_p({\mathcal U}_{A\cap B},{\mathcal
U}_{C\cap D})$, for every $p \geq 0$.  For these groups the
relative Mayer-Vietoris sequence
\begin{eqnarray*}
\begin{array}{cccccccccc}
\cdots\rightarrow\!\!\!&\!\!\! H_{p+1}({\mathcal U},{\mathcal
U}_Y)        \!\!\!&\!\!\!{\rightarrow}\!\!\!&\!\!\!
H_p({\mathcal U}_{A\cap B},{\mathcal U}_{C\cap
D})\!\!\!&\!\!\!{\rightarrow}\!\!\!&\!\!\! H_p({\mathcal
U}_A,{\mathcal U}_C)         \oplus H_p({\mathcal U}_B,{\mathcal
U}_D)   \!\!\!&\!\!\!{\rightarrow}\!\!\!&\!\!\!    H_p({\mathcal
U},{\mathcal U}_Y)
\!\!\!&\!\!\!\rightarrow\!\!\!&\!\!\! \cdots
\end{array}
\end{eqnarray*}
 is exact (cf. \cite{Ha02}, page 152).

We now recall that the $p$th \Cech homology group of a pair of
spaces $(X,Y)$ over $G$ is the inverse limit of the system of
groups  $\{H_p({\mathcal U}, {\mathcal U}_Y, G), \pi_{{\mathcal
U}{\mathcal V}}\}$ defined on the direct set of all open coverings
of the pair $(X,Y)$  (cf. \cite{EiSt52}, Chap. IX, Thm. 3.2 and
Def. 3.3). The claim is proved recalling that, given an inverse
system of exact lower sequences, where all the terms of the
sequence belong to the category of vector spaces over a field, the
limit sequence is also exact (cf. \cite{EiSt52}, Chap. VIII, Thm.
5.7, and \cite{Ke61}).
\end{proof}
The following result, concerning homomorphisms of relative
Mayer-Vietoris exact sequences, holds. We omit the proof, which
can be obtained in a standard way.

\begin{thm}\label{MVrelhom}
If $(X, A, B)$, $(Y, C, D)$, $(X', A', B')$, $(Y', C', D')$ are
compact proper triads with $X=A\cup B$, $Y=C\cup D$, $Y\subseteq
X$, $C\subseteq A$, $D\subseteq B$, and $X'=A'\cup B'$, $Y'=C'\cup
D'$, $Y'\subseteq X'$, $C'\subseteq A'$, $D'\subseteq B'$, and
$f:X\rightarrow X'$ is a map such that $f(Y)\subseteq Y'$,
$f(A)\subseteq A'$, $f(B)\subseteq B'$, $f(C)\subseteq C'$,
$f(D)\subseteq D'$, then $f$ induces a homomorphism of the
relative Mayer-Vietoris sequences such that commutativity holds in
the diagram
\begin{eqnarray*}
\begin{array}{cccccccccc}
\cdots\rightarrow\!\!\!&\!\!\!\check{H}_{p+1}(X,Y)\!\!\!&\!\!\!{\rightarrow}
\!\!\!&\!\!\!\check{H}_p(A\cap B, C\cap
D)\!\!\!&\!\!\!{\rightarrow}\!\!\!&\!\!\!\check{H}_p(A, C)\oplus
\check{H}_p(B,
D)\!\!\!&\!\!\!{\rightarrow}\!\!\!&\!\!\!\check{H}_{p}(X,
Y)\!\!\!&\!\!\!\rightarrow\!\!\!&\!\!\! \cdots\\
&&&&&&&&&\\
&\downarrow &&\downarrow &&\downarrow &&\downarrow &&\\
&&&&&&&&&\\
\cdots\rightarrow\!\!\!&\!\!\!\check{H}_{p+1}(X',Y')\!\!\!&\!\!\!{\rightarrow}\!\!\!&\!\!\!\check{H}_p(A'\cap
B', C'\cap
D')\!\!\!&\!\!\!{\rightarrow}\!\!\!&\!\!\!\check{H}_p(A',
C')\oplus \check{H}_p(B',
D')\!\!\!&\!\!\!{\rightarrow}\!\!\!&\!\!\!\check{H}_{p}(X',
Y')\!\!\!&\!\!\!\rightarrow\!\!\!&\!\!\! \cdots\\
\end{array}
\end{eqnarray*}
\end{thm}

\subsection*{Acknowledgments}
We wish to thank  F. Cagliari, M. Grandis, R. Piccinini for their
helpful suggestions and P. Frosini for suggesting the example in
Figure \ref{patologicEx} (a). Thanks to A. Cerri and F. Medri for
their invaluable help with the software. Anyway, the authors are
solely responsible for any possible errors.

Finally, we wish to express our gratitude to M. Ferri and P. Frosini for
their indispensable support and friendship.

This work was partially performed within the activity of ARCES (University of Bologna).

\end{document}